\documentclass[10pt]{article}
\usepackage{amssymb,amsmath,amsthm,epsfig}
\usepackage{txfonts,mathrsfs}
\usepackage{latexsym, enumerate,cite}
\usepackage{eepic}
\usepackage{booktabs}
\usepackage{afterpage}
\usepackage{epic}
\usepackage{graphicx}
\usepackage{color}
\usepackage{ifpdf}
\usepackage{subfigure}
\usepackage{tikz}
\usepackage{dsfont}
\usepackage{multirow}
\usepackage{makecell}
\usepackage{algorithm}
\usepackage{bm}
\usepackage{multirow}
\usepackage{color}
\usepackage[colorlinks, linkcolor=blue,anchorcolor=blue,citecolor=blue,urlcolor=blue]{hyperref}
\usepackage[title]{appendix}
\usepackage{comment}
\graphicspath{{figs/}}

\numberwithin{equation}{section}

\theoremstyle{plain}
\newtheorem{lem}{Lemma}[section]
\newtheorem{thm}[lem]{Theorem}

\newtheorem{prop}{Proposition}[section]

\theoremstyle{definition}
\newtheorem{defn}{Definition}[section]

\newtheorem{rem}{Remark}[section]

\DeclareMathOperator{\diag}{diag}

\newcommand{\p}{\partial}
\newcommand{\ds}{\displaystyle}

\newcommand{\RR}{\mathbb{R}}
\newcommand{\rmr}{\mathrm{r}}

\newcommand{\nm}{\noalign{\smallskip}}
\newcommand{\Bx}{\mathbf{x}}
\newcommand{\By}{\mathbf{y}}

\renewcommand{\(}{\left(}
\renewcommand{\)}{\right)}

\def \e{\ensuremath{\mathrm{e}}}
\def \ii{\ensuremath{\mathrm{i}}}
\def \d{\ensuremath{\mathrm{d}}}
\begin{document}

\title{
	On subwavelength guided seismic waves in Matryoshka-type elastic media with ring defects
}

\author{
	{Lingzheng Kong} \thanks{School of Mathematics and Statistics, Central South University, Changsha, 410083, Hunan Province, China. Email:
		\href{mailto:math_klz@csu.edu.cn}{math\_klz@csu.edu.cn};\ \ \href{mailto:math_klz@163.com}{math\_klz@163.com}}
}

\date{}%
\maketitle
% ----------------------------------------------------------------
\begin{abstract}
Seismic surface waves, particularly low-frequency Rayleigh waves, are notoriously destructive and remain difficult to control using conventional engineering methods. Recent advances in elastic metamaterials have opened new avenues for manipulating and guiding seismic waves at subwavelength scales. In this work, we present a rigorous mathematical study of subwavelength resonances and seismic-wave localization in arbitrarily shaped Matryoshka-type elastic metamaterials, which are composed of nested high-contrast resonators made of materials much stiffer than the background medium. By employing the displacement-to-traction map and a variational formulation, we derive necessary and sufficient conditions that characterize subwavelength resonances. Using the Gohberg--Sigal theory and Puiseux series of multivalued algebraic functions, we establish the existence of subwavelength resonances and obtain their asymptotic characterization in terms of the eigensystem of the generalized stiffness tensor, which serves as the elastic analogue of the capacitance matrix in the celebrated Minnaert acoustic-cavitation systems. Furthermore, in the concentric spherical case, the stiffness tensor exhibits a block-diagonal structure separating translational and rotational components, with each block being tridiagonal and possessing positive, simple eigenvalues. Based on this matrix formulation, we rigorously demonstrate that, for structures with a sufficiently large number of layers, ring defects in layered subwavelength resonators can act as effective radially laminated seismic waveguides, supporting both wave localization and guided propagation along the defects. In particular, appropriately placed ring defects can induce eigenmodes that are exponentially localized at multiple defect sites simultaneously, with precisely quantified amplitude ratios. Finally, some numerical results are provided to corroborate the theoretical findings.

%nearly degenerate eigenfrequencies
%
%create nearly degenerate eigenfrequencies above the bulk spectrum, whose corresponding eigenmodes are  localized simultaneously at multiple defects.
\end{abstract}


%    35B30 	Dependence of solutions to PDEs on initial and/or boundary data and/or on parameters of PDEs
% 35B34 Resonance in context of PDEs
%35Q74 — PDEs in connection with mechanics of deformable solids
%74J20  	Wave scattering in solid mechanics
%74B05  	Classical linear elasticity

% ----------------------------------------------------------------
\section{Introduction}

\subsection{Background and motivation}
Millions of earthquakes occur worldwide each year, with more than a thousand reaching magnitudes of 5.0 or higher \cite{LOF}. When an earthquake occurs, seismic waves carrying vast amounts of energy are emitted in all directions. Upon reaching the Earth’s surface, this energy propagates along it in the form of surface waves. Among these, the Rayleigh wave, whose frequency typically lies in the low range of approximately 0.1–20 Hz \cite{Palermo_SDEE_18,Wu_IJMS2021}, is particularly destructive to building structures, as the resonant frequencies of buildings usually fall within this range \cite{Achaoui_EML2016}. Moreover, these low-frequency surface waves are  difficult to guide or attenuate using conventional engineering approaches. In recent years, however, the advent of metamaterials has opened new possibilities for manipulating and confining low-frequency Rayleigh waves.

Metamaterials, which are artificial composite structures not found in nature, can be engineered to guide or trap wave propagation. Periodically arranged  metamaterials that exhibit band gaps \cite{AFLYZ_JDE2017,LZScience}, frequency ranges where elastic waves are strongly attenuated or do not propagate, offer a novel approach to earthquake protection. Most seismic metamaterials developed to date exploit this band-gap property. The formation mechanisms of band gaps primarily include Bragg scattering and local resonance \cite{Achaoui_PRB2011,LZScience}. For metamaterials based on Bragg scattering, the wavelength associated with the band gap must be comparable to the lattice constant of the periodic array \cite{BJEG_PRL2014}. However, the wavelengths of seismic waves that cause structural damage to buildings are generally much longer. Consequently, constructing seismic metamaterials that rely on Bragg scattering for low-frequency applications would require unrealistically large structures \cite{MBKPD_PNAS2016,AFLYZ_JDE2017}, rendering such designs both economically and physically impractical.

In contrast, locally resonant metamaterials exhibit band gaps whose lattice constants are two orders of magnitude smaller than the relevant wavelength \cite{LZScience,Maldovan_Na2013}. These metamaterials behave as effective media with negative elastic parameters \cite{LZScience,LXarXiv}, and the induced local resonances enable deep-subwavelength wave control\cite{NaturePhysics}, thereby allowing the manipulation of low-frequency seismic wave propagation through microstructured designs.  
This subwavelength mechanism makes local resonance a practically feasible strategy for seismic applications. As a result, periodic arrays of elastic resonant units based on local resonance have been widely proposed. Colombi et al. \cite{Colombi_SR2016} observed strong attenuation of seismic surface waves between 30 and 45 Hz and between 90 and 110 Hz, when seismic waves propagated through a forest, and numerical simulations confirmed that the attenuation arose from the local resonance of the trees. Inspired by forest trees, numerous studies have examined periodic arrays of various elastic resonant units \cite{Wu_IJMS2021,Palermo_SDEE_18} leveraging local resonance mechanisms arising from impedance mismatches. These mismatches originate from the high contrast in physical parameters between soft and hard elastic materials.  The assumption of an infinitely periodic background permits the application of Floquet–Bloch theory \cite{AK_book2018,FBTheory} to characterize guided modes induced by single or linear defects \cite{Khelif_PRE_2004,Pennec_APL_2005}. Similar ideas have been extensively developed in other fields, notably in quantum mechanics for the Schr\"odinger operator \cite{WeinsteinPNAS2014,WeinsteinMAMS2017}, and in subwavelength physics for bubbly crystals perturbed by point \cite{ADH_SIMA2023,ADH_CMP2024} or line defects \cite{AHLMZ_ArXiv,AHYJEMS}. These studies have established a rigorous foundation for understanding wave localization in periodic systems, where a localized defect breaks translational symmetry, producing exponentially confined modes with eigenvalues within the band gaps of the unperturbed periodic operator. It is also worth noting that another class of  anomalous localized resonances can arise directly from materials with negative parameters \cite{DLbook2024,LLL_JMPA2018}. Both negative-index and high-contrast resonant components can act as contrast agents in various imaging techniques \cite{sini_ARMA2025,AZ_CMP2015} by using the field enhancement and focusing of subwavelength resonances.

In most studies on seismic metamaterials, periodic structures often consist of single-layered (homogeneous) resonators. Such configurations, however, typically exhibit narrow band gaps, limiting their effectiveness in engineering applications \cite{LPDV_PRE2007,Zeng_IJSS20202}. This limitation has motivated the development of metamaterials featuring broader or multiple band gaps. Among these, Matryoshka-type metamaterials have gained considerable attention as subwavelength resonators, owing to their high tunability and high-quality resonance characteristics. Experimental and numerical studies \cite{LPDV_PRE2007,Zeng_IJSS20202} have shown that the nested design of Matryoshka resonators enables multiple resonances, resulting in a greater number of band gaps. Nevertheless, the mathematical understanding of the origin of subwavelength resonance and the underlying mechanisms of seismic wave guiding in Matryoshka-type elastic media with general geometries remains limited, apart from the case of concentric radially layered structures which has been well illustrated and studied in \cite{DKLZ_JDE26,KZDF_JCP,DKLLZ_SAPM2025,DKLZ_ESIAM24} focusing on various resonance phenomena. This setting poses several new challenges. First, the lack of translational invariance precludes the application of Floquet--Bloch theory, necessitating alternative analytical tools for spectral analysis. Second, the doubly connected nested geometry introduces topological and analytical complexities. Finally, the generality of the geometric configuration enhances the applicability of the model to a broad class of physical systems.

\subsection{Main novelty and strategy} 

%This work is a nontrivial extension of the method proposed in \cite{DKZ_JLMS2026} for the acoustic scattering governed by the scalar Helmholtz equation to the elastic scattering problem of the Lam\'e system. The elastic wave equation is more challenging due to the coexistence of compressional and shear waves, which propagate at different speeds, and the Green function exhibits higher singularity compared with the Helmholtz case. Consequently, a more sophisticated analysis is required. The results obtained in this work offer a rigorous mathematical foundation for understanding the aforementioned physical phenomena unique to elastic waves. 
To establish our main conclusions, we first decouple the multiple elastic scattering problem in Matryoshka-type metamaterials with arbitrary shapes and high-contrast material parameters using the displacement-to-traction map. This leads to a variational characterization of subwavelength resonances (see Proposition \ref{TCR}). By applying layer potential techniques and introducing the notion of {\em stiffness} in the gap region of nested resonators (see Definition \ref{defnstiff}), we derive an asymptotic expansion of the variational characterization in terms of the {\em stiffness tensor} (see Lemma \ref{EST}, Proposition \ref{ImU}). The notion of stiffness was originally developed to model the spherical bonded rubber bush mountings in the setting of elastostatics \cite{HT_IJSS2005,Hill1975}, 
which serves as a natural analogue of the capacitance matrix 
%in electrostatics \cite{DHJE2011} and  
in the celebrated Minnaert acoustic-cavitation systems \cite{FA_SAM_2022,ADH_SIMA2023}. This expansion motivates the introduction of the scaled resonance variable $\hat{\omega} = \delta^{-\frac{1}{2}}\omega$, which allows the application of Gohberg--Sigal theory to establish the existence of subwavelength resonances (see Theorem \ref{existence_SWLR}). To derive the asymptotic behavior of the subwavelength resonances $\omega(\delta)$ as  the material contrast  $\delta\to 0$, we note that the standard scalar-case analysis \cite{DKZ_JLMS2026,DKLZ_JDE26,FCA_SIAP2023} determines the generalized capacitance eigenvalue problem and its first-order corrections. Extending this to Lam\'e systems is highly nontrivial, since eigenvalues of the generalized elastic stiffness problem may have multiplicities, splitting into distinct eigenvalues whose number equals the multiplicity \cite{FA_SAM_2022,Katpbook1995}. Using theories of symmetric polynomials and algebraic functions, we show that $\omega(\delta)$ are multivalued algebraic functions admitting Puiseux series expansions at $\delta\to0$ (see Theorem \ref{PE_SWLR}).

Having established the theoretical framework for subwavelength resonances, we proceed to examine the problem of subwavelength guided seismic waves.  We first provide an explicit calculation for concentric radially layered structures, generalizing results in \cite{DKLZ_JDE26,LZArxiv} and validating the general theory for subwavelength resonances in arbitrary-shaped Matryoshka metamaterials. In the concentric spherical case, the stiffness tensor has a block-diagonal structure separating translational and rotational components, with each block tridiagonal and possessing positive, simple eigenvalues. Using this matrix formulation, we study wave localization induced by a ring defect in finite-layered dimer-type subwavelength resonators, characterized by nested layers with alternating stiffness values (see Definition \ref{Ccap}). Applying properties of Chebyshev polynomials within the framework of \cite{DKZ_JLMS2026,ABDHKL_SAPM2024}, we prove the existence, uniqueness, and convergence of an eigenfrequency within the spectral gap, and demonstrate exponential localization of the corresponding eigenmode, referred to as a localized defect mode. Although our focus is on geophysical structures, the proposed metamaterials offer alternatives for vibration and noise attenuation \cite{LZScience}, as the spectral gap can be tailored by modulating geometric parameters, controlling its width or opening/closing. When the spectral gap closes, dimer-type resonators degenerate into monomer-type ones, raising the natural question: 

{\em When the spectral gap closes, by what mechanism can a localized defect mode be excited?}

\noindent
 To restore the wave localization, we introduce material‑parameter defects, which act as selective trapping mechanisms for localized modes.  Specifically,  we establish the existence and convergence of a number of eigenfrequencies above the bulk spectrum that equals the number of defects, together with exponential localization of the associated eigenmodes (see Theorem \ref{Propiff2}).
 By strategically placing multiple defects with appropriately chosen parameters, we can engineer nearly degenerate eigenfrequencies whose corresponding eigenmodes are simultaneously localized at multiple defect sites (see Theorem \ref{thm_multimode}). These theoretical results are in agreement with the numerical simulations of double defects reported by \cite{ADH_CMP2024}.
 In contrast, we derive refined asymptotic expansions for these defect eigenvalues, which capture not only their leading-order dependence on the defect parameters but also their fine structure induced by inter-defect interactions (see \eqref{MLexp}). Furthermore, we obtain a detailed characterization of the associated eigenvectors at the defect sites: their amplitudes follow explicit sinusoidal profiles with precisely quantified amplitude ratios, given by \eqref{amplitude_ratio}. These amplitude distributions depend sensitively on the spacing between defects, exhibiting a parity-dependent structure that determines the relative phase and magnitude of the modal components across different defect locations. This provides a complete description of both the spectral splitting and the spatial structure of the multimode localized states.
 Exploiting this phenomenon, multimode localized states can be realized in metamaterial architectures, enabling multi-channel filtering capabilities that are unattainable with conventional single-defect designs. In particular, different quasi-degenerate modes exhibit distinct amplitude distributions across defect sites, allowing selective excitation and routing of waves through multiple spatial or frequency channels. Moreover, when the defects are arranged in hierarchical configurations (see Remark \ref{rem53}), the interaction between different block of defects gives rise to a rich multi-scale structure of localized modes. This provides a systematic framework for designing broadband and multifunctional wave control devices through controlled superposition and coupling of defect-induced modes, and contributes to the mathematical understanding of guided wave phenomena in complex structured media beyond periodicity, and provides a rigorous foundation for designing radial laminate seismic waveguides with nested resonator defects and thin functional coatings.
 
It is worth emphasizing that, as a significant by-product of this study, we also derive a complete spectral classification of tridiagonal 1-Toeplitz matrices, in the presence of finitely many compact defects. To the best of our knowledge, this characterization is the first result of its kind and is of independent mathematical interest. Consequently, it can be naturally applied to a broad class of waveguide problems, in particular to the trapping and guiding of acoustic waves governed by high-contrast Helmholtz equations \cite{ABDHKL_SAPM2024,DKZ_JLMS2026,ADH_SIMA2023,ADH_CMP2024}. Moreover, even in this acoustic setting, our results provide new insights, including a refined description of quasi-degenerate eigenvalues and their associated modal structures.

\subsection{Outlines}

The remainder of the paper is organized as follows.
In Section \ref{section2}, we introduce the problem setting for subwavelength resonances in a Matryoshka-type elastic medium with arbitrarily shaped resonators and review preliminary results on boundary layer potentials.
In Section \ref{D2Nmap}, we derive a new formula characterizing subwavelength resonances in Matryoshka-type elastic metamaterials of general geometry.
Section \ref{concentric_radial} provides explicit calculations for concentric radially layered structures, which serve to validate the general theory and formulas.
In Section \ref{sec5}, we investigate seismic-wave attenuation and localization induced by ring defects in finite-layered subwavelength resonators.
Finally, concluding remarks are presented in Section \ref{sec6}.

\section{Subwavelength resonances in Matryoshka-type elastic media}\label{section2}
\subsection{Problem setting}
We consider subwavelength resonances in a Matryoshka-type elastic medium. The structure is assumed to consist of two types of materials: the host matrix (soft) material and the high-contrast (hard) material, as illustrated in Figure~\ref{MLHCCB}.
We denote by
\[
D = \bigcup_{j=1}^{N}D_{j}
\]
the entire resonator formed by the nested hard materials, where each $D_j$ is the  bounded doubly-connected domain lying between the interior boundary $\Gamma_j^-$ and the exterior boundary $\Gamma_j^+$, $j=1,2,\ldots,N$. 
The boundaries are nested in such a way that $\Gamma_{j}^-$ encloses $\Gamma_{j+1}^+$ for $j=1,2,\ldots,N-1$.
Let $D_j'$ denote the gap region between $\Gamma_j^-$ and $\Gamma_{j+1}^+$, for $j=1,2,\ldots,N-1$. In addition, let $D_0'$ be the unbounded exterior region outside $\Gamma_{1}^+$, and $D_N'$ the bounded interior region inside $\Gamma_{N}^-$.  Then the host matrix material occupies
\[
\RR^3\setminus{\overline{D}} = \bigcup_{j=0}^{N} D'_j.
\]
The material configuration is characterized by the density $\rho(\Bx)$ and the Lamé parameters $(\lambda(\Bx),\mu(\Bx))$, defined by
 \begin{equation}\label{nestedcomplement}
 \rho(\Bx)
 =\left\{
 \begin{array}{ll}
 \rho_{\rmr,j}, & \Bx\in D_j, \; j=1,2,\ldots,N,\\
 \rho, & \Bx\in D_j', \; j=0,1,\ldots,N,
 \end{array}
 \right.\;\; \mbox{ and }\;\;(\lambda(\Bx),\mu(\Bx))
 =\left\{
 \begin{array}{ll}
 (\lambda_\rmr,\mu_\rmr), &\Bx\in D_j, \; j=1,2,\ldots,N,\\
 (\lambda,\mu), & \Bx\in D_j', \; j=0,1,\ldots,N,
 \end{array}
 \right.
 \end{equation}
 where $\rho(\Bx)>0$, and the Lam\'{e} parameters satisfy the following strong convexity conditions:
 \begin{equation}\label{strong_convexity}
  \mu(\Bx)>0 \mbox{ and } 3\lambda(\Bx)+2\mu(\Bx)>0.
 \end{equation}

 \begin{figure}[H]
 	\centering
 	\includegraphics[scale=0.09]{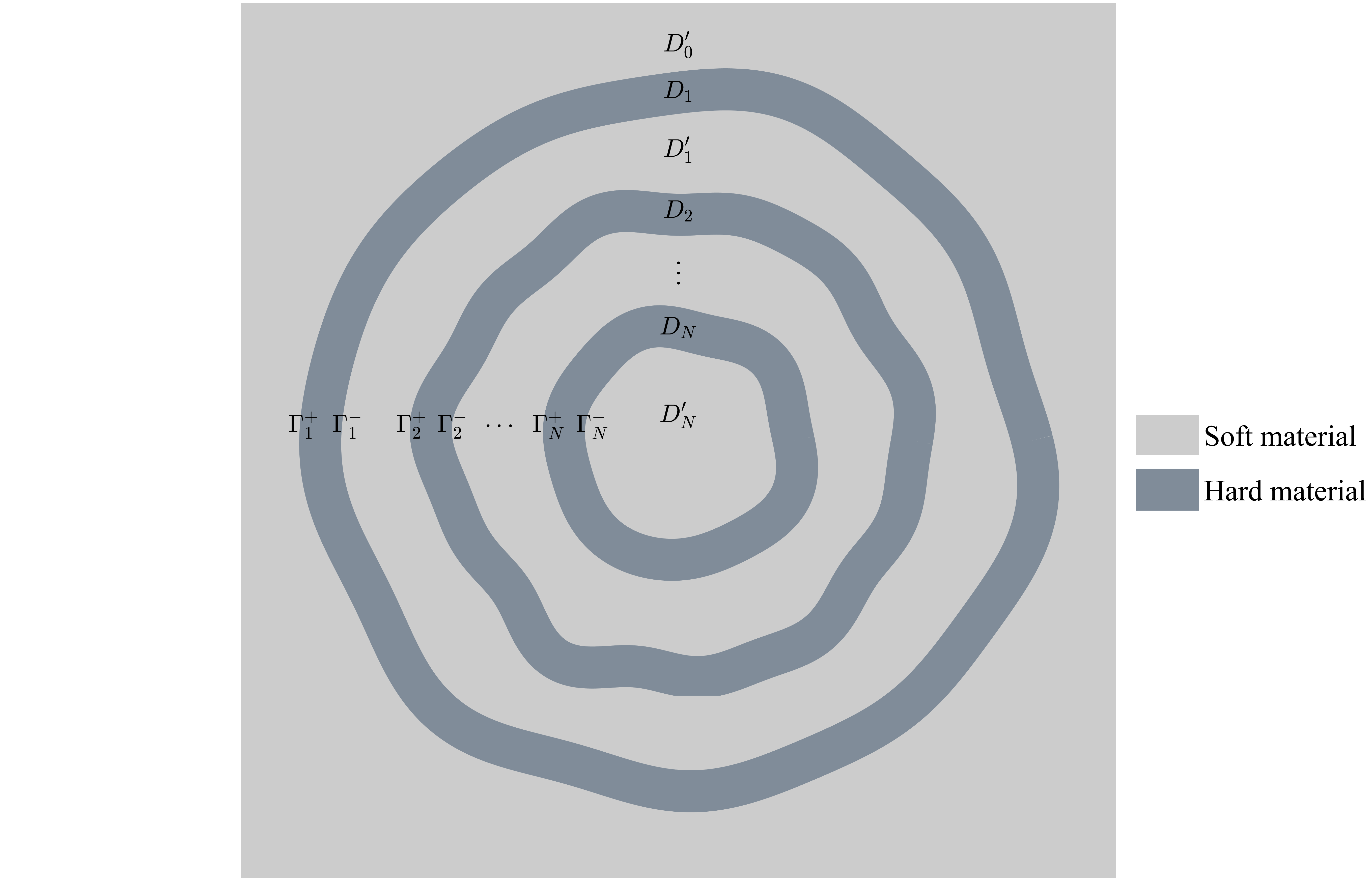}
 	\caption{ Schematic illustration of a structure of $N$-layer Matryoshka-type elastic resonators.}\label{MLHCCB}
 \end{figure}
% We remark that the setting \eqref{nestedcomplement} can be further generalized to accommodate more general material parameters; see Remark~\ref{General_material}.

Let $\mathbf{u}^{\textup{in}}$ be a time-harmonic incident elastic wave satisfying 
\[
\mathcal{L}_{\lambda,\mu } \mathbf{u}^{\textup{in}}  +\rho\omega^2  \mathbf{u}^{\textup{in}} = 0 \; \text{ in  } \RR^3,
\]
where the Lam\'e operator $\mathcal{L}_{\lambda, \mu}$  is defined by
\begin{equation}\label{Lameopert}
\mathcal{L}_{\lambda, \mu} \mathbf{u}^{\textup{in}}:=\mu \Delta \mathbf{u}^{\textup{in}}+(\lambda+\mu) \nabla \nabla \cdot \mathbf{u}^{\textup{in}}.
\end{equation}
Then the total elastic displacement field
$\mathbf{u}: = (\mathbf{u}^{(i)}(\Bx))_{1\leq i\leq 3} $ satisfies  the following Lam\'{e} system:
\begin{equation} \label{main_equation1}
\begin{cases}
\ds\mathcal{L}_{\lambda,\mu } \mathbf{u}  +\rho\omega^2  \mathbf{u} = 0, & \text{in  } \RR^3\setminus \overline{D}, \\
\ds\mathcal{L}_{\lambda_\rmr,\mu_\rmr } \mathbf{u}  +\sum_{j=1}^{N}\rho_{\rmr,j}\chi_{D_j}\omega^2  \mathbf{u}  = 0, & \text{in  } D, \\
\nm
\ds \mathbf{u}|_+ = \mathbf{u}|_-, & \text{on } \Gamma^\pm_j, \; j =1,2,\ldots,N,\\
\nm
\ds \partial_{\bm{\nu}}\mathbf{u}|_+ =  \partial_{\bm{\nu}_\rmr}\mathbf{u}|_- , & \text{on }\Gamma^+_j,\; j =1,2,\ldots,N, \\
\nm
\ds  \partial_{\bm{\nu}_\rmr}\mathbf{u}|_+  = \partial_{\bm{\nu}}\mathbf{u}|_- , & \text{on }\Gamma^-_j,\; j =1,2,\ldots,N, \\
\nm
\ds \mathbf{u}^{\textup{s}} := \mathbf{u} - \mathbf{u}^{\textup{in}} &  \mbox{satisfies the  radiation condition,}
\end{cases}
\end{equation}
where the conormtal derivative (or traction) is defined by
\begin{equation}\label{traction1}
\partial_{\bm{\nu}}\mathbf{u}=\lambda(\nabla \cdot \mathbf{u}) \bm{\nu}+\mu\left(\nabla \mathbf{u}+\(\nabla \mathbf{u}\)^{t}\right) \bm{\nu},
\end{equation}
with $\nabla \mathbf{u}$ denoting the matrix $(\partial_j \mathbf{u}_i)^3_
{i,j=1}$ and the superscript $t$ signifying the matrix transpose.
The operators $\mathcal{L}_{\lambda_\rmr,\mu_\rmr }$ and $\partial_{\bm{\nu}_\rmr}$ are defined in \eqref{Lameopert} and \eqref{traction1}, respectively, with
the parameters $(\lambda,\mu)$ replaced by $(\lambda_\rmr,\mu_\rmr)$.
In \eqref{main_equation1}, $\mathbf{u}^{\textup{s}}$  satisfies the radiation condition means that  there hold the following facts \cite{Kupradze_book1979},
\begin{align*}
\left(\nabla \times \nabla \times \mathbf{u}^{\textup{s}}\right)(\mathbf{x}) \times \frac{\mathbf{x}}{|\mathbf{x}|}-\mathrm{i} k_{s} \nabla \times \mathbf{u}^{\textup{s}}(\mathbf{x}) & =\mathcal{O}(|\mathbf{x}|^{-2}), \\
\frac{\mathbf{x}}{|\mathbf{x}|} \cdot\left[\nabla\left(\nabla \cdot \mathbf{u}^{\textup{s}}\right)\right](\mathbf{x})-\mathrm{i} k_{p} \nabla \mathbf{u}^{\textup{s}}(\mathbf{x}) & =\mathcal{O}(|\mathbf{x}|^{-2}),
\end{align*}
as $|\mathbf{x}| \rightarrow+\infty$. 
It is well-known that the elastic wave can be decomposed into the shear wave ($s$-wave)
and the compressional wave ($p$-wave).
The corresponding wavenumbers  are given by
\begin{equation}\label{auxiliaryparameters}
  k_s =\frac{\omega}{c_s}, \quad k_p=\frac{\omega}{c_p},
\end{equation}
 where $c_s$ is the wave velocity for shear waves and $c_p$ is the wave velocity for compressive waves:
 \begin{equation}\label{wave_velocity}
  c_s =\sqrt{\frac{\mu}{\rho}}, \quad c_p =\sqrt{\frac{\lambda+2\mu}{\rho}}.
 \end{equation}
 In what follows, the parameters $k_{s,\rmr,j}, k_{p,\rmr,j}, c_{s,\rmr,j},c_{p,\rmr,j}$ are defined in \eqref{auxiliaryparameters} and \eqref{wave_velocity} by replacing $(\lambda,\mu,\rho)$ with $(\lambda_\rmr,\mu_\rmr,\rho_{\rmr,j})$.
We also introduce some dimensionless contrast parameters:
 \begin{equation}\label{contrastparameter}
 \delta=\frac{\lambda}{\lambda_\rmr} =\frac{\mu}{\mu_\rmr},  
 %\,\, \epsilon = \frac{\rho}{\rho_\rmr},
 \,\, \tau_j=  \frac{c_s}{c_{s,\rmr,j}} = \frac{c_p}{c_{p,\rmr,j}} = \frac{k_{p,\rmr,j}}{k_p} = \frac{k_{s,\rmr,j}}{k_s}. %=\sqrt{\frac{\delta}{\epsilon}}.
 \end{equation}
 We will assume that  $
 c_s,c_p = \mathcal{O}(1)$, $c_{s,\rmr,j},c_{p,\rmr,j} = \mathcal{O}(1)$,  and $\tau_j = \mathcal{O}(1)$; meanwhile
 $0<\delta \ll 1$.
 This high-contrast assumption is the cause of the underlying system’s subwavelength resonant response and will be at the center of our subsequent analysis.

Using \eqref{contrastparameter}, the Lam\'{e} system \eqref{main_equation1} can be reformulated as follows:
 \begin{equation} \label{main_equation}
 \begin{cases}
 \ds\mathcal{L}_{\lambda,\mu } \mathbf{u}  +\rho\omega^2  \mathbf{u} = 0, & \text{in  } \RR^3\setminus \overline{D}, \\
 \ds\mathcal{L}_{\lambda,\mu } \mathbf{u}  +\sum_{j=1}^N\rho\tau_j^2\chi_{D_j}\omega^2  \mathbf{u}  = 0, & \text{in  } D,\\
 \nm
 \ds \mathbf{u}|_+ = \mathbf{u}|_-, & \text{on } \Gamma^\pm_j, \; j =1,2,\ldots,N,\\
 \nm
 \ds \delta\partial_{\bm{\nu}}\mathbf{u}|_+ =  \partial_{\bm{\nu}}\mathbf{u}|_- , & \text{on }\Gamma^+_j,\; j =1,2,\ldots,N, \\
 \nm
 \ds  \partial_{\bm{\nu}}\mathbf{u}|_+  =\delta \partial_{\bm{\nu}}\mathbf{u}|_- , & \text{on }\Gamma^-_j,\; j =1,2,\ldots,N, \\
 \nm
 \ds \mathbf{u}^{\textup{s}} := \mathbf{u} - \mathbf{u}^{\textup{in}} &  \mbox{satisfies the  radiation condition.}
 \end{cases}
 \end{equation}
 In order to understand the resonant behavior of the $N$-layer nested scatterers, we shall give the definition of the subwavelength resonant frequencies and resonant modes of the Lam\'{e} system \eqref{main_equation1}  based on the high contrast $\delta$  between the materials.
 
 \begin{defn}\label{defn:resonance}
 	Given $\delta>0$, a \emph{subwavelength resonant frequency} (or \emph{eigenfrequency}) is a complex number $\omega=\omega(\delta)\in\mathbb{C}$ such that $\omega(\delta)\to 0$ as $\delta\to 0$, and for which, in the case $\mathbf{u}^{\textup{in}}=0$, there exists a nontrivial solution to \eqref{main_equation}. Such a solution is called an \emph{associated resonant mode} (or \emph{eigenmode}).
 \end{defn}

 \subsection{Layer potentials}\label{subsec22}
In this subsection, we introduce the boundary layer potential operators. For more detailed discussions on the layer potential theory, we refer the reader to \cite{AK_book2018,HAmmariElasticityImaging}.

Let $\bm{G}^{\omega}:=(G^{\omega}_{i,j})_{i,j=1}^3$ denote the outgoing fundamental solution to the partial differential operator (PDO) $\mathcal{L}_{\lambda, \mu}+ \rho\omega^{2}$ in $\mathbb{R}^{3}$. It is given by \cite{HAmmariElasticityImaging},
\begin{equation}\label{fundamentalsolution}
G^{\omega}_{i,j}=-\frac{\e^{\mathrm{i} k_{s}|\mathbf{x}|}}{4 \pi \mu|\mathbf{x}|}\delta_{i,j} +\frac{1}{4 \pi \omega^{2} \rho}\partial_i\partial_j \frac{\e^{\mathrm{i} k_{p}|\mathbf{x}|}-\e^{\mathrm{i} k_{s}|\mathbf{x}|}}{|\mathbf{x}|},
\end{equation}
where $\delta_{i,j}$ is Kronecker delta function, $k_{s},k_{p}$ are defined in \eqref{auxiliaryparameters}. If $\omega=0$,   then the fundamental solution to the PDO $\mathcal{L}_{\lambda, \mu}$  is given by
\begin{equation}\label{gamma0}
{G}_{i,j}^0(\mathbf{x})=-\frac{\alpha_{1}}{4 \pi} \frac{1}{|\mathbf{x}|}\delta_{i,j}-\frac{\alpha_{2}}{4 \pi} \frac{x_ix_j}{|\mathbf{x}|^{3}},%\nonumber
\end{equation}
with $\Bx:=(x_1,x_2,x_3)^t$,
\begin{equation}\label{alpha}
\alpha_{1}=\frac{1}{2}\(\frac{1}{\mu}+\frac{1}{\lambda+2 \mu}\), \;\;\mbox{ and }\;\; \alpha_{2}=\frac{1}{2}\(\frac{1}{\mu}-\frac{1}{\lambda+2 \mu}\).
\end{equation}
 Let $\Omega$ be a bounded simply connected domain with a $C^{1,\eta}\, (0<\eta<1)$ boundary $\Gamma$. The single-  and the double-layer
 potentials associated with the operator  $\mathcal{L}_{\lambda, \mu}+ \rho\omega^{2}$ are defined as
 \begin{equation}\label{Single}
 \mathbf{S}_{{\Gamma}}^{\omega} [\bm{\phi}](\Bx) =  \int_{{\Gamma}}\bm{G}^{\omega}(\Bx- \By) \bm{\phi}(\By) ~\mathrm{d}\sigma(\By),  \quad \Bx \in  \mathbb{R}^3,
 \end{equation}
 \begin{equation}\label{Double}
 \mathbf{D}_{\Gamma}^{\omega} [\bm{\phi}](\Bx) =  \int_{{\Gamma}} \partial_ {\bm{\nu}({\By})} \bm{G}^{\omega}(\Bx- \By) \bm{\phi}(\By) ~\mathrm{d}\sigma(\By),  \quad \Bx \in  \mathbb{R}^3\setminus\Gamma,
 \end{equation}
 where $\bm{\phi}\in L^2(\Gamma)^3$ is the density function. The following jump relations hold almost everywhere on $\Gamma$: 
 \begin{equation} \label{singlejumpk}
 \partial_{\bm{\nu}}\mathbf{S}^\omega_{\Gamma} [\bm{\phi}] \big|_{\pm} (\Bx)= \(\pm \frac{\mathbf{I}}{2}+
 \mathbf{K}_{{\Gamma}}^{\omega,*}\)[\bm{\phi}](\Bx) \;\; \mbox{a.e.}\;\; \Bx\in {\Gamma},
 \end{equation}
  \begin{equation} \label{doublejumpk}
 \mathbf{D}_{\Gamma}^{\omega} [\bm{\phi}] \big|_{\pm} (\Bx) = \(\mp \frac{\mathbf{I}}{2}+
 \mathbf{K}_{{\Gamma}}^{\omega}\)[\bm{\phi}](\Bx) \;\; \mbox{a.e.}\;\; \Bx\in {\Gamma},
 \end{equation}
where $\mathbf{I}$ denotes the identity matrix operator in $\RR^3$.  The boundary integral operator $\mathbf{K}_{{\Gamma}}^{\omega}$ is  defined by
\begin{equation}\label{KKK}
\mathbf{K}_{{\Gamma}}^{\omega}[\bm{\phi}]
(\Bx) = \mbox{p.v.}\;\int_{{\Gamma}}\p_ {\bm{\nu}({\By})}\bm{G}^{\omega}(\Bx-\By)\bm{\phi}(\By)~\mathrm{d} \sigma(\By),
\end{equation}
and the operator
$ \mathbf{K}_{{\Gamma}}^{\omega,*}$, known as the  the Neumann--Poincar\'e (NP) operator, is defined by
\begin{equation}\label{KKK*}
 \mathbf{K}_{{\Gamma}}^{\omega,*}[\bm{\phi}]
(\Bx) = \mbox{p.v.}\;\int_{{\Gamma}}\p_ {\bm{\nu}({\Bx})}\bm{G}^{\omega}(\Bx-\By)\bm{\phi}(\By)~\mathrm{d} \sigma(\By).
\end{equation}
Here p.v. stands for the Cauchy principle value. In what follows, we denote by $\mathbf{S}_{\Gamma}$, $\mathbf{D}_{\Gamma}$, $\mathbf{K}_{\Gamma}$ and $\mathbf{K}_{\Gamma}^{*}$ the the static layer potentials, obtained formally by taking $\omega =0$ in \eqref{Single}--\eqref{Double} and \eqref{KKK}-- \eqref{KKK*}, respectively.
 
Since we are interested in low-frequency regime, we first present the asymptotic expansion of the fundamental solution $\bm{G}^{\omega}$ defined in \eqref{fundamentalsolution} (cf. \cite{HAmmariElasticityImaging,LZArxiv}).
 
 \begin{lem}\label{AE_fundmtl}
 	For $\omega\ll 1$, the fundamental solution $\bm{G}^{\omega}$ has the following asymptotic expansion:
 	\begin{equation}
 	\bm{G}^{\omega}(\mathbf{x})=\sum_{n=0}^{+\infty}\omega^n\bm{G}_n(\mathbf{x}),\nonumber
 	\end{equation}
 	where $\bm{G}_n :=({G}_{n;i,j})_{i,j=1}^3 $ is given by
 	\begin{equation}\label{gamman}
 	\begin{aligned}
 	{G}_{n;i,j}(\mathbf{x})&=-\frac{1}{4\pi\rho}\frac{\mathrm{i}^n}{(n+2)n!}\left(\frac{n+1}{c_{s}^{n+2}}+\frac{1}{c_{p}^{n+2}}\right)|\mathbf{x}|^{n-1}\delta_{i,j}\\
 	&\quad +\frac{1}{4\pi\rho}\frac{\mathrm{i}^n(n-1)}{(n+2)n!}\left(\frac{1}{c_{s}^{n+2}}-\frac{1}{c_{p}^{n+2}}\right)|\mathbf{x}|^{n-3}x_ix_j.
 	\end{aligned}
 	\end{equation}
 \end{lem}

Form Lemma \ref{AE_fundmtl}, we can obtain the asymptotic expansions for the single-layer potential and
the Neumann-Poincar\'{e} operators.
 
 \begin{lem}\label{lemgamma}
  There hold the following asymptotic expansions:
 \begin{equation}\label{singlelayerasymptotic}
 \mathbf{S}_{\Gamma}^{\omega}=\mathbf{S}_{\Gamma}+\sum_{n=1}^\infty\omega^n\mathbf{S}_{\Gamma,n}, \mbox{ and }\; \mathbf{K}_{\Gamma}^{\omega,*}=\mathbf{K}_{\Gamma}^{*}+\sum_{n=1}^\infty\omega^n\mathbf{K}_{\Gamma,n}^{*},
 \end{equation}
 where 
 \[
 \mathbf{S}_{{\Gamma},n} [\bm{\phi}](\Bx) =  \int_{{\Gamma}}\bm{G}_{n}(\Bx- \By) \bm{\phi}(\By) ~\mathrm{d}\sigma(\By),
 \]
 \[
 \mathbf{K}_{{\Gamma,n}}^{*}[\bm{\phi}]
 (\Bx) = \int_{{\Gamma}}{\p_{\bm{\nu}({\Bx})} \bm{G}_{n}(\Bx-\By)}\bm{\phi}(\By)~\mathrm{d} \sigma(\By),\quad \Bx \in  \Gamma.
 \]
 In particular, one has that
 \[
 \mathbf{K}_{\Gamma,1}^{*}=0,
 \]
 \begin{equation}\label{SL1}
 \mathbf{S}_{\Gamma,1}[\bm{\phi}](\mathbf{x})=-\ii \alpha\int_{\Gamma}\bm{\phi}(\mathbf{y})~\d \sigma(\mathbf{y}),
 \end{equation}
 where
 \begin{equation}\label{alphaa}
 \alpha:=\frac{1}{12\pi\rho}\left(\frac2{c_{s}^3}+\frac1{c_{p}^3}\right).
 \end{equation}
Moreover, the norms
$\| \mathbf{S}_{\Gamma, n} \|_{
	\mathcal{L}(L^2(\Gamma)^3, H^1(\Gamma)^3)}$
and $\| \mathbf{K}_{\Gamma, n}^* \|_{ \mathcal{L}(L^2(\Gamma)^3)}$
are uniformly bounded with respect to $n$, and  the series in \eqref{singlelayerasymptotic} are convergent in
$\mathcal{L}(L^2(\Gamma)^3, H^1(\Gamma)^3)$ and $\mathcal{L}(L^2(\Gamma)^3,L^2(\Gamma)^3)$, respectively.
 \end{lem}

We next introduce the vector space $\mho $ consisting of all linear solutions to the
equation 
\begin{equation}\label{elastostatics}
\begin{cases}
\mathcal{L}_{\lambda,\mu}\mathbf{u}=0,&\mbox{ in } \Omega,\\
\ds \partial_{\bm{\nu}}\mathbf{u}=0,&\mbox{ on }\Gamma.
\end{cases}
\end{equation}
It is straightforward to verify that $\mho$ is the space spanned by the displacement fields corresponding to rigid body motions \cite{HAmmariElasticityImaging}, given by
%By direct computations, we have that the space $\mho $  spanned by the displacement fields of the rigid motions \cite{HAmmariElasticityImaging}
%spanned by the following rigid displacements 
\begin{equation}\label{rigid_motions}
\bm{\varkappa}_{1}=
\begin{pmatrix}
1\\0\\0
\end{pmatrix},\;
\bm{\varkappa}_2=
\begin{pmatrix}
0\\1\\0
\end{pmatrix},\;
\bm{\varkappa}_3=
\begin{pmatrix}
0\\0\\1
\end{pmatrix},\;
\bm{\varkappa}_4=
\begin{pmatrix}
x_2\\-x_1\\0
\end{pmatrix},\;
\bm{\varkappa}_5=
\begin{pmatrix}
x_3\\0\\-x_1
\end{pmatrix},\;
\bm{\varkappa}_6=
\begin{pmatrix}
0\\x_3\\-x_2
\end{pmatrix},
\end{equation}
where $\bm{\varkappa}_p$ for $p = 1, 2, 3$ represent rigid translations, while $\bm{\varkappa}_p$ for $p = 4, 5, 6$ correspond to rigid rotations.

We end this section by collecting some properties of the static layer potential operators, which will be useful in the following.

\begin{lem} [\cite{AK_book2018}]
	The following statements hold:
	\begin{enumerate}
		\item[\textup{(i)}] The single-layer potential $-\mathbf{S}_{\Gamma}:
		L^2(\Gamma)^3 \rightarrow L^2(\Gamma)^3$ is positive and self-adjoint;
		
		\item[\textup{(ii)}] The double layer potential $\mathbf{D}_{\Gamma}$ satisfies that for $1\leq p\leq 6,$
		\begin{equation}\label{0105}
		\mathbf{D}_{\Gamma}[\bm{\varkappa}_p](\mathbf{x}) =\begin{cases} 
		\bm{0},& \mbox{ if } \mathbf{x}\in \RR^3\setminus\overline{\Omega},\\
		\bm{\varkappa}_p,& \mbox{ if } \mathbf{x}\in {\Omega},
		\end{cases} \;\mbox{ and }\, \mathbf{K}_{\Gamma}[\bm{\varkappa}_p](\mathbf{x}) = \frac{1}{2}\bm{\varkappa}_p \,\mbox{ for } \,\mathbf{x}\in \Gamma.
		\end{equation}
		
		\item[\textup{(iii)}]  The kernel space of $-\frac{\mathbf{I}}{2}+\mathbf{K}_{\Gamma}^{*}$ is consist of $\mathbf{S}_{\Gamma}^{-1}[\bm{\varkappa}_p]$ with $1\leq p\leq 6.$
	\end{enumerate}

\end{lem}

\section{Variational characterization of resonances based on the displacement-to-traction map}\label{D2Nmap}

In this section, we characterize the displacement-to-traction map of the Lam\'e system \eqref{main_equation} in the Matryoshka-type elastic medium $D = \cup_{j=1}^{N} D_{j}$. 
To this end, we first define the matrix-type single- and double-layer potential operators associated with the bounded doubly-connected domain $D'_j$, which lies between the interior boundary $\Gamma_{j+1}^+$ and the
exterior boundary $\Gamma_{j}^-$, respectively, as 
\begin{equation}\label{singlelayermatrix}
\mathbb{S}_{j,j+1}^\omega := \begin{pmatrix}
\mathbf{S}_{{\Gamma_j^-}}^{\omega}  & \mathbf{S}_{{\Gamma_j^-,\Gamma_{j+1}^+}}^{\omega} \\
\mathbf{S}_{{\Gamma_{j+1}^+,\Gamma_j^-}}^{\omega}  & \mathbf{S}_{{\Gamma_{j+1}^+}}^{\omega}
\end{pmatrix}, \mbox{ and } \mathbb{K}^{\omega,*}_{j,j+1} = \begin{pmatrix}
-\mathbf{K}_{{\Gamma_j^-}}^{\omega,*}  & -\mathbf{K}_{{\Gamma_j^-,\Gamma_{j+1}^+}}^{\omega,*} \\
\nm 
\mathbf{K}_{{\Gamma_{j+1}^+,\Gamma_j^-}}^{\omega,*}  & \mathbf{K}_{{\Gamma_{j+1}^+}}^{\omega,*}
\end{pmatrix}.
\end{equation} 
Here, we introduced the operators $\mathbf{S}^{\omega}_{\Gamma_{i},\Gamma_{j}}:L^2(\Gamma_{j})^3\to L^2(\Gamma_i)^3$ and $\mathbf{K}_{\Gamma_{i},\Gamma_{j}}^{\omega, *}:L^2(\Gamma_{j})^3\to L^2(\Gamma_i)^3$ as defined respectively by
\begin{equation}\label{SKGiGj}
\mathbf{S}^{\omega}_{\Gamma_{i},\Gamma_{j}}[\bm{\phi}] = \mathbf{S}^{\omega}_{\Gamma_{j}}[\bm{\phi}] \big|_{\Gamma_i}\;\mbox{ and }\; \mathbf{K}_{\Gamma_{i},\Gamma_{j}}^{\omega, *}[\bm{\phi}] = \partial_{\bm{\nu}}\mathbf{S}^{\omega}_{\Gamma_{j}}[\bm{\phi}]\big|_{\Gamma_i}, \mbox{ for } \forall\, \bm{\phi}\in L^2(\Gamma_{j})^3.
\end{equation}
Moreover, we also introduce the $L^2$-adjoint operator of $\mathbf{K}_{\Gamma_{i},\Gamma_{j}}^{\omega, *}$, $\mathbf{K}_{\Gamma_{j},\Gamma_{i}}^{\omega}:L^2(\Gamma_{i})^3\to L^2(\Gamma_j)^3$, defined by
\begin{equation}\label{Kji}
\mathbf{K}_{\Gamma_{j},\Gamma_{i}}^{\omega}[\bm{\phi}] = \mathbf{D}^{\omega}_{\Gamma_{i}}[\bm{\phi}] \big|_{\Gamma_j}, \mbox{ for } \forall\, \bm{\phi}\in L^2(\Gamma_{i})^3.
\end{equation}
Thus the $L^2$-adjoint of $\mathbb{K}^{\omega,*}_{j,j+1} $, $\mathbb{K}^{\omega}_{j,j+1} $, is given by
\begin{equation}\label{Kjjp1}
\mathbb{K}^{\omega}_{j,j+1} = \begin{pmatrix}
-\mathbf{K}^{\omega}_{{\Gamma_j^-}}  & \mathbf{K}^{\omega}_{{\Gamma_j^-,\Gamma_{j+1}^+}} \\
\nm 
-\mathbf{K}^{\omega}_{{\Gamma_{j+1}^+,\Gamma_j^-}}  & \mathbf{K}^{\omega}_{{\Gamma_{j+1}^+}}
\end{pmatrix}.
\end{equation}
In what follows, we abbreviate  
 $\mathbb{S}_{j,j+1}: = \mathbb{S}_{j,j+1}^0$,   $\mathbb{K}^{*}_{j,j+1} := \mathbb{K}_{j,j+1}^{0,*}$, and $\mathbb{K}_{j,j+1} := \mathbb{K}_{j,j+1}^{0}$
  when $\omega=0$ in \eqref{singlelayermatrix}  and \eqref{Kjjp1}, respectively.

\subsection{Definition of displacement-to-traction map}

The following lemma provides the representation formula for the solution
to exterior displacement problems on $\RR^3\setminus \overline{D}$.

\begin{lem}
	\label{def:DTNGS}
	Assume that $\omega$ is chosen such that  the  operators   $ \mathbf{S}_{\Gamma_1^+}^{\omega}$, $\mathbf{S}_{{\Gamma_N^-}}^{\omega}$, and $\mathbb{S}_{j,j+1}^\omega$	
	are invertible
	for $1\leq j\leq N-1$. Then, for any  $\bm{f}_j^{\pm}\in L^2(\Gamma_j^{\pm})^3$, $j=1,2,\ldots,N$, there exists a unique solution  $\mathbf{v}_{\bm{f}}\in
	H^{1}(\RR^3)^3$ to the following exterior problem:
	\begin{equation} \label{exterior_problemGS}
	\begin{cases}
	\ds\mathcal{L}_{\lambda,\mu } \mathbf{v}_{\bm{f}}  +\rho\omega^2  \mathbf{v}_{\bm{f}} =0, & \text{in  } \RR^3\setminus \overline{D}, \\
	\nm
	\ds  \mathbf{v}_{\bm{f}}|_{\Gamma_j^\pm}   = \bm{f}_j^\pm ,&  j =1,2,\ldots,N, \\
	\nm
	\ds \mathbf{v}_{\bm{f}} &  \mbox{satisfies the radiation condition.}
	\end{cases}
	\end{equation}
	Furthermore, the solution $\mathbf{v}_{\bm{f}}$ admits the representation
	\begin{equation}\label{eq3dqrq}
	\mathbf{v}_{\bm{f}} = \begin{cases}
	\ds \mathbf{S}_{\Gamma_1^+}^{\omega} [\bm\psi_1^+], & x\in D_0', \\
	\nm
	\ds  \mathbf{S}_{{\Gamma_j^-}}^{\omega} [\bm\phi_j^-](x)+\mathbf{S}_{{\Gamma_{j+1}^+}}^{\omega} [\bm\psi_{j+1}^+](x), & x\in D_j'\quad j =1,2,\ldots,N-1, \\
	\nm
	\ds \mathbf{S}_{{\Gamma_N^-}}^{\omega} [\bm\phi_N^-], & x\in D_N', \\
	\end{cases}
	\end{equation}
	where $\bm\psi_1^+ = (\mathbf{S}_{\Gamma_1^+}^{\omega})^{-1} [\bm{f}_1^+]$, $\bm\phi_N^- = (\mathbf{S}_{{\Gamma_N^-}}^{\omega})^{-1} [\bm{f}_N^-]$, and $\bm\phi_j^-$, $\bm\psi_{j+1}^+$, $j=1,2,\ldots,N-1$, are determined by  
	\begin{equation}\label{eqn:pxggr}
	\begin{pmatrix}
	\bm\phi_j^-\\
	\nm 
	\bm\psi_{j+1}^+
	\end{pmatrix} = \(\mathbb{S}_{j,j+1}^\omega\)^{-1}
	\begin{pmatrix}
	\bm{f}_j^{-}\\
	\nm 
	\bm{f}_{j+1}^{+}
	\end{pmatrix}.
	\end{equation}
\end{lem}
\begin{proof}[\bf Proof]
It is easy to see that the solution $\mathbf{v}_{\bm{f}}$ to \eqref{exterior_problemGS} admits the representation \eqref{eq3dqrq}. By using the Dirichlet boundary conditions of \eqref{exterior_problemGS}, the density functions $\bm\phi_j^-$  and $\bm\psi_{j+1}^+$, $j=1,2,\ldots,N-1$, satisfy
		\[
		\begin{pmatrix}
		\mathbf{S}_{{\Gamma_j^-}}^{\omega}  & \mathbf{S}_{{\Gamma_j^-,\Gamma_{j+1}^+}}^{\omega} \\
		\mathbf{S}_{{\Gamma_{j+1}^+,\Gamma_j^-}}^{\omega}  & \mathbf{S}_{{\Gamma_{j+1}^+}}^{\omega}
		\end{pmatrix}
		\begin{pmatrix}
		\bm\phi_j^-\\
		\nm 
		\bm\psi_{j+1}^+
		\end{pmatrix} = 
		\begin{pmatrix}
		\bm{f}_j^{-}\\
		\nm 
		\bm{f}_{j+1}^{+}
		\end{pmatrix}.
		\]
	Inverting above equality, 
	we can get \eqref{eqn:pxggr}. The proof is complete.
\end{proof}

\begin{defn}
For any $\omega \in \mathbb{C}$ is chosen such that  the  operators   $ \mathbf{S}_{\Gamma_1^+}^{\omega}$, $\mathbf{S}_{{\Gamma_N^-}}^{\omega}$,  $\mathbb{S}_{j,j+1}^\omega$  are invertible	for  $1\leq j\leq N - 1$, the displacement-to-traction map with frequency $\omega$ is
	the linear operator $\mathcal{T}^\omega : \mathbb{C}^{6N} \rightarrow \mathbb{C}^{6N}$ defined by
	\begin{equation}
	\label{eqD2N}
	\mathcal{T}^{\omega}[(\bm{f}_j^{\pm})_{1\leq j\leq N}]=
	\left(  {\partial_{\bm{v}} \mathbf{v}_{\bm{f}}}|_{\Gamma_j^\pm}\right)_{1 \leq j \leq N},
	\end{equation}
	where $\mathbf{v}_f$ is the unique solution to \eqref{exterior_problemGS}.
\end{defn}

\begin{rem}
The map \eqref{eqD2N} is the analogue of the Dirichlet-to-Neumann map for 
%the conductivity problem (see \cite{ADKL_AIHPCAN}) and 
the Helmholtz equation (see \cite{DKZ_JLMS2026}).	The condition that  $\omega$ is chosen such that the operators $ \mathbf{S}_{\Gamma_1^+}^{\omega}$, $\mathbf{S}_{{\Gamma_N^-}}^{\omega}$,  $\mathbb{S}_{j,j+1}^\omega$ for  $1\leq j\leq N - 1$ are invertible is equivalent to assuming that $\omega$ is not a Dirichlet eigenvalue of \eqref{exterior_problemGS} (cf. \cite{WMCbook}). This guarantees the uniqueness of the exterior Dirichlet problem \eqref{exterior_problemGS} and thus the well-defined of displacement-to-traction map \eqref{eqD2N}.	This condition is naturally satisfied in the regime	$\omega \to0$.
\end{rem}

In the next proposition, we give the matrix representation of the displacement-to-traction map $\mathcal{T}^\omega$.

\begin{prop} \label{prop:DTN}
	The displacement-to-traction map
	$\mathcal{T}^{\omega}$ admits the following matrix representation: for any $ \bm{f}:= (\bm{f}_j^{\pm})_{1\leq j\leq N}$, $\mathcal{T}^{\omega}[\bm{f}]:= (\mathcal{T}^{\omega}[\bm{f}]_{j}^{\pm})_{1\leq j\leq N}$ is given by
	\begin{equation}\label{matrixrepre}
	\begin{pmatrix}
	\mathcal{T}^{\omega}[\bm{f}]_1^{+}\\
	\mathcal{T}^{\omega}[\bm{f}]_{1}^{-}\\
	\mathcal{T}^{\omega}[\bm{f}]_2^{+}\\
	\vdots\\
	\mathcal{T}^{\omega}[\bm{f}]_{N-1}^{-}\\
	\mathcal{T}^{\omega}[\bm{f}]_N^{+}\\
	\mathcal{T}^{\omega}[\bm{f}]_{N}^{-}
	\end{pmatrix} =  \begin{pmatrix}
	\(\frac{\mathbf{I}}{2}+\mathbf{K}_{{\Gamma_1^+}}^{\omega,*}\)(\mathbf{S}_{\Gamma_1^+}^{\omega})^{-1}&&&&\\
	&\mathbb{A}^\omega_{1,2}&&&\\
	&&\ddots &&\\
	&&&\mathbb{A}^\omega_{N-1,N}&\\
	&&&&\(-\frac{\mathbf{I}}{2}+\mathbf{K}_{{\Gamma_N^-}}^{\omega,*}\) (\mathbf{S}_{{\Gamma_N^-}}^{\omega})^{-1}
	\end{pmatrix} \begin{pmatrix}
	\bm{f}_1^{+}\\
	\bm{f}_{1}^{-}\\
	\bm{f}_2^{+}\\
	\vdots\\
	\bm{f}_{N-1}^{-}\\
	\bm{f}_N^{+}\\
	\bm{f}_{N}^{-}
	\end{pmatrix}
	\end{equation}
	where 
	\[
	\begin{aligned}
	\mathbb{A}^\omega_{j,j+1} &:=\begin{pmatrix}
	-\frac{\mathbf{I}}{2}+\mathbf{K}_{{\Gamma_j^-}}^{\omega,*}  & \mathbf{K}_{{\Gamma_j^-,\Gamma_{j+1}^+}}^{\omega,*} \\
	\nm 
	\mathbf{K}_{{\Gamma_{j+1}^+,\Gamma_j^-}}^{\omega,*}  & \frac{\mathbf{I}}{2}+\mathbf{K}_{{\Gamma_{j+1}^+}}^{\omega,*}
	\end{pmatrix}
	\begin{pmatrix}
	\mathbf{S}_{{\Gamma_j^-}}^{\omega}  & \mathbf{S}_{{\Gamma_j^-,\Gamma_{j+1}^+}}^{\omega} \\
	\mathbf{S}_{{\Gamma_{j+1}^+,\Gamma_j^-}}^{\omega}  & \mathbf{S}_{{\Gamma_{j+1}^+}}^{\omega}
	\end{pmatrix}^{-1}\\
	& := \begin{pmatrix}
	-\mathbf{I} & \mathbf{O} \\
	\nm
\mathbf{O}  &\mathbf{I}
	\end{pmatrix}\(\frac{\mathbb{I}}{2}+\mathbb{K}^{\omega,*}_{j,j+1}\) (\mathbb{S}^{\omega}_{j,j+1})^{-1},\quad j =1,2,\ldots,N-1,
	\end{aligned}
	\]
	with  
	\[
	\mathbb{I} = \begin{pmatrix}
	\mathbf{I}  & \mathbf{O} \\
	\nm 
	\mathbf{O}  & \mathbf{I}
	\end{pmatrix},
%\;\; \mathbb{K}^{\omega,*}_{j,j+1} = \begin{pmatrix}
%	-\mathbf{K}_{{\Gamma_j^-}}^{\omega,*}  & -\mathbf{K}_{{\Gamma_j^-,\Gamma_{j+1}^+}}^{\omega,*} \\
%	\nm 
%	\mathbf{K}_{{\Gamma_{j+1}^+,\Gamma_j^-}}^{\omega,*}  & \mathbf{K}_{{\Gamma_{j+1}^+}}^{\omega,*}
%	\end{pmatrix},
	\]
	and $\mathbf{O}$ being the zero matrix operator in $\RR^3$.
\end{prop}

\begin{proof}[\bf Proof]
	 By using \eqref{singlejumpk}, and \eqref{eq3dqrq}--\eqref{eqD2N}, we can obtain that
	 \[
	 \mathcal{T}^{\omega}[\bm{f}]_{1}^{+} = \(\frac{\mathbf{I}}{2}+\mathbf{K}_{{\Gamma_1^+}}^{\omega,*}\)(\mathbf{S}_{\Gamma_1^+}^{\omega})^{-1} [\bm{f}_1^+],
	 \]
	\[
	\begin{aligned}
	\begin{pmatrix}
	\mathcal{T}^{\omega}[\bm{f}]_j^{-}\\
	\nm 
	\mathcal{T}^{\omega}[\bm{f}]_{j+1}^{+}
	\end{pmatrix} &= \begin{pmatrix}
	-\frac{\mathbf{I}}{2}+\mathbf{K}_{{\Gamma_j^-}}^{\omega,*}  & \mathbf{K}_{{\Gamma_j^-,\Gamma_{j+1}^+}}^{\omega,*} \\
	\nm 
	\mathbf{K}_{{\Gamma_{j+1}^+,\Gamma_j^-}}^{\omega,*}  & \frac{\mathbf{I}}{2}+\mathbf{K}_{{\Gamma_{j+1}^+}}^{\omega,*}
	\end{pmatrix}
	\begin{pmatrix}
	\bm\phi_j^-\\
	\nm 
	\bm\psi_{j+1}^+
	\end{pmatrix}\\
	& = \begin{pmatrix}
	-\frac{\mathbf{I}}{2}+\mathbf{K}_{{\Gamma_j^-}}^{\omega,*}  & \mathbf{K}_{{\Gamma_j^-,\Gamma_{j+1}^+}}^{\omega,*} \\
	\nm 
	\mathbf{K}_{{\Gamma_{j+1}^+,\Gamma_j^-}}^{\omega,*}  & \frac{\mathbf{I}}{2}+\mathbf{K}_{{\Gamma_{j+1}^+}}^{\omega,*}
	\end{pmatrix}
	\begin{pmatrix}
	\mathbf{S}_{{\Gamma_j^-}}^{\omega}  & \mathbf{S}_{{\Gamma_j^-,\Gamma_{j+1}^+}}^{\omega} \\
	\mathbf{S}_{{\Gamma_{j+1}^+,\Gamma_j^-}}^{\omega}  & \mathbf{S}_{{\Gamma_{j+1}^+}}^{\omega}
	\end{pmatrix}^{-1} \begin{pmatrix}
	\bm{f}_j^{-}\\
	\nm 
	\bm{f}_{j+1}^{+}
	\end{pmatrix}, \;\mbox{ for } j=1,2,\ldots,N-1,
	\end{aligned}
	\]
	and
	\[
	\mathcal{T}^{\omega}[\bm f]_N^{-} = \(-\frac{\mathbf{I}}{2}+\mathbf{K}_{{\Gamma_N^-}}^{\omega,*}\) (\mathbf{S}_{{\Gamma_N^-}}^{\omega})^{-1} [\bm f_N^-].
	\]
	The proof is complete.
\end{proof}

It is noteworthy that, in the low-frequency regime, the displacement data $ \bm{f}:= (\bm{f}_j^{\pm})_{1\leq j\leq N}$ associated with the exterior problem  \eqref{exterior_problemGS}  can be expressed as a linear combination of the rigid motion basis $\{\bm{\varkappa}_{p}\}_{p=1}^6$ defined in \eqref{rigid_motions}. However, the rigid rotations  $\{\bm{\varkappa}_{p}\}_{p=4}^6$ are not orthogonal in the $L^2(D_i)^3$ inner product unless the domain $D_i$ possesses certain geometric symmetries (e.g., the concentric radial configuration considered in Section~\ref{concentric_radial}). For later use, we apply the Gram--Schmidt procedure to $\{\bm{\varkappa}_{p}\}_{p=1}^6$ to construct the orthogonal basis $\{\bm{\xi}_{i,p}\}_{p=1}^6$ satisfying 
\[
\(\bm{\xi}_{i,p},\bm{\xi}_{i,q}\)_{L^2(D_i)^3} = \|\bm{\xi}_{i,p}\|_{L^2(D_i)^3}\delta_{p,q}.
\]
Particularly, the rigid translations $\{\bm{\varkappa}_{p}\}_{p=1}^3$ are orthogonal, but for convenience we still write
$
\bm{\xi}_{i,p} = \bm{\varkappa}_{p}$ for $ p=1,2,3.
$
%\[
%\bm{\xi}_{i,4} = \begin{pmatrix}
%x_2-c_{i,2} \\
%-x_1+ c_{i,1} \\
%0
%\end{pmatrix},\; \; \bm{\xi}_{i,5} = \begin{pmatrix}
%x_3-c_{i,3}-L_{i,1}(x_2-c_{i,2})\\
%L_{i,1}(x_1-c_{i,1}) \\
%-x_1+c_{i,1}
%\end{pmatrix},
%\]
%and 
%\[
%\bm{\xi}_{i,6} = \begin{pmatrix}
%-L_{i,2}(x_2-c_{i,2})-L_{i,3}(x_3-c_{i,3})+L_{i,3} L_{i,1}(x_2-c_{i,2})\\
%x_3-c_{i,3}+L_{i,2}(x_1-c_{i,1})+L_{i,3} L_{i,1}(c_{i,1}-x_1) \\
%-x_2+c_{i,2}+L_{i,3}(x_1-c_{i,1})
%\end{pmatrix},
%\]
%where the constants $c_{i,1}, c_{i,2}, c_{i,3}$ represent the centroid of the domain $D_i$ and are defined as:
%\begin{align*}
%c_{i,1}=\frac{\int_{D_i}x_1~\d \Bx}{|D_i|},\quad c_{i,2}=\frac{\int_{D_i}x_2~\d \Bx}{|D_i|},\quad c_{i,3}=\frac{\int_{D_i}x_3~\d \Bx}{|D_i|}.
%\end{align*}
%The coefficients $L_{i,1}, L_{i,2}, L_{i,3}$ are given, respectively, by
%\[
%L_{i,1}=\frac{\int_{D_i} x_2x_3~\d \Bx-|{D_i}|c_{i,2} c_{i,3}}{\int_{D_i}((x_2-c_{i,2})^2+(x_1-c_{i,1})^2)~\d \Bx},\quad L_{i,2}=\frac{-\int_{D_i}x_1x_3~\d \Bx-|{D_i}|c_{i,1} c_{i,3}}{\int_{D_i}((x_2-c_{i,2})^2+(x_1-c_{i,1})^2)~\d \Bx},
%\]
%and 
%\[ L_{i,3}=\frac{\int_{D_i}(-L_{i,1} x_3(x_2-c_2)-x_2(-x_1+c_{i,1}))~\d \Bx}{\int_{D_i}((x_3-c_{i,3}-L_{i,1} x_2+L_{i,1} c_{i,2})^2+(L_{i,1} x_1-L_{i,1} c_{i,1})^2+(c_{i,1}-x_1)^2)~\d \Bx}.
%\]

 In order to give the definition of  the stiffness $\textbf{Stiff}(D'_j;D_i)$ of $D_j'$, $j=1,2,\ldots,N-1$, under the orthonormal basis $\{\bm{\xi}_{i,p}\}_{p=1}^6$,  we introduce the follow functions for any fixed $i=1,2,\ldots,N$,
 \begin{equation}\label{Uipm}
 \bm\Upsilon_{i,p;1} := \begin{pmatrix}
 \bm{\xi}_{i,p}\\
 \nm
 \bm{0}
 \end{pmatrix},\; \bm\Upsilon_{i,p;2} := \begin{pmatrix}
 \bm{0}\\
 \nm
 \bm{\xi}_{i,p}
 \end{pmatrix}\in L^2(\Gamma_j^-)^3\times L^2(\Gamma_{j+1}^+)^3,\;\mbox{for}\; p=1,2,\ldots,6.
 \end{equation}
Here $\bm{0}$ denotes the three-dimensional zero vector, and so $\bm\Upsilon_{i,p;1}$ and $\bm\Upsilon_{i,p;2}$ are six-dimensional
vector-valued functions. 
%\[
%e_1=(1,0)^T, \;e_2 = (0,1)^T\in L^2(\Gamma_j^-)\times L^2(\Gamma_{j+1}^+).
%\] 
\begin{defn}\label{defnstiff} For any fixed $i=1,2,\ldots,N$, 
	the stiffness $\textbf{Stiff}(D'_j;D_i):=\(\text{Stiff}(D'_j;D_i)_{pq}\)_{p,q=1}^6$ of $D_j'$, $j=1,2,\ldots,N-1$, can be defined by
	\begin{equation}\label{stifD'j}
	\text{Stiff}(D'_j;D_i)_{pq} := \left( -\mathbb{S}_{j,j+1}^{-1} [\bm\Upsilon_{i,p;2}],\bm\Upsilon_{i,q;2}\right)_{L^2(\Gamma_j^-)^3\times L^2(\Gamma_{j+1}^+)^3}.
	\end{equation}
\end{defn}

\begin{lem}
For any fixed $i=1,2,\ldots,N$, the stiffness $\textbf{\textup{Stiff}}(D'_j;D_i)$ of $D'_j$, $j=1,2,\ldots,N-1$, deifined in \eqref{stifD'j} is positive-definite. Moreover, let
	\begin{equation}\label{S^-101GS}
	\begin{pmatrix}
	\hat{\bm\phi}_{j;i,p}^{-}\\
	\nm
	\hat{\bm{\psi}}_{j+1;i,p}^{+}
	\end{pmatrix}
	:=\mathbb{S}_{j,j+1}^{-1}[\bm\Upsilon_{i,p;1}] 
	\;\mbox{ and }\; 
	\begin{pmatrix}
	\hat{\bm\varphi}_{j;i,p}^{-}\\
	\nm
	\hat{\bm{\zeta}}_{j+1;i,p}^{+}
	\end{pmatrix}
	:=\mathbb{S}_{j,j+1}^{-1} [\bm\Upsilon_{i,p;2}].
	\end{equation}
 Then, we have
	\begin{equation}\label{cap_pp}
	\textup{Stiff}(D'_j;D_i)_{pp} = -\int_{\Gamma_{j+1}^+} \hat{\bm{\zeta}}_{j+1;i,p}^{+}\cdot \overline{ \bm{\xi}_{i,p}} ~\d\sigma > 0,
	\end{equation}  
	\begin{equation}\label{cap1_pp}
	-\int_{\Gamma_{j}^-} \hat{\bm{\phi}}_{j;i,p}^{-}\cdot \overline{ \bm{\xi}_{i,p}} ~\d\sigma >0,
	\end{equation}
	and  for any fixed $1\leq m,n\leq N$,
	\begin{equation}\label{cap_pq}
	\int_{\Gamma_{j+1}^+} \hat{\bm{\zeta}}_{j+1;m,p}^{+}\cdot \overline{ \bm{\xi}_{n,q}} ~\d\sigma = \int_{\Gamma_{j+1}^+} \hat{\bm{\zeta}}_{j+1;n,q}^{+}\cdot \overline{ \bm{\xi}_{m,p}} ~\d\sigma,
	\end{equation}
	\begin{equation}\label{cap1_pq}
	\int_{\Gamma_{j}^-} \hat{\bm{\phi}}_{j;m,p}^{-}\cdot \overline{ \bm{\xi}_{n,q}} ~\d\sigma = \int_{\Gamma_{j}^-} \hat{\bm{\phi}}_{j;n,q}^{-}\cdot \overline{ \bm{\xi}_{m,p}} ~\d\sigma,
	\end{equation}
	\begin{equation}\label{cap_selfad}
	\int_{\Gamma_{j+1}^+} \hat{\bm\psi}_{j+1;m,p}^{+} \cdot \overline{\bm{\xi}_{n,q}} ~\d\sigma  = \int_{\Gamma_{j}^-} \hat{\bm\varphi}_{j;n,q}^{-} \cdot \overline{\bm{\xi}_{m,p}} ~\d\sigma.
	\end{equation}
\end{lem}
\begin{proof}[\bf Proof]
	By using \eqref{stifD'j}, and the self-adjointness of the operator $-\mathbb{S}_{j,j+1}$, we imediately have that $\textbf{\textup{Stiff}}(D'_j;D_i)$ is symmetric. It follows from the positive definiteness of the operator $-\mathbb{S}_{j,j+1}$ that for any $\bm{a} = (a_p)_{p=1}^6\in \RR^6$, 
	\[
	\begin{aligned}
	\bm{a}^t\textbf{\textup{Stiff}}(D'_j;D_i)\bm{a} &= \sum_{p,q=1}^{6}a_p \( \(-\mathbb{S}_{j,j+1}^{-1} [\bm\Upsilon_{i,p;2}],\bm\Upsilon_{i,q;2}\)_{L^2(\Gamma_j^-)^3\times L^2(\Gamma_{j+1}^+)^3}\) a_q \\
	&=  \( -\mathbb{S}_{j,j+1}^{-1} \left[\sum_{p=1}^{6}a_p\bm\Upsilon_{i,p;2}\right],\sum_{p=1}^{6}a_p\bm\Upsilon_{i,p;2}\)_{L^2(\Gamma_j^-)^3\times L^2(\Gamma_{j+1}^+)^3}\geq  0,
	\end{aligned}
	\]
	with equality if and only if $\bm{a}=0$. Moreover, using again the positive definiteness and self-adjointness of the operator $-\mathbb{S}_{j,j+1}$, we can conclude that \eqref{cap_pp}--\eqref{cap_selfad} hold true.
\end{proof}

Based on \eqref{Uipm},
we introduce the $6$-by-$6$ matrices
\begin{equation}\label{6L6C}
\bm\Upsilon_{i;1}:=\(\bm\Upsilon_{i,1;1},\bm\Upsilon_{i,2;1},\ldots,\bm\Upsilon_{i,6;1}\)\;\;\mbox{ and }\;\;\bm\Upsilon_{i;2}:=\(\bm\Upsilon_{i,1;2},\bm\Upsilon_{i,2;2},\ldots,\bm\Upsilon_{i,6;2}\),
\end{equation}
and let
\[
\begin{pmatrix}
\hat{\bm\phi}_{j;i}^{-}\\
\nm
\hat{\bm{\psi}}_{j+1;i}^{+}
\end{pmatrix}
:=\mathbb{S}_{j,j+1}^{-1}[\bm\Upsilon_{i;1}] 
\;\mbox{ and }\; 
\begin{pmatrix}
\hat{\bm\varphi}_{j;i}^{-}\\
\nm
\hat{\bm{\zeta}}_{j+1;i}^{+}
\end{pmatrix}
:=\mathbb{S}_{j,j+1}^{-1} [\bm\Upsilon_{i;2}].
\]  

\begin{prop}
	If $\omega \in \mathbb{C}$ belongs to a neighborhood of zero, then there exists a family of $6N\times 6N$ matrices $(\mathcal{T}_{n})_{n\geq 0}$ such that the convergent series representation holds:
	\begin{equation}
	\label{DTN_expansion}
	\mathcal{T}^{\omega} =
	\sum_{n=0}^{+ \infty} \omega^{n} \mathcal{T}_{n}.
	\end{equation}
	Particularly,  for any fixed $i=1,2,\ldots,N$, let 
	\[
	\bm{f}_j^{\pm} = \sum_{p=1}^6 t_{j,p}^\pm \bm{\xi}_{i,p} := \bm{\xi}_{i}\bm{t}_j^{\pm},\;\;j=1,2,\ldots,N,
	\]
	where $\bm{\xi}_{i} = (\bm{\xi}_{i,1},\bm{\xi}_{i,2},\ldots,\bm{\xi}_{i,6})$, and $\bm{t}_j^{\pm} = (t_{j,1}^\pm,t_{j,2}^\pm,\ldots,t_{j,6}^\pm)^T$ is an any fixed 6-dimensional vector, then we have
	\begin{equation} \label{T_0expre}
	\begin{cases}
	\ds\mathcal{T}_0[\bm f]_1^{+} = \hat{\bm\zeta}_{1;i}^+ \bm t_1^+, &  \\
	\nm
	\ds   \mathcal{T}_0[\bm f]_j^{+} = \hat{\bm \zeta}_{j;i}^{+} (\bm t_{j}^{+}-\bm t_{j-1}^{-}) , &j =2,3,\ldots,N, \\
	\nm
	\ds   \mathcal{T}_0[\bm f]_j^{-} = \hat{\bm\varphi}_{j;i}^{-}(\bm t_j^{-}-\bm t_{j+1}^{+}) 
	, & j =1,2,\ldots,N-1,  \\
	\nm
	\ds  \mathcal{T}_0[\bm f]_N^{-} = \bm 0,&
	\end{cases}
	\end{equation}
	where $\hat{\bm\zeta}_{1;i}^+ = \(\hat{\bm\zeta}_{1;i,1}^+,\hat{\bm\zeta}_{1;i,2}^+,\ldots,\hat{\bm\zeta}_{1;i,6}^+\) $ with $\hat{\bm\zeta}_{1;i,p}^+= \mathbf{S}_{\Gamma_1^+}^{-1}[\bm{\xi}_{i,p}]$, 
	and 
	\begin{equation} \label{T_1expre}
	\begin{cases}
	\ds\mathcal{T}_1[\bm f]_1^{+} 
	%= \frac{\i}{4\pi}||\zeta_{1}^+||^2_{L^2(\Gamma_1^+)}f_1^+ 
	= %\i\alpha\sum_{p=1}^3 \(\int_{\Gamma_1^+} \mathcal{T}_0[\bm{\xi}_{i,p}]_{1}^{+}\cdot\bm f_1^+  ~\d\sigma\)  \mathcal{T}_0[\bm{\xi}_{i,p}]_{1}^{+} = 
	\ii\alpha\sum_{p=1}^3 \(\int_{\Gamma_1^+} \hat{\bm\zeta}_{1;i,p}^+ \cdot\overline{\bm f_1^+}  ~\d\sigma\)  \hat{\bm\zeta}_{1;i,p}^+, &  \\
	\nm
	\ds   \mathcal{T}_1[\bm f]_j^{+} =\bm  0 , &j =2,3,\ldots,N, \\
	\nm
	\ds   \mathcal{T}_1[\bm f]_j^{-} =\bm  0
	, & j =1,2,\ldots,N.
	\end{cases}
	\end{equation}
%	where $Cap(D'_0) := \langle-\mathcal{S}_{\Gamma_1^+}^{-1}[1],1\rangle_{L^2(\Gamma_1^+)}$.
\end{prop}
\begin{proof}[\bf Proof]
	The convergent series representation \eqref{DTN_expansion} of displacement-to-traction map $\mathcal{T}^{\omega}$ follows immediately from the operators $\(\frac{\mathbf{I}}{2}+\mathbf{K}_{{\Gamma_1^+}}^{\omega,*}\)(\mathbf{S}_{\Gamma_1^+}^{\omega})^{-1}$, $\(-\frac{\mathbf{I}}{2}+\mathbf{K}_{{\Gamma_N^-}}^{\omega,*}\)(\mathbf{S}_{\Gamma_N^-}^{\omega})^{-1}$ and 
	$\mathbb{A}_{j,j+1}^{\omega}$, $j=1,2,\ldots,N-1$, are analytic for $\omega$ small enough. 
	
	Next, we give the proofs of \eqref{T_0expre} and \eqref{T_1expre}. 
%	we have
%	\[
%	\mathcal{T}_0 = \begin{pmatrix}
%	\(\frac{\mathbf{I}}{2}+\mathbf{K}_{{\Gamma_1^+}}^{*}\)\mathbf{S}_{\Gamma_1^+}^{-1}&&&&\\
%	&\mathbb{A}^0_{1,2}&&&\\
%	&&\ddots &&\\
%	&&&\mathbb{A}^0_{N-1,N}&\\
%	&&&&\(-\frac{\mathbf{I}}{2}+\mathbf{K}_{{\Gamma_N^-}}^{*}\) \mathbf{S}_{{\Gamma_N^-}}^{-1}
%	\end{pmatrix}.
%	\]
	By using \eqref{singlelayerasymptotic}, \eqref{0105}, \eqref{matrixrepre}  and  Calder\'on’s identity (cf. \cite{AK_book2018,KZDF_JCP}), we have 
	\begin{equation}\label{T0f1}
	\mathcal{T}_0[\bm{f}]_{1}^{+} = \(\frac{\mathbf{I}}{2}+\mathbf{K}_{{\Gamma_1^+}}^{*}\)\mathbf{S}_{\Gamma_1^+}^{-1}[\bm f_1^+] = \mathbf{S}_{\Gamma_1^+}^{-1}[\bm{\xi}_{i}]\bm{t}_1^{+},
	\end{equation}
	\[
	\mathcal{T}_0[\bm f]_{N}^{-} = \(-\frac{\mathbf{I}}{2}+\mathbf{K}_{{\Gamma_N^-}}^{*}\)\mathbf{S}_{\Gamma_N^-}^{-1}[\bm f_N^-] = 0.
	\]
	Moreover, by using \eqref{Kjjp1} and 
	\[
	\begin{pmatrix}
	\bm f_j^{-}\\
	\nm
	\bm f_{j+1}^{+}
	\end{pmatrix} = \(\bm\Upsilon_{i;1},\bm\Upsilon_{i;2}\)\begin{pmatrix}
	\bm t_j^{-}\\
	\nm
	\bm t_{j+1}^{+}\end{pmatrix},
	\]
	where $\bm\Upsilon_{i;m}$, $m=1,2$, is defined in \eqref{6L6C}, we get 
\[
\begin{aligned}
\begin{pmatrix}
\mathcal{T}_0[\bm f]_j^{-}\\
\nm 
\mathcal{T}_0[\bm f]_{j+1}^{+}
\end{pmatrix} :&=  \begin{pmatrix}
-\mathbf{I} & \mathbf{O} \\
\nm
\mathbf{O}  &\mathbf{I}
\end{pmatrix}\(\frac{1}{2}\mathbb{I}+\mathbb{K}^{*}_{j,j+1}\)\mathbb{S}_{j,j+1}^{-1} \begin{pmatrix}
\bm f_j^{-}\\
\nm
\bm f_{j+1}^{+}
\end{pmatrix}\\ 
&=    \begin{pmatrix}
-\mathbf{I} & \mathbf{O} \\
\nm
\mathbf{O}  &\mathbf{I}
\end{pmatrix}
\mathbb{S}_{j,j+1}^{-1}\(\frac{1}{2}\mathbb{I}+\mathbb{K}_{j,j+1}\) \(\bm\Upsilon_{i;1},\bm\Upsilon_{i;2}\)\begin{pmatrix}
\bm t_j^{-}\\
\nm
\bm t_{j+1}^{+}\end{pmatrix}\\
& = \begin{pmatrix}
-\mathbf{I} & \mathbf{O} \\
\nm
\mathbf{O}  &\mathbf{I}
\end{pmatrix}
\mathbb{S}_{j,j+1}^{-1} \(-\bm\Upsilon_{i;2},\bm\Upsilon_{i;2}\)\begin{pmatrix}
\bm t_j^{-}\\
\nm
\bm t_{j+1}^{+}\end{pmatrix}\\
%&= \begin{pmatrix}
%-\mathbf{I} & \mathbf{O} \\
%\nm
%\mathbf{O}  &\mathbf{I}
%\end{pmatrix}
%\begin{pmatrix}
%-\hat{\bm\varphi}_{j;i}^{-}&\hat{\bm\varphi}_{j;i}^{-}\\
%\nm
%-\hat{\bm{\zeta}}_{j+1;i}^{+}& \hat{\bm{\zeta}}_{j+1;i}^{+}
%\end{pmatrix}\begin{pmatrix}
%\bm t_j^{-}\\
%\nm
%\bm t_{j+1}^{+}\end{pmatrix}\\
&= 
\begin{pmatrix}
\hat{\bm\varphi}_{j;i}^{-}&-\hat{\bm\varphi}_{j;i}^{-}\\
\nm
-\hat{\bm{\zeta}}_{j+1;i}^{+}& \hat{\bm{\zeta}}_{j+1;i}^{+}
\end{pmatrix}\begin{pmatrix}
\bm t_j^{-}\\
\nm
\bm t_{j+1}^{+}\end{pmatrix},
%& = \begin{pmatrix}
%-\mathbf{I} & \mathbf{O} \\
%\nm
%\mathbf{O}  &\mathbf{I}
%\end{pmatrix}
%\mathbb{S}_{j,j+1}^{-1} \begin{pmatrix}
%\bm{0}\\
%\bm f_{j+1}^{+}-\bm f_j^{-}
%\end{pmatrix}\\
%& = (\bm f_{j+1}^{+} - \bm f_j^{-}) \begin{pmatrix}
%-\hat\varphi_{j}^{-}\\
%\hat\zeta_{j+1}^{+}
%\end{pmatrix},
\quad j =1,2,\ldots,N-1.
\end{aligned}
\]	
	Hence, the formula \eqref{T_0expre} holds ture. 
%	Moreover, it is easy to see that 
%	\[
%	\mathcal{T}_1 = 
%	\begin{pmatrix}
%	-\(\frac{\mathbf{I}}{2}+\mathbf{K}_{{\Gamma_1^+}}^{*}\)\mathbf{S}_{\Gamma_1^+}^{-1}\mathbf{S}_{\Gamma_1^+,1}\mathbf{S}_{\Gamma_1^+}^{-1}&&&&\\
%	&\mathbb A_{1,2,(1)}&&&\\
%	&&\ddots &&\\
%	&&&\mathbb A_{N-1,N,(1)}&\\
%	&&&&-\(\frac{\mathbf{I}}{2}-\mathbf{K}_{{\Gamma_N^-}}^{*}\) \mathbf{S}_{{\Gamma_N^-}}^{-1}\mathbf{S}_{{\Gamma_N^-},1}\mathbf{S}_{{\Gamma_N^-}}^{-1}
%	\end{pmatrix},
%	\]
%	with 
%	\[
%\mathbb	A_{j,j+1,(1)} =-  \begin{pmatrix}
%	-\mathbf{I}  & \mathbf{O} \\
%	\nm 
%	\mathbf{O}  & \mathbf{I}
%	\end{pmatrix}\(\frac{1}{2}\mathbb{I}+\mathbb{K}^{*}_{j,j+1}\) \mathbb{S}_{j,j+1}^{-1}\mathbb{S}_{j,j+1;1}\mathbb{S}_{j,j+1}^{-1},\quad j =1,2,\ldots,N-1.
%	\]
Moreover, it follows from \eqref{SL1}, \eqref{0105} and \eqref{T0f1} that
	\[
	\begin{aligned}
	\mathcal{T}_1[\bm f]_{1}^{+} &= -\(\frac{\mathbf{I}}{2}+\mathbf{K}_{{\Gamma_1^+}}^{*}\)\mathbf{S}_{\Gamma_1^+}^{-1}\mathbf{S}_{\Gamma_1^+,1}\mathbf{S}_{\Gamma_1^+}^{-1}[\bm f_1^+]\\
%	&=\ii\alpha \(\frac{\mathbf{I}}{2}+\mathbf{K}_{{\Gamma_1^+}}^{*}\)\mathbf{S}_{\Gamma_1^+}^{-1}\int_{\Gamma_1^+}\mathbf{S}_{\Gamma_1^+}^{-1}[\bm f_1^+]~\d\sigma\\
	& = \ii\alpha\sum_{p=1}^3 \(\frac{\mathbf{I}}{2}+\mathbf{K}_{{\Gamma_1^+}}^{*}\)\mathbf{S}_{\Gamma_1^+}^{-1}\(\int_{\Gamma_1^+}\mathbf{S}_{\Gamma_1^+}^{-1}[\bm f_1^+]\cdot \bm{\xi}_{i,p} ~\d\sigma\)  \bm{\xi}_{i,p}\\
	& = \ii\alpha\sum_{p=1}^3 \(\int_{\Gamma_1^+} \hat{\bm\zeta}_{1;i,p}^+ \cdot\bm f_1^+  ~\d\sigma\)  \hat{\bm\zeta}_{1;i,p}^+,
	\end{aligned}
	\]
	and 
	\[
	\mathcal{T}_1[\bm f]_{N}^{-} = -\(\frac{\mathbf I}{2}-\mathbf{K}_{{\Gamma_N^-}}^{*}\) \mathbf{S}_{{\Gamma_N^-}}^{-1}\mathbf{S}_{{\Gamma_N^-},1}\mathbf{S}_{{\Gamma_N^-}}^{-1}[\bm f]_{N}^{-}  =\bm 0.
	\]
	Moreover, from \eqref{S^-101GS}, we have
	\[
	\begin{aligned}
	\mathbb{S}_{j,j+1}^{-1}
	\begin{pmatrix}
	\bm f_j^{-}\\
	\nm
	\bm f_{j+1}^{+}
	\end{pmatrix} = \mathbb{S}_{j,j+1}^{-1}\(\bm\Upsilon_{i;1},\bm\Upsilon_{i;2}\)\begin{pmatrix}
	\bm t_j^{-}\\
	\nm
	\bm t_{j+1}^{+}\end{pmatrix}= \begin{pmatrix}
	\hat{\bm\phi}_{j;i}^{-}&\hat{\bm\varphi}_{j;i}^{-}\\
	\nm
	\hat{\bm{\psi}}_{j+1;i}^{+}&\hat{\bm{\zeta}}_{j+1;i}^{+}
	\end{pmatrix}\begin{pmatrix}
	\bm t_j^{-}\\
	\nm
	\bm t_{j+1}^{+}\end{pmatrix} = \begin{pmatrix}
	 \hat{\bm\phi}_{j;i}^{-} \bm t_j^{-} + \hat{\bm\varphi}_{j;i}^{-}\bm t_{j+1}^{+}\\
	\nm
	\hat{\bm{\psi}}_{j+1;i}^{+} \bm t_j^{-} + \hat{\bm{\zeta}}_{j+1;i}^{+}\bm t_{j+1}^{+}\end{pmatrix}
	\end{aligned},
	\]
	and by \eqref{SL1}, we get 
	\[
	\begin{aligned}
&\quad	\mathbb{S}_{j,j+1;1}\begin{pmatrix}
	\hat{\bm\phi}_{j;i}^{-} \bm t_j^{-} + \hat{\bm\varphi}_{j;i}^{-}\bm t_{j+1}^{+}\\
	\nm
	\hat{\bm{\psi}}_{j+1;i}^{+} \bm t_j^{-} + \hat{\bm{\zeta}}_{j+1;i}^{+}\bm t_{j+1}^{+}\end{pmatrix} = 	\begin{pmatrix}
	\mathbf{S}_{{\Gamma_j^-},1}  & \mathbf{S}_{{\Gamma_j^-,\Gamma_{j+1}^+},1} \\
	\nm 
	\mathbf{S}_{{\Gamma_{j+1}^+,\Gamma_j^-},1}  & \mathbf{S}_{{\Gamma_{j+1}^+},1}
	\end{pmatrix}\begin{pmatrix}
	\hat{\bm\phi}_{j;i}^{-} \bm t_j^{-} + \hat{\bm\varphi}_{j;i}^{-}\bm t_{j+1}^{+}\\
	\nm
	\hat{\bm{\psi}}_{j+1;i}^{+} \bm t_j^{-} + \hat{\bm{\zeta}}_{j+1;i}^{+}\bm t_{j+1}^{+}\end{pmatrix} \\
	&= -\ii \alpha \begin{pmatrix}
	\int_{\Gamma_j^-}\(\hat{\bm\phi}_{j;i}^{-} \bm t_j^{-} + \hat{\bm\varphi}_{j;i}^{-}\bm t_{j+1}^{+}\)  ~\d\sigma + \int_{\Gamma_{j+1}^{+}}\(\hat{\bm{\psi}}_{j+1;i}^{+} \bm t_j^{-} + \hat{\bm{\zeta}}_{j+1;i}^{+}\bm t_{j+1}^{+}\)  ~\d\sigma\\
	\nm 
	\int_{\Gamma_j^-}\(\hat{\bm\phi}_{j;i}^{-} \bm t_j^{-} + \hat{\bm\varphi}_{j;i}^{-}\bm t_{j+1}^{+}\)  ~\d\sigma + \int_{\Gamma_{j+1}^{+}}\(\hat{\bm{\psi}}_{j+1;i}^{+} \bm t_j^{-} + \hat{\bm{\zeta}}_{j+1;i}^{+}\bm t_{j+1}^{+}\)  ~\d\sigma
	\end{pmatrix}.
	\end{aligned}
	\]
	Hence, 	we finally obtain that
	\[
	\begin{aligned}
	&\quad \begin{pmatrix}
	\mathcal{T}_1[\bm f]_j^{-}\\
	\mathcal{T}_1[\bm f]_{j+1}^{+}
	\end{pmatrix} :=  -\begin{pmatrix}
	-\mathbf{I}  & \mathbf{O} \\
	\nm 
	\mathbf{O}  & \mathbf{I}
	\end{pmatrix}\(\frac{1}{2}\mathbb{I}+\mathbb{K}^{*}_{j,j+1}\)\mathbb{S}_{j,j+1}^{-1}\mathbb{S}_{j,j+1;1}\mathbb{S}_{j,j+1}^{-1}
	\begin{pmatrix}
	\bm f_j^{-}\\
	\bm f_{j+1}^{+}
	\end{pmatrix}\\ 
	&=    \ii\alpha \begin{pmatrix}
	-\mathbf{I}  & \mathbf{O} \\
	\nm 
	\mathbf{O}  & \mathbf{I}
	\end{pmatrix}
	\mathbb{S}_{j,j+1}^{-1}\(\frac{1}{2}\mathbb{I}+\mathbb{K}_{j,j+1}\) 
	\begin{pmatrix}
	\int_{\Gamma_j^-}\(\hat{\bm\phi}_{j;i}^{-} \bm t_j^{-} + \hat{\bm\varphi}_{j;i}^{-}\bm t_{j+1}^{+}\)  ~\d\sigma + \int_{\Gamma_{j+1}^{+}}\(\hat{\bm{\psi}}_{j+1;i}^{+} \bm t_j^{-} + \hat{\bm{\zeta}}_{j+1;i}^{+}\bm t_{j+1}^{+}\)  ~\d\sigma\\
	\nm 
	\int_{\Gamma_j^-}\(\hat{\bm\phi}_{j;i}^{-} \bm t_j^{-} + \hat{\bm\varphi}_{j;i}^{-}\bm t_{j+1}^{+}\)  ~\d\sigma + \int_{\Gamma_{j+1}^{+}}\(\hat{\bm{\psi}}_{j+1;i}^{+} \bm t_j^{-} + \hat{\bm{\zeta}}_{j+1;i}^{+}\bm t_{j+1}^{+}\)  ~\d\sigma
	\end{pmatrix}\\
	& = \begin{pmatrix}
	\bm 0\\
	\nm
	\bm 0\end{pmatrix},\quad j =1,2,\ldots,N-1,
	\end{aligned}
	\]
	where we used the fact that $\(\frac{1}{2}\mathbb{I}+\mathbb{K}_{j,j+1}\)  \begin{pmatrix}
	\bm b\\
	\nm 
	\bm b
	\end{pmatrix} = \begin{pmatrix}
	\bm 0\\
	\nm
	\bm 0\end{pmatrix}$, for any $\bm b\in \RR^3.$
	The proof is complete.
\end{proof}

\subsection{Variational characterization of subwavelength resonances}
Based on the displacement-to-traction map introduced above, we now provide a variational characterization of resonances in the Matryoshka-type elastic medium $D = \cup_{j=1}^{N}D_{j}$, as illustrated in Figure \ref{MLHCCB}.
The  elastic scattering problem \eqref{main_equation} in $D$ can be characterized in terms of the displacement-to-traction map $\mathcal{T}^\omega$ for $\mathbf{u}\in H^{1}(D)^3$ as follows:
\begin{equation} \label{D2N}
\begin{cases}
\ds\mathcal{L}_{\lambda,\mu } \mathbf{u}  +\sum_{j=1}^N\rho\tau_j^2\chi_{D_j}\omega^2  \mathbf{u}  = 0, & \text{in  } D, \\
\nm
\ds   \partial_{\bm{\nu}} \mathbf{u}|_\mp = \delta \mathcal{T}^{\omega}[\mathbf{u}-\mathbf{u}^{\textup{in}}]_j^\pm + \delta \partial_{\bm{\nu}}{\mathbf{u}^{\textup{in}}} , & \text{on }\Gamma^\pm_j, \; j =1,2,\ldots,N.
\end{cases}
\end{equation}

By using variational principle and Green's formula for Lam\'e operator (cf. \cite{HAmmariElasticityImaging}), \eqref{D2N} can be formulated as the following weak form:
\begin{equation}\label{eq:varpro}
\bm{a}(\mathbf{u},\mathbf{v}) = \langle \bm{h},\mathbf{v}\rangle_{{H}^{-1}(D)^3,{H}^1(D)^3},  \text{ for }\forall\, \mathbf{v}\in {H}^1(D)^3,
\end{equation}
where the bilinear form $\bm{a}(\mathbf{u},\mathbf{v})$ for $\mathbf{u},\mathbf{v}\in {H}^1(D)^3 $ is defined by
\[
\bm{a}(\mathbf{u},\mathbf{v}):=
\sum_{i=1}^N\left( Q_{\lambda,\mu,D_i}(\mathbf{u},\mathbf{v}) - \rho\tau_i^2\omega^2 \int_{D_i} \mathbf{u}\cdot \overline{\mathbf{v}} ~\d \mathbf{x}\right)  - \delta\sum_{i=1}^N  \left(\int_{\Gamma_i^+}
\overline{\mathbf{v}}\cdot \mathcal{T}^{\omega}[\mathbf{u}]_i^+~\d \sigma  -\int_{\Gamma_i^-}
\overline{\mathbf{v}}\cdot \mathcal{T}^{\omega}[\mathbf{u}]_i^- ~\d \sigma \right),
\]
and
\begin{equation}\label{rightlinear}
\langle \bm{h},\mathbf{v}\rangle_{{H}^{-1}(D)^3,{H}^1(D)^3} := \delta \sum_{i=1}^N \left(\int_{\Gamma_i^+} \overline{\mathbf{v}}\cdot \(-\mathcal{T}^{\omega}[\mathbf{u}^{\textup{in}}]_i^+ +  \partial_{\bm{\nu}} \mathbf{u}^{\textup{in}} \)~\d \sigma - \int_{\Gamma_i^-} \overline{\mathbf{v}}\cdot \(-\mathcal{T}^{\omega}[\mathbf{u}^{\textup{in}}]_i^- +  \partial_{\bm{\nu}}\mathbf{u}^{\textup{in}}\)~\d \sigma\right),
\end{equation}
where  
\begin{equation}\label{QLambdaMu}
Q_{\lambda,\mu,D_i}(\mathbf{u},\mathbf{v}) =  \int_{D_i}\(\lambda(\nabla \cdot\mathbf{u})(\overline{\nabla\cdot\mathbf{v}}) +\frac{\mu}{2} \(\nabla\mathbf{u}+ (\nabla\mathbf{u})^t\):\(\overline{ \nabla\mathbf{v} +(\nabla\mathbf{v})^t}\) \)~\d \mathbf{x},
\end{equation}
and the notation $\mathbf{A}:\mathbf{B}=\mathrm{tr} (\mathbf{A}\mathbf{B}^t)$ is used for matrices $\mathbf{A}$ and $\mathbf{B}$. Moreover, the strong convexity condition \eqref{strong_convexity} shows that the quadratic form $\mathbf{u}\in  {H}_{0}^1(D)^3  \mapsto  Q_{\lambda,\mu,D_i}(\mathbf{u},\mathbf{u})$ is positive definite \cite{HAmmariElasticityImaging}.

Next, we introduce a new bilinear form:
\begin{equation}\label{NBF}
\bm{a}_{\omega,\delta}(\mathbf{u},\mathbf{v}) := \bm{a}(\mathbf{u},\mathbf{v})+\sum_{i=1}^N \(\sum_{p=1}^6\(\int_{D_i}\mathbf{u}\cdot\overline{\bm{\xi}_{i,p}}~\d \Bx \int_{D_i}\overline{\mathbf{v}} \cdot\bm{\xi}_{i,p}~\d \Bx\)\),
\end{equation}
where $\{\bm{\xi}_{i,p}\}_{p=1}^6$ is an orthogonal basis with respect to the $L^2(D_i)$ inner product, obtained by applying the Gram--Schmidt orthogonalization process to the rigid motion basis $\{\bm{\varkappa}_{p}\}_{p=1}^6$ given by \eqref{rigid_motions}. 
For sufficiently small $\delta$ and $\omega$, the bilinear form $a_{\omega, \delta}$ is an analytic perturbation of the continuous, coercive bilinear form
\[
\bm{a}_{0,0}(\mathbf{u},\mathbf{v}) = \sum_{i=1}^N Q_{\lambda,\mu,D_i}(\mathbf{u},\mathbf{v})   +\sum_{i=1}^N \(\sum_{p=1}^6\(\int_{D_i}\mathbf{u}\cdot\overline{\bm{\xi}}_{i,p}~\d \Bx \int_{D_i}\overline{\mathbf{v}} \cdot\bm{\xi}_{i,p}~\d \Bx\)\),
\]
hence remains coercive.
Consequently, for any $\bm{h}\in H^{-1}(D)^3$, the Lax–Milgram lemma yields a unique solution  $\mathbf{u}_{\bm{h}}(\omega ,\delta )$ to the problem
\begin{equation}\label{Lax_Milgram_solu}
\bm{a}_{\omega,\delta}(\mathbf{u}_{\bm{h}}(\omega,\delta),\mathbf{v}) = \langle \bm{h},\mathbf{v}\rangle_{H^{-1}(D)^3,H^1(D)^3},  \text{ for }\forall\, \mathbf{v}\in H^1(D)^3,
\end{equation}
which is analytic with respect to $\omega$ and $\delta$ (cf. \cite{HW_book,FA_JMPA2024}).
In order to characterize the subwavelength resonances, we introduce $\mathbf{u}_{j,q}(\omega,\delta)$ with $q = 1,2,\ldots,6$ and $j=1,2,\ldots,N$, satisfying the following variational problems
\begin{equation}\label{ujq}
\bm{a}_{\omega,\delta}(\mathbf{u}_{j,q}(\omega,\delta),\mathbf{v})=\int_{D_j}\overline{\mathbf{v}} \cdot\bm{\xi}_{j,q}\d \Bx,\;   \text{ for }\forall\, \mathbf{v}\in H^1(D)^3.
\end{equation}

We next introduce  the $6N$-dimensional vectors
\begin{equation}\label{xF}
\bm{s}(\omega,\delta) = \begin{pmatrix}
\bm{s}_1(\omega,\delta)\\\bm{s}_2(\omega,\delta)\\ \vdots\\ \bm{s}_6(\omega,\delta)
\end{pmatrix}, 
\bm{H}(\omega,\delta) = \begin{pmatrix}
\bm{H}_1(\omega,\delta)\\\bm{H}_2(\omega,\delta)\\ \vdots\\ \bm{H}_6(\omega,\delta)
\end{pmatrix},
\end{equation}
and the $6N$-by-$6N$ matrix
\begin{equation}\label{6NU}
	\bm{U}(\omega,\delta):=(\bm{U}_{pq}(\omega,\delta))_{1\leq p,q\leq 6}.
\end{equation}
Here each sub-vector is an 
$N$-dimensional vector defined as
\[
\bm{s}_p(\omega ,\delta):=\(s_{p,i}(\omega ,\delta)\)_{1\leq i \leq N}:=\(\int_{D_i}\mathbf{u}(\omega ,\delta )\cdot\overline{\bm{\xi}}_{i,p} ~\d \Bx\)_{1\leq i \leq N},
\] 
and 
\[
\bm{H}_p:=\(H_{p,i}(\omega ,\delta)\)_{1\leq i \leq N}:=\(\int_{D_i} \mathbf{u}_{\bm{h}}(\omega ,\delta )\cdot\overline{\bm{\xi}}_{i,p}~\d \Bx\)_{1\leq i \leq N},
\] 
respectively, each sub-matrix $\bm{U}_{pq}(\omega,\delta)$ is an 
$N$-by-$N$ matrix defined as
\begin{equation}\label{6NUpq}
\bm{U}_{pq}(\omega,\delta):= (\bm{U}_{pqij}(\omega,\delta))_{1\leq i,j\leq N}:=
\left(\int_{D_i}\mathbf{u}_{j,q}\cdot\overline{\bm{\xi}_{i,p}}~\d \Bx\right)_{1\leq i,j\leq N}.
\end{equation}

\begin{lem}\label{lem32}
	Let $\omega \in \mathbb{C}$ and $\delta \in \mathbb{R}$ belong to a neighborhood of zero.  For any  $\bm{h}\in H^{-1}(D)^3$, the variational problem \eqref{eq:varpro} has a unique solution $\mathbf{u}:= \mathbf{u}(\omega,\delta)$ if and only if 
	\begin{equation}\label{eq654}
	\det (\bm{I} - \bm{U}(\omega,\delta))\neq 0.
	\end{equation}
	Moreover, when the condition \eqref{eq654} is satisfied, the solution to \eqref{eq:varpro} can be given by
	\begin{equation}\label{eq656}
	\mathbf{u}(\omega,\delta) = \sum_{j=1}^N \sum_{q=1}^6 {s}_{q,j}(\omega ,\delta)\mathbf{u}_{j,q}(\omega ,\delta)+\mathbf{u}_{\bm{h}}(\omega,\delta),
	\end{equation}
where $\mathbf{u}_{\bm{h}}(\omega,\delta)$ and $\mathbf{u}_{j,q}(\omega ,\delta)$  are the unique solutions to the variational problems \eqref{Lax_Milgram_solu} and \eqref{ujq},
respectively, and $\bm{s}(\omega,\delta)$ defined in \eqref{xF} satisfies the following $6N\times 6N$ linear system:
\begin{equation}\label{6Nsystem}
\(\bm{I} - \bm{U}(\omega,\delta)\)\bm{s}(\omega,\delta) = \bm{H}(\omega,\delta).
\end{equation}
\end{lem}

\begin{proof}[\bf Proof]
	From \eqref{NBF} and \eqref{Lax_Milgram_solu}, we have that
	the variational problem \eqref{eq:varpro}  is equivalent to
	\[
	\bm{a}_{\omega,\delta}(\mathbf{u},\mathbf{v})-\sum_{j=1}^N\(\sum_{q=1}^6 \(\bm{a}_{\omega,\delta}(\mathbf{u}_{j,q},\mathbf{v})\int_{D_j}\mathbf{u}\cdot\overline{\bm{\xi}_{j,q}}~\d \Bx \)\)=\bm{a}_{\omega,\delta}(\mathbf{u}_{\bm{h}}(\omega ,\delta ),\mathbf{v}),
	\]
	which implies that
	\begin{equation}\label{eqn:q48jq}
		\mathbf{u}-\sum_{j=1}^N \(\sum_{q=1}^6\(\mathbf{u}_{j,q}\int_{D_j}\mathbf{u}\cdot\overline{\bm{\xi}}_{j,q} ~\d \Bx\) \)=\mathbf{u}_{\bm{h}}(\omega ,\delta ).
	\end{equation}
	It follows that for $i = 1,2,\ldots,N$ and $p = 1,2,\ldots,6$,
	\[
	\int_{D_i}\mathbf{u}\cdot\overline{\bm{\xi}}_{i,p} ~\d \Bx -\sum_{j=1}^N \(\sum_{q=1}^6\(\int_{D_i} \mathbf{u}_{j,q}\cdot\overline{\bm{\xi}}_{i,p} ~\d \Bx \int_{D_j}\mathbf{u}\cdot\overline{\bm{\xi}}_{j,q}~\d \Bx  \)\)= \int_{D_i} \mathbf{u}_{\bm{h}}(\omega ,\delta )\cdot\overline{\bm{\xi}}_{i,p}~\d \Bx,
	\]
	which is precisely the linear system \eqref{6Nsystem} for the vector $\bm{s}(\omega,\delta)$ defined in \eqref{xF}.  Therefore, the variational problem \eqref{eq:varpro} admits a unique solution $\mathbf{u}(\omega,\delta)$ given by \eqref{eq656} if and only if	the condition \eqref{eq654} is satisfied. 
\end{proof}

Therefore, from the Definition \ref{defn:resonance}, we can obtain the following result.

\begin{prop}\label{TCR}
	Subwavelength resonant frequencies $\omega=\omega(\delta)\in\mathbb{C}$ are determined by $\det(\bm{I} - \bm{U}(\omega,\delta))=0$.
\end{prop}

Consequently, to derive the subwavelength resonant frequencies $\omega=\omega(\delta)\in\mathbb{C}$,  it suffices to analyse the strong form of the variational problem \eqref{ujq}
\begin{equation}\label{SFVP}
\begin{cases}
\ds -\mathcal{L}_{\lambda,\mu } \mathbf{u}_{j,q}  -\sum_{i=1}^N\rho\tau_i^2\chi_{D_i}\omega^2  \mathbf{u}_{j,q} + \sum_{i=1}^{N}
\sum_{p=1}^6\left(\int_{D_i} \mathbf{u}_{j,q}\cdot \overline{\bm{\xi}_{i,p}} ~\d \Bx\right)\bm{\xi}_{i,p} \chi_{D_i} =\bm{\xi}_{j,q}
\chi_{D_j}, &  \text{in } D,\\
\nm
\ds \partial_{\bm{\nu}}\mathbf{u}_{j,q}|_\pm = \delta \mathcal{T}^{\omega}[\mathbf{u}_{j,q}]_i^\mp , & \text{on }\Gamma^\mp_i,\; i =1,2,\ldots,N.
\end{cases}
\end{equation}

\begin{prop}\label{prop33}
	The solutions $\mathbf{u}_{j,q}(\omega, \delta)$, $j=1,2,\ldots,N$ and $q=1,2,\ldots,6$,  to the variational problem \eqref{ujq} have the following asymptotic behavior as $\omega,\delta \to 0:$
	\begin{equation}\label{asy_ujwd}
	\begin{aligned}
	&\quad\mathbf{u}_{j,q}(\omega,\delta)\\ &= \(\frac{1}{\|\bm{\xi}_{j,q}\|_{L^2(D_j)}^2}+ \frac{\rho\tau_j^2\omega^2}{\|\bm{\xi}_{j,q}\|_{L^2(D_j)}^4}\)\bm{\xi}_{j,q}\chi_{D_j}\\
	&\quad +\delta\left[\chi_{D_{j-1}} \chi_{\{2\leq j\leq N\}} \sum_{p=1}^6\frac{\int_{\Gamma_{j-1}^+} \hat{\bm{\varphi}}^{-}_{j-1;j,q} \cdot\overline{\bm{\xi}_{j-1,p}}~\d \sigma}{\|\bm{\xi}_{j-1,p}\|_{L^2(D_{j-1})}^4 \|\bm{\xi}_{j,q}\|_{L^2(D_{j})}^2}  \bm{\xi}_{j-1,p}\right.\\
	&\quad\quad\quad +\chi_{D_{j}} \chi_{\{1\leq j\leq N\}} \sum_{p=1}^6\(\frac{\int_{\Gamma_{j}^+} \hat{\bm{\zeta}}^{+}_{j;j,q} \cdot\overline{\bm{\xi}_{j,p}}~\d \sigma}{\|\bm{\xi}_{j,p}\|_{L^2(D_{j})}^4 \|\bm{\xi}_{j,q}\|_{L^2(D_{j})}^2}-\frac{\int_{\Gamma_{j}^-} \hat{\bm{\varphi}}^{-}_{j;j,q} \cdot\overline{\bm{\xi}_{j,p}}~\d \sigma}{\|\bm{\xi}_{j,p}\|_{L^2(D_{j})}^4 \|\bm{\xi}_{j,q}\|_{L^2(D_{j})}^2}\)  \bm{\xi}_{j,p}\\
	&\quad\quad\quad\left.  - \chi_{D_{j+1}} \chi_{\{1\leq j\leq N-1\}} \sum_{p=1}^6\frac{\int_{\Gamma_{j+1}^+} \hat{\bm{\zeta}}^{+}_{j+1;j,q} \cdot\overline{\bm{\xi}_{j+1,p}}~\d \sigma}{\|\bm{\xi}_{j+1,p}\|_{L^2(D_{j+1})}^4 \|\bm{\xi}_{j,q}\|_{L^2(D_{j})}^2}  \bm{\xi}_{j+1,p} + \widetilde{\mathbf{u}}_{j,q;0,1}\right]\\
	&\quad +\omega\delta\(\frac{\ii\alpha\chi_{D_j}  \delta_{1,j}\delta_{1,i}}{\|\bm{\xi}_{j,q}\|_{L^2(D_j)}^2}\sum_{p=1}^6\frac{1}{\|\bm{\xi}_{i,p}\|_{L^2(D_i)}^4 } \sum_{t=1}^{3}\int_{\Gamma_1^+}\hat{\bm\zeta}^{+}_{1;j,t}\cdot \overline{\bm{\xi}_{j,q}}~\d \sigma  \int_{\Gamma_1^+}\hat{\bm\zeta}^{+}_{1;j,t} \cdot \overline{\bm{\xi}_{i,p}}~\d \sigma\bm{\xi}_{j,p}+\widetilde{\mathbf{u}}_{j,q;1,1}\) +\mathcal{O}((\omega^2+\delta)^2),
	\end{aligned}
	\end{equation}
	where  the functions $\widetilde{\mathbf{u}}_{j,q;0,1}$ and $\widetilde{\mathbf{u}}_{j,q;1,1}$  satisfy
	\[
	\int_{D_i}\widetilde{\mathbf{u}}_{j,q;0,1}\cdot\overline{\bm{\xi}_{i,p}} ~~\d \Bx=0,\quad  \int_{D_i}\widetilde{\mathbf{u}}_{j,q;1,1}\cdot\overline{\bm{\xi}_{i,p}} ~\d \Bx=0,\;\mbox{ for }
	i = 1,2,\ldots, N\mbox{ and }p = 1,2,\ldots,6.
	\]

\end{prop}
\begin{proof}[\bf Proof]
It follows from \eqref{DTN_expansion} that  there exist  $(\mathbf{u}_{j,q;m,n})_{m{\geq}0,	n{\geq}0}$ such that $\mathbf{u}_{j,q}(\omega,\delta)$ admits the convergent expansion
\begin{equation}\label{ujcs}
\mathbf{u}_{j,q}(\omega,\delta) = \sum_{m,n=0}^{+\infty} \omega^{m}\delta^n
\mathbf{u}_{j,q;m,n}.
\end{equation}
Substituting \eqref{ujcs} into \eqref{SFVP} and matching powers of $\omega$ and
$\delta$, we obtain
\begin{equation}\label{ujmn}
\begin{cases}
\ds -\mathcal{L}_{\lambda,\mu } \mathbf{u}_{j,q;m,n}  + \sum_{i=1}^{N}\sum_{p=1}^6
\left({\int_{D_i} \mathbf{u}_{j,q;m,n} \cdot\overline{\bm{\xi}_{i,p}} ~\d \Bx}\right) \bm{\xi}_{i,p} \chi_{D_i} =  \sum_{i=1}^N\rho\tau_i^2\chi_{D_i} \mathbf{u}_{j,q;m-2,n}+
\bm{\xi}_{j,q}\chi_{D_j} \delta_{m,0}\delta_{n,0},&  \text{ in } D,\\
\nm
\ds \partial_{\bm{\nu}} \mathbf{u}_{j,q;m,n}|_\pm =  \sum_{s=0}^{m}\mathcal{T}_s[\mathbf{u}_{j,q;m-s,n-1}]_i^\mp , & \text{on }\Gamma^\mp_i,\; i =1,2,\ldots,N,
\end{cases}
\end{equation}
where $\mathbf{u}_{j,q;m,n}=0$ for negative indices $m$ and $n$.
Taking $m=n=0$ in \eqref{ujmn}, we have that $\mathbf{u}_{j,q;0,0}$ satisfies
\begin{equation}\label{uj00}
\begin{cases}
\ds -\mathcal{L}_{\lambda,\mu } \mathbf{u}_{j,q;0,0}  + \sum_{i=1}^{N}\sum_{p=1}^{6}
\left(\int_{D_i} \mathbf{u}_{j,q;0,0} \cdot\overline{\bm{\xi}_{i,p}} ~\d\Bx\right) \bm{\xi}_{i,p} \chi_{D_i} =  
\bm{\xi}_{j,q}\chi_{D_j}, &  \text{ in } D,\\
\nm
\ds \partial_{\bm{\nu}}\mathbf{u}_{j,q;0,0}|_\pm =  0 , & \text{on }\Gamma^\mp_i,\; i =1,2,\ldots,N,
\end{cases}
\end{equation}
which implies that $\mathbf{u}_{j,q;0,0} = \frac{\bm{\xi}_{j,q}\chi_{D_j}}{\|\bm{\xi}_{j,q}\|_{L^2(D_j)}^2}$.
It follows from induction that
\begin{equation}\label{ujevenodd0}
\mathbf{u}_{j,q;2m,0} = \rho^m\tau_j^{2m}\frac{ \bm{\xi}_{j,q} \chi_{D_j}}{ \|\bm{\xi}_{j,q}\|_{L^2(D_j)}^{2m+2}}  \text{ and }\mathbf{u}_{j,q;2m+1,0} = \bm{0} \text{ for	any } m \geq 0.
\end{equation}

Taking $m=0$ and $n=1$ in \eqref{ujmn}, we obtain that $\mathbf{u}_{j,q;0,1}$ satisfies
\begin{equation}\label{uj01}
\begin{cases}
\ds -\mathcal{L}_{\lambda,\mu } \mathbf{u}_{j,q;0,1}  + \sum_{i=1}^{N}\sum_{p=1}^{6}
\left(\int_{D_i} \mathbf{u}_{j,q;0,1}\cdot\overline{\bm{\xi}_{i,p}} ~\d \Bx\right)\bm{\xi}_{i,p} \chi_{D_i} =  0,&  \text{ in } D,\\
\nm
\ds\partial_{\bm{\nu}} \mathbf{u}_{j,q;0,1}|_\pm =  \mathcal{T}_0[\mathbf{u}_{j,q;0,0}]_i^\mp, & \text{on }\Gamma^\mp_i,\; i =1,2,\ldots,N.
\end{cases}
\end{equation}
From \eqref{T_0expre} with $\bm{f}_i^\pm = \mathbf{u}_{j,q;0,0}|_{\Gamma_i^\pm}=\frac{\delta_{i,j}}{\|\bm{\xi}_{j,q}\|_{L^2(D_j)}^2}\bm{\xi}_{j,q}|_{\Gamma_i^\pm}$, we get
\begin{equation}
\begin{cases}
\ds\mathcal{T}_0[\mathbf{u}_{j,q;0,0}]_1^{+} = \frac{\delta_{1,j}}{\|\bm{\xi}_{j,q}\|_{L^2(D_j)}^2}\hat{\bm{\zeta}}^+_{1;j,q},
&  \\
\nm
\ds   \mathcal{T}_0[\mathbf{u}_{j,q;0,0}]_i^{+} =\frac{\delta_{i,j} - \delta_{i-1,j}}{\|\bm{\xi}_{j,q}\|_{L^2(D_j)}^2}\hat{\bm{\zeta}}^+_{i;j,q},
&i =2,3,\ldots,N, \\
\nm
\ds   \mathcal{T}_0[\mathbf{u}_{j,q;0,0}]_i^{-} = \frac{\delta_{i,j} - \delta_{i+1,j}}{\|\bm{\xi}_{j,q}\|_{L^2(D_j)}^2}\hat{\bm{\varphi}}^{-}_{i;j,q},
& i =1,2,\ldots,N-1,  \\
\nm
\ds  \mathcal{T}_0[\mathbf{u}_{j,q;0,0}]_N^{-} = \bm 0.&
\end{cases}
\end{equation}
Multiplying \eqref{uj01} by $ \overline{\bm{\xi}_{i,p}} \chi_{D_i}$ and using Green's second formula \cite{HAmmariElasticityImaging} for Lam\'e operator, we have that
\begin{equation}
\begin{aligned}
&\quad \int_{D_i} \mathbf{u}_{j,q;0,1} \cdot\overline{\bm{\xi}_{i,p}} ~\d \Bx = \frac{1}{\|\bm{\xi}_{i,p}\|_{L^2(D_i)}^2}\(\int_{\Gamma_i^+} \mathcal{T}_0[\mathbf{u}_{j,q;0,0}]_i^{+} \cdot\overline{\bm{\xi}_{i,p}}~\d \sigma - \int_{\Gamma_i^-} \mathcal{T}_0[\mathbf{u}_{j,q;0,0}]_i^{-} \cdot\overline{\bm{\xi}_{i,p}}~\d \sigma\)\\
& = \frac{1}{\|\bm{\xi}_{i,p}\|_{L^2(D_i)}^2}\( \frac{\delta_{1,j}}{\|\bm{\xi}_{j,q}\|_{L^2(D_j)}^2} \int_{\Gamma_1^+} \hat{\bm{\zeta}}^{+}_{1;j,q} \cdot\overline{\bm{\xi}_{1,p}}~\d \sigma \chi_{\{ i=1\}} + \frac{\delta_{i,j} - \delta_{i-1,j}}{\|\bm{\xi}_{j,q}\|_{L^2(D_j)}^2}
\int_{\Gamma_i^+} \hat{\bm{\zeta}}^{+}_{i;j,q} \cdot\overline{\bm{\xi}_{i,p}}~\d \sigma \chi_{\{ 2\leq i\leq N\}}\right.\\
&\left. \quad\quad\quad\quad\quad\quad - \frac{\delta_{i,j} - \delta_{i+1,j}}{\|\bm{\xi}_{j,q}\|_{L^2(D_j)}^2}
\int_{\Gamma_i^-} \hat{\bm{\varphi}}^{-}_{i;j,q} \cdot\overline{\bm{\xi}_{i,p}}~\d \sigma \chi_{\{ 1\leq i\leq N-1\}}
\).
\end{aligned}
\end{equation}
Hence, $\mathbf{u}_{j,q;0,1}$ can be given by
\begin{equation}
\begin{aligned}
 \mathbf{u}_{j,q;0,1} 
&=\chi_{D_{j-1}} \chi_{\{2\leq j\leq N\}} \sum_{p=1}^6\frac{\int_{\Gamma_{j-1}^-} \hat{\bm{\varphi}}^{-}_{j-1;j,q} \cdot\overline{\bm{\xi}_{j-1,p}}~\d \sigma}{\|\bm{\xi}_{j-1,p}\|_{L^2(D_{j-1})}^4 \|\bm{\xi}_{j,q}\|_{L^2(D_{j})}^2}  \bm{\xi}_{j-1,p}\\
&\quad +\chi_{D_{j}} \chi_{\{1\leq j\leq N\}} \sum_{p=1}^6\(\frac{\int_{\Gamma_{j}^+} \hat{\bm{\zeta}}^{+}_{j;j,q} \cdot\overline{\bm{\xi}_{j,p}}~\d \sigma}{\|\bm{\xi}_{j,p}\|_{L^2(D_{j})}^4 \|\bm{\xi}_{j,q}\|_{L^2(D_{j})}^2}-\frac{\int_{\Gamma_{j}^-} \hat{\bm{\varphi}}^{-}_{j;j,q} \cdot\overline{\bm{\xi}_{j,p}}~\d \sigma}{\|\bm{\xi}_{j,p}\|_{L^2(D_{j})}^4 \|\bm{\xi}_{j,q}\|_{L^2(D_{j})}^2}\)  \bm{\xi}_{j,p}\\
&\quad - \chi_{D_{j+1}} \chi_{\{1\leq j\leq N-1\}} \sum_{p=1}^6\frac{\int_{\Gamma_{j+1}^+} \hat{\bm{\zeta}}^{+}_{j+1;j,q} \cdot\overline{\bm{\xi}_{j+1,p}}~\d \sigma}{\|\bm{\xi}_{j+1,p}\|_{L^2(D_{j+1})}^4 \|\bm{\xi}_{j,q}\|_{L^2(D_{j})}^2}  \bm{\xi}_{j+1,p} + \widetilde{\mathbf{u}}_{j,q;0,1},
\end{aligned}
\end{equation}
where $\hat{\bm{\varphi}}^{-}_{N;j,q}=\bm{0}$, and  $\widetilde{\mathbf{u}}_{j,q;0,1}$  satisfies $\int_{D_i}\widetilde{\mathbf{u}}_{j,q;0,1} \cdot\overline{\bm{\xi}_{i,p}} ~\d \Bx=0$ for $i = 1,2,\ldots, N$ and $p=1,2,\ldots,6$.

Taking $m=n=1$ in \eqref{ujmn}, it folllows from \eqref{ujmn} and \eqref{ujevenodd0} that $\mathbf{u}_{j,q;1,1}$ satisfies
\begin{equation}\label{uj11}
\begin{cases}
\ds -\mathcal{L}_{\lambda,\mu } \mathbf{u}_{j,q;1,1}  + \sum_{i=1}^{N}\sum_{p=1}^6
\left(\int_{D_i} \mathbf{u}_{j,q;1,1} \cdot\overline{\bm{\xi}_{i,p}} ~\d \Bx\right) \bm{\xi}_{i,p} \chi_{D_i} =  0,&  \text{ in } D,\\
\nm
\ds\partial_{\bm{\nu}} \mathbf{u}_{j,q;1,1}|_\pm = \mathcal{T}_1[\mathbf{u}_{j,q;0,0}]_i^\mp , & \text{on }\Gamma^\mp_i,\; i =1,2,\ldots,N.
\end{cases}
\end{equation}
By \eqref{T_1expre} and \eqref{ujevenodd0}, we have that
\[
\mathcal{T}_1[\mathbf{u}_{j,q;0,0}]_i^{+} = \begin{cases}
\ii\alpha\frac{\delta_{1,j}}{\|\bm{\xi}_{j,q}\|_{L^2(D_j)}^2}\sum_{t=1}^{3}\int_{\Gamma_1^+}\hat{\bm\zeta}^{+}_{1;j,t}\cdot \overline{\bm{\xi}_{j,q}}~\d \sigma \hat{\bm\zeta}^{+}_{1;j,t}&\text{ if }i=1,\\
\bm 0&\text{ if }i{\geq}2,
\end{cases}
\text{ and } \mathcal{T}_1[\mathbf{u}_{j,q;0,0}]_i^{-} = \bm 0 \text{  for } i=1,2,\ldots,N.
\]
Consequently,  we obtain that
\[
    \int_{D_i} \mathbf{u}_{j,q;1,1} \cdot\overline{\bm{\xi}_{i,p}} ~\d \Bx=\frac{\ii\alpha \delta_{1,j}\delta_{1,i}}{\|\bm{\xi}_{i,p}\|_{L^2(D_i)}^2 \|\bm{\xi}_{j,q}\|_{L^2(D_j)}^2} \sum_{t=1}^{3}\int_{\Gamma_1^+}\hat{\bm\zeta}^{+}_{1;j,t}\cdot \overline{\bm{\xi}_{j,q}}~\d \sigma  \int_{\Gamma_1^+}\hat{\bm\zeta}^{+}_{1;j,t} \cdot \overline{\bm{\xi}_{i,p}}~\d \sigma,
\]
which implies  that
\[
\mathbf{u}_{j,q;1,1}=\frac{\ii\alpha\chi_{D_j}  \delta_{1,j}\delta_{1,i}}{\|\bm{\xi}_{j,q}\|_{L^2(D_j)}^2}\sum_{p=1}^6\frac{1}{\|\bm{\xi}_{i,p}\|_{L^2(D_i)}^4 } \sum_{t=1}^{3}\int_{\Gamma_1^+}\hat{\bm\zeta}^{+}_{1;j,t}\cdot \overline{\bm{\xi}_{j,q}}~\d \sigma  \int_{\Gamma_1^+}\hat{\bm\zeta}^{+}_{1;j,t} \cdot \overline{\bm{\xi}_{i,p}}~\d \sigma\bm{\xi}_{j,p}+\widetilde{\mathbf{u}}_{j,q;1,1},
\]
where $\widetilde{\mathbf{u}}_{j,q;1,1}$  satisfies $\int_{D_i}
\widetilde{\mathbf{u}}_{j,q;1,1}\cdot\overline{\bm{\zeta}_p} ~\d \Bx=0$ for any $1\leq i\leq N$ and $1\leq p\leq 6$. 
The proof is complete.
\end{proof}

%Next, we introduce the definition of the elastic stiffness tensor, which serves as the elastic counterpart of the generalized capacitance matrix \cite{FA_SAM_2022,ADH_SIMA2023}  in the celebrated Minnaert acoustic-cavitation systems \cite{AFGLZ_AIHPCAN,MPS_JMPA2022}.   The notion of stiffness was originally developed to model the spherical bonded rubber bush mountings in elastostatics \cite{HT_IJSS2005,Hill1975} (for Matryoshka-type spherical bonded rubber bush mountings, see Lemma \ref{EST_sphere}). For the first time, we present the stiffness tensor for arbitrarily shaped Matryoshka-type elastic metamaterials.

Next, we introduce, for the first time, the definition of the elastic stiffness tensor for arbitrarily shaped Matryoshka-type elastic metamaterials. This tensor serves as the elastic counterpart of the generalized capacitance matrix \cite{FA_SAM_2022,ADH_SIMA2023} in the celebrated Minnaert acoustic-cavitation systems \cite{AFGLZ_AIHPCAN,MPS_JMPA2022}. The notion of stiffness was originally developed to model spherical bonded rubber bush mountings in elastostatics \cite{HT_IJSS2005,Hill1975} (for Matryoshka-type spherical bonded rubber bush mountings, see Section \ref{concentric_radial}).

\begin{defn}\label{def: Capacitance matrix}
	Consider the solutions $\mathbf{V}_{j,q} : \RR^3\to \mathbb{R}^3$ of the problem
	\begin{equation}\label{V_i}
	\begin{cases}
	\mathcal{L}_{\lambda,\mu} \mathbf{V}_{j,q}(\Bx) =0, & \Bx\in \RR^3 \setminus \overline{D}, \\
	\mathbf{V}_{j,q}(\Bx)=\delta_{ij}{\bm{\xi}_{i,q}}, & \Bx\in D_i ,\\
	\mathbf{V}_{j,q}(\Bx) = \mathcal{O}(|\Bx|^{-1}), & \text{ as }|x|\to\infty,
	\end{cases}
	\end{equation}
	for $1\leq i,j\leq N$ and $1\leq p,q \leq 6$. Then the elastic stiffness tensor  $\bm{\mathcal{S}}:=(\bm{\mathcal{S}}_{pqij})_{1\leq p,q\leq 6;1\leq i,j\leq N}$ is defined coefficient-wise by
	\begin{equation}\label{cap_matrix}
	\bm{\mathcal{S}}_{pqij} = -\(\int_{\Gamma_i^+}{\partial_{\bm\nu} \mathbf{V}_{j,q}}\cdot\overline{\bm{\xi}_{i,p}} \d \sigma- \int_{\Gamma_i^-}{\partial_{\bm\nu} \mathbf{V}_{j,q}} \cdot\overline{\bm{\xi}_{i,p}} \d \sigma\).
	\end{equation}
\end{defn}

\begin{lem}\label{EST}
	The elastic stiffness tensor defined in \eqref{cap_matrix} for the Matryoshka-type elastic medium $D = \cup_{j=1}^{N}D_{j}$ can be given by
	\begin{equation}\label{CapmatrixGS}
	\begin{aligned}
	\bm{\mathcal{S}}_{pqij} &= \delta_{i,j-1}\int_{\Gamma_j^+}  \hat{\bm \zeta}_{j;j,q}^+ \cdot\overline{\bm{\xi}_{j-1,p}} ~\d \sigma + \delta_{i,j} \(\int_{\Gamma_{j+1}^+}    \hat{\bm\psi}_{j+1;j,q}^+ \cdot\overline{\bm{\xi}_{j,p}} ~\d \sigma - \int_{\Gamma_j^+} \hat{\bm \zeta}_{j;j,q}^+~ \cdot\overline{\bm{\xi}_{j,p}}\d \sigma\)\\
	&\quad  - \delta_{i,j+1} \int_{\Gamma_{j+1}^+}  \hat{\bm\psi}_{j+1;j,q}^+  \cdot\overline{\bm{\xi}_{j+1,p}} ~\d \sigma,
	\end{aligned}
	\end{equation}
	for $1\leq i,j\leq N$ and $1\leq p,q\leq 6$, 
	where $\hat{\bm\zeta}_{1;1,q}^+ = \mathbf{S}_{\Gamma_1^+}^{-1}[\bm{\xi}_{1,q}]$, $\hat{\bm\psi}_{N+1;N,q}^+ = 0$, $\hat{\bm\zeta}_{j+1,j+1,q}^+$ and $\hat{\bm\psi}_{j+1,j,q}^-$, $j=1,2,\ldots,N-1$, are determined by \eqref{S^-101GS}.
\end{lem}
\begin{proof}[\bf Proof]
	The solutions $\mathbf V_{j,q}$ to \eqref{V_i}  can be given by
	\begin{equation}\label{V_i_solu1}
\mathbf	V_{1,q}(\Bx) =
	\begin{cases}
	\ds \mathbf{S}_{{\Gamma_{1}^+}} [\hat{\bm\zeta}_{1;1,q}^+](\Bx), & \Bx\in D_{0}',\\
	\bm{\xi}_{1,q}, & \Bx\in D_1,\\
	\ds \mathbf{S}_{{\Gamma_1^-}} [\hat{\bm\phi}_{1;1,q}^-](\Bx)+\mathbf{S}_{{\Gamma_{2}^+}} [\hat{\bm\psi}_{2;1,q}^+](\Bx), &\Bx\in D_1',\\
	\bm 0, & \text{else},
	\end{cases}
	\end{equation}
	\begin{equation}\label{V_i_soluj2N}
	\mathbf V_{j,q}(\Bx) =
	\begin{cases}
	\ds \mathbf{S}_{{\Gamma_{j-1}^-}} [\hat{\bm\varphi}_{j-1;j,q}^-](\Bx)+\mathbf{S}_{{\Gamma_{j}^+}} [\hat{\bm\zeta}_{j;j,q}^+](\Bx), &\Bx\in D_{j-1}',\\
	\bm{\xi}_{j,q}, & \Bx\in D_j,\\
	\ds \mathbf{S}_{{\Gamma_j^-}} [\hat{\bm\phi}_{j;j,q}^-](\Bx)+\mathbf{S}_{{\Gamma_{j+1}^+}} [\hat{\bm\psi}_{j+1;j,q}^+](\Bx), &\Bx\in D_j',\\
	\bm 0, & \text{else},
	\end{cases}
	\mbox{ for } j =2,3,\ldots,N-1,
	\end{equation}
	\begin{equation}\label{V_i_soluN}
	\mathbf V_{N,q}(\Bx) =
	\begin{cases}
	\ds \mathbf{S}_{{\Gamma_{N-1}^-}} [\hat{\bm\varphi}_{N-1;N,q}^-](\Bx)+\mathbf{S}_{{\Gamma_{N}^+}} [\hat{\bm\zeta}_{N;N,q}^+](\Bx), &\Bx\in D_{N-1}',\\
	\bm{\xi}_{N,q}, & \Bx\in D_N,\\
	\ds \mathbf{S}_{{\Gamma_{N}^-}} [\hat{\bm\phi}_{N;N,q}^-](\Bx), & \Bx\in D_{N}',\\
	\bm 0, & \text{else},
	\end{cases}
	\end{equation}
	where $\hat{\bm\zeta}_{1;1,q}^+ = \mathbf{S}_{\Gamma_1^+}^{-1}[\bm{\xi}_{1,q}]$, $\hat{\bm\phi}_{N;N,q}^- = \mathbf{S}^{-1}_{{\Gamma_{N}^-}}[\bm{\xi}_{N,q}]$, and $\hat{\bm\varphi}_{j;j+1,q}^{-}$, $\hat{\bm\zeta}_{j+1;j+1,q}^{+}$, $\hat{\bm\phi}_{j;j,q}^{-}$,
	$\hat{\bm\psi}_{j+1;j,q}^{+}$, $j=1,2,\ldots,N-1$, are determined by \eqref{S^-101GS}.
	By using \eqref{0105}, \eqref{Kji} and \eqref{cap_matrix}, we have
	\[
	\begin{aligned}
	\bm{\mathcal{S}}_{pqij} &= \delta_{i,j} \int_{\Gamma_i^-}  \(\(-\frac{\mathbf I}{2}+\mathbf{K}_{{\Gamma_j^-}}^{*}\)[\hat{\bm\phi}_{j;j,q}^-] + \mathbf{K}^{*}_{\Gamma_j^-,\Gamma_{j+1}^+}[\hat{\bm\psi}_{j+1;j,q}^+]\)\cdot\overline{\bm{\xi}_{i,p}}~\d \sigma\\
	&\quad  + \delta_{i,j-1}\int_{\Gamma_i^-}  \(\(-\frac{\mathbf I}{2}+\mathbf{K}_{{\Gamma_{j-1}^-}}^{*}\)[\hat{\bm\varphi}_{j-1;j,q}^-] + \mathbf{K}^{*}_{\Gamma_{j-1}^-,\Gamma_{j}^+}[\hat{\bm\zeta}_{j;j,q}^+]\)\cdot\overline{\bm{\xi}_{i,p}}~\d \sigma\\
	&\quad -\delta_{i,j} \int_{\Gamma_i^+}  \(\mathbf{K}^{*}_{\Gamma_{j}^+,\Gamma_{j-1}^-}[\hat{\bm\varphi}_{j-1;j,q}^-]+\(\frac{\mathbf I}{2}+\mathbf{K}_{{\Gamma_{j}^+}}^{*}\)[\hat{\bm\zeta}_{j;j,q}^+]  \) \cdot\overline{\bm{\xi}_{i,p}}~\d \sigma\\
	&\quad  - \delta_{i,j+1} \int_{\Gamma_i^+}  \(\mathbf{K}^{*}_{\Gamma_{j+1}^+,\Gamma_j^-}[\hat{\bm\phi}_{j;j,q}^-]+\(\frac{\mathbf I}{2}+\mathbf{K}_{{\Gamma_{j+1}^+}}^{*}\)[\hat{\bm\psi}_{j+1;j,q}^+]  \) \cdot\overline{\bm{\xi}_{i,p}} ~\d \sigma\\
	%& =\delta_{i,j} \int_{\Gamma_{j+1}^+}    \psi_{j+1}^+ ~\d \sigma + \delta_{i,j-1}\int_{\Gamma_j^+}  \zeta_{j}^+~\d \sigma(x) -\(\delta_{i,j} \int_{\Gamma_j^+} \zeta_{j}^+~\d \sigma(x) + \delta_{i,j+1} \int_{\Gamma_{j+1}^+}  \psi_{j+1}^+ ~\d \sigma(x)\)\\
	& = \delta_{i,j-1}\int_{\Gamma_j^+}  \hat{\bm \zeta}_{j;j,q}^+ \cdot\overline{\bm{\xi}_{j-1,p}} ~\d \sigma + \delta_{i,j} \(\int_{\Gamma_{j+1}^+}    \hat{\bm\psi}_{j+1;j,q}^+ \cdot\overline{\bm{\xi}_{j,p}} ~\d \sigma - \int_{\Gamma_j^+} \hat{\bm \zeta}_{j;j,q}^{+} \cdot\overline{\bm{\xi}_{j,p}}\d \sigma\) - \delta_{i,j+1} \int_{\Gamma_{j+1}^+}  \hat{\bm\psi}_{j+1;j,q}^+  \cdot\overline{\bm{\xi}_{j+1,p}} ~\d \sigma.
	%	& = \delta_{i,j-1}\int_{\Gamma_j^+}  \zeta_{j}^+~\d \sigma + \delta_{i,j} \(\int_{\Gamma_j^-}\varphi_j^- ~\d \sigma - \int_{\Gamma_j^+} \zeta_{j}^+~\d \sigma(x)\) - \delta_{i,j+1} \int_{\Gamma_j^-}\varphi_j^- ~\d \sigma.\\
	\end{aligned}
	\]
	The proof is complete.
\end{proof}

In what follows, we introduce the $6N \times 6N$ matrices
\begin{equation}\label{Volumn_xi}
\bm{V} = \diag({V}_1,{V}_2,\ldots,{V}_6),\mbox{ and } \bm{T} = \diag(T,T,T,T,T,T),
\end{equation}
where 
\begin{equation}\label{element_volumn_xi}
{V}_p = \diag\(\|\bm{\xi}_{1,p}\|_{L^2(D_1)}^2,\|\bm{\xi}_{2,p}\|_{L^2(D_2)}^2,\ldots,\|\bm{\xi}_{N,p}\|_{L^2(D_N)}^2\).
\end{equation}
and 
\begin{equation}\label{element_T}
{T} = \diag\(\tau_1,\tau_2,\ldots,\tau_N\).
\end{equation}
By applying a flattening procedure similar to \eqref{6NU}--\eqref{6NUpq} to the stiffness tensor \eqref{CapmatrixGS}, we obtain a 
$6N$-by-$6N$ matrix, which we still denote by 
$\bm{\mathcal{S}}$.

\begin{prop}\label{ImU}
	The following asymptotic expansion for the matrix $\bm{U}(\omega,\delta)$ defined in \eqref{6NU} holds:
	\begin{equation}\label{bmU}
	\bm{U}(\omega,\delta) = \bm{I}+\rho\omega^2\bm{T}^2\bm{V}^{-1}-\delta \bm{V}^{-1}\bm{\mathcal{S}} \bm{V}^{-1} +\ii\alpha \omega\delta \bm{V}^{-1}\bm{E} \bm{V}^{-1}+\mathcal{O}((\omega^2+\delta)^2),
	\end{equation}
	where $\alpha$ is given by \eqref{alphaa}, and  $\bm{E} = (E_{pqij})_{1\leq p,q\leq 6;1\leq i,j\leq N}$ with 
	\begin{equation}\label{EPQIJ}
	E_{pqij} = \delta_{1,j}\delta_{1,i}\sum_{t=1}^{3}\int_{\Gamma_1^+}\hat{\bm\zeta}^{+}_{1;j,t}\cdot \overline{\bm{\xi}_{j,q}}~\d \sigma  \int_{\Gamma_1^+}\hat{\bm\zeta}^{+}_{1;j,t} \cdot \overline{\bm{\xi}_{i,p}}~\d \sigma.
	\end{equation}
\end{prop}

\begin{proof}[\bf Proof]
It follows from \eqref{6NU}, \eqref{6NUpq} and \eqref{asy_ujwd}  that
\begin{equation}\label{Upqij}
\begin{aligned}
\bm{U}_{pqij}(\omega,\delta) &= \(1+ \frac{\rho\tau_j^2\omega^2}{\|\bm{\xi}_{j,q}\|_{L^2(D_j)}^2}\)\delta_{p,q}\delta_{i,j}\\
&\quad +\delta\left[ \frac{\delta_{i,j-1}}{\|\bm{\xi}_{i,p}\|_{L^2(D_i)}^2\|\bm{\xi}_{j,q}\|_{L^2(D_j)}^2}\int_{\Gamma_{j-1}^-} \hat{\bm{\varphi}}^{-}_{j-1;j,q} \cdot\overline{\bm{\xi}_{j-1,p}}~\d \sigma
\right.\\
&\quad \quad \quad\left.+ \frac{\delta_{i,j}}{\|\bm{\xi}_{i,p}\|_{L^2(D_i)}^2\|\bm{\xi}_{j,q}\|_{L^2(D_j)}^2} \(\int_{\Gamma_j^+} \hat{\bm{\zeta}}^{+}_{j;j,q} \cdot\overline{\bm{\xi}_{j,p}}~\d \sigma - \int_{\Gamma_j^-} \hat{\bm{\varphi}}^{-}_{j;j,q} \cdot\overline{\bm{\xi}_{j,p}}~\d \sigma \)\right.\\
&\quad \quad \quad \left.-\frac{\delta_{i,j+1}}{\|\bm{\xi}_{i,p}\|_{L^2(D_i)}^2\|\bm{\xi}_{j,q}\|_{L^2(D_j)}^2} \int_{\Gamma_{j+1}^+} \hat{\bm{\zeta}}^{+}_{j+1;j,q} \cdot\overline{\bm{\xi}_{j+1,p}}~\d \sigma \right]\\
&\quad +\omega\delta\(\frac{\ii\alpha \delta_{1,j}\delta_{1,i}}{\|\bm{\xi}_{i,p}\|_{L^2(D_i)}^2 \|\bm{\xi}_{j,q}\|_{L^2(D_j)}^2} \sum_{t=1}^{3}\int_{\Gamma_1^+}\hat{\bm\zeta}^{+}_{1;j,t}\cdot \overline{\bm{\xi}_{j,q}}~\d \sigma  \int_{\Gamma_1^+}\hat{\bm\zeta}^{+}_{1;j,t} \cdot \overline{\bm{\xi}_{i,p}}~\d \sigma\) +\mathcal{O}((\omega^2+\delta)^2).
\end{aligned}
\end{equation}
Moreover, from the definition \eqref{cap_matrix} and variational principle, we have that the solutions $\mathbf V_{j,q}$ to \eqref{V_i} satisfies that  for any  $\mathbf w\in \mathbf{D}^{1,2}(\RR^3\setminus \overline{D}):=\{\mathbf v:\nabla \mathbf v\in L^2(\RR^3\setminus \overline{D})^{3\times 3}\}$,
\begin{equation}
Q_{\lambda,\mu,\RR^3\setminus \overline{D}}(\mathbf{V}_{j,q},\mathbf{w}) = -\sum_{n=1}^N\(\int_{\Gamma_n^+}{\partial_{\bm\nu} \mathbf{V}_{j,q}}\cdot\overline{\mathbf w}~\d \sigma- \int_{\Gamma_n^-}{\partial_{\bm\nu}\mathbf V_{j,q}}\cdot \overline{\mathbf{w}}~\d \sigma\),
\end{equation}
where $Q_{\lambda,\mu,\RR^3\setminus \overline{D}}$ is defined in \eqref{QLambdaMu}.
Taking $\mathbf w = \mathbf V_{i,p}$, we obtain
\begin{equation}\label{symCMGS}
\begin{aligned}
Q_{\lambda,\mu,\RR^3}(\mathbf{V}_{j,q},\mathbf{V}_{i,p})& = Q_{\lambda,\mu,\RR^3\setminus \overline{D}}(\mathbf{V}_{j,q},\mathbf{V}_{i,p})\\
& = -\sum_{n=1}^N\(\int_{\Gamma_n^+}{\partial_{\bm\nu} \mathbf V_{j,q}}\cdot\overline{\mathbf{V}_{i,p}} ~\d \sigma- \int_{\Gamma_n^-}{\partial_{\bm\nu} \mathbf V_{j,q}}\cdot\overline{\mathbf{V}_{i,p}}~\d \sigma\)\\
& = -\(\int_{\Gamma_i^+}{\partial_{\bm\nu} \mathbf V_{j,q}}\cdot\overline{\bm\xi_{i,p}} ~\d \sigma- \int_{\Gamma_i^-}{\partial_{\bm\nu} \mathbf V_{j,q}}\cdot\overline{\bm\xi_{i,p}} ~\d \sigma\) = \bm{\mathcal{S}}_{pqij},
\end{aligned}
\end{equation}
which implies that $\bm{\mathcal{S}}_{pqij} = \bm{\mathcal{S}}_{qpji}$. It follows that
\begin{equation}\label{zetaeqpsi}
\int_{\Gamma_j^+}  \hat{\bm \zeta}_{j;j,q}^+ \cdot\overline{\bm{\xi}_{j-1,p}} ~\d \sigma = -\int_{\Gamma_{j}^+}  \hat{\bm\psi}_{j;j-1,p}^+  \cdot\overline{\bm{\xi}_{j,q}} ~\d \sigma.
\end{equation}
This, together with \eqref{cap_pq}--\eqref{cap_selfad} and \eqref{Upqij}, implies that \eqref{bmU} holds true.
	The proof is complete.
\end{proof}

\subsection{Existence and asymptotics of the subwavelength resonances}
In the previous section, we have provided a variational characterization of subwavelength resonances for the Matryoshka-type elastic medium $D = \cup_{j=1}^{N}D_{j}$ based on the displacement-to-traction map. Nevertheless, it remains unclear whether subwavelength resonance exists for a given high contrast $\delta$. Furthermore, to better understand the limiting behavior of subwavelength resonance as $\delta\to 0$, it is necessary to investigate the asymptotic expansion of resonance with respect to $\delta$. These two issues will be the primary focus of this subsection.

From Propositions \ref{TCR} and \ref{ImU}, it follows that $\omega = \mathcal{O}(\delta^{\frac{1}{2}})$. Hence, we set $\beta := \delta^{\frac{1}{2}}$ and introduce the scaled frequency $\hat{\omega}: = \frac{\omega}{\beta}$. Then, by \eqref{bmU}, we have
\[
\bm{I}-\bm{U}(\hat{\omega}\beta,\beta^2) = -\hat{\omega}^2\beta^2\rho\bm{T}^2\bm{V}^{-1}+\beta^2 \bm{V}^{-1}\bm{\mathcal{S}} \bm{V}^{-1} -\ii\alpha \beta^3\hat{\omega} \bm{V}^{-1}\bm{E} \bm{V}^{-1}+\mathcal{O}((\hat{\omega}^2\beta^2+\beta^2)^2).
\]
Define 
\begin{equation}\label{operator_AE}
\begin{aligned}
\hat{\bm{U}}(\hat{\omega},\beta):&= \beta^{-2}\(\bm{I}-\bm{U}(\hat{\omega}\beta,\beta^2)\)\\
&= \hat{\bm{U}}_0(\hat{\omega})+\beta \hat{\bm{U}}_1(\hat{\omega}) +\mathcal{O}(\beta^2),
\end{aligned}
\end{equation}
where 
\[
\hat{\bm{U}}_0(\hat{\omega}):= -\hat{\omega}^2 \rho\bm{T}^2\bm{V}^{-1}+ \bm{V}^{-1}\bm{\mathcal{S}} \bm{V}^{-1},
\]
\[
\hat{\bm{U}}_1(\hat{\omega}):=-\ii\alpha  \hat{\omega} \bm{V}^{-1}\bm{E} \bm{V}^{-1}.
\]
Clearly,  $\hat{\omega}\beta$ is a  characteristic value of $\bm{I}-\bm{U}(\hat{\omega}\beta,\beta^2)$ if and only if  $\hat{\omega}$ is a characteristic value of $\hat{\bm{U}}(\hat{\omega},\beta)$.

Let $(\lambda_{m},\bm{v}_{m})_{1\leq m\leq 6N}$ be the eigenpairs satisfying the generalized elastic stiffness eigenvalue problem:
%Next, we consider the $6N$ eigenvalues $\lambda_{m}$ and eigenvectors $\bm{a}_{m}$ of the generalized elastic stiffness eigenvalue problem:
\begin{equation}\label{GEP}
\bm{T}^{-2}\bm{\mathcal{S}}\bm{v}_{m} = \lambda_{m} \bm{V}\bm{v}_{m}, \mbox{ for } 1\leq m\leq 6N,
\end{equation}
where the eigenvectors form an orthonormal basis with respect to the following inner product
\begin{equation}\label{orthonormalV}
\bm{v}_{m}^T\bm{V}\bm{v}_{n} = \delta_{m,n}.
\end{equation}
It follows that the leading term of the characteristic values for $\hat{\bm{U}}(\hat{\omega},\beta)$ can be given by 
\[
\hat{\omega}_{m;0}^\pm = \pm\(\frac{\lambda_{m}}{\rho}\)^{\frac{1}{2}},\; \mbox{ for } 1\leq m\leq 6N,
\]
 counted with their multiplicities. Furthermore, by Kato's asymptotic perturbation theory \cite{Katpbook1995},  the eigenvectors $\bm{V}\bm{v}_{m}(\beta)$ of $\hat{\bm{U}}(\hat{\omega},\beta)$ can be represented as
 \begin{equation}\label{eigenvetor_AE}
 \bm{v}_{m}(\beta) = \bm{v}_{m}+\beta \bm{v}_{m;1} +\mathcal{O}(\beta^2).
 \end{equation}
Given that the eigenvalues of the problem \eqref{GEP} exhibit multiplicities, then the method of deriving the eigenfrequencies based on the classical implicit function theorem (cf. \cite{FA_SAM_2022,FCA_SIAP2023,DKLZ_JDE26}) is no longer applicable, since the corresponding Jacobian matrices may not be invertible. Therefore, a much more subtle analysis is required.

Next, we denote the distinct eigenvalues of $\bm{T}^{-2}\bm{V}^{-1}\bm{\mathcal{S}}$ (cf. \eqref{GEP}--\eqref{orthonormalV}) by $\{\lambda_m^{\mathrm{d}}\}_{m=1}^{N_{\mathrm{d}}}$ with $N_{\mathrm{d}}\leq 6N$, ignoring algebraic multiplicity. Let $a_m$ denote the algebraic multiplicity of $\lambda_m^{\mathrm{d}}$ so that $\sum_{m=1}^{N_{\mathrm{d}}} a_m =6N$. Due to the Hermitian property of $\bm{\mathcal{S}}$ (see \eqref{symCMGS}), we know that the algebraic and geometric multiplicities coincide for each eigenvalue. 
Moreover, in view of \eqref{symCMGS} and from the solutions $\mathbf{V}_{j,q}(\Bx)\in {H}_{loc}^1(\RR^3 \setminus \overline{D})^3$, the elastic stiffness tensor \eqref{cap_matrix} is positive-definite.
Hence, the corresponding eigenfrequencies can be rewritten by
\begin{equation}\label{wdmopm}
\hat{\omega}_{m;0}^{\textup{d},\pm} = \pm\(\frac{\lambda_{m}^{\textup{d}}}{\rho}\)^{\frac{1}{2}}, \mbox{ for }  1\leq m\leq N_{\mathrm{d}}.
\end{equation}

The existence of subwavelength resonant frequencies is given as follows; see \cite{dyatlov2019mathematical,Gohberg_book} for a review of Gohberg--Sigal theory and the related concepts used here.

\begin{thm}\label{existence_SWLR}
For an $N$-layered Matryoshka-type elastic structure $D = \bigcup_{j=1}^{N} D_{j}$, there exist $12N$ subwavelength resonant frequencies, counted with multiplicity.
\end{thm}
\begin{proof}[\bf Proof]

For any fixed $m=1,2,\ldots,N_{\mathrm{d}}$, define the $\varepsilon_0$-neighborhood of $\hat{\omega}_{m;0}^{\textup{d},\pm}$ as
\[
B_{\varepsilon_0}(\hat{\omega}_{m;0}^{\textup{d},\pm}): = \left\{\hat{\omega}\in \mathbb{C}: |\hat{\omega} -\hat{\omega}_{m;0}^{\textup{d},\pm} | <\varepsilon_0 \right\}, \mbox{ for sufficiently small  } \varepsilon_0 >0.
\]
Let $M(\hat{\omega}_{m;0}^{\textup{d},\pm},\hat{\bm{U}}_0(\hat{\omega}))$ denote the algebraic multiplicity of $\hat{\omega}_{m;0}^{\textup{d},\pm}$ for the operator $\hat{\bm{U}}_0(\hat{\omega})$ with respect to $\partial B_{\varepsilon_0}(\hat{\omega}_{m;0}^{\textup{d},\pm})$  (cf. \cite[Page 205]{Gohberg_book}), defined by
\[
M(\hat{\omega}_{m;0}^{\textup{d},\pm},\hat{\bm{U}}_0(\hat{\omega})):=\frac{1}{2\pi \ii}\textup{tr}\int_{\partial B_{\varepsilon_0}(\hat{\omega}_{m;0}^{\textup{d},\pm})} \hat{\bm{U}}_0(\hat{\omega})^{-1} \frac{\d }{\d \hat{\omega}} \hat{\bm{U}}_0(\hat{\omega})~\d \hat{\omega}.
\]
For $\beta\ll 1$, one has that for all $\hat{\omega}\in \partial B_{\varepsilon_0}(\hat{\omega}_{m;0}^{\textup{d},\pm})$, 
\[
\| \hat{\bm{U}}_0(\hat{\omega})^{-1}\(\hat{\bm{U}}(\hat{\omega},\beta) - \hat{\bm{U}}_0(\hat{\omega})\)\|<1.
\]
Hence, by the generalized Rouch\'e theorem \cite[Theorem 1.15]{AK_book2018}, it follows that 
\[
M(\partial B_{\varepsilon_0}(\hat{\omega}_{m;0}^{\textup{d},\pm}),\hat{\bm{U}}(\hat{\omega},\beta)) = M(\hat{\omega}_{m;0}^{\textup{d},\pm},\hat{\bm{U}}_0(\hat{\omega})) = a_m,
\]
where $M(\partial B_{\varepsilon_0}(\hat{\omega}_{m;0}^{\textup{d},\pm}),\hat{\bm{U}}(\hat{\omega},\beta)) $ is given by
\[
M(\partial B_{\varepsilon_0}(\hat{\omega}_{m;0}^{\textup{d},\pm}),\hat{\bm{U}}(\hat{\omega},\beta)):=\frac{1}{2\pi \ii}\textup{tr}\int_{\partial B_{\varepsilon_0}(\hat{\omega}_{m;0}^{\textup{d},\pm})} \hat{\bm{U}}(\hat{\omega},\beta)^{-1} \frac{\partial }{\partial \hat{\omega}} \hat{\bm{U}}(\hat{\omega},\beta)~\d \hat{\omega}.
\]
This implies that there exist $a_m$ characteristic values of of $\hat{\bm{U}}(\hat{\omega},\beta)$ in $B_{\varepsilon_0}(\hat{\omega}_{m;0}^{\textup{d},\pm})$ for each $m=1,2,\ldots,N_{\mathrm{d}}$.
Let $ C_r>0$ be a given constant such that $\hat{\bm{U}}_0(\hat{\omega})$ is invertible on $\partial B_0(C_r)$ and $B_{\varepsilon_0}(\hat{\omega}_{m;0}^{\textup{d},\pm})\subset B_{C_r}(0)$  for any sufficiently small  $\varepsilon_0 >0$. Similarly, we have that these characteristic values near $\hat{\omega}_{m;0}^{\textup{d},\pm}$ are all the characteristic values of  $\hat{\bm{U}}(\hat{\omega},\beta)$ in $B_{C_r}(0)$. Therefore, using Proposition \ref{TCR}, \eqref{operator_AE} and $\sum_{m=1}^{N_{\mathrm{d}}} a_m =6N$, we can conclude that for sufficiently small $\delta>0$, there exist $12N$ subwavelength resonant frequencies, up to multiplicity.
\end{proof}
 
Having established the existence of eigenfrequencies, we now proceed to study their asymptotic expansions.
Specifically,  we focus on the characteristic values of $\hat{\bm{U}}(\hat{\omega},\beta)$ located near $\hat{\omega}_{m;0}^{\textup{d},+}$, which we denote by $\{\hat{\omega}_{m;0,j}^{\textup{d},+}(\beta)\}_{j=1}^{a_m}$.
Our analysis follows the general framework developed in \cite{AF_JDE2004, ALZ_TAMS2023}.
For each integer $\ell\geq 1$,  define the following functions
\begin{equation}\label{pmell}
 p_{m;\ell}(\beta):=\frac{1}{2\pi \ii}\textup{tr}\int_{\partial B_{\varepsilon_0}(\hat{\omega}_{m;0}^{\textup{d},+})} \(\hat{\omega} - \hat{\omega}_{m;0}^{\textup{d},+}\)^{\ell} \hat{\bm{U}}(\hat{\omega},\beta)^{-1} \frac{\partial }{\partial \hat{\omega}} \hat{\bm{U}}(\hat{\omega},\beta)~\d \hat{\omega}.
 \end{equation}
It follows from \cite[Theorem 1.14]{AK_book2018} that $\left\{  p_{m;\ell} \right\}_{\ell\geq 1}$ are the power sum polynomials of the variables $z_{m;j}(\beta):=\hat{\omega}_{m;0,j}^{\textup{d},+}(\beta) - \hat{\omega}_{m;0}^{\textup{d},+}$ with $z_{m;j}(0) = 0$ for $j=1,2,\ldots,a_m$:
 \[
 p_{m;\ell}(\beta) = \sum_{j=1}^{a_m} z_{m;j}(\beta)^{\ell}.
 \] 
From the definition \eqref{operator_AE} of $\hat{\bm{U}}(\hat{\omega},\beta)$, it follows that each $p_{m,\ell}(\beta)$ admits an analytic continuation into $B_{\beta_0}(0)$  for small enough $\beta_0>0$.
Moreover,  for any fixed $1\leq m \leq N_{\mathrm{d}}$,  the function defined by   
 \begin{equation}\label{det:U}
 {f}(\hat{\omega},\beta):=\det(\hat{\bm{U}}(\hat{\omega},\beta)), \;\mbox{ for } (\hat{\omega},\beta)\in B_{\varepsilon_0}(\hat{\omega}_{m;0}^{\textup{d},+})\times B_{\beta_0}(0),
 \end{equation}
is analytic. We next derive the factorization of the function $f(\hat{\omega},\beta)$ using the Weierstrass preparation theorem; see \cite{AM_book1991, KS_book2001} for background on this classical result and related analytic tools.

 \begin{prop}
For any fixed $1\leq m \leq N_{\mathrm{d}}$, let the analytic functions 	$p_{m;\ell}(\beta)$, $\ell\geq 1$, be defined by \eqref{pmell} on $B_{\beta_0}(0)$. Define $ r_{m;j}(\beta) $, $1\leq j\leq a_m$, to satisfy the following recurrence relation
\begin{equation}\label{RRR}
\sum_{s=0}^{j-1} \(p_{m;j-s}(\beta)r_{m;s}(\beta)\)+ jr_{m;j}(\beta) = 0 \;\mbox{ with } r_{m;0}(\beta) = 1. 
\end{equation}
Then we have 
\begin{enumerate}
	\item[\textup{(i)}]
	the functions $r_{m;j}(\beta)$, $1\leq j\leq a_m$, are analytic in $B_{\beta_0}(0)$ and satisfy $r_{m;j}(0) = 0 $;
	\item[\textup{(ii)}]
	$\hat{\omega}_{m;0,j}^{\textup{d},+}(\beta) - \hat{\omega}_{m;0}^{\textup{d},+}$, $1\leq j\leq a_m$, are all the roots of the Weierstrass polynomial:
	\begin{equation}
	P_{m;\beta}(z) =  \sum_{j=0}^{a_m}r_{m;j}(\beta) z^{a_m-j};
	\end{equation}
	\item[\textup{(iii)}]
	the function ${f}(\hat{\omega},\beta)$ defined in \eqref{det:U} admits the following factorization:
	\begin{equation}\label{WPT}
	{f}(\hat{\omega},\beta)=P_{m;\beta}\(\hat{\omega}-\hat{\omega}^{\textup{d},+}_{m;0}\){g}\(\hat{\omega},\beta \),
	\end{equation}
	where ${g}(\hat{\omega},\beta )$ is analytic with respect to both $\hat{\omega}$ and $\beta$, and does not vanish identically.
\end{enumerate}
 \end{prop}
\begin{proof}[\bf Proof]
%The proof follows from an argument similar to that of \cite[Proposition 3.13]{ALZ_TAMS2023},  we only sketch the main steps here. 
Define
 \[
 P_{m;\beta}(z) := \prod_{j=1}^{a_m} (z - z_{m;j}(\beta)).
 \]
 Clearly, the values $z_{m;j}(\beta)$, $j = 1,2,\ldots,a_m$, are the roots of $P_{m;\beta}(\theta)$. By elementary algebra, we have
 \[
 P_{m;\beta}(z) = z^{a_m} +\sum_{j=1}^{a_m}(-1)^j \hat{e}_{m;j}(\beta) z^{a_m-j},
 \]
 where $\hat{e}_{m;j}(\beta) $ are elementary symmetric polynomials of the variables $z_{m;j}(\beta)$. 
 It follows that $\hat{e}_{m;j}(0) =0.$ If we set $\hat{e}_{m;0}(\beta)= 1$, then the functions $p_{m;\ell}(\beta)$ and $\hat{e}_{m;j}(\beta)$ satisfy Newton's identities(cf. \cite{AM_book1991}):
 \[
 \sum_{s=0}^{j-1} \((-1)^sp_{m;j-s}(\beta)\hat{e}_{m;s}(\beta)\)+(-1)^j j\hat{e}_{m;j}(\beta) = 0, \;j=1,2,\ldots,a_m.
 \]
This recurrence relation implies that $\hat{e}_{m;j}(\beta)$ are analytic in $B_{\beta_0}(0)$. Let $ r_{m;j}(\beta) :=(-1)^j\hat{e}_{m;j}(\beta)$, $j=1,2,\ldots,a_m.$ Therefore, 
\[
P_{m;\beta}(z) =  \sum_{j=0}^{a_m}r_{m;j}(\beta) z^{a_m-j},
\]
where the functions $r_{m;j}(\beta)$ are analytic in $B_{\beta_0}(0)$ and satisfy $r_{m;j}(0) = 0 $ and $r_{m;0}(\beta)=1$. Thus, $P_{m;\beta}(z)$ is a Weierstrass polynomial (cf. \cite[Definition 6.4.4]{KS_book2001}).

The factorization \eqref{WPT} of the function $f(\hat{\omega},\beta)$ directly follows from Weierstrass preparation theorem \cite[Theorem 6.4.5]{KS_book2001}.
Moreover, if there exists a pair of $(\hat{\omega}_{1},\beta_1)\in B_{\varepsilon_0}(\hat{\omega}^{\mathrm{d}}_{m;0})\times B_{\beta_0}(0)$ such that ${g}(\hat{\omega}_1,\beta_1)=0$, then the total number of the roots of ${f}(\hat{\omega},\cdot)$ at the point $\beta=\beta_1$  would be at least $a_m+1$. which is a contradiction. This completes the proof.
\end{proof}

The following result concerns the factorization property of Weierstrass polynomials. Detailed proofs can be found in \cite[Appendix A2, Theorem 1]{HB_book1985} and \cite[Lemma 6.4.2]{KS_book2001}.

\begin{lem}\label{Weierstrass}
	Let $P_{m;\beta}(\widehat{\omega}-\widehat{\omega}^{\textup{d},+}_{m;0})$ be a Weierstrass polynomial. 
	Then it admits a unique factorization into prime   factors of the form
	\[
	P_{m;\beta}(\hat{\omega}-\hat{\omega}^{\textup{d},+}_{m;0})=\prod^{b_m}_{j=1}p_{m;\beta,j}
	\(\hat{\omega}-\hat{\omega}^{\textup{d},+}_{m;0}\)^{c_{m;j}},
	\]
	where each $p_{m;\beta,j}(\cdot-\hat{\omega}^{\textup{d},+}_{m;0})$ is an irreducible and mutually distinct Weierstrass polynomial in $\hat{\omega}$ of degree $d_{m;j}$ with $\sum^{b_m}_{j=1}c_{m;j}d_{m;j}=a_m$.
\end{lem}

It follows from Lemma \ref{Weierstrass} that the equation $P_{m;\beta}(\hat{\omega}-\hat{\omega}^{\textup{d},+}_{m;0})=0$ is equivalent to the system of equations  $p_{m;\beta,j}
(\hat{\omega}-\hat{\omega}^{\textup{d},+}_{m;0})=0,j=1,\cdots,b_m$. The roots of $p_{m;\beta,j}
(\hat{\omega}-\hat{\omega}^{\mathrm{d}}_{m;0})=0$ are given by a $d_{m;j}$-valued algebraic function, which admits a Puiseux series expansion near $\beta=0$ (cf. \cite[Appendix A5, Theorems 2 and 3]{HB_book1985}). Therefore, we can obtain the asymptotic expansion of $\hat{\omega}_{m;0,j}^{\textup{d},+}(\beta) $ for sufficiently small $\beta := \delta^{\frac{1}{2}}$.

\begin{thm}\label{PE_SWLR}
There exists a sufficiently small $\beta_0>0$ such that, for all $\beta\in B_{\beta_0}(0)$,	 the $a_m$ roots of ${f}(\cdot,\beta)$ defined in \eqref{det:U} near $\hat{\omega}^{\textup{d},+}_{m;0}$ take $\tilde{a}_m$ distinct values, independent of $\beta$, denoted by $\left\{\hat{\omega}^{{\textup{d}},+}_{m;0,j}(\beta)\right\}^{\tilde{a}_m}_{j=1}$, with $\tilde{a}_m\leq a_m$.
	Moreover, the set $\left\{\hat{\omega}^{{\textup{d}},+}_{m;0,j}(\beta)\right\}^{\tilde{a}_m}_{j=1}$ can 
	be divided into $b_m$ subsets:
	\[
	\left\{\hat{\omega}^{{\textup{d}},+}_{m;0,j}(\beta)\right\}^{\tilde{a}_m}_{j=1} = \bigcup_{t=1}^{b_m}  \left\{\hat{\omega}^{{\textup{d}},+}_{m;0,j}(\beta)\right\}^{\sum_{s=1}^{t}d_{m;s}}_{j=1+\sum_{s=1}^{t-1}d_{m;s}},
	\]
	where the positive integers $d_{m;s}$ satisfy $\sum_{s=1}^{b_m}d_{m;s} = \tilde{a}_m$. Each subset corresponds to a $d_{m;t}$-valued algebraic function, which admits a Puiseux series expansion of the form
	\begin{align}\label{w:expansion}
	\hat{\omega}^{{\textup{d}},+}_{m;0,j}(\beta)-\hat{\omega}^{\textup{d},+}_{m;0}
	=\sum^{+\infty}_{n=0}C_{m;t,n}\Big(\beta^{\frac{1}{d_{m;t}}}\big(\e^{\mathrm{i}\frac{2\pi}{d_{m;t}}}\big)^{i_m}\Big)^n,\quad j=i_m+\sum^{t-1}_{s=1}d_{m;s},\quad
	i_m=1,\ldots,d_{m;t}.
	\end{align}
\end{thm}

For each fixed  $m=1,2,\ldots,N_{\mathrm{d}}$, let  $\bm{w}_{m;j},j=1,2,\ldots,a_m$, denote the eigenvectors corresponding to the eigenvalue $\lambda^{\textup{d}}_{m}$ of  $\bm{T}^{-2}\bm{V}^{-1}\bm{\mathcal{S}}$ (cf. \eqref{GEP}--\eqref{orthonormalV}). Accordingly, the eigenvectors of $\hat{\bm{U}}(\hat{\omega},\beta)$ near $\hat{\omega}_{m;0}^{\textup{d},+}$  can be expressed as $\bm{V}{\bm{w}}_{m;j}(\beta),j=1,2,\ldots,a_m$, where each ${\bm{w}}_{m;j}(\beta)$ satisfies \eqref{eigenvetor_AE}. By combining \eqref{operator_AE} and \eqref{eigenvetor_AE}, there exists $\tilde{a}\in[1,a_m]$ such that
\[
\hat{\bm{U}}(\hat{\omega}^{{\mathrm{d}},+}_{m;0,j}(\beta),\beta)\bm{V}
{\bm{w}}_{m;\tilde{a}}(\beta)=
\bm{T}^2(-\rho\hat{\omega}^2+\lambda^{\textup{d}}_{m}){\bm{w}}_{m;\tilde{a}}(\beta)\big|_{\hat{\omega}
	=\hat{\omega}^{{\mathrm{d}},+}_{m;0,j}(\beta)}+\mathcal{O}(\beta).
\]
Substituting \eqref{w:expansion} into the equality above, we have
\[
C_{m;t,n}=0,\quad n=0,1,\cdots,d_{m;t}-1.
\]
Consequently, all the characteristic values $\hat{\omega}_{m;0,j}^{\textup{d},+}(\beta),j=1,2,\ldots,a_m$, counted with multiplicity, of $\hat{\bm{U}}(\hat{\omega},\beta)$ near $\hat{\omega}_{m;0}^{\textup{d},+}$ can be represented as
\begin{align}\label{eq:muti}
\hat{\omega}^{\textup{d},+}_{m;0,j}(\beta)=\hat{\omega}^{\textup{d},+}_{m;0}+C_{m;t,d_{m;t}}\beta+o(\beta).
\end{align}

Our main result in this subsection is the following:

\begin{thm}\label{theoasy}
	  The $12N$ subwavelength resonant frequencies $\omega^{\pm}_i(\delta^{\frac{1}{2}})$, counted with their multiplicities,  admit the following asymptotic expansions as $\delta \rightarrow 0^+$:
	\begin{align}\label{full:asymptotic}
	\omega^{\pm}_i(\delta^{\frac{1}{2}})=\pm\sqrt{\frac{\lambda_i}{\rho}}\delta^{\frac{1}{2}}-\ii \frac{\alpha {\bm{v}}_{i}^T \bm{T}^{-2}\bm{E}  {\bm{v}}_{i} }{2\rho} \delta+o(\delta),\quad i=1,2,\ldots,6N,
	\end{align}
	with the relation
	\begin{align}\label{-reson}
	\omega^-_i(\delta^{\frac{1}{2}})=-\overline{\omega^+_i}(\delta^{\frac{1}{2}}),
	\end{align}
	where  $(\lambda_{i},\bm{v}_{i})_{1\leq m\leq 6N}$ are the eigenpairs of generalized elastic stiffness eigenvalue problem \eqref{GEP}, $\bm{T}$ and $\bm{E}$ are defined in \eqref{Volumn_xi} and \eqref{EPQIJ}, respectively,  and $\alpha$ is given by \eqref{alphaa}.
	Moreover, 
	let $\mathbf{u}_{i}$ denote the eigenmode corresponding to $\omega_i^+$, and let $\bm{v}^{(j)}_i$ be the $j$-th entry of the eigenvector $\bm{v}_i$. Then
	\begin{equation}\label{eigenmode}
	\mathbf{u}_{i} = \sum_{j=1}^{N} \sum_{q=1}^{6} \bm{v}^{\(N(q-1)+j\)}_{i}\mathbf{V}_{j;q}+\mathcal{O}(\delta^{\frac{1}{2}}),
	\end{equation}
	where $\mathbf{V}_{j;q}$, for $j=1,2,\ldots,N$ and $q=1,2,\ldots,6$, are given by \eqref{V_i_solu1}--\eqref{V_i_soluN}.
\end{thm}

\begin{proof}[\bf Proof]
	By using \eqref{operator_AE}, \eqref{GEP}, \eqref{wdmopm}, and \eqref{eq:muti}, we can derive that
	\begin{align*}
	\hat{\bm{U}}(\hat{\omega}^{\textup{d},+}_{m;0,j}(\beta),\beta)\bm{V}{\bm{w}}_{m;j}(\beta)
	&=\bm{T}^2(-\lambda^{\textup{d}}_{m}\bm{I}+\bm{T}^{-2}\bm{V}^{-1}\bm{\mathcal{S}}){\bm{w}}_{m;j}
	\\
	&
	\quad+
	\beta(-2\rho\hat{\omega}^{\textup{d},+}_{m;0}C_{m;t,d_{m;t}}\bm{T}^2-\ii\alpha  \hat{\omega}^{\textup{d},+}_{m;0} \bm{V}^{-1}\bm{E} ){\bm{w}}_{m;j}\\
	&\quad+\beta\bm{T}^2(-\lambda^{\textup{d}}_{m}\bm{I}+\bm{T}^{-2}\bm{V}^{-1}\bm{\mathcal{S}}){\bm{w}}_{m;j,1}+o(\beta)=0.
	\end{align*}
	Therefore, we obtain
	\begin{equation*}
	C_{m;t,d_{m;t}}=-\ii \frac{\alpha }{2\rho} {\bm{w}}_{m;j}^T \bm{T}^{-2}\bm{E}  {\bm{w}}_{m;j}.
	\end{equation*}
	From $\beta :=\delta^{\frac{1}{2}}$ and $\omega:=\beta \hat{\omega}$, it follows that
	\begin{align}\label{Omega:delta}
	\omega^{+}_i(\delta^{\frac{1}{2}})=\sqrt{\frac{\lambda_i}{\rho}}\delta^{\frac{1}{2}}-\ii \frac{\alpha {\bm{v}}_{i}^T \bm{T}^{-2}\bm{E}  {\bm{v}}_{i} }{2\rho} \delta+\mathcal{O}(\delta^{1+\frac{1}{2d_{i}}}),\quad i=1,2,\ldots,6N,
	\end{align}
	counted with their multiplicities. Moreover, by taking the complex conjugate on both sides of \eqref{D2N} and using the Difinition \ref{defn:resonance}, we find that \eqref{-reson} holds.
	
	In view of \eqref{operator_AE} and \eqref{eigenvetor_AE}, we obtain the leading term of the $\bm{s}$ in \eqref{6Nsystem} with $\bm{H}=\bm{0}$ is $\bm{V}\bm{v}_i+\mathcal{O}(\delta^{\frac{1}{2}})$.
	By Lemma~\ref{lem32}, the eigenmode inside the resonators $D = \cup_{j=1}^{N}D_{j}$ can be expressed as
	\[
	\mathbf{u}_{(i)} = \sum_{j=1}^{N} \sum_{q=1}^{6}\|\bm{\xi}_{j,q}\|_{L^2(D_j)}^2 \bm{v}^{\(N(q-1)+j\)}_{i}\mathbf{u}_{j;q}.
	\]
	The asymptotic expansion of Proposition \ref{prop33} shows that $\mathbf{u}_{j;q} =\frac{1}{\|\bm{\xi}_{j,q}\|_{L^2(D_j)}^2}\bm{\xi}_{j,q}\chi_{D_j}  +\mathcal{O}(\delta)$, so we obtain the desired result inside the resonators.
	On the other hand, by expanding the result of Lemma~\ref{def:DTNGS} for small $\omega$, we obtain a Lam\'e-free function that interpolates between the boundaries of the resonators, that is, $\mathbf{V}_{j,q}$ outside the resonators.
	The proof is complete. 
\end{proof}

\section{Explicit formulas for layered concentric elastic sphere}\label{concentric_radial}

In this section, we provide some explicit calculations for the radial case of an $N$-layered concentric sphere (i.e., Matryoshka-type spherical bonded rubber bush mountings) concerning the elastic stiffness tensor in \eqref{CapmatrixGS} and  the asymptotic expansions of %displacement-to-traction map in \eqref{T_0expre}--\eqref{T_1expre} and  
subwavelength resonances in \eqref{full:asymptotic}.
Precisely, we give a sequence of resonators, $D_j,$ $j=1,2,\ldots,N$, by
\begin{equation}\label{Dj}
D_j = \{r_{j}^-<r:=|\Bx|< r_{j}^+\},
\end{equation}
the host matrix material $D_j'$, $j=0,1,\ldots,N,$ by
\begin{equation}\label{eq:aj}
D_{0}'=\{r>r_{1}^+\}, \quad D_j'=\{r_{j+1}^+<r< r_{j}^-\}, \quad  j=1,2,\ldots, N-1, \quad D_N'=\{r< r_N^-\},
\end{equation}
and the interfaces between the adjacent layers can be rewritten by
\begin{equation}\label{interface}
\Gamma_j^\pm=\left\{|\Bx|=r_j^\pm\right\}, \quad  j=1,2,\ldots,N,
\end{equation}
where $N\in \mathbb{N}$ and $r_j^\pm>0$. 

We remark that, in the concentric radial configuration, the rigid motion basis 
$\{\bm{\varkappa}_{p}\}_{p=1}^6$ in \eqref{rigid_motions} is automatically orthogonal 
in the $L^2(D_i)^3$ sense for any $i=1,2,\ldots,N$. Consequently, we may simply take 
$\bm{\xi}_{i,p}=\bm{\varkappa}_{p}$. Moreover, the quantities 
$\hat{\bm\phi}_{j;i,p}^{-}, \hat{\bm{\psi}}_{j+1;i,p}^{+}, 
\hat{\bm\varphi}_{j;i,p}^{-}, \hat{\bm{\zeta}}_{j+1;i,p}^{+}$ in \eqref{S^-101GS}, for $j=1,2,\ldots,N-1$, reduce to  
\begin{equation}\label{S^-101GSDi}
\begin{pmatrix}
\hat{\bm\phi}_{j,p}^{-}\\
\nm
\hat{\bm{\psi}}_{j+1,p}^{+}
\end{pmatrix}
:=\mathbb{S}_{j,j+1}^{-1}\begin{pmatrix}
\bm{\varkappa}_{p}\\
\nm
\bm{0}
\end{pmatrix},
\qquad
\begin{pmatrix}
\hat{\bm\varphi}_{j,p}^{-}\\
\nm
\hat{\bm{\zeta}}_{j+1,p}^{+}
\end{pmatrix}
:=\mathbb{S}_{j,j+1}^{-1}\begin{pmatrix}
\bm{0}\\
\nm
\bm{\varkappa}_{p}
\end{pmatrix},
\end{equation}
and $\hat{\bm\zeta}_{1;1,q}^+$ in Lemma \ref{EST} reduces to $\hat{\bm\zeta}_{1,q}^+ = \mathbf{S}_{\Gamma_1^+}^{-1}[\bm{\varkappa}_{p}]$, 
which shows that these quantities are independent of the domain $D_i$.
Moreover, the elastic stiffness tensor \eqref{CapmatrixGS} reduces to 
\begin{equation}\label{Spqij}
\begin{aligned}
\bm{\mathcal{S}}_{pqij} &= \delta_{i,j-1}\int_{\Gamma_j^+}  \hat{\bm \zeta}_{j,q}^+ \cdot \bm{\varkappa}_{p} ~\d \sigma + \delta_{i,j} \(\int_{\Gamma_{j+1}^+}    \hat{\bm\psi}_{j+1,q}^+ \cdot \bm{\varkappa}_{p}  ~\d \sigma - \int_{\Gamma_j^+} \hat{\bm \zeta}_{j,q}^+~ \cdot \bm{\varkappa}_{p}\d \sigma\)\\
&\quad  - \delta_{i,j+1} \int_{\Gamma_{j+1}^+}  \hat{\bm\psi}_{j+1,q}^+  \cdot \bm{\varkappa}_{p} ~\d \sigma.
\end{aligned}
\end{equation}

Let $\Gamma_0: = \left\{|\Bx|=r_0\right\}.$ It is known from \cite{DLbook2024} that for $p=1,2,3$,
\begin{equation}\label{staticTI123}
\mathbf{S}_{\Gamma_0}[\bm{\varkappa}_{p}](\mathbf{x})=
\begin{cases}
\ds -\frac{(2\lambda+5\mu)}{3\mu(\lambda+2\mu)}\frac{r_0^2}{|\mathbf{x}|}\bm{\varkappa}_{p}, & |\Bx|\geq r_0, \\
\nm
\ds  -\frac{(2\lambda+5\mu)r_0}{3\mu(\lambda+2\mu)}\bm{\varkappa}_{p}, & |\Bx|\le r_0, 
\end{cases}\quad \frac{\partial}{\partial\bm{v}
}\mathbf{S}_{\Gamma_0}[\bm{\varkappa}_{p}](\mathbf{x})=\begin{cases}
\ds \frac{r_0^2}{|\mathbf{x}|^2}\bm{\varkappa}_{p}, & |\Bx|\geq r_0, \\
\nm
\ds  0, & |\Bx|\le r_0, 
\end{cases}
\end{equation}
 and  for $p=4,5,6$, 
\begin{equation}\label{staticTI456}
\mathbf{S}_{\Gamma_0}[\bm{\varkappa}_{p}](\mathbf{x})=\begin{cases}
\ds -\frac {r_0^4}{3\mu |\mathbf{x}|^3}\bm{\varkappa}_{p}, & |\Bx|\geq r_0, \\
\nm
\ds  -\frac{r_0}{3\mu}\bm{\varkappa}_{p}, & |\Bx|\le r_0, 
\end{cases},\quad
\frac{\partial}{\partial\bm{v}
}\mathbf{S}_{\Gamma_0}[\bm{\varkappa}_{p}](\mathbf{x})=\begin{cases}
\ds \frac {r_0^4}{ |\mathbf{x}|^4}\bm{\varkappa}_{p}, & |\Bx|\geq r_0, \\
\nm
\ds  0, & |\Bx|\le r_0, 
\end{cases}
\end{equation}
It then follows from \eqref{singlejumpk} that for $p=1,2,\ldots,6$
\begin{equation}\label{NPstaticTI}
\mathbf{K}_{\Gamma_0}^{*}[\bm{\varkappa}_{p}]=\frac{1}{2}\bm{\varkappa}_{p}.
\end{equation}
Moreover, it follows from \eqref{SL1} that for $\Bx\in \RR^3$
\begin{equation}\label{SL1TI}
\mathbf{S}_{\Gamma_0,1}[\bm{\varkappa}_{p}](\mathbf{x}) =\begin{cases}
\ds -4\pi\ii \alpha r_0^2 \bm{\varkappa}_{p}, & \mbox{ if } p=1,2,3, \\
\nm
\ds  0, & \mbox{ if } p=4,5,6.
\end{cases}  
\end{equation}

\begin{lem}\label{EST_sphere}
	The elastic stiffness tensor  \eqref{CapmatrixGS} for an $N$-layered concentric sphere admits the following exact representation:
	\begin{equation}\label{ECM}
	\begin{aligned}
	\ds \bm{\mathcal{S}}_{pqij}: = \begin{cases}
	-\mathrm{S}^{\mathrm{T}}_{j-1} \chi_{\{1\le q\le 3\}} \delta_{p,q}- \mathrm{S}^{\mathrm{R}}_{j-1} \chi_{\{4\le q\le 6\}}\delta_{p,q}, &\text{ if } i=j-1,\\
	\nm
	\(\mathrm{S}^{\mathrm{T}}_{j-1}+\mathrm{S}^{\mathrm{T}}_j\)\chi_{\{1\le q\le 3\}}\delta_{p,q}+ \(\mathrm{S}^{\mathrm{R}}_{j-1}+\mathrm{S}^{\mathrm{R}}_j\) \chi_{\{4\le q\le 6\}}\delta_{p,q},
	&\text{ if }  1<i=j<N,  \\
	\nm
-\mathrm{S}^{\mathrm{T}}_j \chi_{\{1\le q\le 3\}} \delta_{p,q}- \mathrm{S}^{\mathrm{R}}_j \chi_{\{4\le q\le 6\}}\delta_{p,q}, &\text{ if } i=j+1,
	\end{cases}
	\end{aligned}
	\end{equation}
	where the translational  and rotational stiffness $\mathrm{S}^{\mathrm{T}}_j$ and $\mathrm{S}^{\mathrm{R}}_j$ of $D'_j$, for $j=0,1,\ldots,N$, are defined by, respectively,
	\begin{equation}\label{stif_tra}
	\mathrm{S}^{\mathrm{T}}_j =\begin{cases}
	\ds \frac{12\pi \mu(\lambda+2\mu)}{(2\lambda+5\mu)}{r_{1}^+},& \mbox{ for } j = 0,\\
	\nm
	\ds \frac{12\pi \mu(\lambda+2\mu)}{(2\lambda+5\mu)}\frac{r_{j}^- r_{j+1}^+ }{r_{j}^- -r_{j+1}^+},& \mbox{ for } j = 1,2,\ldots,N-1,\\
	\nm 
	\ds 0,& \mbox{ for } j = N,
	\end{cases} 
	\end{equation}
	and 
	\begin{equation}\label{stif_tor}
	\mathrm{S}^{\mathrm{R}}_j =\begin{cases}
	\ds {8\pi \mu}{\(r_{1}^+\)^3},& \mbox{ for } j = 0,\\
	\nm
	\ds {8\pi \mu}\frac{\(r_{j}^-\)^3 \(r_{j+1}^+\)^3}{\(r_{j}^-\)^3 -\(r_{j+1}^+\)^3},& \mbox{ for } j = 1,2,\ldots,N-1,\\
	\nm
	\ds 0, & \mbox{ for } j = N.
	\end{cases} 
	\end{equation}
	That is, $\bm{\mathcal{S}} = \diag\(\bm{\mathcal{S}}^{\mathrm{T}},\bm{\mathcal{S}}^{\mathrm{T}},\bm{\mathcal{S}}^{\mathrm{T}},\bm{\mathcal{S}}^{\mathrm{R}},\bm{\mathcal{S}}^{\mathrm{R}},\bm{\mathcal{S}}^{\mathrm{R}}\)$, where 
	\begin{equation}\label{capmatrix123}
	\begin{split}
	\bm{\mathcal{S}}^{\textup{m}}:= \begin{pmatrix}
	\ds \mathrm{S}^{\mathrm{m}}_{0}+\mathrm{S}^{\mathrm{m}}_1 &  \ds - \mathrm{S}^{\mathrm{m}}_1 & &  &  \\
	\nm
	\ds -\mathrm{S}^{\mathrm{m}}_1 & \ds \mathrm{S}^{\mathrm{m}}_{1}+\mathrm{S}^{\mathrm{m}}_2 & \ds - \mathrm{S}^{\mathrm{m}}_2&  &  \\
	\nm
	& \ddots &\ddots & \ddots& \\
	\nm
	& &\ds-\mathrm{S}^{\mathrm{m}}_{N-2} &\ds \mathrm{S}^{\mathrm{m}}_{N-2}+\mathrm{S}^{\mathrm{m}}_{N-1} &\ds -\mathrm{S}^{\mathrm{m}}_{N-1}\\
	\nm
	& &  &\ds-\mathrm{S}^{\mathrm{m}}_{N-1} &\ds \mathrm{S}^{\mathrm{m}}_{N-1}
	\end{pmatrix},
	\end{split}
	\end{equation}
	with $\mathrm{m}\in\{\mathrm{T,R}\}$.
\end{lem}
\begin{proof}[\bf Proof]
 By using \eqref{staticTI123} and \eqref{staticTI456}, we have that for $j=1,2,\ldots,N-1$,
 \[
 \begin{aligned}
 \begin{pmatrix}
 \mathbf{S}_{{\Gamma_j^-}}  & \mathbf{S}_{{\Gamma_j^-,\Gamma_{j+1}^+}} \\
 \nm 
 \mathbf{S}_{{\Gamma_{j+1}^+,\Gamma_j^-}}  & \mathbf{S}_{{\Gamma_{j+1}^+}}
 \end{pmatrix}
 \begin{pmatrix}
\bm{\varkappa}_{p}& 0\\
\nm 
 0&\bm{\varkappa}_{p}\\
 \end{pmatrix} =\begin{cases}
 \ds -\frac{(2\lambda+5\mu)}{3\mu(\lambda+2\mu)}\begin{pmatrix}
 \bm{\varkappa}_{p}& 0\\
 \nm 
 0&\bm{\varkappa}_{p}\\
 \end{pmatrix}
 \begin{pmatrix}
 r_j^-& \frac{\(r_{j+1}^+\)^2}{ r_j^-} \\
 r_{j}^- &r_{j+1}^+ \\
 \end{pmatrix},&\mbox{ if } p=1,2,3, \\
 \nm 
\ds -\frac{1}{3\mu}\begin{pmatrix}
\bm{\varkappa}_{p}& 0\\
\nm 
0&\bm{\varkappa}_{p}\\
\end{pmatrix}
\begin{pmatrix}
r_j^-& \frac{\(r_{j+1}^+\)^4}{ \(r_j^-\)^3} \\
r_{j}^- &r_{j+1}^+ \\
\end{pmatrix}, &\mbox{ if } p=4,5,6.
 \end{cases}  
 \end{aligned}
 \]
 Hence, we can get that
 \[
 \begin{aligned}
 \begin{pmatrix}
 \mathbf{S}_{{\Gamma_j^-}}  & \mathbf{S}_{{\Gamma_j^-,\Gamma_{j+1}^+}} \\
 \nm 
 \mathbf{S}_{{\Gamma_{j+1}^+,\Gamma_j^-}}  & \mathbf{S}_{{\Gamma_{j+1}^+}}
 \end{pmatrix}^{-1}
 \begin{pmatrix}
 \bm{\varkappa}_{p}& 0\\
 \nm 
 0&\bm{\varkappa}_{p}\\
 \end{pmatrix} =\begin{cases}
 \ds -\frac{3\mu(\lambda+2\mu)}{(2\lambda+5\mu)}\begin{pmatrix}
 \bm{\varkappa}_{p}& 0\\
 \nm 
 0&\bm{\varkappa}_{p}\\
 \end{pmatrix}
 \begin{pmatrix}
 \frac{1}{r_{j}^- -r_{j+1}^+}& -\frac{r_{j+1}^+}{r_{j}^-\(r_{j}^- -r_{j+1}^+\)} \\
 -\frac{r_{j}^-}{r_{j+1}^+\(r_{j}^- -r_{j+1}^+\)} &\frac{r_{j}^-}{r_{j+1}^+\(r_{j}^- -r_{j+1}^+\)} \\
 \end{pmatrix},&\mbox{ if } p=1,2,3, \\
 \nm 
 \ds -3\mu \begin{pmatrix}
 \bm{\varkappa}_{p}& 0\\
 \nm 
 0&\bm{\varkappa}_{p}\\
 \end{pmatrix}
 \begin{pmatrix}
  \frac{\(r_{j}^-\)^2}{\(r_{j}^-\)^3 -\(r_{j+1}^+\)^3}& -\frac{\(r_{j+1}^+\)^3}{r_j^-\(\(r_{j}^-\)^3 -\(r_{j+1}^+\)^3\)} \\
 -\frac{\(r_{j}^-\)^3}{r_{j+1}^+\(\(r_{j}^-\)^3 -\(r_{j+1}^+\)^3\)} &\frac{\(r_{j}^-\)^3}{r_{j+1}^+\(\(r_{j}^-\)^3 -\(r_{j+1}^+\)^3\)} \\
 \end{pmatrix}, &\mbox{ if } p=4,5,6.
 \end{cases}  
 \end{aligned}
 \]
 In view of \eqref{S^-101GSDi}, we have that for $j=1,2,\ldots,N-1$ and $p=1,2,3$,
 \[
 \hat{\bm\phi}_{j,p}^{-} = -\frac{3\mu(\lambda+2\mu)}{(2\lambda+5\mu)}\frac{\bm{\varkappa}_{p}}{r_{j}^- -r_{j+1}^+}, \quad \hat{\bm{\psi}}_{j+1,p}^{+} = \frac{3\mu(\lambda+2\mu)}{(2\lambda+5\mu)}\frac{r_{j}^- \bm{\varkappa}_{p}}{r_{j+1}^+\(r_{j}^- -r_{j+1}^+\)},
 \]
 \[
 \hat{\bm\varphi}_{j,p}^{-} = \frac{3\mu(\lambda+2\mu)}{(2\lambda+5\mu)}\frac{r_{j+1}^+ \bm{\varkappa}_{p}}{r_{j}^-\(r_{j}^- -r_{j+1}^+\)},\quad \hat{\bm{\zeta}}_{j+1,p}^{+} =  -\frac{3\mu(\lambda+2\mu)}{(2\lambda+5\mu)}\frac{r_{j}^- \bm{\varkappa}_{p}}{r_{j+1}^+\(r_{j}^- -r_{j+1}^+\)},
 \]
 and for $p=4,5,6$,
 \[
 \hat{\bm\phi}_{j,p}^{-} = -3\mu\frac{\(r_{j}^-\)^2 \bm{\varkappa}_{p}}{\(r_{j}^-\)^3 -\(r_{j+1}^+\)^3}, \quad \hat{\bm{\psi}}_{j+1,p}^{+} =  3\mu\frac{\(r_{j}^-\)^3 \bm{\varkappa}_{p}}{r_{j+1}^+\(\(r_{j}^-\)^3 -\(r_{j+1}^+\)^3\)},
 \]
 \[
 \hat{\bm\varphi}_{j,p}^{-} = 3\mu\frac{\(r_{j+1}^+\)^3 \bm{\varkappa}_{p}}{r_j^-\(\(r_{j}^-\)^3 -\(r_{j+1}^+\)^3\)},\quad \hat{\bm{\zeta}}_{j+1,p}^{+} =  -3\mu\frac{\(r_{j}^-\)^3 \bm{\varkappa}_{p}}{r_{j+1}^+\(\(r_{j}^-\)^3 -\(r_{j+1}^+\)^3\)}.
 \]
% \[
% \hat{\bm\phi}_{j,p}^{-} = \begin{cases}
% \ds -\frac{3\mu(\lambda+2\mu)}{(2\lambda+5\mu)}\frac{\bm{\varkappa}_{p}}{r_{j}^- -r_{j+1}^+},&\mbox{ if } p=1,2,3, \\
% \nm 
% \ds -3\mu\frac{\(r_{j}^-\)^2 \bm{\varkappa}_{p}}{\(r_{j}^-\)^3 -\(r_{j+1}^+\)^3},&\mbox{ if } p=4,5,6,
% \end{cases}
% \quad \hat{\bm{\psi}}_{j+1,p}^{+} = \begin{cases}
% \ds \frac{3\mu(\lambda+2\mu)}{(2\lambda+5\mu)}\frac{r_{j}^- \bm{\varkappa}_{p}}{r_{j+1}^+\(r_{j}^- -r_{j+1}^+\)},&\mbox{ if } p=1,2,3, \\
% \nm 
% \ds 3\mu\frac{\(r_{j}^-\)^3 \bm{\varkappa}_{p}}{r_{j+1}^+\(\(r_{j}^-\)^3 -\(r_{j+1}^+\)^3\)},&\mbox{ if } p=4,5,6,
% \end{cases}
% \]
% and 
% \[
% \hat{\bm\varphi}_{j,p}^{-} = \begin{cases}
% \ds \frac{3\mu(\lambda+2\mu)}{(2\lambda+5\mu)}\frac{r_{j+1}^+ \bm{\varkappa}_{p}}{r_{j}^-\(r_{j}^- -r_{j+1}^+\)},&\mbox{ if } p=1,2,3, \\
% \nm 
% \ds 3\mu\frac{\(r_{j+1}^+\)^3 \bm{\varkappa}_{p}}{r_j^-\(\(r_{j}^-\)^3 -\(r_{j+1}^+\)^3\)},&\mbox{ if } p=4,5,6,
% \end{cases}
% \quad 
% \hat{\bm{\zeta}}_{j+1,p}^{+} = \begin{cases}
% \ds -\frac{3\mu(\lambda+2\mu)}{(2\lambda+5\mu)}\frac{r_{j}^- \bm{\varkappa}_{p}}{r_{j+1}^+\(r_{j}^- -r_{j+1}^+\)},&\mbox{ if } p=1,2,3, \\
% \nm 
% \ds -3\mu\frac{\(r_{j}^-\)^3 \bm{\varkappa}_{p}}{r_{j+1}^+\(\(r_{j}^-\)^3 -\(r_{j+1}^+\)^3\)},&\mbox{ if } p=4,5,6.
% \end{cases}
% \]
 Moreover, we have
\[
\hat{\bm{\zeta}}_{1,p}^{+} = \begin{cases}
\ds -\frac{3\mu(\lambda+2\mu)}{(2\lambda+5\mu)}\frac{ \bm{\varkappa}_{p}}{r_{1}^+},&\mbox{ if } p=1,2,3, \\
\nm 
\ds -3\mu\frac{ \bm{\varkappa}_{p}}{r_{1}^+},&\mbox{ if } p=4,5,6.
\end{cases}
\]
Hence, in the concentric radial configuration,  by direct computations, we obtain that
 \[
 \int_{\Gamma_1^+}  \hat{\bm \zeta}_{1,q}^+ \cdot \bm{\varkappa}_{p} ~\d \sigma = \begin{cases}
 \ds -\frac{12\pi \mu(\lambda+2\mu)}{(2\lambda+5\mu)}{r_{1}^+}\delta_{p,q}, & \mbox{ if } p=1,2,3,\\
 \nm 
 \ds 
 -{8\pi \mu}{\(r_{1}^+\)^3}\delta_{p,q}, & \mbox{ if } p=4,5,6,
 \end{cases}
 \]
 and for $j=1,2,\ldots,N-1$,
 \[
 \int_{\Gamma_{j+1}^+}  \hat{\bm \zeta}_{j+1,q}^+ \cdot \bm{\varkappa}_{p} ~\d \sigma = \begin{cases}
 \ds -\frac{12\pi \mu(\lambda+2\mu)}{(2\lambda+5\mu)}\frac{r_{j}^- r_{j+1}^+ }{\(r_{j}^- -r_{j+1}^+\)}\delta_{p,q}, & \mbox{ if } p=1,2,3,\\
 \nm 
 \ds 
 -{8\pi \mu}\frac{\(r_{j}^-\)^3 \(r_{j+1}^+\)^3}{\(\(r_{j}^-\)^3 -\(r_{j+1}^+\)^3\)}\delta_{p,q}, & \mbox{ if } p=4,5,6,
 \end{cases}
 \]
 and 
 \[
 \int_{\Gamma_{j+1}^+}  \hat{\bm \psi}_{j+1,q}^+ \cdot \bm{\varkappa}_{p} ~\d \sigma = -\int_{\Gamma_{j+1}^+}  \hat{\bm \zeta}_{j+1,q}^+ \cdot \bm{\varkappa}_{p} ~\d \sigma, 
 \]
 which in turn verifies \eqref{zetaeqpsi}.
Therefore, it follow from \eqref{Spqij} that the conclusion holds ture. The proof is complete.
\end{proof}

Direct calculation
shows that
\[
\|\bm{\varkappa}_{p}\|_{L^2(D_j)}^2 = \begin{cases}
\ds \frac{4\pi}{3}\(\(r_{j}^+\)^3 -\(r_{j}^-\)^3\):={V}_{\mathrm{T}}(D_j),& \mbox{ if } p=1,2,3, \\
\nm
\ds \frac{8\pi}{15} \(\(r_{j}^+\)^5 -\(r_{j}^-\)^5\): = V_{\mathrm{R}}(D_j),& \mbox{ if } p=4,5,6.
\end{cases}
\]
We refer to $V_{\mathrm{T}}(D_j)$ and $V_{\mathrm{R}}(D_j)$ as the translational norm integral and the rotational norm integral, respectively, which represent the $L^2$-norms of the rigid motion basis functions over the spherical shell $D_j$.
We next introduce $N \times N$ matrices
\[
\bm{V}^{\mathrm{m}} :=\diag\(V_{\mathrm{m}}(D_1),V_{\mathrm{m}}(D_2),\ldots,V_{\mathrm{m}}(D_N)\), \mbox{ with } \mathrm{m}\in\{\mathrm{T,R}\}.
\]
Let $(\lambda^{\textup{m}}_{i},\bm{v}^{\textup{m}}_{i})_{1\leq i\leq N}$ be the eigenpairs satisfying 
\begin{equation}\label{GEP14}
{T}^{-2}\bm{\mathcal{S}}^{\textup{m}}\bm{v}^{\textup{m}}_{i} = \lambda^{\textup{m}}_{i} \bm{V}^{\textup{m}}\bm{v}^{\textup{m}}_{i}, \mbox{ for } 1\leq i\leq N, 	\mbox{ with } \mathrm{m}\in\{\mathrm{T,R}\}.
\end{equation}
where $T$ is given by \eqref{element_T}, and the eigenvectors form an orthonormal basis with respect to the following inner product
\begin{equation}\label{orthonormalV14}
\(\bm{v}_{i}^{\textup{m}}\)^T\bm{V}^{\textup{m}}\bm{v}^{\textup{m}}_{j} = \delta_{i,j}.
\end{equation}

In view of \eqref{capmatrix123} and using \cite[Lemma 7.7.1]{SEP1998}, we have the following lemma.

\begin{lem}\label{eigenvalueECM}
	The $N$ eigenvalues of the matrix $\bm{\mathcal{S}}^{\textup{m}}$ defined by \eqref{capmatrix123}, satisfy the following relations:
	\begin{equation}\label{eigenpositive}
	0<\lambda_{1}^{\mathrm{m}}<\lambda_{2}^{\mathrm{m}}<\cdots<\lambda_{N}^{\mathrm{m}}.
	\end{equation}
\end{lem}

Combining Theorem \ref{theoasy} with our above preparations, we can obtain the following result.

\begin{prop}\label{propMLCB}
	In the concentric radial configuration given by  \eqref{Dj}--\eqref{interface}, the subwavelength resonant frequencies admit the following explicit representation as $\delta\to 0$:
\begin{equation}\label{Omega:delta_CR}
\begin{aligned}
\omega^{\mathrm{T},\pm}_{i}(\delta^{\frac{1}{2}})&=\pm\sqrt{\frac{\lambda_i^{\mathrm{T}}}{\rho}}\delta^{\frac{1}{2}}- \ii\frac{72\alpha\pi^2 (r_1^+)^2}{\rho\tau_1^2}\(\frac{\mu(\lambda+2\mu)}{2\lambda+5\mu}\)^2 \(\bm{v}^{\mathrm{T},(1)}_{i}\)^2 \delta+o(\delta),\quad i=1,2,\ldots,N,\\
\omega^{\mathrm{R},\pm}_i(\delta^{\frac{1}{2}})&=\pm\sqrt{\frac{\lambda^{\mathrm{R}}_i}{\rho}}\delta^{\frac{1}{2}}+o(\delta),\quad i=1,2,\ldots,N,
\end{aligned}
\end{equation}	
each of which has multiplicity 3 up to order  $\mathcal{O}(\delta)$.
Moreover, 
let $\mathbf{u}^{\mathrm{m}}_{i}$ be a eigenmode corresponding to $\omega_i^{\mathrm{m},+}$, and let $\bm{v}^{\mathrm{m},(j)}_i$ be the $j$-th entry of the eigenvector $\bm{v}^{\mathrm{m}}_i$ with $\mathrm{m}\in\{\mathrm{T,R}\}$. Then we have
\begin{equation}\label{eigenmodesphere}
\mathbf{u}^{\mathrm{T}}_{i} = \sum_{j=1}^{N} \sum_{q=1}^{3} \bm{v}^{\mathrm{T},\(j\)}_{i}\mathbf{V}_{j;q}+\mathcal{O}(\delta^{\frac{1}{2}}), \mbox{ and } \mathbf{u}^{\mathrm{R}}_{i} = \sum_{j=1}^{N} \sum_{q=4}^{6} \bm{v}^{\mathrm{R},\(j\)}_{i}\mathbf{V}_{j;q}+\mathcal{O}(\delta^{\frac{1}{2}}), 
\end{equation}
where $\mathbf{V}_{j;q}$, for $j=1,2,\ldots,N$ and $q=1,2,\ldots,6$, can be explicit given by 
\begin{equation}\label{V_i_solu}
\mathbf{V}_{j,q}(\Bx) =
\begin{cases}
\ds \(- \frac{r_{j+1}^{+}}{|\Bx|} + 1\)\frac{r_j^{-}\bm{\varkappa}_q  \chi_{\{1\le q\le 3\}}}{r_j^{-}-r_{j+1}^{+}}+ \(- \frac{\(r_{j+1}^{+}\)^3}{|\Bx|^3} + 1\) \frac{\(r_j^{-}\)^3\bm{\varkappa}_q \chi_{\{4\le q\le 6\}}}{\(r_j^{-}\)^3-\(r_{j+1}^{+}\)^3} , & \mbox{ for }r_{j+1}^{+}\leq |\Bx|\leq r_j^{-},\\
\nm 
\ds \bm{\varkappa}_q, & \mbox{ for }r_{j}^{-}\leq |\Bx|\leq r_j^{+},\\
\nm 
\ds \( \frac{r_{j-1}^{-}}{|\Bx|} - 1\)\frac{r_j^{+} \bm{\varkappa}_q \chi_{\{1\le q\le 3\}}}{r_{j-1}^{-}-r_{j}^{+}} + \( \frac{\(r_{j-1}^{-}\)^3}{|\Bx|^3} - 1\)\frac{\(r_j^{+}\)^3 \bm{\varkappa}_q \chi_{\{4\le q\le 6\}}}{\(r_{j-1}^{-}\)^3-\(r_{j}^{+}\)^3}, & \mbox{ for }r_{j}^{+}\leq |\Bx|\leq r_{j-1}^{-},\\
\nm
\ds \bm{0}, & \text{else},
\end{cases}
\end{equation}
with $r_{N+1}^+ = 0$ and $r_0^- = +\infty$.
\end{prop}

\begin{rem}
	We now summarize some remarks.
First, if ${D}$ is a single-layered homogeneous resonator, i.e., $N=1$ and $r_1^-=0$, it then follows from \eqref{capmatrix123}, \eqref{GEP14} and \eqref{Omega:delta_CR} that
	\[
	\omega_1^{\mathrm{m},+}(\delta^{\frac{1}{2}}) =   \delta^{\frac{1}{2}}\(\frac{\mathrm{S}^{\mathrm{m}}_{0}}{\tau_1^2 V_{\mathrm{m}}(D_1)}\)^{\frac{1}{2}}(1+o(1)) = \begin{cases}
	\ds \frac{1}{\tau_1 r_1^+}\sqrt{\frac{9 \mu(\lambda+2\mu)}{(2\lambda+5\mu)\rho }} \delta^{\frac{1}{2}}(1+o(1)),& \mathrm{ if \;m = T}, \\
	\nm 
	\ds \frac{1}{\tau_1 r_1^+}\sqrt{\frac{15\mu}{\rho}} \delta^{\frac{1}{2}}(1+o(1)),& \mathrm{ if \;m = R}.
	\end{cases}
	\]
	which is the result obtained in  \cite{LZArxiv}. 
	Moreover, the translational stiffness \eqref{stif_tra} can be seen as the elastic analogue of the capacity in multi-layered high-contrast acoustic resonators \cite{DKLZ_JDE26} governed by scalar Helmholtz systems. In contrast to the scalar case, however, the null space of the elastostatic equation \eqref{elastostatics} is six-dimensional, consisting not only of rigid translations 
	$\bm{\varkappa}_p$, $p = 1, 2, 3$, but also of rigid rotations 
	$\bm{\varkappa}_p$, $p = 4, 5, 6$. These rotational rigid body modes are essential for the rigorous derivation of the rotational stiffness \eqref{stif_tor}, which has no counterpart in scalar wave models.
	Finally,
	from Proposition \ref{propMLCB}, we can see that the translational  and rotational eigenmodes $\mathbf{u}^{\mathrm{T}}_{(i)}$ and $\mathbf{u}^{\mathrm{R}}_{(i)}$ can be approximated as a point scatterer with monopole and dipole modes, respectively. That is $\mathbf{u}^{\mathrm{T}}_{i} = \mathcal{O}\({1}/{|\Bx|}\)$ and $\mathbf{u}^{\mathrm{R}}_{i} = \mathcal{O}\({1}/{|\Bx|^2}\)$ when $|\Bx|$ is sufficiently large.	
\end{rem}

\section{Subwavelength guided modes in the defect}\label{sec5}
In this section, we study finite-layered metamaterials composed of dimer- and monomer-type resonators, considering configurations with geometric and material-parameter defects, respectively, in order to establish a rigorous exponential localization criterion for the associated wave modes. We only consider systems of identically translational or rotational norm integral resonators, that is
\begin{equation}\label{identicalvolume}
V_{\mathrm{m}}(D_1)=V_{\mathrm{m}}(D_2)=\cdots=V_{\mathrm{m}}(D_N) = 1, \mbox{ with } \mathrm{m}\in\{\mathrm{T,R}\}.
\end{equation}
Hence, from \eqref{GEP14}, it suffices to calculate the eigenpairs $(\lambda^{\textup{m}}_{i},\bm{v}^{\textup{m}}_{i})_{1\leq i\leq N}$ of matrix ${T}^{-2}\bm{\mathcal{S}}^{\textup{m}}$, with $\mathrm{m}\in\{\mathrm{T,R}\}$.
Before that, we first recall the elastic stiffness matrix \eqref{capmatrix123} has the  following tridiagonal form:
\begin{equation}\label{Ccap}
\begin{split}
\bm{\mathcal{S}}^{\textup{m}}:=  \begin{pmatrix}
\ds
\alpha^{\textup{m}}_1 &  \ds  \beta^{\textup{m}}_1 & &  & & \\
\nm
\ds \beta^{\textup{m}}_1 & \ds \alpha^{\textup{m}}_2 & \ds  \beta^{\textup{m}}_2&  & &\\
\nm
& \ds\beta^{\textup{m}}_2 & \ds \alpha^{\textup{m}}_3 & \ds\beta^{\textup{m}}_3 & & \\
\nm
& & \ddots &\ddots & \ddots& \\
\nm
&  & &\ds\beta^{\textup{m}}_{N-2} &\ds \alpha^{\textup{m}}_{N-1} &\ds \beta^{\textup{m}}_{N-1}\\
\nm
&  & &  &\ds\beta^{\textup{m}}_{N-1} &\ds \alpha^{\textup{m}}_N
\end{pmatrix},
\end{split}
\end{equation}
with $\mathrm{m}\in\{\mathrm{T,R}\}$, where $\alpha^{\textup{m}}_N = \mathrm{S}^{\mathrm{m}}_{N-1}$ and for $i = 1,2,\ldots,N-1$,
\begin{equation}\label{alphabeta}
\alpha^{\textup{m}}_i = \mathrm{S}^{\mathrm{m}}_{i-1}+\mathrm{S}^{\mathrm{m}}_{i} \mbox{ and }	\beta^{\textup{m}}_i = -\mathrm{S}^{\mathrm{m}}_{i}.
\end{equation}

\subsection{Localized modes generated by geometrical defects}\label{dimerdefect}
In this section, we investigate finite-layered metamaterials composed of dimer-type resonators, focusing on configurations that include a geometric defect. 
We also simplify $\rho_{\rmr,j} = \rho_{\rmr}$ for all $1\leq j\leq N$, i.e., $\tau_j = \tau$ for all $1\leq j\leq N$. One
remarks that in this case the eigenvalue problem  \eqref{GEP14} may be simplified by finding
eigenvalues of $ \bm{\mathcal{S}}^{\textup{m}}$ and multiplying the eigenvalues by $\tau^{-2}$.

We next give the definition of multi-layered dimer-type metamaterials both with and without defects, respectively. 

\begin{defn}
	Let $D$ denote the $N$-layered metamaterial illustrated in Figure \ref{MLHCCB} with material parameters given by \eqref{nestedcomplement}, and let $\rho_{\rmr,j} = \rho_{\rmr}$ for all $1\leq j\leq N$.
	We define $D$ to be an $N$-layered dimer-type metamaterial without defects, if the repeating pattern of alternating stiffness values is strictly  maintained. That is, for $N=2n$, $n\in \mathbb{N}$ large enough, the elastic stiffness matrix \eqref{Ccap} satisfies:
	\begin{equation}\label{dimersetup}
	\beta^{\textup{m}}_i =\begin{cases}
	\ds \beta^{\textup{m}}_1, & \mbox{if } i \mbox{ is odd}, \\
	\nm
	\ds   \beta^{\textup{m}}_2, & \mbox{if } i \mbox{ is even},
	\end{cases} \;
	1\leq i\leq 2n-1,
	\end{equation}
	with $\beta^{\textup{m}}_1<\beta^{\textup{m}}_2<0$. 
	Conversely, we define $D$ to be an $N$-layered dimer-type metamaterial with a defect if the repeating pattern is disrupted. Specifically, for $N=4n+1$, the elastic stiffness matrix \eqref{Ccap} satisfies:
	\begin{equation}\label{dimersetupdefect}
	\beta^{\textup{m}}_i =\begin{cases}
	\ds \beta^{\textup{m}}_1, & \mbox{if } i \mbox{ is odd}, \\
	\nm
	\ds   \beta^{\textup{m}}_2, & \mbox{if } i \mbox{ is even},
	\end{cases} \;
	1\leq i\leq 2n,
	\mbox{ and }
	\beta^{\textup{m}}_i =\begin{cases}
	\ds \beta^{\textup{m}}_1, & \mbox{if } i \mbox{ is even}, \\
	\nm
	\ds   \beta^{\textup{m}}_2, & \mbox{if } i \mbox{ is odd},
	\end{cases} \;
	2n+1\leq i\leq 4n.
	\end{equation}
\end{defn}

\begin{defn}[Asymptotic spectral bulk and gap] \label{def:spectralgap}
	Consider a structure consisting of finite-layered nested resonators. The \emph{asymptotic spectral bulk} $\Xi $ and \emph{asymptotic spectral gap} $\Theta$ of the structure are defined as the spectral bulk and spectral gap of the corresponding infinite-layered structure respectively.
\end{defn}

By using the properties of Chebyshev polynomials, we can prove the existence of a spectral gap for  dimer-type metamaterials without defects.

\begin{prop}\label{propABG}
	Consider an $N$-layered dimer-type metamaterial without defects, which satisfies the setup \eqref{dimersetup}. Let $\bm{\mathcal{S}}^{\textup{m}}\in\mathbb{R}^{2n \times 2n}$ be the associated stiffness matrix of the unperturbed structure. If for $n$ large enough, there holds
	\begin{equation}\label{gapcondition}
	|\beta^{\textup{m}}_0|\leq \left|\frac{4\beta^{\textup{m}}_2(\beta^{\textup{m}}_1+\beta^{\textup{m}}_2)}{\beta^{\textup{m}}_1+2\beta^{\textup{m}}_2}\right|.
	\end{equation}
	Then
	\begin{equation}\label{ASB}
	\Xi = \overline{\lim_{n\to\infty} \sigma(\bm{\mathcal{S}}^{\textup{m}})} = \left[0,-2\beta^{\textup{m}}_2\right] \cup \left[-2\beta^{\textup{m}}_1, -2(\beta^{\textup{m}}_1+\beta^{\textup{m}}_2)\right],
	\end{equation}
	where $\lim$ denotes the Hausdorff limit. Consequently, the asymptotic spectral gap is
	\begin{equation}\label{ASG}
	\Theta = \left(-2\beta^{\textup{m}}_2, -2\beta^{\textup{m}}_1 \right)\subset \RR.
	\end{equation}
\end{prop}

\begin{rem}
	For any fixed outermost boundary~$\Gamma_1^+$, preserving the nested geometry of the layered structure as the number of layers tends to infinity necessitates  that both the thickness of each resonator and the spacing between adjacent resonators be sufficiently small. Specifically, in view of \eqref{alphabeta}, \eqref{dimersetup}, and the definition of the stiffness coefficients $\mathrm{S}_j^{\textup{m}}$ of $D_j'$ for $j =0, 1,2,\ldots,N$, and $\mathrm{m}\in \{\mathrm{T,R}\}$ (see \eqref{stif_tra}--\eqref{stif_tor}), this geometric constraint implies that $|\beta_1^{\mathrm{m}}|$ and $|\beta_2^{\mathrm{m}}|$ must both be sufficiently large. Consequently, for any multi-layered structure with a fixed outermost boundary and a sufficiently large number of layers, condition~\eqref{gapcondition} is automatically satisfied. Furthermore, since the structure consists of finitely many layers made of prescribed materials, it can be fabricated using additive layer manufacturing techniques \cite{MBKPD_PNAS2016}, taking the form of a radially laminated medium with ultra-thin coatings.
\end{rem}

We next consider an $N$-layered dimer-type metamaterial with a defect, which satisfies the setup \eqref{dimersetupdefect}. Following the framework established in recent work \cite{DKZ_JLMS2026}, we can derive a localization criterion based on the position of eigenvalues in the spectrum  (cf. \cite[Proposition 4.3]{DKZ_JLMS2026}), particularly, eigenvector associated with the mid-gap eigenvalue in the asymptotic spectral gap exhibits exponential decay in both directions away from the defect located in the middle layer.  We refer to the corresponding eigenmode as a localized defect mode.
In contrast, eigenvectors corresponding to bulk eigenvalues display oscillatory behavior with trigonometric function characteristics. At the boundary of the asymptotic spectral bulk, eigenvectors show approximately linear growth behavior, representing a transitional state between the exponentially localized and oscillatory regimes.  As shown in \eqref{eigenmodesphere}, eigenmodes, similar to eigenvectors, exhibit three characteristic states depending on the distribution of eigenvalues across the spectrum. Furthermore, for structures with a sufficiently large number of layers, we can prove the existence and uniqueness of an  mid-gap eigenvalue in the spectral gap of the dimer-type structure with a defect, 
which converges to
\begin{equation}\label{Lexpdimer}
 {\frac{1}{2}\(-\sqrt{9 \(\beta_1^{\textup{m}}\)^2-14 \beta_1^{\textup{m}} \beta_2^{\textup{m}}+9 \(\beta_2^{\textup{m}}\)^2}-3\beta_1^{\textup{m}}-3\beta_2^{\textup{m}}\)}\in  \Theta,
\end{equation}
 as $N\to+\infty$, thereby establishing the existence of a unique localized defect mode (cf. \cite[Theorem 4.4]{DKZ_JLMS2026}).

From \eqref{ASG}, we can see that the spectral gap can be tailored by modulating the stiffness values \eqref{stif_tra}--\eqref{stif_tor}, thereby enabling control over its width or the opening and closing of the spectral gap.  Motivated by this observation, it is natural to pose the following question: 
\[
\mbox{{\em When the spectral gap closes} (i.e.,  }\beta^{\textup{m}}_1=\beta^{\textup{m}}_2), \mbox{{\em by what mechanism can a localized defect mode be excited?}}
\] 
To address this question, we next investigate a monomer-type metamaterial containing material-parameter defects in the following subsection.

\subsection{Localized modes generated by material-parameter defects}\label{LM_mpd}

In this section, we investigate finite-layered metamaterials composed of monomer-type resonators, focusing on configurations that include a material-parameter defect.

We first give the definition of finite-layered monomer-type metamaterials both with and without defects, respectively. 

\begin{defn}
	Let $D$ denote the $N$-layered metamaterial illustrated in Figure \ref{MLHCCB} with material parameters given by \eqref{nestedcomplement}.
	We call $D$ an $N$-layered monomer-type metamaterial without defects if the following conditions hold:
	\[
	\tau_i = 1 \mbox{ and }\beta^{\textup{m}}_i = \beta^{\textup{m}}<0 \mbox{ for all } 1\leq i\leq N.
	\]
	Conversely, we call $D$ an $N$-layered monomer-type metamaterial with finite multiple material-parameter defects in $D_{\Lambda}: = \cup_{j\in \Lambda}D_j$ with  $\Lambda := \{i_1,i_2,\ldots,i_M:1<i_1<i_2<,\ldots,<i_M< N\; \mbox{and }\;M\ll N \}$ if  
	\begin{equation}\label{monomersetupdefect1}
	\beta^{\textup{m}}_i = \beta^{\textup{m}}<0 \mbox{ for all } 1\leq i\leq N, \mbox{ and }
	\tau_{i} =\begin{cases}
	\ds 1+\eta_{l}, & \mbox{if } i\in \Lambda, \\
	\nm
	\ds   1, & \mbox{else} ,
	\end{cases}
	\end{equation}
	for some parameters $\eta_{l}>-1$.
\end{defn}

Observe that $\eta_{l} =0$ corresponds to the unperturbed case, and $|\eta_{l}|$ describes the magnitude of the perturbation.
 Hence, defective materials are characterised by the eigenvalue problem
\begin{equation}\label{GEP142}
{T}^{-2}\bm{\mathcal{S}}^{\textup{m}}\bm{v}^{\textup{m}}_{i} = \lambda^{\textup{m}}_{i} \bm{v}^{\textup{m}}_{i}, \mbox{ for } 1\leq i\leq N, 	\mbox{ with } \mathrm{m}\in\{\mathrm{T,R}\}.
\end{equation}

We next recall results from \cite{ANZIAM, CDF_AMS} concerning the spectrum of tridiagonal $1$-Toeplitz matrices with perturbations at the diagonal corners.
Denote by
\begin{equation}\label{1toeplitzCP}
A_{N}^{(a, b)}(\alpha, \beta):=\left(\begin{array}{ccccccc}
\alpha+a & \beta &  &  &  &   \\
\beta & \alpha & \beta &  &  &    \\
& \beta & \alpha & \beta &  &    \\
&	 & \ddots &\ddots & \ddots&  \\
&  &  & \beta & \alpha & \beta \\
&  &  &  & \beta & \alpha+b
\end{array}\right) \in \RR^{N\times N}
\end{equation}
the tridiagonal $1$-Toeplitz matrices with perturbations on the diagonal corners.
The characteristic polynomial of $A_{N}(\alpha, \beta):=A_{N}^{(0, 0)}(\alpha, \beta)$ is
\begin{equation}\label{Pn_chebpoly}
P_N^*(\lambda) := \beta^NU_N\(z(\lambda)\),\mbox{ with } z(\lambda):=\frac{\lambda-\alpha}{2\beta}, 
\end{equation}
where the Chebyshev polynomial of the second kind  satisfying the three-term recurrence relations
\begin{equation}\label{3termChebyshev}
U_{n+1}(s)  =2 s U_n(s)-U_{n-1}(s), \;\mbox{ for all } n=1,2,\ldots
\end{equation}
with initial conditions $U_0(s) =1$ and $U_1(s)  =2s$.
More specifically,  the Chebychev polynomials of second kind admit the following form
\begin{equation}\label{Chebypoly}
U_{n}(s) =\begin{cases}
\ds (\textup{sgn}(s))^{n}\frac{\sinh\((n+1)\textup{arccosh}(|s|)\)}{\sqrt{s^2-1}}, & \mbox{if } |s|>1, \\
\nm
\ds   \frac{\sin\((n+1)\textup{arccos}(s)\)}{\sqrt{1-s^2}}, & \mbox{if } |s|<1,\\
\nm
\ds s^{n}(n+1), &\mbox{if } |s|=1.
\end{cases}
\end{equation}

\begin{lem}\label{eigenvalue}
 The characteristic polynomial $\mathcal{P}_{A_{N}^{(a, b)}(\alpha, \beta)}(\lambda)$ of $A_{N}^{(a, b)}(\alpha, \beta)$ can be given by
 \[
 \mathcal{P}_{A_{N}^{(a, b)}(\alpha, \beta)}(\lambda) = P_N^*(\lambda) - (a+b)P_{N-1}^*(\lambda)+abP_{N-2}^*(\lambda).
 \]
\end{lem}
\begin{proof}[\bf Proof]
	It follows from \eqref{1toeplitzCP} that
	\[
	A_{N}^{(a, b)}(\alpha, \beta) = A_{N}(\alpha, \beta)+\bm{X}\begin{pmatrix}
	a&0\\
	\nm
	0&b
	\end{pmatrix}\bm{X}^t,
	\]
where $\bm{X} := \(\bm{e}_1,\bm{e}_N\)$ is the $N$-by-2 matrix and  
${\bm{e}}_i $ denotes the $N$-dimensional vector whose $i$-th entry is equal to one.
Then, by straightforward algebraic manipulations, we obtain
\begin{equation}\label{PAab}
\mathcal{P}_{A_{N}^{(a, b)}(\alpha, \beta)}(\lambda) =  P_N^*(\lambda)\begin{vmatrix}
1-a\(A^{-1}_{N}(\alpha, \beta)\)_{11} &-a \(A^{-1}_{N}(\alpha, \beta)\)_{1N}\\
\nm 
-b \(A^{-1}_{N}(\alpha, \beta)\)_{N1}&1-b\(A^{-1}_{N}(\alpha, \beta)\)_{NN}
\end{vmatrix}.
\end{equation}\label{invA11}
	From \eqref{1toeplitzCP} and \eqref{Pn_chebpoly}, we obtain
	\begin{equation}
	\(A^{-1}_{N}(\alpha, \beta)\)_{11} =  \(A^{-1}_{N}(\alpha, \beta)\)_{NN} = \frac{P_{N-1}^*(\lambda)}{P_N^*(\lambda)},
	\end{equation}
	and 
	\begin{equation}\label{invA1N}
		\(A^{-1}_{N}(\alpha, \beta)\)_{1N} = \(A^{-1}_{N}(\alpha, \beta)\)_{N1} = \frac{\beta^{N-1}}{P_N^*(\lambda)}.
	\end{equation}
	Therefore, combining \eqref{PAab}--\eqref{invA1N} with the identity
	\[
	U_N^2(\lambda) -U_{N-1}(\lambda) U_{N+1}(\lambda) = 1,
	\]
	completes the proof.
\end{proof}

By Taking $\beta^{\textup{m}}_1=\beta^{\textup{m}}_2: =\beta^{\textup{m}} $ in Proposition \ref{propABG}, we obtain the following result.

\begin{prop}\label{propABG1}
	Consider an $N$-layered monomer-type metamaterial without defects. Let $\bm{\mathcal{S}}^{\textup{m}}\in\mathbb{R}^{N \times N}$ be the associated stiffness matrix of the unperturbed structure. If for $N$ large enough, there holds 
	$
	-3\beta^{\textup{m}}_0\leq -8\beta^{\textup{m}}.
	$
	Then
	\begin{equation}\label{ASB1}
	\Xi = \overline{\lim_{N\to\infty} \sigma(\bm{\mathcal{S}}^{\textup{m}})} = \left[0,-4\beta^{\textup{m}}\right],
	\end{equation}
	where $\lim$ denotes the Hausdorff limit. 
\end{prop}

For such a layered monomer-type metamaterial, the associated stiffness matrix has the  following tridiagonal form:
\begin{equation}\label{CcapWOD}
\begin{split}
\bm{\mathcal{S}}^{\textup{m}}:=  \begin{pmatrix}
\ds
-\beta_0^{\textup{m}}-\beta^{\textup{m}} &  \ds  \beta^{\textup{m}} & &  & & \\
\nm
\ds \beta^{\textup{m}} & \ds -2\beta^{\textup{m}} & \ds  \beta^{\textup{m}}&  & &\\
\nm
& \ds\beta^{\textup{m}} & \ds -2\beta^{\textup{m}} & \ds\beta^{\textup{m}} & & \\
\nm
& & \ddots &\ddots & \ddots& \\
\nm
&  & &\ds\beta^{\textup{m}} &\ds -2\beta^{\textup{m}} &\ds \beta^{\textup{m}}\\
\nm
&  & &  &\ds\beta^{\textup{m}} &\ds -\beta^{\textup{m}}
\end{pmatrix} = A_{N}^{(\beta^{\textup{m}}- \beta^{\textup{m}}_0, \beta^{\textup{m}})}(-2\beta^{\textup{m}},\beta^{\textup{m}}),
\end{split}
\end{equation}
with $\mathrm{m}\in\{\mathrm{T,R}\}$. It is easy to see that 
\begin{equation}\label{1Toeplitz}
\bm{\mathcal{S}}^{\textup{m}} 
= A_{N}^{(0, 0)}(-2\beta^{\textup{m}},\beta^{\textup{m}}) + \(\bm{e}_1,\bm{e}_N\) 
\begin{pmatrix}
\beta^{\textup{m}}-\beta_0^{\textup{m}}&0\\
\nm
0&\beta^{\textup{m}}
\end{pmatrix}\begin{pmatrix}
\bm{e}_1^t\\
\nm
\bm{e}_N^t\end{pmatrix}:=A_{N}^{(0, 0)}(-2\beta^{\textup{m}},\beta^{\textup{m}})+\bm{X}\bm{R}^{\textup{m}}\bm{X}^t, 
\end{equation}
where $\bm{X} := \(\bm{e}_1,\bm{e}_N\)$ is the $N$-by-2 matrix, $\bm{R}^{\textup{m}}:=\diag\(\beta^{\textup{m}}-\beta_0^{\textup{m}},\beta^{\textup{m}}\)$,  
 ${\bm{e}}_i $ is the $N$-dimensional vector with the $i$-th entry being one.

\begin{lem}\label{Lemm52}
		Consider an $N$-layered monomer-type metamaterial with multiple material-parameter defects in $D_{\Lambda} = \cup_{j=i_1}^{i_M}D_j$ with $1<i_1<i_2<\ldots<i_M< N$ and $M\ll N$, which satisfies the setup \eqref{monomersetupdefect1}.
		Let $\bm{\mathcal{S}}^{\textup{m}}\in\mathbb{R}^{N \times N}$ be the stiffness matrix \eqref{CcapWOD} of the monomer-type structure and let $(\lambda^{\textup{m}},\bm{v}^{\textup{m}})$ be an eigenpair of ${T}^{-2}\bm{\mathcal{S}}^{\textup{m}}$ but $\lambda^{\textup{m}}\notin \sigma (\bm{\mathcal{S}}^{\textup{m}})$.
		Then the eigenvector $\bm{v}^{\textup{m}}$ admits the following representation
		\[
		\bm{v}^{\textup{m}} 
%		=  \lambda^{\textup{m}}\eta(2+\eta) \bm{v}^{\textup{m},(j)}\(\bm{\mathcal{S}}^{\textup{m}} - \lambda^{\textup{m}}\bm{I} \)^{-1}\bm{e}_j
		 = \sum_{l=1}^{M}\((1+\eta_l)^{-2}-1\) \(\bm{\mathcal{S}}^{\textup{m}}  \bm{v}^{\textup{m}}\)^{(i_l)} \(\lambda^{\textup{m}}\bm{I}_N-\bm{\mathcal{S}}^{\textup{m}}  \)^{-1}\bm{e}_{i_l},
		\]
where $\bm{I}_N$ is the $N$-by-$N$ identity matrix.
\end{lem}
\begin{proof}[\bf Proof]
	Note that $T$ is diagonal and is given by
	\begin{equation}\label{T}
	T_{ii} =\begin{cases}
	\ds 1+\eta_{l}, & \mbox{if } i=i_l, l = 1,2,\ldots,M, \\
	\nm
	\ds   1, & \mbox{else } ,
	\end{cases}
	\end{equation}
which implies that 
	\[
	T^{-2} = \bm{I}_N + \sum_{l=1}^M\((1+\eta_l)^{-2}-1\)\bm{e}_{i_l}\bm{e}_{i_l}^t.
	\]
	A direct computation yields, 
\[
{T}^{-2}\bm{\mathcal{S}}^{\textup{m}}\bm{v}^{\textup{m}} = \lambda^{\textup{m}} \bm{v}^{\textup{m}} \Leftrightarrow  \(\lambda^{\textup{m}}\bm{I}_N-\bm{\mathcal{S}}^{\textup{m}}  \)\bm{v}^{\textup{m}} =   \sum_{l=1}^{M}\((1+\eta_l)^{-2}-1\) \bm{e}_{i_l} \bm{e}_{i_l}^t\bm{\mathcal{S}}^{\textup{m}}  \bm{v}^{\textup{m}},
\]	
which concludes the proof.
\end{proof}

\begin{thm}\label{Propiff2}
	Consider an $N$-layered monomer-type metamaterial containing  $M$ material-parameter defects in $D_{\Lambda} = \cup_{j=i_1}^{i_M}D_j$ with $1<i_1<i_2<\ldots<i_M\ll N$, satisfying the setup \eqref{monomersetupdefect1}, where the perturbation parameters $\eta_l$ are mutually distinct.
	Let $\bm{\mathcal{S}}^{\textup{m}}\in\mathbb{R}^{N \times N}$ be the stiffness matrix \eqref{CcapWOD} of the monomer-type structure and let $(\lambda^{\textup{m}},\bm{v}^{\textup{m}})$ be an eigenpair of ${T}^{-2}\bm{\mathcal{S}}^{\textup{m}}$. 
	If the following conditions holds: 
	\begin{equation}
-1<\eta_l<-1+\frac{\sqrt{4i_l-1}}{2\sqrt{i_l}},\;\mbox{ and}\; \min_{2\leq l\leq M}|i_l-i_{l-1}|\gg 1,
	\end{equation}
	then for $N$ large  enough, there exist  $M$ distinct eigenvalues   above the spectrum bulk $\Xi = \left[0,-4\beta^{\textup{m}}\right]$, denoted by $\lambda^{\textup{m}}_{(i_l,\eta_l),N}$, $l=1,2,\ldots,M$.
Each eigenvalue $\lambda^{\textup{m}}_{(i_l,\eta_l),N}$, $l=1,2,\ldots,M$, converges at leading order to the unique solution   $\lambda^{\textup{m}}_{(i_l,\eta_l)}$ of the transcendental equation in $\lambda$:
	\begin{equation}\label{Lexp}
	\frac{\sqrt{\lambda}}{\sqrt{\lambda+4\beta^{\textup{m}}}}\(1 - z_+^{2i_l}\) = -\frac{1}{\eta_l(\eta_l+2)}, \mbox{ with } z_+ = \frac{\lambda+2\beta^{\textup{m}}-\sqrt{\lambda(\lambda+4\beta^{\textup{m}})}}{2\beta^{\textup{m}}} \mbox{ and } |z_+|<1,
%	\frac{\sqrt{\lambda}}{\sqrt{4\beta^{\textup{m}}+\lambda}}\left[1 - \left(\frac{2\beta^{\textup{m}}+\lambda - \sqrt{\lambda(4\beta^{\textup{m}}+\lambda)}}{2\beta^{\textup{m}}}\right)^{2i_l}\right] = -\frac{1}{\eta_l(\eta_l+2)},
	\end{equation}
	as $N\to+\infty$. 
%	Then, only for $-1<\eta_l<0$, there exist  $M$ eigenvalues   above the spectrum bulk $\Xi = \left[0,-4\beta^{\textup{m}}\right]$, denoted by $\lambda^{\textup{m}}_{(i_l,\eta_l),N}$, $l=1,2,\ldots,M$.   These eigenvalues converge  to, up to leading order:
%%	\[
%%	{\frac{\sqrt{\lambda}}{\sqrt{4\beta+\lambda}}\left[1 - \left(\frac{2\beta+\lambda - \sqrt{\lambda(4\beta+\lambda)}}{2\beta}\right)^{2i_l}\right]} = -\frac{1}{\eta_l(\eta_l+2)}
%%	\]
%	\begin{equation}\label{Lexp}
%	\lambda^{\textup{m}}_{(i_l,\eta_l),\textup{exp}} %= \sqrt{\frac{-2\beta^{\textup{m}}(\eta+1)^2+\sqrt{(\eta+1)^2\(4\(\eta\beta^{\textup{m}}\)^2+8\eta\(\beta^{\textup{m}}\)^2+4\(\beta^{\textup{m}}\)^2\)}}{2\eta+1}} 
%	=\frac{-4\beta^{\textup{m}}}{1-\eta_l^2(2+\eta_l)^2}
%	\end{equation}
%	exponentially as $N\to+\infty$, and for $N$ large enough,
%	\begin{align}\label{CVG_EF}
%	| \lambda^{\textup{m}}_{\textup{exp}} - \lambda^{\textup{m}}_{\textup{exp},N}|
%	< A\e^{-BN},
%	\end{align}
%	for some $A,B>0$ independent of $N$.
	Denote $\xi = \xi(\lambda^{\textup{m}}):= \(2\beta^{\textup{m}} + \lambda^{\textup{m}}\)/(2\beta^{\textup{m}})$. Moreover, the following characterizations hold for the eigenvectors:
	\begin{itemize}
		\item[\textup{(i)}] if $\lambda^{\textup{m}}\notin \Xi$ (i.e., $\xi<-1$), then the set
		\[
		\{\lambda^{\textup{m}}:\lambda^{\textup{m}}\notin \Xi \mbox{ and  } \lambda^{\textup{m}}\in \sigma({T}^{-2}\bm{\mathcal{S}}^{\textup{m}})\} = \{  \lambda^{\textup{m}}_{(i_l,\eta_l),N}:l=1,2,\ldots,M \},
		\]  
		the eigenvector $\bm{v}^{\textup{m}}_{(i_l,\eta_l)}$ corresponding to $\lambda^{\textup{m}}_{(i_l,\eta_l),N}$  is exponentially localized near the defect site $i_l$:
		\[
		\bm{v}_{(i_l,\eta_l)}^{\textup{m},(i)}= A(-\e)^{-|i-i_l|\textup{arccosh}(|\xi|)},
		\]
		with $A$  bounded as $N\to\infty$;
		
		\item[\textup{(ii)}] if $\lambda^{\textup{m}}\in \Xi\setminus \partial\Xi$ (i.e., $|d|<1$), then
		\begin{align*}
		\bm{v}^{\textup{m},(i)} = A\cos(i\arccos(d))+B\sin(i\arccos(d)),
		\end{align*}
		with  $A,B$  bounded as $N\to\infty$;
		\item[\textup{(iii)}] if $\lambda^{\textup{m}}\in \partial\Xi$ (i.e., $d = \pm1$), then
		\begin{align*}
		\bm{v}^{\textup{m},(i)} = Ad^i+Bd^i\cdot i,
		\end{align*}
		with  $A,B,$  bounded as $N\to\infty$.
	\end{itemize}
\end{thm}

\begin{proof}[\bf Proof]
	
Let $N$-by-$M$ matrices $\bm{X}_{\Lambda}$ and $\bm{X}_{\Lambda_\eta}$  be difined as
\[
\bm{X}_{\Lambda} := \(\bm{\mathcal{S}}^{\textup{m}}\bm{e}_{i_l} \)_{1\leq l\leq M}\;\mbox{ and }\; \bm{X}_{\Lambda_{\eta}}: = \(\(\frac{1}{(1+\eta_l)^2}-1\)\bm{e}_{i_l} \)_{1\leq l\leq M}. 
\]
By using \eqref{T}, Lemma \ref{eigenvalue} and applying Cauchy’s formula for the determinant of a rank-one perturbation (cf. \cite[identity (0.8.5.11)]{HJMA2013}), the characteristic polynomial $\mathcal{P}_{{T}^{-2}\bm{\mathcal{S}}^{\textup{m}}}(\lambda^{\textup{m}})$ of ${T}^{-2}\bm{\mathcal{S}}^{\textup{m}}\in\RR^{N\times N}$ can be written as
\[
\begin{aligned}
\mathcal{P}_{{T}^{-2}\bm{\mathcal{S}}^{\textup{m}}}(\lambda^{\textup{m}}):&=\det\(\lambda^{\textup{m}} \bm{I}_N- {T}^{-2}\bm{\mathcal{S}}^{\textup{m}}\) = \det\(\lambda^{\textup{m}} \bm{I}_N- \bm{\mathcal{S}}^{\textup{m}} - \sum_{l=1}^M \(\frac{1}{(1+\eta_l)^2}-1\) \bm{e}_{i_l}  \(\bm{\mathcal{S}}^{\textup{m}}\bm{e}_{i_l}\)^t \)\\
&  = \det\(\lambda^{\textup{m}} \bm{I}_N- \bm{\mathcal{S}}^{\textup{m}} - \bm{X}_{\Lambda_{\eta}} \bm{X}_{\Lambda}^t \)\\
& = \det\(\lambda^{\textup{m}} \bm{I}_N - \bm{\mathcal{S}}^{\textup{m}}\) \det\(\bm{I}_M - \bm{X}_{\Lambda}^t\(\lambda^{\textup{m}} \bm{I}_N - \bm{\mathcal{S}}^{\textup{m}}\)^{-1} \bm{X}_{\Lambda_\eta}\),
\end{aligned}
\]
where 
\[
\begin{aligned}
\(\bm{X}_{\Lambda}^t\(\lambda^{\textup{m}} \bm{I}_N - \bm{\mathcal{S}}^{\textup{m}}\)^{-1} \bm{X}_{\Lambda_\eta}\)_{l,s} &=\(\frac{1}{(1+\eta_s)^2}-1\) \bm{e}_{i_l}^t \bm{\mathcal{S}}^{\textup{m}}  \(\lambda^{\textup{m}} \bm{I}_N - \bm{\mathcal{S}}^{\textup{m}}\)^{-1} \bm{e}_{i_s}\\
& = \(\frac{1}{(1+\eta_s)^2}-1\) \bm{e}_{i_l}^t
\(\lambda^{\textup{m}}\(\lambda^{\textup{m}} \bm{I}_N - \bm{\mathcal{S}}^{\textup{m}}\)^{-1} - \bm{I}_N\)
\bm{e}_{i_s}\\
& = \(\frac{1}{(1+\eta_s)^2}-1\) 
\(\lambda^{\textup{m}}\(\lambda^{\textup{m}} \bm{I}_N - \bm{\mathcal{S}}^{\textup{m}}\)_{i_l,i_s}^{-1} - \delta_{l,s}\).
\end{aligned}
\]
Consequently,
\[
\begin{aligned}
\quad \bm{I}_M - \bm{X}_{\Lambda}^t\(\lambda^{\textup{m}} \bm{I}_N - \bm{\mathcal{S}}^{\textup{m}}\)^{-1} \bm{X}_{\Lambda_\eta} = \bm{\mathcal{C}}_\eta -\( \bm{\mathcal{C}}_\eta - \bm{I}_M\) \bm{\mathcal{R}}^{\textup{m}},
\end{aligned}
\]
with $\bm{\mathcal{C}}_\eta := \diag\(\frac{1}{(1+\eta_1)^2},\frac{1}{(1+\eta_2)^2},\ldots,\frac{1}{(1+\eta_M)^2}\)$ and  $\bm{\mathcal{R}}^{\textup{m}} = \bm{\mathcal{R}}^{\textup{m}}(\lambda^{\textup{m}})
: = \(\lambda^{\textup{m}}\(\lambda^{\textup{m}} \bm{I}_N - \bm{\mathcal{S}}^{\textup{m}}\)_{i_l,i_s}^{-1}\)_{1\leq l,s\leq M}$.
Hence, a value $\lambda^{\textup{m}}\notin \Xi$ is an eigenvalue of ${T}^{-2}\bm{\mathcal{S}}^{\textup{m}}$  if and only if
\[
f_N(\lambda^{\textup{m}}):=\det(\bm{\mathcal{C}}_{\eta} -\( \bm{\mathcal{C}}_{\eta} - \bm{I}\) \bm{\mathcal{R}}^{\textup{m}}) = 0.
\]
After straightforward algebraic manipulations using \eqref{Pn_chebpoly} and Lemma \ref{eigenvalue}, we obtain 
\begin{equation}\label{Rls}
\begin{aligned}
&\quad   \(\bm{\mathcal{R}}^{\textup{m}}\)_{l,s} = \frac{ (-1)^{i_l+i_s}\lambda^{\textup{m}}(-\beta^\textup{m})^{|i_l-i_s|} \mathcal{P}_{A_{\min\{i_l,i_s\}-1}^{(\beta^{\textup{m}}-\beta_0^{\textup{m}}, 0)}(-2\beta^{\textup{m}}, \beta^{\textup{m}})} (\lambda^{\textup{m}}) \mathcal{P}_{A_{N-\max\{i_l,i_s\}}^{(0, \beta^{\textup{m}})}(-2\beta^{\textup{m}}, \beta^{\textup{m}})}(\lambda^{\textup{m}}) }{\mathcal{P}_{A_{N}^{(\beta^{\textup{m}}-\beta_0^{\textup{m}}, \beta^{\textup{m}})}(-2\beta^{\textup{m}}, \beta^{\textup{m}})}(\lambda^{\textup{m}})}\\
& = \frac{ \lambda^{\textup{m}}
	\(U_{\min\{i_l,i_s\}-1}(d) - (\beta^{\textup{m}}-\beta_0^{\textup{m}} )U_{\min\{i_l,i_s\}-2}(d)/\beta^{\textup{m}}\) \(U_{N-\max\{i_l,i_s\}}(d) -  U_{N-\max\{i_l,i_s\}-1}(d)\) }{\beta^{\textup{m}}U_N(d) - (2\beta^{\textup{m}}-\beta_0^{\textup{m}} )U_{N-1}(d)+\(\beta^{\textup{m}}-\beta_0^{\textup{m}}\) U_{N-2}(d)},
\end{aligned}
\end{equation}
where $\xi:= \(2\beta^{\textup{m}} + \lambda^{\textup{m}}\)/(2\beta^{\textup{m}}).$ Notice that $\lambda^{\textup{m}}\notin \Xi$  is equivalent to $\xi<-1$. 
Define 
\begin{equation}\label{zpm}
z_\pm := \xi\pm\sqrt{\xi^2-1} = -\e ^{\mp \textup{arccosh}(|\xi|)}.
\end{equation} For $\xi<-1$, we have  $|z_+|<1$,  $|z_-|>1$ and $z_+ z_- = 1$.
From \eqref{Chebypoly}, we have that for $\xi<-1$ and $N$ large enough,
\begin{equation}\label{Uninf}
U_N(\xi) = \frac{z_-^{N+1} - z_+^{N+1}}{z_- - z_+} =\frac{z_-^{N+1}}{z_- - z_+} \(
1+O\(\(|{z_+}/{z_-}|\)^{N+1}\)
\).
\end{equation}
Therefore,  by using \eqref{Rls} and \eqref{Uninf}, we have that for $N$ and $N-i_M$ large enough,
\begin{equation}\label{Rlsinf}
\begin{aligned}
\lim_{N\to\infty } \(\bm{\mathcal{R}}^{\textup{m}}\)_{l,s}& = \lambda^{\textup{m}}  \frac{z_+^{\max\{i_l,i_s\}}\(z_+^{-\min\{i_l,i_s\}}- z_+^{\min\{i_l,i_s\}}\)}{\beta^{\textup{m}}(z_- - z_+)} \frac{}{}  \\
&	= \lambda^{\textup{m}} \frac{z_+^{|i_l-i_s|}- z_+^{(i_l+i_s)}}{\beta^{\textup{m}}(z_- - z_+)}  =\begin{cases} 
\ds %\frac{\lambda^{\textup{m}}\(1 - z_+^{2i_l} \)}{\beta^{\textup{m}}(z_- - z_+)} = 
\frac{\sqrt{\lambda^{\textup{m}}}}{\sqrt{\lambda^{\textup{m}}+4\beta^{\textup{m}}}}\(1 - z_+^{2i_l}  \), & \text{if } l = s, \\ 
\nm 
\ds\mathcal{O}\(|z_+|^{|i_l-i_s|}\) = \mathcal{O} (\e^{-|i_s-i_l|\textup{arccosh}(|\xi|)}), & \text{if } l \neq s.
\end{cases}
\end{aligned}
\end{equation}
Hence, from $\min_{2\leq l\leq M}|i_l-i_{l-1}|\gg 1$, we can use the diagonal approximation to obtain
\begin{equation}\label{eigenabove2}
f_{\infty}(\lambda^{\textup{m}}): =\lim_{N\to\infty}f_N(\lambda^{\textup{m}})  = \prod_{l=1}^{M}\frac{1}{(1+\eta_l)^{2}}\(1+\eta_l (2+\eta_l )\frac{\sqrt{\lambda^{\textup{m}}}}{\sqrt{\lambda^{\textup{m}}+4\beta^{\textup{m}}}}\(1 - z_+^{2i_l}  \)\).
\end{equation}
Combining \eqref{eigenpositive}  with   \eqref{eigenabove2}, we conclude that $f_{\infty}(\lambda^{\textup{m}}) = 0$ is equivalent to
\begin{equation}\label{Lexpequation}
g(\lambda^{\textup{m}}) := \frac{\sqrt{\lambda^{\textup{m}}}}{\sqrt{4\beta^{\textup{m}}+\lambda^{\textup{m}}}}\(1 - z_+^{2i_l}\) = -\frac{1}{\eta_l(\eta_l+2)}, \mbox{ with } z_+ = \frac{\lambda^{\textup{m}}+2\beta^{\textup{m}}-\sqrt{\lambda^{\textup{m}}(\lambda^{\textup{m}}+4\beta^{\textup{m}})}}{2\beta^{\textup{m}}} \mbox{ and } |z_+|<1.
\end{equation}
for $l=1,2,\ldots,M$. Since $g(\lambda^{\textup{m}})$ is monotonically decreasing for $\lambda^{\textup{m}}>-4\beta^{\textup{m}}$ with range $(1,4i_l)$, the equation $f_\infty(\lambda^{\textup{m}})=0$ admits $M$ distinct roots $\lambda^{\textup{m}}_{(i_l,\eta_l)}$ lying  above the spectrum bulk $\Xi = \left[0,-4\beta^{\textup{m}}\right]$  if $-1<\eta_l<-1+\frac{\sqrt{4i_l-1}}{2\sqrt{i_l}}$, with each root satisfying \eqref{Lexpequation}.
Owing to the sign-preserving nature of the limiting process, it follows that for sufficiently large $N$, there exist corresponding roots $\lambda^{\textup{m}}_{(i_l,\eta_l),N}$ above $\Xi$ such that $f_N(\lambda^{\textup{m}}_{(i_l,\eta_l),N})=0$.

We now  show that for $N$ large enough, the eigenvector $\bm{v}^{\textup{m}}_{(i_l,\eta_l)}$ corresponding to $\lambda^{\textup{m}}_{(i_l,\eta_l),N}$  is exponentially localized near the defect site $i_l$. Set $a_l:= \((1+\eta_s)^{-2}-1\)  \(\bm{\mathcal{S}}^{\textup{m}}  \bm{v}^{\textup{m}}\)^{(i_s)}$ for $s=1,2,\ldots,M$. It follows from Lemma \ref{Lemm52} that 
\begin{equation}\label{vmieta}
\bm{v}^{\textup{m}} 
= \sum_{s=1}^{M}a_s \(\lambda^{\textup{m}}\bm{I}-\bm{\mathcal{S}}^{\textup{m}}  \)^{-1}\bm{e}_{i_s},
\end{equation}
this, together with  \eqref{CcapWOD} and the identity  $\bm{\mathcal{S}}^{\textup{m}}\(\lambda^{\textup{m}}\bm{I}-\bm{\mathcal{S}}^{\textup{m}}  \)^{-1} = \lambda^{\textup{m}}\(\lambda^{\textup{m}}\bm{I}-\bm{\mathcal{S}}^{\textup{m}}  \)^{-1}- \bm{I}$, implies that
\begin{equation}\label{svietal}
\begin{aligned}
\(\bm{\mathcal{S}}^{\textup{m}}  \bm{v}^{\textup{m}}\)^{(i_l)}&  = \sum_{j=i_l-1}^{i_l+1} \bm{\mathcal{S}}^{\textup{m}}_{i_l,j} \bm{v}^{\textup{m},(j)} =  \sum_{j=i_l-1}^{i_l+1} \bm{\mathcal{S}}^{\textup{m}}_{i_l,j} \sum_{s=1}^{M}a_s \(\lambda^{\textup{m}}\bm{I}-\bm{\mathcal{S}}^{\textup{m}}  \)^{-1}_{j,i_s}\\
& = \sum_{j=i_l-1}^{i_l+1} \bm{\mathcal{S}}^{\textup{m}}_{i_l,j} \sum_{s\neq l}^{M}a_s \(\lambda^{\textup{m}}\bm{I}-\bm{\mathcal{S}}^{\textup{m}}  \)^{-1}_{j,i_s} + \sum_{j=i_l-1}^{i_l+1} a_l\bm{\mathcal{S}}^{\textup{m}}_{i_l,j}  \(\lambda^{\textup{m}}\bm{I}-\bm{\mathcal{S}}^{\textup{m}}  \)^{-1}_{j,i_l}\\
& = \sum_{s=1}^{M}a_s \(\bm{\mathcal{R}}_{i_l,i_s}^{\textup{m}}(\lambda^{\textup{m}})-\delta_{l,s}\).
\end{aligned}
\end{equation}
From \eqref{Rlsinf} and \eqref{svietal}, the vector $\bm{a}:=\(a_l\)_{1 \leq l\leq M}$ satisfies, for $N$ large enough,  the diagonal linear system
\begin{equation}\label{Dddy0}
\bm{D}\bm{a} = 0, \mbox{ with } \bm{D} = \diag\(\( (1+\eta_l)^{-2}\(1+\eta_l (2+\eta_l )\frac{\sqrt{\lambda^{\textup{m}}}}{\sqrt{\lambda^{\textup{m}}+4\beta^{\textup{m}}}}\(1 - z_+^{2i_l}  \)\)\)_{l=1}^M\).
\end{equation}
Finally, using  \eqref{Rlsinf}, \eqref{vmieta} and \eqref{Dddy0}, we have that for $N$ large enough
\begin{equation}
\bm{v}_{(i_l,\eta_l)}^{\textup{m},(i)} 
= A \(\lambda^{\textup{m}}\bm{I}-\bm{\mathcal{S}}^{\textup{m}}  \)^{-1}_{i,i_l} = Az_+^{|i-i_l|} = A(-\e)^{-|i-i_l|\textup{arccosh}(|\xi|)},
\end{equation}
where $A$ remains bounded as $N\to\infty$. 
In a similar manner, employing \eqref{Chebypoly}, one readily establishes the corresponding conclusions (ii)–(iii) for eigenvectors inside or on the boundary of the spectral bulk $\Xi$.
\end{proof}

\begin{rem}\label{rem51}
	Theorem \ref{Propiff2} shows that a defect at $(i_l,\eta_l)$ generates an eigenvector localized around the perturbation position $i_l$, with the corresponding eigenvalue increasing as the perturbation magnitude $|\eta_l|$ grows. It is noteworthy that, in fact, $\lambda^{\textup{m}}_{(i_l,\eta_l),N}$ converges  to the root of the equation \eqref{Lexp}.
 Consequently, by selecting multiple distinct yet comparable perturbation parameters, one may engineer a broadband of defect states above the spectrum bulk.
	Moreover, 
	even when the perturbation parameters $\eta_l$ are equal, the resulting eigenvalues differ due to the distinct defect positions $i_l$.  
	Through appropriate choice of $(i_l,\eta_l)$ for $l=1,2,\ldots,M$, one may engineer quasi-degenerate eigenvalues, yielding eigenvectors with simultaneous localization at multiple defect sites. %This property proves advantageous for constructing multimode localized states in metamaterial architectures.
	This phenomenon facilitates the design of multimode localized states in metamaterial architectures, opening avenues for multi-channel filtering and related wave-manipulation functionalities.
\end{rem}

The above observations suggest that, through appropriate choices of $(i_l,\eta_l)$ for $l=1,2,\ldots,M$, one may engineer quasi-degenerate eigenvalues, yielding eigenvectors with simultaneous localization at multiple defect sites. 
The following theorem, based on a more refined analysis, establishes precise conditions for this simultaneous localization phenomenon and derives sharper asymptotics for the quasi-degenerate eigenvalues together with the amplitude ratios of the eigenvector at each defect site.

\begin{thm}\label{thm_multimode}
	Consider an $N$-layered monomer-type metamaterial containing  $M$ material-parameter defects in $D_{\Lambda} = \cup_{j=i_1}^{i_M}D_j$,  with equidistant defect sites and identical perturbation parameter, i.e., 
	\[
	i_l = i_1+(l-1)d, \;\mbox{ and }\; \eta_l = \eta \;\mbox{ for }\; l=1,2,\ldots,M.
	\]
	Denote $\xi:= \(2\beta^{\textup{m}} + \lambda^{\textup{m}}\)/(2\beta^{\textup{m}})$.	 
	If the following conditions holds: 
	\begin{equation}
	\min\{2d,2i_1, N-i_M\}\gg (\textup{arccosh}(|\xi|))^{-1},\;\mbox{ and}\; -1<\eta<0,
	\end{equation}
	then for $N$ large  enough, there exist $M$ eigenvalues above the spectrum bulk $\Xi = \left[0,-4\beta^{\textup{m}}\right]$, denoted by $\lambda^{\textup{m}}_{(i_l,\eta),N}$, for $l = 1,\dots,M$. These eigenvalues are quasi-degenerate in the sense that
	\begin{equation}\label{quasi-degenerate}
	\max_{l\neq s}|\lambda^{\textup{m}}_{(i_l,\eta),N} - \lambda^{\textup{m}}_{(i_s,\eta),N}| = \mathcal{O}\(\beta^{\textup{m}} \(\frac{1-\gamma}{1+\gamma}\)^{d}\),
	\end{equation}
	with $\gamma = -\eta(\eta+2)\in (0,1)$, and as $N \to +\infty$, they converge at leading order to
	\begin{equation}\label{MLexp}
	\lambda^{\textup{m}}_{(i_l,\eta),d} = \frac{-4\beta^{\textup{m}}}{1-\gamma^2}\(1-\frac{4\gamma^2}{1-\gamma^2}\(\frac{1-\gamma}{1+\gamma}\)^{d}\cos\frac{l\pi}{M+1}   \) +o(1), \mbox{ for } l=1,2,\ldots,M. %\mathcal{O}\( \beta^{\textup{m}} \e ^{- 2\min\{d,i_1\} \textup{arccosh}(|\xi|)}\),
	\end{equation}
	Moreover, 	the eigenvector $\bm{v}^{\textup{m}}_{(i_l,\eta)}$ corresponding to $\lambda^{\textup{m}}_{(i_l,\eta),N}$  becomes, in the limit $N \to \infty$, exponentially localized at multiple defect sites:
	\begin{equation}\label{vmi}
	\bm{v}_{(i_l,\eta)}^{\textup{m},(i)}= \sum_{s=1}^M A_s(-\e)^{-|i-i_s|\textup{arccosh}(|\xi|)},
	\end{equation}
	where the amplitude $A_s$ of the corresponding eigenvector $\bm{v}^{\textup{m}}_{(i_l,\eta)}$ at the defect sites can be given by
	\begin{equation}\label{amplitude_ratio}
	(A_s)_{1\leq s\leq M}  \propto 
	\begin{cases}
	\(\sin\frac{(M+1-l)s\pi}{M+1} \)_{1\leq s\leq M},&\mbox{if }\  d \ \mbox{is even},\\
	\(\sin\frac{ls\pi}{M+1} \)_{1\leq s\leq M},&\mbox{if }\ d \ \mbox{is odd}.\\
	\end{cases}
	\end{equation}
\end{thm}
\begin{proof}[\bf Proof] 
	The existence of $M$ eigenvalues, $\lambda^{\textup{m}}_{(i_l,\eta),N}$, for $l = 1,\dots,M$,   above the spectrum bulk $\Xi = \left[0,-4\beta^{\textup{m}}\right]$, for $N$ large  enough, can be proved by Theorem \ref{Propiff2}. We next to show \eqref{quasi-degenerate}--\eqref{amplitude_ratio}. By using \eqref{svietal}, we have that the vector $\bm{a}:=\(a_l\)_{1\leq l\leq M}$, with $a_l:= \((1+\eta)^{-2}-1\)  \(\bm{\mathcal{S}}^{\textup{m}}  \bm{v}^{\textup{m}}\)^{(i_l)}$ satisfies the linear system
	\begin{equation}\label{linearA}
	\(\bm{I} -\gamma \bm{\mathcal{R}}^{\textup{m}}\)\bm{a} = 0,\;\mbox{ with }\;\ \gamma =  -\eta(\eta+2)\in (0,1).
	\end{equation}
	From $i_l = i_1+(l-1)d$, and by using \eqref{Rlsinf}, we have that for $N$ large enough,
	\[
	\(\bm{I} -\gamma \bm{\mathcal{R}}^{\textup{m}}\)_{l,s} = \begin{cases} 
	\ds %\frac{\lambda^{\textup{m}}\(1 - z_+^{2i_l} \)}{\beta^{\textup{m}}(z_- - z_+)} = 
	1-\gamma t\(1 - z_+^{2i_l}  \), & \text{if } l = s, \\ 
	\nm 
	\ds -\gamma t \(z_+^{jd }- z_+^{(i_l+i_s)}\), & \text{if } |l-s| = j.
	\end{cases}
	\]
	where $ t:=\frac{\sqrt{\lambda^{\textup{m}}}}{\sqrt{\lambda^{\textup{m}}+4\beta^{\textup{m}}}}>1$. 
	From $\min\{2d,2i_1\}\gg (\textup{arccosh}(|\xi|))^{-1}$ and $z_+<0$, we can use the nearest-neighbor approximation to obtain 
	\[
	\(\bm{I}  -\gamma  \bm{\mathcal{R}}^{\textup{m}}\) = A_{M}^{(0, 0)}\(1-\gamma t,  -\gamma t(-|z_+|)^{d } \) +\mathcal{O}(|z_+|^{2\min\{d,i_1\}}). 
	\]
	By using the spectral theory of $1$-Toeplitz matrices \eqref{1toeplitzCP} (cf. \cite{ANZIAM, CDF_AMS}), we have that for the case that $d$ is odd,
	\[
	 A_{M}^{(0, 0)}\(1-\gamma t,  \gamma t|z_+|^{d } \) \bm{a}_l =  \(1-\gamma t   +2\gamma t|z_+|^{d } \cos\frac{l\pi}{M+1} \)  \bm{a}_l, 
	\]
	with $\bm{a}_l = \(\sin\frac{ls\pi}{M+1}\)_{1\leq s\leq M}$.
	Hence, the linear system \eqref{linearA} admits nontrival solutions if and only if 
	\begin{equation}\label{matrixvanish}
	\gamma t\(1-2|z_+|^{d } \cos\frac{l\pi}{M+1}\)  = 1.
	\end{equation}
	From \eqref{zpm} and $  t=\frac{\sqrt{\lambda^{\textup{m}}}}{\sqrt{\lambda^{\textup{m}}+4\beta^{\textup{m}}}}>1$, we have that $|z_+| = \frac{t-1}{t+1}<1$ and $\lambda^{\textup{m}}(t) = \frac{-4\beta^{\mathrm{m}}t^2}{t^2-1}$. Suppose that $t = t_0+\varepsilon t$ with $\varepsilon t = \mathcal{O}(|z_+|^{d })$ and $\gamma t_0 = 1$. Substituting  these into \eqref{matrixvanish}, we have\[
	\varepsilon t = |z_+|^{d } \frac{2}{\gamma}\cos\frac{l\pi}{M+1} = \(\frac{1-\gamma}{1+\gamma}\)^{d } \frac{2}{\gamma}\cos\frac{l\pi}{M+1}+o(1).
	\]
	Therefore, we have that for odd $d$, the root $\lambda^{\textup{m}}_{(i_l,\eta),d}$, for $l=1,2,\ldots,M$, to \eqref{matrixvanish} has the asymptotic expansion
	\begin{equation}\label{dodd}
	\begin{aligned}
	\lambda^{\textup{m}}_{(i_l,\eta),d} &= \lambda^{\textup{m}}(t_0)+\(\lambda^{\textup{m}}\)'(t_0)(t-t_0) +O(|t-t_0|^2)\\
	%	& = \frac{-4\beta^{\textup{m}}}{1-\gamma^2}+\(\frac{16\beta^{\textup{m}}\gamma^2}{(1-\gamma^2)^2}\)\(\frac{1-\gamma}{1+\gamma}\)^{d } \cos\(\frac{l\pi}{M+1}\)+o(1)\\
	& = \frac{-4\beta^{\textup{m}}}{1-\gamma^2}\(1-\frac{4\gamma^2}{1-\gamma^2}\(\frac{1-\gamma}{1+\gamma}\)^{d } \cos\frac{l\pi}{M+1}\)+o(1).
	\end{aligned}
	\end{equation}
In a similar manner, for the case that $d$ is even,  the linear system \eqref{linearA} admits nontrival solutions if and only if 
\[
\gamma t\(1+2|z_+|^{d } \cos\frac{l\pi}{M+1}\)  = 1.
\]
Therefore, we have that for even $d$, the root $\lambda^{\textup{m}}_{(i_l,\eta),d}$, for $l=1,2,\ldots,M$, to \eqref{matrixvanish} has the asymptotic expansion
\begin{equation}\label{deven}
\begin{aligned}
\lambda^{\textup{m}}_{(i_l,\eta),d} 
& = \frac{-4\beta^{\textup{m}}}{1-\gamma^2}\(1+\frac{4\gamma^2}{1-\gamma^2}\(\frac{1-\gamma}{1+\gamma}\)^{d } \cos\frac{l\pi}{M+1}\)+o(1).
\end{aligned}
\end{equation}
Consequently, \eqref{dodd} and \eqref{deven} can be simplified and written compactly as \eqref{MLexp} with 
\[
\bm{a}_l^{(s)}=\begin{cases}
\sin\frac{(M+1-l)s\pi}{M+1}, & \mbox{ if } d \ \mbox{ is even},\\
\sin\frac{ls\pi}{M+1}, & \mbox{ if } d \ \mbox{ is odd},\\
\end{cases}
\]
Finally, it follows from   \eqref{Rlsinf} and \eqref{vmieta} that \eqref{vmi}--\eqref{amplitude_ratio} hold. The proof is complete.
\end{proof}

\begin{rem}\label{rem53}
Theorem \ref{thm_multimode} developed for $M$  equidistant defects with identical perturbation parameter admits a natural generalization to hierarchical configurations.
 Specifically, the defect set $\{(i_l,\eta_l)\}_{l=1}^M$ may be partitioned into $b$  blocks, 
 \[
 \{(i_l,\eta_l)\}_{l=1}^M = \bigcup_{t=1}^{b}  \left\{(i_l,\eta^{(t)}):i_{l+1}-i_{l} = d_t,\;1+\sum_{s=1}^{t-1}M_{s}\leq l \leq  \sum_{s=1}^{t}M_{s}-1  \right\}^{\sum_{s=1}^{t}M_{s}}_{l=1+\sum_{s=1}^{t-1}M_{s}},
 \]
 where the positive integers $ M_t$ satisfy $\sum_{t=1}^b M_t = M$.
 This means that in the intra-block structure, $M_t$ defects within block $t$ are equidistant with spacing $d_t$ and share identical perturbation parameter $\eta^{(t)}\in(-1,0)$. Moreover, if the distance of inter-block separation is much greater than the spacing of intra-block structure, i.e., 
 \begin{equation}\label{inter-block-weak}
  \min_{2\leq t\leq b-1}\(i_{1+\sum_{s=1}^{t-1}M_s} - i_{\sum_{s=1}^{t-1}M_s} \)\gg \max_{1\leq t\leq b}d_t,
 \end{equation}
 then at the intra-block level, each block $t$  
 generates a miniband of $M_t$ 
 quasi-degenerate   eigenvalues with width $\mathcal{O}\(\beta^{\textup{m}} \(\frac{1- \gamma^{(t)}}{1+\gamma^{(t)}}\)^{d_t}\)$ 
 centered at $\frac{-4\beta^{\textup{m}}}{1-\(\gamma^{(t)}\)^2}$ with $\gamma^{(t)} = -\eta^{(t)}(\eta^{(t)}+2)\in (0,1)$, precisely as characterized by
 \begin{equation}\label{MLexpblock}
 \lambda^{\textup{m}}_{(i_l,\eta^{(t)}),d_t} = \frac{-4\beta^{\textup{m}}}{1-\(\gamma^{(t)}\)^2}\(1-\frac{4\(\gamma^{(t)}\)^2}{1-\(\gamma^{(t)}\)^2}\(\frac{1-\(\gamma^{(t)}\)}{1+\(\gamma^{(t)}\)}\)^{d_t}\cos\frac{l\pi}{M_t+1}   \) +o(1), \mbox{ for } 1\leq t\leq b, 1\leq l\leq M_t. %\mathcal{O}\( \beta^{\textup{m}} \e ^{- 2\min\{d,i_1\} \textup{arccosh}(|\xi|)}\),
 \end{equation}
   At the inter-block level, weak tunneling couplings \eqref{inter-block-weak}
 between blocks induce further splittings, leading to $b$ super-minibands.
\end{rem}

\section{Numerical computations}

In this section, we conduct numerical computations to corroborate our theoretical findings in the previous sections. To enhance computational efficiency,  we focus on the multi-layer concentric sphere structure.  First, we  analyse the resonant mode splitting in multilayered concentric spheres.
Moreover, we verify the existence and convergence of a number of eigenvalues above the bulk spectrum that equals the number of defects, together with the localization criterion based on eigenvalue position, thereby supporting the conclusions from Subsection \ref{LM_mpd}. Finally, we perform numerical simulations of the localized defect modes.

\subsection{Resonant mode splitting}

In this subsection, we shall compute  the eigenfrequencies based on the elastic multipole   expansion method.  
Let $j_n(t)$ and $h_n(t)$ be the spherical Bessel and Hankel functions of the
first kind of order $n$, respectively, and let $Y_{n}^{m}$ denote spherical harmonic functions of the order $n$ with the degree $m$ defined on the unit sphere $\mathbb{S}$. 
Recall that \cite{LZArxiv,DLbook2024}, the family $(\mathcal{I}_{n}^{m},\mathcal{T}_{n}^{m},\mathcal{N}_{n}^{m})$ are the vectorial spherical harmonics of order $n$,
\begin{equation}\label{vectorial_spherical_harmonics}
\begin{aligned}
\mathcal{T}_{n}^{m}&=\nabla_{\mathbb{S}}Y_{n}^{m}\times \bm{\nu},\ n\geq1,\ n\geq m\geq-n,\\
\mathcal{I}_{n}^{m}&=\nabla_{\mathbb{S}}Y_{n+1}^{m}+(n+1)Y_{n+1}^{m}\bm{\nu},\ n\geq0,n+1\geq m\geq-(n+1),\\
\mathcal{N}_{n}^{m}&=-\nabla_{\mathbb{S}}Y_{n-1}^{m}+nY_{n-1}^{m}\bm{\nu},\ n\geq1,n-1\geq m\geq-(n-1),
\end{aligned}
\end{equation}
forms an orthogonal basis of $(L^2(\mathbb{S}))^3$. 
By applying certain spectral results of layer potentials (cf. \cite{LZArxiv,DLbook2024}) along with some recursion formulas for $j_n(t)$ and $h_n(t)$ (cf. \cite{hongyuliu_JE2021,Morsebook1953}), 
we define, in a novel and concise manner, the exterior and interior solutions to the equation $\left( \mathcal{L}_{\lambda,\mu}+\omega^2\rho\right)\mathbf{u}=0$ in $\mathbb{R}^3\backslash\{\bm{0}\}$ and $\mathbb{R}^3$, respectively, as
\begin{equation}\label{ext_solution}
\begin{cases}
\mathbf{L}_{nm}^{e} \left(k_{s}|\mathbf{x}|\right) = h_{n}\left(k_{s}|\mathbf{x}|\right)\mathcal{T}_{n}^{m},&n\geq1, n\geq m\geq-n,\\
\mathbf{P}_{nm}^e \left(k_{p}|\mathbf{x}|\right) = h_{n-1}(k_{p}|\mathbf{x}|)\mathcal{I}_{n-1}^{m} -  h_{n+1}(k_{p}|\mathbf{x}|)\mathcal{N}_{n+1}^{m}, &n\geq1,n\geq m\geq-n,\\
{\mathbf{Q}}_{nm}^e \left(k_{s}|\mathbf{x}|\right) = (n+1)h_{n-1}\left(k_{s}|\mathbf{x}|\right)\mathcal{I}_{n-1}^{m} + n  h_{n+1}\left(k_{s}|\mathbf{x}|\right)\mathcal{N}_{n+1}^{m}, &n\geq1,n\geq m\geq-n,
\end{cases}
\end{equation} 
%are radiating solutions to the equation $\left( \mathcal{L}_{\lambda,\mu}+\omega^2\rho\right)\mathbf{u}=0$ in $\mathbb{R}^3\backslash\{\bm{0}\}$
and
\begin{equation}\label{int_solution}
\begin{cases}
\mathbf{L}_{nm}^{i} \left(k_{s}|\mathbf{x}|\right) = j_{n}\left(k_{s}|\mathbf{x}|\right)\mathcal{T}_{n}^{m},&n\geq1,n\geq m\geq-n,\\
\mathbf{P}_{nm}^{i }\left(k_{p}|\mathbf{x}|\right) = j_{n-1}(k_{p}|\mathbf{x}|)\mathcal{I}_{n-1}^{m} -  j_{n+1}(k_{p}|\mathbf{x}|)\mathcal{N}_{n+1}^{m}, &n\geq1,n\geq m\geq-n,\\
{\mathbf{Q}}_{nm}^{i} \left(k_{s}|\mathbf{x}|\right) = (n+1)j_{n-1}(k_{s}|\mathbf{x}|)\mathcal{I}_{n-1}^{m} + n  j_{n+1}(k_{s}|\mathbf{x}|)\mathcal{N}_{n+1}^{m}, &n\geq1,n\geq m\geq-n.\\
\end{cases}
\end{equation} 
%are entire solutions to the equation $\left( \mathcal{L}_{\lambda,\mu}+\omega^2\rho\right)\mathbf{u}=0$ in $\mathbb{R}^3$. 
In what follows, we denote by 
\[
\begin{aligned}
D^i_{L,n}(k_{s}|\Bx|) &= j_n\left(k_s|\mathbf{x}|\right),\\
D^{i,1}_{P,n}(k_{p}|\Bx|) & = j_{n-1}(k_{p}|\mathbf{x}|),\;
D^{i,2}_{P,n}(k_{p}|\Bx|)  =-j_{n+1}(k_{p}|\mathbf{x}|), \\
D^{i,1}_{Q,n}(k_{s}|\Bx|) & = (n+1)j_{n-1}(k_{s}|\mathbf{x}|),\;
D^{i,2}_{Q,n}(k_{s}|\Bx|)  = n j_{n+1}(k_{s}|\mathbf{x}|),\\
\end{aligned}
\]
and $D^e_{L,n}(k_{s}|\Bx|), D^{e,1}_{P,n}(k_{p}|\Bx|), D^{e,2}_{P,n}(k_{p}|\Bx|), D^{e,1}_{Q,n}(k_{s}|\Bx|), D^{e,2}_{Q,n}(k_{s}|\Bx|)$   can be obtained by replacing $j_n$ with $h_n$.

By using spherical coordinates,  the elastic displacement field $\mathbf{u}(\Bx)$ to \eqref{main_equation}, with $\mathbf{u}^{\textup{in}}(\Bx) = 0$ and the material parameters given in \eqref{nestedcomplement}--\eqref{contrastparameter}, can be written as
\begin{equation}\label{lame_solution2}
\mathbf{u}(\Bx) = \begin{cases}
\ds \sum_{n=1}^{+\infty} \sum_{m=-n}^{n} \(a^+_{1,n}{\mathbf{L}}_{nm}^{e} \left(k_{s}|\mathbf{x}|\right) + c^+_{1,n}\mathbf{P}_{nm}^e (k_{p}|\mathbf{x}|) + e^+_{1,n}\mathbf{Q}_{nm}^e (k_{s}|\mathbf{x}|)\), & \quad x \in D_0',\\
\nm
\ds \sum_{n=1}^{+\infty} \sum_{m=-n}^{n} \(b^+_{j,n}\mathbf{L}_{nm}^i (k_{s,\rmr,j}|\mathbf{x}|) + d^+_{j,n}\mathbf{P}_{nm}^i (k_{p,\rmr,j}|\mathbf{x}|) + f^+_{j,n}\mathbf{Q}_{nm}^i (k_{s,\rmr,j}|\mathbf{x}|)\) \\
\ds+  \sum_{n=1}^{+\infty}\sum_{m=-n}^{n} \(a^-_{j,n}\mathbf{L}_{nm}^e (k_{s,\rmr,j}|\mathbf{x}|) + c^-_{j,n}\mathbf{P}_{nm}^e (k_{p,\rmr,j}|\mathbf{x}|) + e^-_{j,n}\mathbf{Q}_{nm}^e (k_{s,\rmr,j}|\mathbf{x}|)\),  & \quad x \in {D_j},\; j=1,2,\ldots,N,\\
\nm
\ds \sum_{n=1}^{+\infty}\sum_{m=-n}^{n} \(b^-_{j,n}\mathbf{L}_{nm}^i (k_{s}|\mathbf{x}|) +  d^-_{j,n}\mathbf{P}_{nm}^i (k_{p}|\mathbf{x}|) + f^-_{j,n}\mathbf{Q}_{nm}^i (k_{s}|\mathbf{x}|)\)\\
\ds +\sum_{n=1}^{+\infty}\sum_{m=-n}^{n} \( a^+_{j+1,n}\mathbf{L}_{nm}^e (k_{s}|\mathbf{x}|) + c^+_{j+1,n}\mathbf{P}_{nm}^e (k_{p}|\mathbf{x}|) + e^+_{j+1,n}\mathbf{Q}_{nm}^e (k_{s}|\mathbf{x}|)\),  & \quad x \in {D_j'},\; j=1,2,\ldots,N,
\end{cases}
\end{equation}
where $a^+_{N+1,n}=c^+_{N+1,n}=e^+_{N+1,n}=0$.  
Moreover,  from \eqref{traction1}, we have  that (cf. \cite{LZArxiv,DLbook2024})  the surface traction of the interior and exterior multipole elastic fields can be written by
\begin{align*}
\partial_{\bm{\nu}}\({\mathbf{L}}_{nm}^{f} (k_{s}|\mathbf{x}|)\)&=T^f_{L,n}(k_{s}|\mathbf{x}|)\mathcal{T}_{n}^{m},\\
\partial_{\bm{v}}\(\mathbf{P}_{nm}^f (k_{p}|\mathbf{x}|)\) &= T^{f,1}_{P,n}(k_{p}|\Bx|)\mathcal{I}_{n-1}^{m} + T^{f,2}_{P,n}(k_{p}|\Bx|)\mathcal{N}_{n+1}^{m},\\
\partial_{\bm{v}}\(\mathbf{Q}_{nm}^f (k_{s}|\mathbf{x}|)\) &= T^{f,1}_{Q,n}(k_{s}|\Bx|)\mathcal{I}_{n-1}^{m} + T^{f,2}_{Q,n}(k_{s}|\Bx|)\mathcal{N}_{n+1}^{m}, \;\mbox{ for }\; f\in\{i,e\},
\end{align*}
where 
\[
\begin{aligned}
T^i_{L,n}(k_{s}|\Bx|) &= \mu \(k_s j'_n \left(k_s|\mathbf{x}|\right) - |\Bx|^{-1}j_n\left(k_s|\mathbf{x}|\right) \),\\
T^{i,1}_{P,n}(k_{p}|\Bx|) & = {2(n-1) \mu|\mathbf{x}|^{-1} } {j_{n-1}(k_{p} |\mathbf{x}|)  } -(\lambda+2\mu )k_{p}j_{n}(k_{p} |\mathbf{x}|),\\
T^{i,2}_{P,n}(k_{p}|\Bx|) & ={2 (n+2)\mu|\mathbf{x}|^{-1} } {j_{n+1}(k_{p} |\mathbf{x}|) }-(\lambda+2\mu )k_{p}j_{n}(k_{p} |\mathbf{x}|), \\
T^{i,1}_{Q,n}(k_{s}|\Bx|) & = \mu \({2(n^2-1)|\mathbf{x}|^{-1} }{j_{n-1}(k_{s} |\mathbf{x}|) }  -(n+1) k_{s} j_{n}(k_{s} |\mathbf{x}|) \),\\
T^{i,2}_{Q,n}(k_{s}|\Bx|) & = \mu \( {-2 n(n+2)|\mathbf{x}|^{-1}  }{j_{n+1}\left( k_{s} |\mathbf{x}|\right) }  +n k_{s} j_{n}\left(k_{s} |\mathbf{x}|\right) \),\\
\end{aligned}
\]
and $T^e_{L,n}(k_{s}|\Bx|), T^{e,1}_{P,n}(k_{p}|\Bx|), T^{e,2}_{P,n}(k_{p}|\Bx|), T^{e,1}_{Q,n}(k_{s}|\Bx|), T^{e,2}_{Q,n}(k_{s}|\Bx|)$   can be obtained by replacing $j_n$ with $h_n$.
By using the transmission conditions across $\Gamma^\pm_j$, $j=1,2,\ldots,N$ and making use of the orthogonality of the vectorial spherical
harmonics,  we find that the constants
$
\bm{a}^\pm = \(\bm{a}_1^{\pm},\bm{a}_2^{\pm},\ldots,\bm{a}_N^{\pm}\)^T$ and $ \bm{c}^{\pm} = \(\bm{c}_1^{\pm},\bm{c}_2^{\pm},\ldots,\bm{c}_N^{\pm}\)^T,
$
with 
$\bm{a}_j^{\pm} = \(a^+_{j,n},b^+_{j,n}, a^-_{j,n},b^-_{j,n}\)$
 and $\bm{c}^{\pm} = \(c^+_{j,n},d^+_{j,n}, e^+_{j,n},f^+_{j,n},c^-_{j,n},d^-_{j,n}, e^-_{j,n},f^-_{j,n}\)^T$, satisfy, respectively, 
 \begin{equation}
 \bm{A}^{\mathrm{R}}_{(n)}(\omega,\delta)
 [\bm{a}^\pm] = 0, \mbox{ and } \bm{A}^{\mathrm{T}}_{(n)}(\omega,\delta)
 [\bm{c}^\pm] = 0,
 \end{equation}
for all $n\in \mathbb{N}$. Here   the $4N$-by-$4N$ matrix $\bm{A}^{\mathrm{R}}_{(n)}(\omega,\delta)$ and the $8N$-by-$8N$ matrix  $\bm{A}^{\mathrm{T}}_{(n)}(\omega,\delta)$ all have the block tridiagonal form
\begin{equation}\label{4Nby4N}
\bm{A}^{\textup{m}}_{(n)}(\omega,\delta):=\begin{pmatrix}
{M}^+_{1,n} & {R}^{+,-}_{1,1,n} & &  & && \\
\nm
{L}^{-,+}_{1,1,n} & {M}^-_{1,n} &{R}^{-,+}_{1,2,n}& &&&\\
\nm
&{L}^{+,-}_{2,1,n} & {M}^+_{2,n} & {R}^{+,-}_{2,2,n} & &&\\
\nm
&&{L}^{-,+}_{2,2,n} & {M}^-_{2,n} & {R}^{-,+}_{2,3,n} & &&\\
&& & \ddots &\ddots & \ddots& \\
&&  & & {L}^{+,-}_{N,N-1,n} & {M}^+_{N,n} & {R}^{+,-}_{N,N,n}\\
\nm
&&  & &  &{L}^{-,+}_{N,N,n} & {M}^-_{N,n}
\end{pmatrix}, \; \mbox{for } \mathrm{m}\in\{\mathrm{T,R} \},
\end{equation}
where for $\mathrm{m = R} $,
\[
{M}_{j,n}^+ = \begin{pmatrix}
-D^e_{L,n}(k_{s}r_j^+) & D^i_{L,n}(k_{s,\rmr,j}r_j^+) \\
\nm
-\delta T^e_{L,n}(k_{s}r_j^+) & T^i_{L,n}(k_{s,\rmr,j}r_j^+)
\end{pmatrix},
\quad 
{M}_{j,n}^- = \begin{pmatrix}
-D^e_{L,n}(k_{s,\rmr,j}r_j^-) & D^i_{L,n}(k_{s}r_j^-) \\
\nm
-T^e_{L,n}(k_{s,\rmr,j}r_j^-) & \delta T^i_{L,n}(k_{s}r_j^-)
\end{pmatrix},
\]
\[
{R}^{+,-}_{j,j,n} = \begin{pmatrix}
D^e_{L,n}(k_{s,\rmr,j}r_j^+)& 0\\
\nm
T^e_{L,n}(k_{s,\rmr,j}r_j^+)& 0
\end{pmatrix},
\quad
{L}^{-,+}_{j,j,n} =  \begin{pmatrix}
0&-D^i_{L,n}(k_{s,\rmr,j}r_j^-)\\
\nm
0&-T^i_{L,n}(k_{s,\rmr,j}r_j^-)
\end{pmatrix},
\]
and
\[
{L}^{+,-}_{j,j-1,n} = \begin{pmatrix}
0&-D^i_{L,n}(k_{s}r_j^+)\\
\nm
0&-\delta T^i_{L,n}(k_{s}r_j^+)
\end{pmatrix},
\quad
{R}^{-,+}_{j,j+1,n} = \begin{pmatrix}
D^e_{L,n}(k_{s}r_j^-)& 0\\
\nm
\delta T^e_{L,n}(k_{s}r_j^-)& 0
\end{pmatrix};
\]
for  $\mathrm{m = T} $,
\[
{M}_{j,n}^+ = \begin{pmatrix}
-D^{e,1}_{P,n}(k_{p}r_j^+) & D^{i,1}_{P,n}(k_{p,\rmr,j}r_j^+)&-D^{e,1}_{Q,n}(k_{s}r_j^+) & D^{i,1}_{Q,n}(k_{s,\rmr,j}r_j^+) \\
\nm
-\delta T^{e,1}_{P,n}(k_{p}r_j^+) & T^{i,1}_{P,n}(k_{p,\rmr,j}r_j^+)&-\delta T^{e,1}_{Q,n}(k_{s}r_j^+) & T^{i,1}_{Q,n}(k_{s,\rmr,j}r_j^+)\\
\nm
-D^{e,2}_{P,n}(k_{p}r_j^+) & D^{i,2}_{P,n}(k_{p,\rmr,j}r_j^+)&-D^{e,2}_{Q,n}(k_{s}r_j^+) & D^{i,2}_{Q,n}(k_{s,\rmr,j}r_j^+) \\
\nm
-\delta T^{e,2}_{P,n}(k_{p}r_j^+) & T^{i,2}_{P,n}(k_{p,\rmr,j}r_j^+)&-\delta T^{e,2}_{Q,n}(k_{s}r_j^+) & T^{i,2}_{Q,n}(k_{s,\rmr,j}r_j^+)\\
\end{pmatrix},
\]
\[
{M}_{j,n}^- = \begin{pmatrix}
-D^{e,1}_{P,n}(k_{p,\rmr,j}r_j^-) & D^{i,1}_{P,n}(k_{p}r_j^-)&-D^{e,1}_{Q,n}(k_{s,\rmr,j}r_j^-) & D^{i,1}_{Q,n}(k_{s}r_j^-) \\
\nm
- T^{e,1}_{P,n}(k_{p,\rmr,j}r_j^-) & \delta T^{i,1}_{P,n}(k_{p}r_j^-)&- T^{e,1}_{Q,n}(k_{s,\rmr,j}r_j^-) & \delta T^{i,1}_{Q,n}(k_{s}r_j^-)\\
\nm
-D^{e,2}_{P,n}(k_{p,\rmr,j}r_j^-) & D^{i,2}_{P,n}(k_{p}r_j^-)&-D^{e,2}_{Q,n}(k_{s,\rmr,j}r_j^-) & D^{i,2}_{Q,n}(k_{s}r_j^-) \\
\nm
- T^{e,2}_{P,n}(k_{p,\rmr,j}r_j^-) & \delta T^{i,2}_{P,n}(k_{p}r_j^-)&- T^{e,2}_{Q,n}(k_{s,\rmr,j}r_j^-) & \delta T^{i,2}_{Q,n}(k_{s}r_j^-)\\
\end{pmatrix},
\]
\[
{R}^{+,-}_{j,j,n} = \begin{pmatrix}
D^{e,1}_{P,n}(k_{p,\rmr,j}r_j^+)& 0&D^{e,1}_{Q,n}(k_{s,\rmr,j}r_j^+)& 0\\
\nm
T^{e,1}_{P,n}(k_{p,\rmr,j}r_j^+)& 0&T^{e,1}_{Q,n}(k_{s,\rmr,j}r_j^+)& 0\\
\nm
D^{e,2}_{P,n}(k_{p,\rmr,j}r_j^+)& 0&D^{e,2}_{Q,n}(k_{s,\rmr,j}r_j^+)& 0\\
\nm
T^{e,2}_{P,n}(k_{p,\rmr,j}r_j^+)& 0&T^{e,2}_{Q,n}(k_{s,\rmr,j}r_j^+)& 0
\end{pmatrix}.
\quad
{L}^{-,+}_{j,j,n} =  \begin{pmatrix}
0&-D^{i,1}_{P,n}(k_{p,\rmr,j}r_j^-)& 0&-D^{i,1}_{Q,n}(k_{s,\rmr,j}r_j^-)\\
\nm
0&-T^{i,1}_{P,n}(k_{p,\rmr,j}r_j^-)& 0&-T^{i,1}_{Q,n}(k_{s,\rmr,j}r_j^-)\\
\nm 
0&-D^{i,2}_{P,n}(k_{p,\rmr,j}r_j^-)& 0&-D^{i,2}_{Q,n}(k_{s,\rmr,j}r_j^-)\\
\nm
0&-T^{i,2}_{P,n}(k_{p,\rmr,j}r_j^-)& 0&-T^{i,2}_{Q,n}(k_{s,\rmr,j}r_j^-)
\end{pmatrix},
\]
and
\[
{L}^{+,-}_{j,j-1,n} = \begin{pmatrix}
0&-D^{i,1}_{P,n}(k_{p}r_j^+)& 0&-D^{i,1}_{Q,n}(k_{s}r_j^+)\\
\nm
0&-\delta T^{i,1}_{P,n}(k_{p}r_j^+)& 0&-\delta T^{i,1}_{Q,n}(k_{s}r_j^+)\\
\nm 
0&-D^{i,2}_{P,n}(k_{p}r_j^+)& 0&-D^{i,2}_{Q,n}(k_{s}r_j^+)\\
\nm
0&-\delta T^{i,2}_{P,n}(k_{p}r_j^+)& 0&-\delta T^{i,2}_{Q,n}(k_{s}r_j^+)
\end{pmatrix},
\quad
{R}^{-,+}_{j,j+1,n} = \begin{pmatrix}
D^{e,1}_{P,n}(k_{p}r_j^-)& 0&D^{e,1}_{Q,n}(k_{s}r_j^-)& 0\\
\nm
\delta T^{e,1}_{P,n}(k_{p}r_j^-)& 0&\delta T^{e,1}_{Q,n}(k_{}r_j^-)& 0\\
\nm
D^{e,2}_{P,n}(k_{p}r_j^-)& 0&D^{e,2}_{Q,n}(k_{s}r_j^-)& 0\\
\nm
\delta T^{e,2}_{P,n}(k_{p}r_j^-)& 0&\delta T^{e,2}_{Q,n}(k_{s}r_j^-)& 0
\end{pmatrix}.
\]
It is important to emphasize that, from a physical perspective, we are concerned with the subwavelength resonance of nested elastic materials, which corresponds to the system's lowest resonant frequency. At this frequency, the system exhibits a factor corresponding to $n = 1$, as the lowest resonance is characterized by the fewest number of oscillations \cite{DKLLZ_SAPM2025,LZArxiv}.
Consequently,  the  matrices
\begin{equation}\label{MSW}
\bm{A}^{\textup{m}}(\omega,\delta):=\bm{A}^{\textup{m}}_{(1)}(\omega,\delta), \; \mbox{for } \mathrm{m}\in\{\mathrm{T,R} \},
\end{equation}
become singular. 
In what follows, we set $\rho = \lambda = \mu = 1$, $\rho_{\rmr,j}  = 8000$, and $\lambda_{\rmr} = \mu_{\rmr} = 1/\delta$ with $\delta =1/10000$.
Let $ f^{\textup{m}}(\omega) = \det(\bm{A}^{\textup{m}}(\omega,\delta) )$ for $  \mathrm{m}\in\{\mathrm{T,R} \}$.  Hence,  calculating the subwavelength resonant frequency $\omega^{\textup{m}}$ is equivalent to determining the following complex root-finding problem
\begin{equation}\label{CRFP}
\omega^{\textup{m}} = \min\limits_{\omega \in \mathbb{C}}\{\omega:  f^{\textup{m}}(\omega) = 0\},
\end{equation}
which can be calculated by using Muller's method \cite{AK_book2018}.
We consider the radii of the layers to be 
equidistant.  For $N$-layer structure, set
\begin{equation}\label{str01}
r_i^+=(N-i+1)\;\text{ and }\;r_i^-=(N-i+0.5), \; \quad i=1, 2, \ldots N.
\end{equation}
It can be seen that, for fixed small $\delta>0$,  the $N$-layered nested elastic resonators possesses $2N$ eigenfrequencies with negative imaginary
part. The eigenfrequencies of 15-layer  nested resonators are shown in Figure \ref{15layerfrqc}. It is observed that these eigenfrequencies lie in the lower half of the complex plane and exhibit symmetry about the imaginary axis. 
To assign a physical meaning to the complex subwavelength resonance frequencies, we next select those with a positive real part.
The negative imaginary part corresponds to the loss of energy to the far field.
From the difinition of the resonant quality factor $Q_{\omega}:=\frac{\Re \omega}{-2\Im\omega}$, 
notably, compared to the translational subwavelength resonant frequencies $\omega^{\mathrm{T}}$, the rotational ones $\omega^{\mathrm{R}}$  exhibit lower scattering losses $|\Im(\omega^{\mathrm{R}})|$, indicating higher quality factors and longer lifetimes.
\begin{figure}
	\centering
	\includegraphics[scale=0.17]{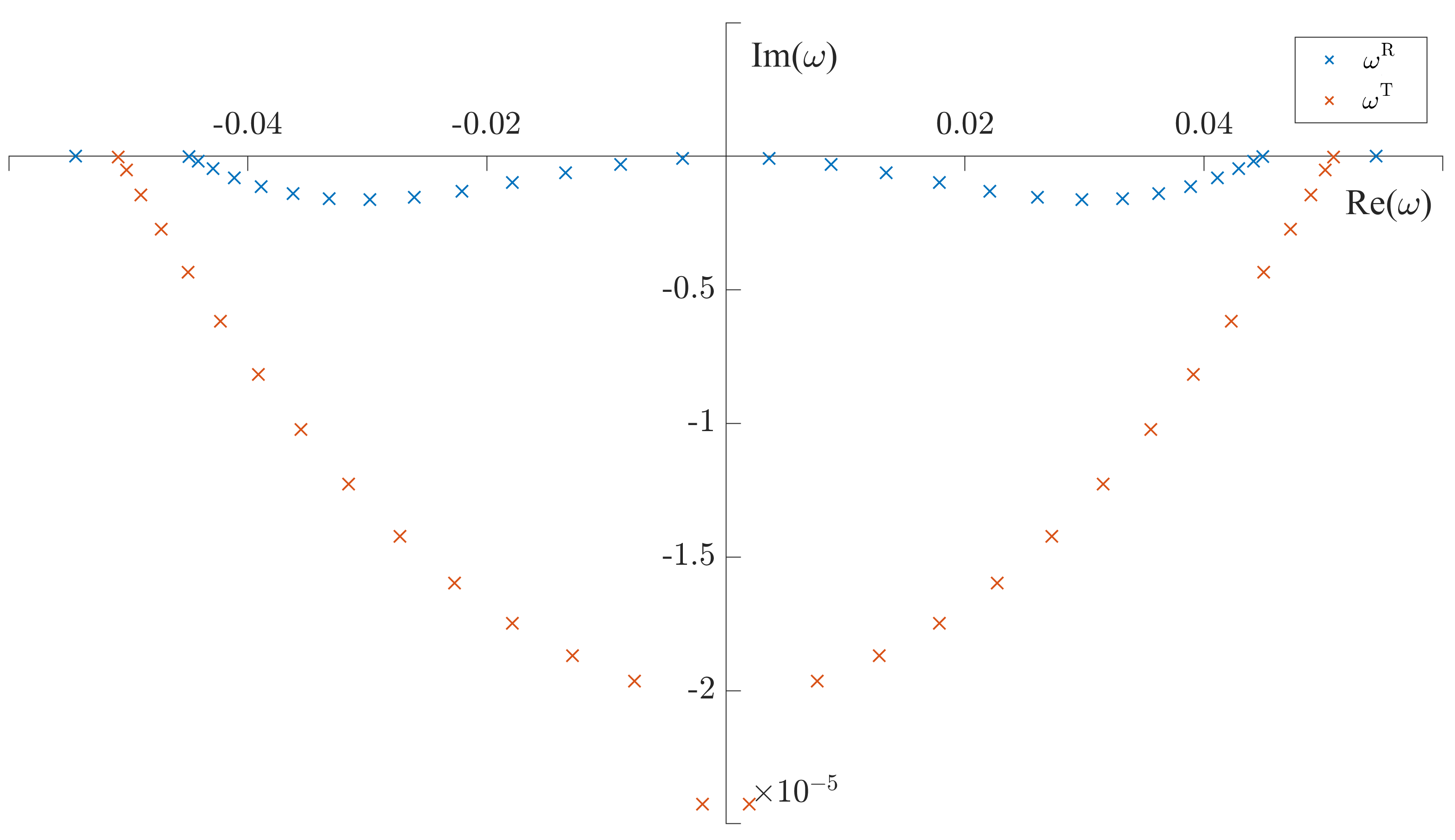}
	\caption{The subwavelength resonant frequencies of the 15-layered  nested resonators designed by \eqref{str01}. }\label{15layerfrqc}
\end{figure}

Next, we examine the spatial distribution of the eigenmodes $\mathbf{u}^{\mathrm{m}}_{j}$ associated with each eigenfrequency $\omega^{\mathrm{m},+}_{j}$.
We refer to the eigenmodes $\mathbf{u}^{\mathrm{T}}_{j}$ and $\mathbf{u}^{\mathrm{R}}_{j}$ corresponding to $\omega^{\mathrm{T},+}_{j}$ and $\omega^{\mathrm{R},+}_{j}$ as the translational and rotational eigenmodes, respectively.
 In order to use  \eqref{lame_solution2} and \eqref{MSW}, we first apply \eqref{vectorial_spherical_harmonics} to derive 
\begin{equation}\label{eq:I}
\mathcal{I}_0^{0}=\sqrt{\frac{3}{4\pi}}
(0,0,1)^t,
%\;\mathcal{I}_0^{-1}=\sqrt{\frac{3}{8\pi}}(1,-\mathrm{i},0)^t,
\;     \mathcal{I}_0^{1}=-\overline{\mathcal{I}_0^{-1}}= \sqrt{\frac{3}{8\pi}}
(-1,-\mathrm{i},0)^t,
\end{equation}
\begin{equation}\label{eq:T}
\mathcal{T}_1^{0}=\sqrt{\frac{3}{4\pi}}\frac{(-x_2,x_1,0)^t}{r}
,%\;\mathcal{T}_1^{-1}=\sqrt{\frac{3}{8\pi}}  \frac{(-\mathrm{i} x_3,-x_3,\mathrm{i} x_1 + x_2)^t}{r},
\;   \mathcal{T}_1^{1}=-\overline{\mathcal{T}_1^{-1}}=\sqrt{\frac{3}{8\pi}}\frac{(-\mathrm{i} x_3,x_3,\mathrm{i} x_1 - x_2)^t}{r},
\end{equation}
and 
\begin{equation}\label{eq:N}
	\mathcal{N}_2^{0}=\scalebox{1.2}{$\sqrt{\frac{3}{4\pi}}\frac{\(3x_1x_3, 3x_2x_3, 3x_3^2-r^2\)^t}{r^2}$}, %\mathcal{N}_2^{-1}=\sqrt{\frac{3}{8\pi}} & \frac{\(3x_1(x_1-\mathrm{i} x_2)-r^2, 3x_2(x_1-\mathrm{i} x_2)+\mathrm{i} r^2, 3x_3(x_1-\mathrm{i} x_2)\)^t}{r^2}, \quad   \\
 \mathcal{N}_2^{1}= -\overline{\mathcal{N}_2^{-1}}= \scalebox{1.2}{$\sqrt{\frac{3}{8\pi}}  \frac{\(3x_1(-x_1-\mathrm{i} x_2)+r^2, 3x_2(-x_1-\mathrm{i} x_2)+\mathrm{i} r^2, 3x_3(-x_1-\mathrm{i} x_2)\)^t}{r^2}.$}
\end{equation}
Using \eqref{eq:I}--\eqref{eq:N}, and substituting the eigenfrequencies obtained from \eqref{CRFP} together with the associated coefficients from \eqref{MSW} into \eqref{lame_solution2}, we numerically compute the corresponding eigenmodes. 

\begin{figure}[H]
	\centering
	\includegraphics[scale=0.72]{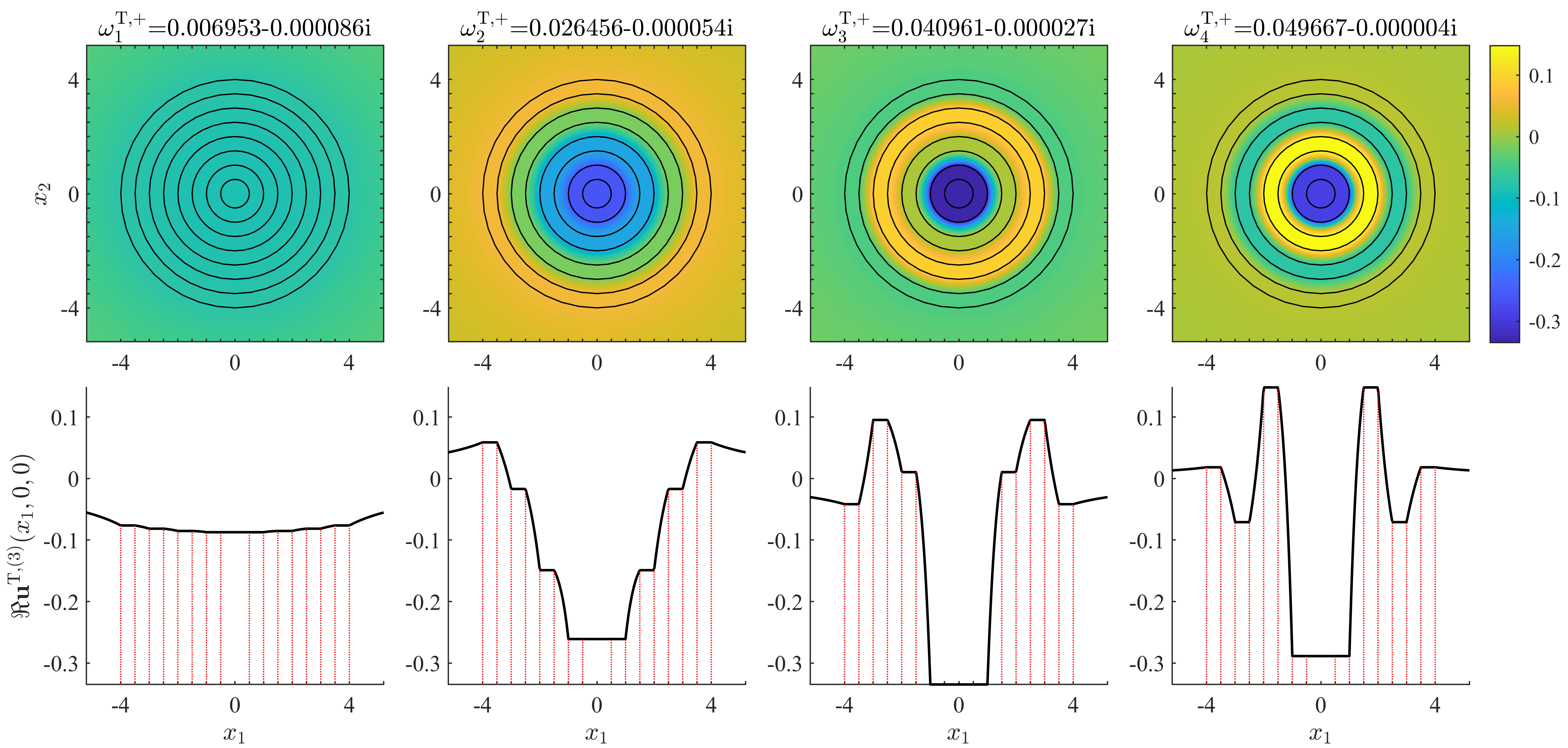}
	\caption{The third components of the translational elastic displacement field distributions $\mathbf{u}_1^{\mathrm{T},(3)},\mathbf{u}_2^{\mathrm{T},(3)},\mathbf{u}_3^{\mathrm{T},(3)},\mathbf{u}_4^{\mathrm{T},(3)}$ for the four-layer nested resonators designed by \eqref{str01}. Each pair of plots corresponds to one of the four eigenfrequencies $\omega^{\mathrm{T},+}_{j}$, $j=1,2,3,4$. The upper panel displays a contour plot of the function $\Re \mathbf{u}_j^{\mathrm{T},(3)}(x_1, x_2,0)$, with the four-layer concentric sphere designed by \eqref{str01} represented as solid black lines. The lower panel shows the cross section of the upper plot, taken along the line $x_2 = 0$ (passing through the centres of the multi-layered structures). Additionally, red dotted lines represent vertical lines at the coordinates of the radii.}\label{4layereigenmodeT}
	
	\vspace{2mm}
	
	\includegraphics[scale=0.72]{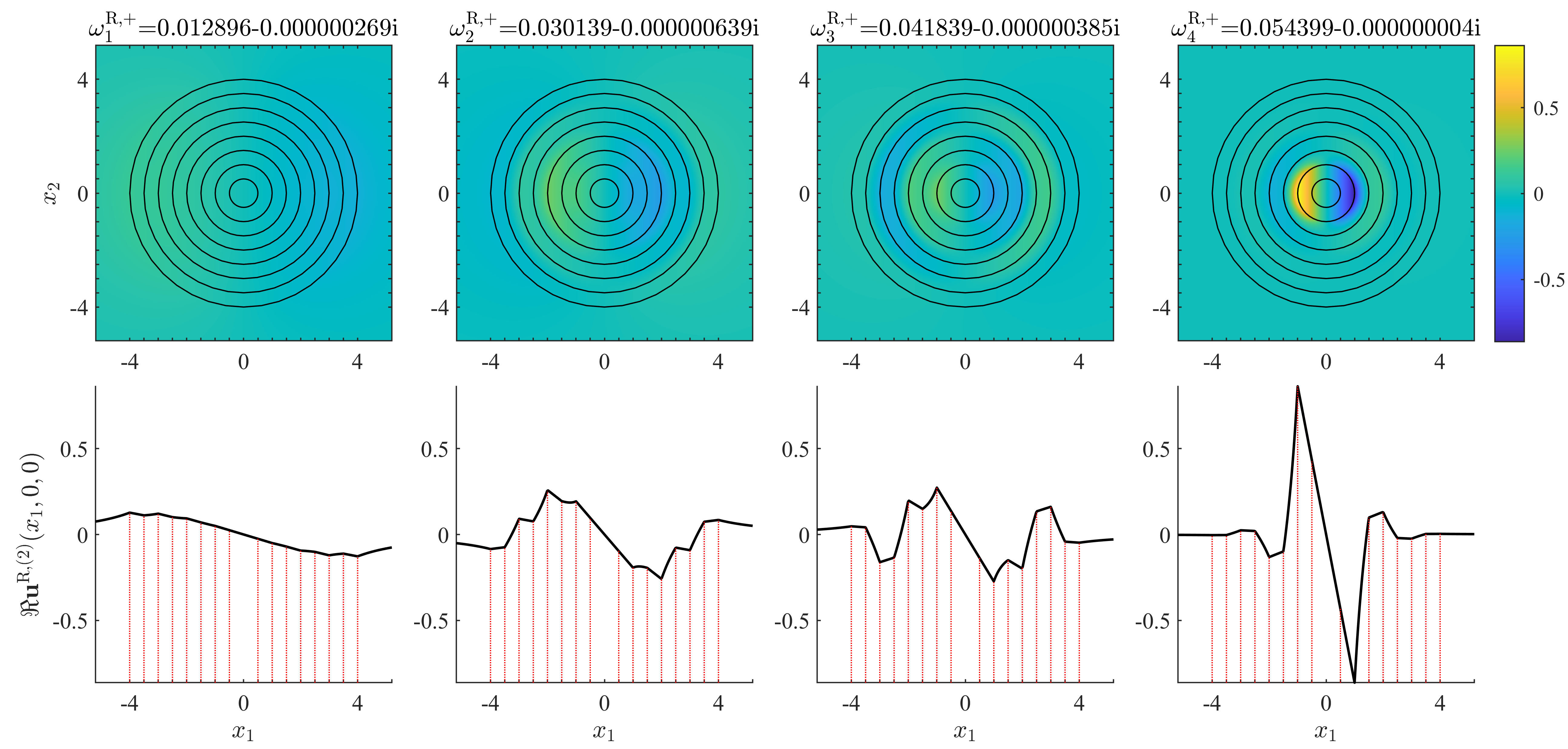}
	\caption{The second components of the rotational elastic displacement field distributions $\mathbf{u}_1^{\mathrm{R},(2)},\mathbf{u}_2^{\mathrm{R},(2)},\mathbf{u}_3^{\mathrm{R},(2)},\mathbf{u}_4^{\mathrm{R},(2)}$ for the four-layer nested resonators designed by \eqref{str01}. 
	}\label{4layereigenmodeR}
\end{figure}

To facilitate visualization of the results, we focus on four-layer nested resonators and examine the spatial distribution of selected components of the eigenmodes. The  third components of the translational elastic displacement distributions $\mathbf{u}_j^{\mathrm{T},(3)}$  and  the  second components of the rotational elastic displacement distributions $\mathbf{u}_j^{\mathrm{R},(2)}$ for the four-layer nested resonators, as designed by \eqref{str01}, are shown in Figures \ref{4layereigenmodeT} and \ref{4layereigenmodeR}, respectively.  The eigenfrequencies $\omega^{\mathrm{m},+}_j$, with $\mathrm{m}\in \{\mathrm{T,R}\}$ and $j=1,2,3,4$, are listed in ascending order of their real parts. 
It is evident from the lower panel of Figure \ref{4layereigenmodeT} that the translational elastic displacement is approximately constant within each resonator and exhibits a symmetric profile about its center, whereas the lower panel of Figure \ref{4layereigenmodeR} shows that the rotational elastic displacement varies approximately linearly within each resonator and exhibits an antisymmetric profile about its center.
The nearly constant profile characterizes a monopole-type mode, while the antisymmetric linear profile characterizes a dipole-type mode, in agreement with the modal decomposition established in Proposition \ref{propMLCB}.
In particular, the nearly constant displacement profile observed in the translational eigenmode reflects an isotropic compression–expansion pattern characteristic of monopole sources, whereas the antisymmetric linear profile of the rotational eigenmode represents a shear-dominated dipole field. This behavior is consistent with the elastic multipole expansion in \eqref{ext_solution}–\eqref{int_solution} at subwavelength scales.

\subsection{Localization criterion based on eigenvalue position}

In this subsection, we present numerical simulations to verify the exponential localization of defect modes, as well as to validate the localization criterion based on the position of the eigenvalue. As mentioned in subsection \ref{dimerdefect}, the theory and computation for the exponential localization criterion of dimer-type resonators include a geometric defect have been established in \cite{DKZ_JLMS2026}. Next, we mainly consider an $N$-layered monomer-type metamaterial with multiple material-parameter defects, under the assumptions \eqref{interface}, \eqref{identicalvolume}, and \eqref{monomersetupdefect1}. 
Since \eqref{stif_tra} and \eqref{stif_tor} show that the matrix $\bm{\mathcal{S}}^{\textup{m}}$ takes the similar form for both $\mathrm{m}=\mathrm{T}$ and $\mathrm{m}=\mathrm{R}$, we proceed, without loss of generality, with numerical verification only for the rotational case $\mathrm{m}=\mathrm{R}$.

We first consider material-parameter defects with multiple different perturbation intensities $(i_l,\eta_l)$ for $1\leq l\leq M$.
We set
$N=45$, $r_1^+ = 2.4$, $\beta^{\textup{R}} = 7500$,
so that the geometric configuration of the 45-layered concentric spheres is uniquely determined.  We further take $\rho=\lambda=\mu=1$, and $\lambda_{\rmr}=\mu_{\rmr}=1/\delta$ with $\delta=1/10000$. The reference densities are chosen as $\rho_{\rmr,j}=10000$ for all layers except $\rho_{\rmr,4}=4900$, $\rho_{\rmr,22}=8100$, and $\rho_{\rmr,34}=6400$. Under this setting, the 45-layered structure possesses three material-parameter defects $\{(i_l,\eta_l)\}_{l=1}^3=\{(22,-0.1),(34,-0.2),(4,-0.3)\}$.

We compare the numerical roots $\lambda^{\mathrm{R}}_{(i_l,\eta_l)}$ of transcendental equation \eqref{Lexp} and the eigenvalues $\lambda^{\mathrm{R}}_{N-M+l}$  of the generalized elastic stiffness eigenvalue problem \eqref{GEP14}. The comparison between the two methods for computing the defect eigenvalues, together with the corresponding relative errors for selected values of $(i_l,\eta_l)$, is presented  in Table \ref{table-results}.
The results indicate that the relative error decreases as  $i_l$ increases, while the defect eigenvalues grow with perturbation magnitude  $|\eta_l|$, in full agreement with the theoretical predictions.

\begin{table}[h]% ed d=1
	\centering
	\begin{tabular}{cccc}
		\toprule
		$(i_l,\eta_l)$ & $\lambda^{\mathrm{R}}_{42+l}$ & $\lambda^{\mathrm{R}}_{(i_l,\eta_l)}$ & \text{Relative error} \\
		\midrule
		$(22,-0.1)$ &   31123.0730289473 &31123.5604313636 & 1.5660$\times 10^{-5}$  \\
		$(34,-0.2)$ & 34466.9116208634   & 34466.9117647049    & 4.1733 $\times 10^{-9}$ \\
		$(4,-0.3)$ & 40535.4432657877 &40542.5104042893& 1.7434$\times 10^{-4}$  \\
		\bottomrule
	\end{tabular}
	\caption{A comparison between the eigenvalues $\lambda^{\mathrm{R}}_{(i_l,\eta_l)}$ of \eqref{Lexp} and the eigenvalues $\lambda^{\mathrm{R}}_{42+l}$  of the generalized elastic stiffness eigenvalue problem \eqref{GEP14} with $N=45$, $r_1^+ = 2.4$, and $\beta^{\textup{R}} = 7500$, over several values of $(i_l,\eta_l)$.}
	\label{table-results}
\end{table}

The spectral analysis of the generalized elastic stiffness eigenvalue problem \eqref{GEP14}, illustrated in Figure \ref{Eigenvector_With_diffDefects}, clearly demonstrates the existence of several eigenvalues located above the bulk spectrum. The number of these eigenvalues equals the number of defects, and they are clearly separated from the bulk eigenvalues shown as black dots.
The behavior of the eigenvectors exhibits three characteristic regimes depending on the spectral location of the corresponding eigenvalues, as illustrated in Figure \ref{Eigenvector_With_diffDefects}.
Eigenvectors associated with eigenvalues lying above the spectral bulk exhibit exponential decay in both directions away from the defect site. Moreover, the degree of localization and the amplitude increase with the magnitude of the perturbation magnitude 
$|\eta_l|$.
In contrast, eigenvectors corresponding to bulk eigenvalues display oscillatory behavior resembling trigonometric functions.
At the boundary of the asymptotic spectral bulk, the eigenvectors exhibit an approximately linear profile, representing a transitional regime between the exponentially localized and oscillatory states.
These numerical results show excellent agreement with the theoretical predictions of Theorem \ref{Propiff2}, confirming the fundamental relationship between the spectral position of an eigenvalue and the spatial localization properties of its associated eigenvector.
It is worth emphasizing that, in contrast to the higher-order plasmonic resonant modes in multilayered negative-index structures, which are typically concentrated at the outermost interface \cite{KZDF_JCP}, the present framework enables the placement of localized modes at arbitrary layers within the structure.

\begin{figure}[H]
	\centering  %图片全局居中
	\subfigbottomskip=-5pt %两行子图之间的行间距
	\subfigcapskip=-5pt %设置子图与子标题之间的距离
	\subfigure[]{
		\includegraphics[width=0.322\linewidth]{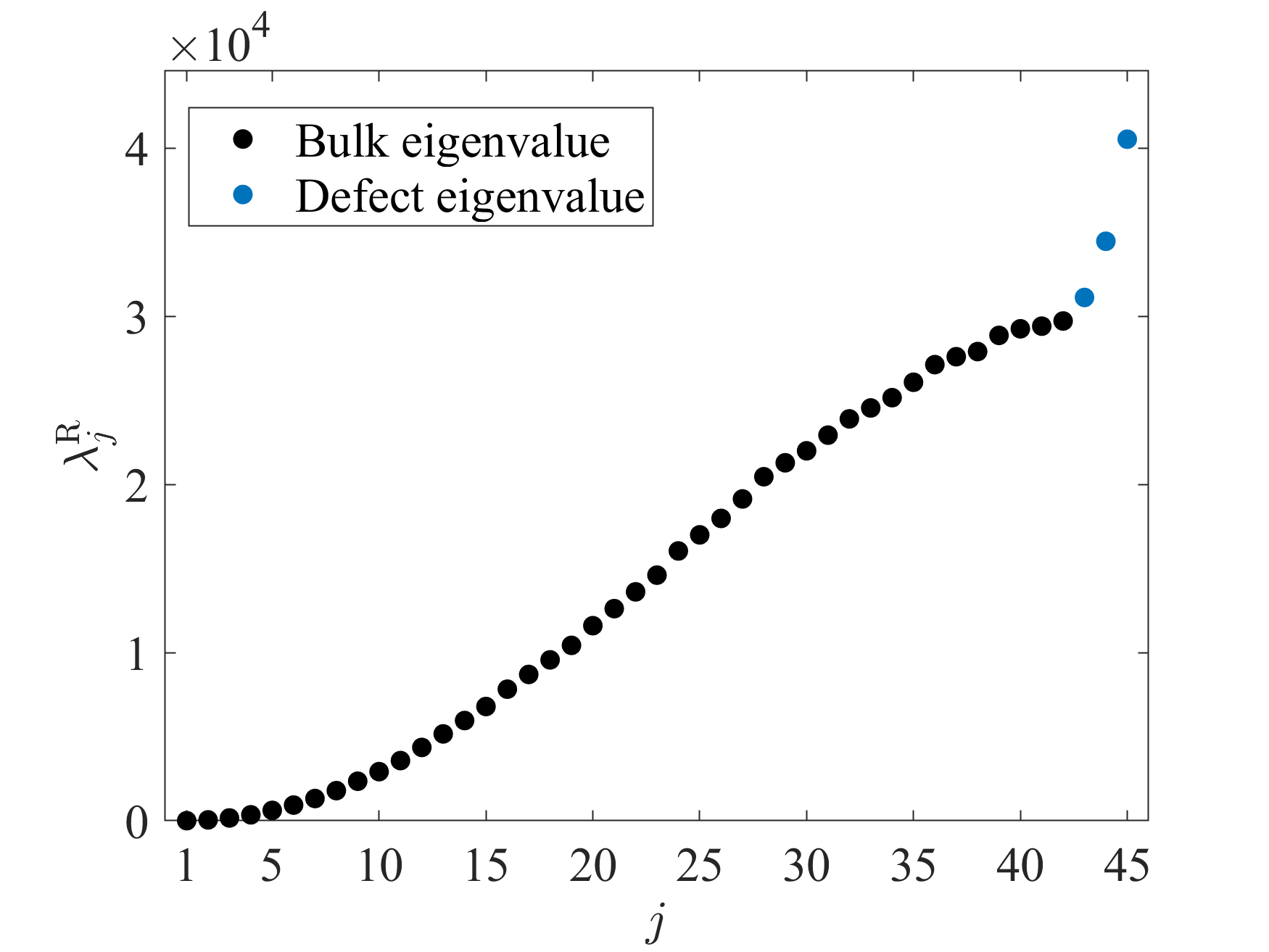}}
	\subfigure[]{
		\includegraphics[width=0.322\linewidth]{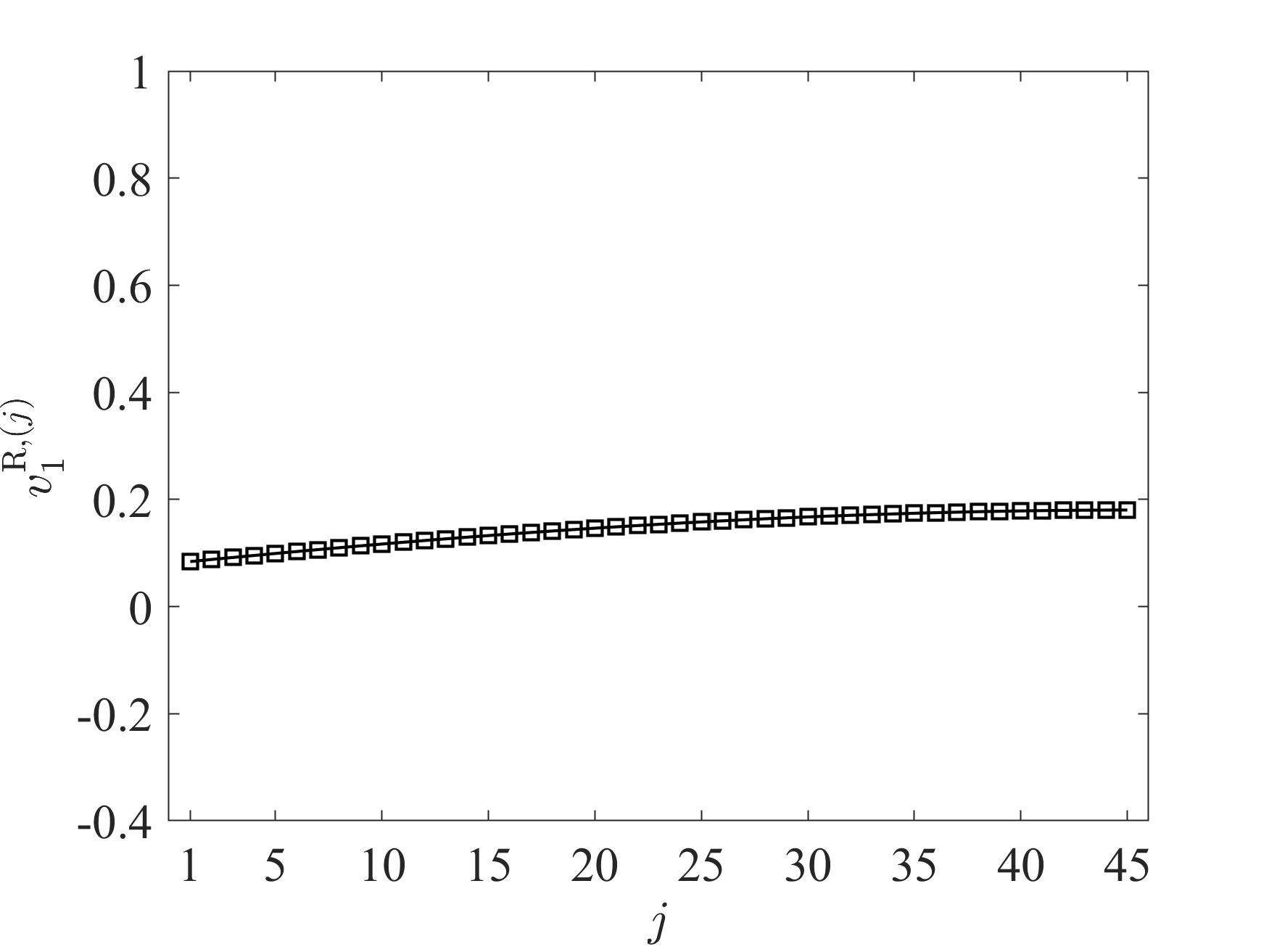}}
	%\quad
	\subfigure[]{
		\includegraphics[width=0.322\linewidth]{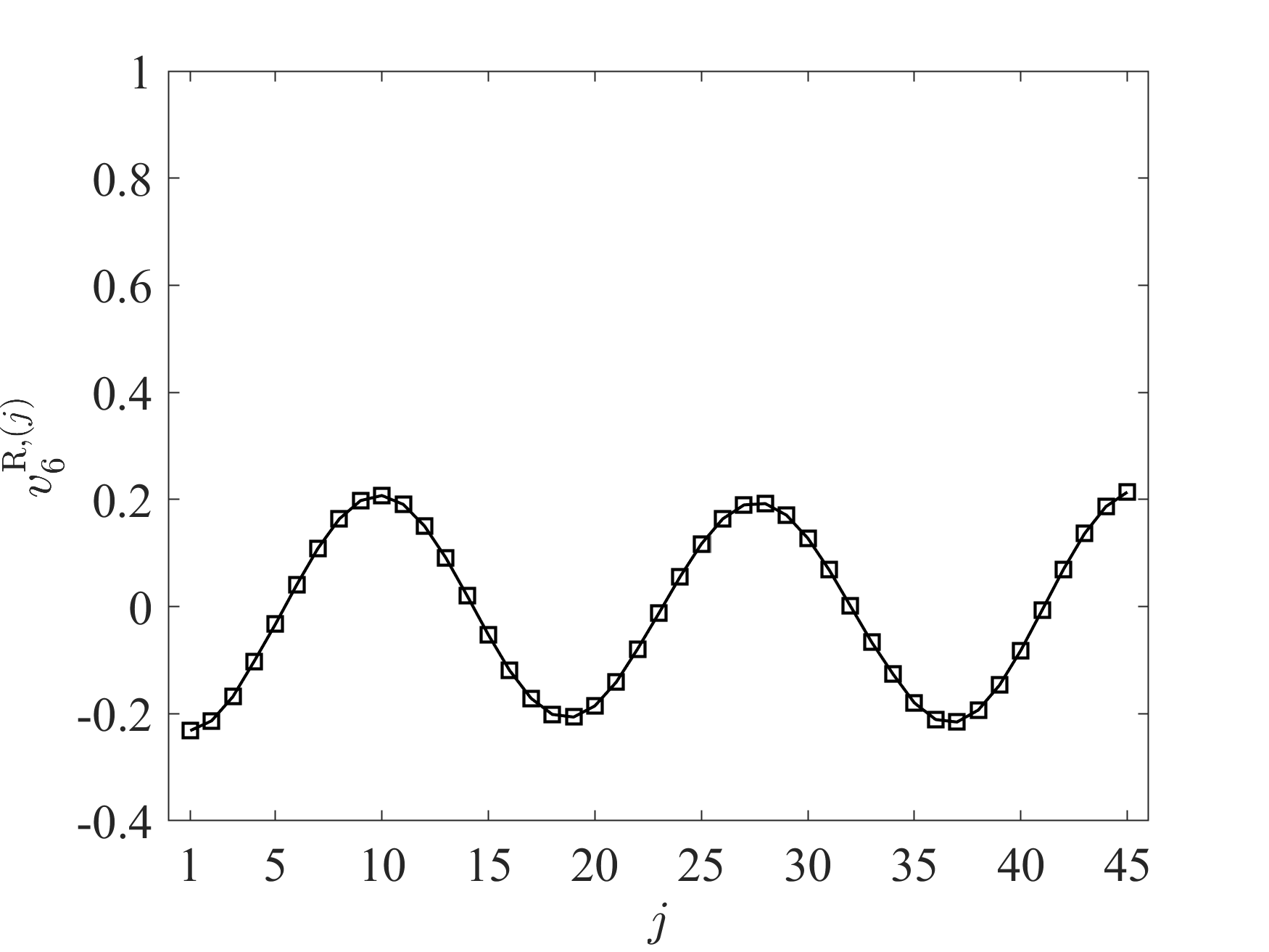}}
	\subfigure[]{
		\includegraphics[width=0.322\linewidth]{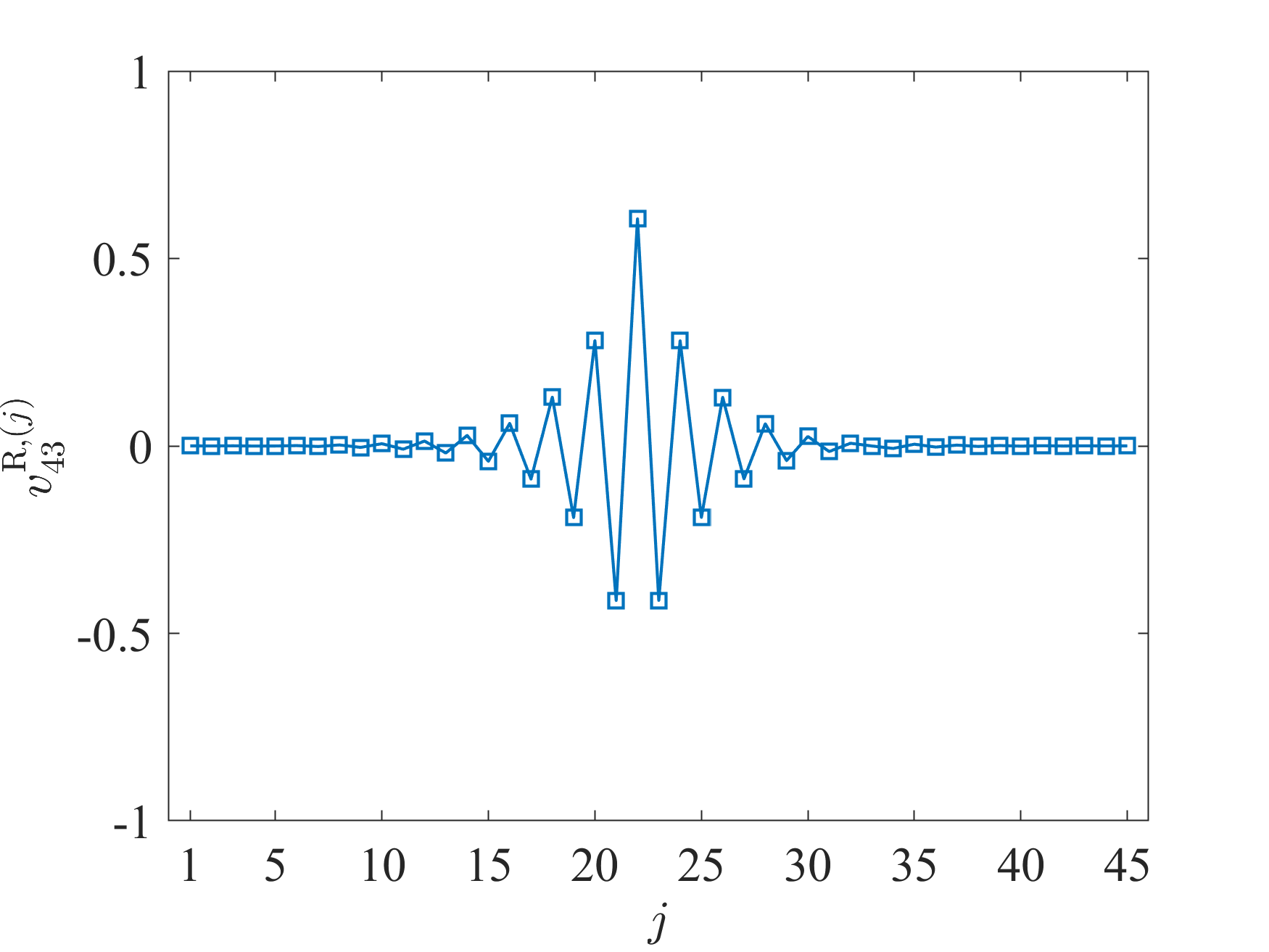}}
	\subfigure[]{
		\includegraphics[width=0.322\linewidth]{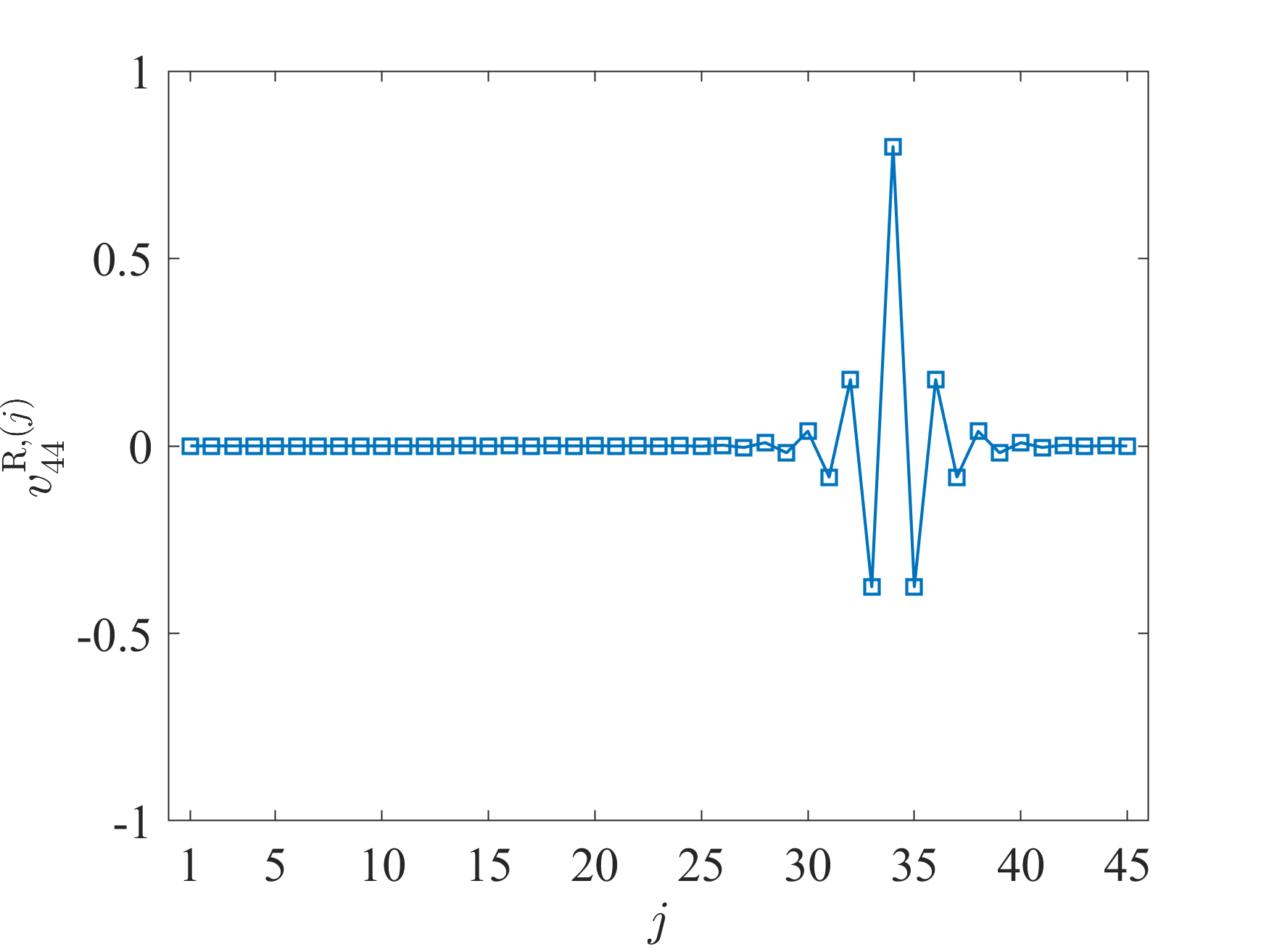}}
	\subfigure[]{
		\includegraphics[width=0.322\linewidth]{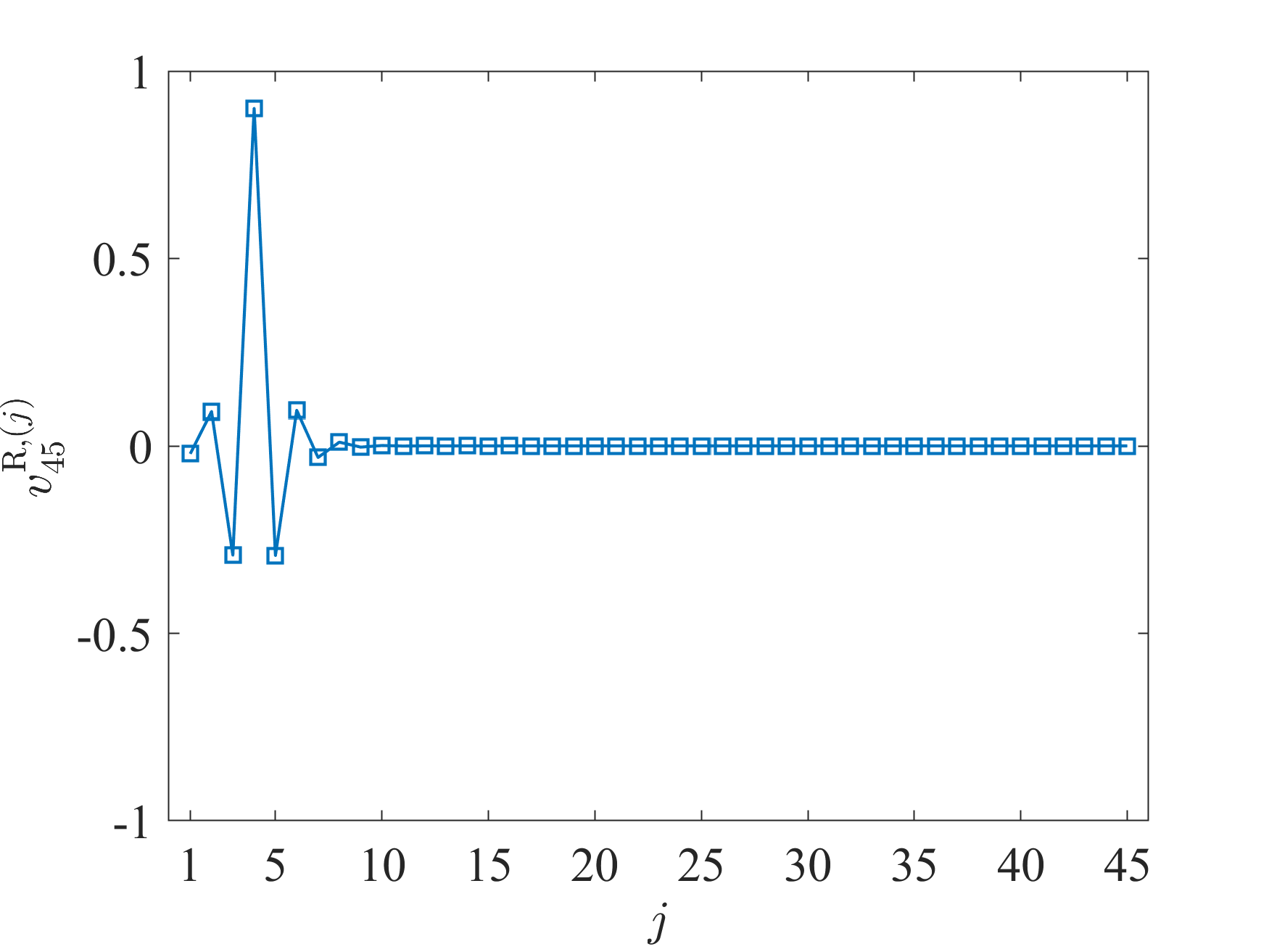}}
	\caption{Eigenpair profiles of the generalized elastic stiffness eigenvalue problem \eqref{GEP14} with $N=45$, $r_1^+=2.4$, $\beta^{\textup{R}} = 7500$ and $\{(i_l,\eta_l)\}_{l=1}^3=\{(22,-0.1),(34,-0.2),(4,-0.3)\}$. 
		(a) Distribution of eigenvalues, where black dots represent bulk eigenvalues and blue dots represent defect eigenvalues.
		(b) An approximately linear eigenvector corresponding to an eigenvalue located at the boundary of the spectral bulk.
		(c) A trigonometric-like oscillatory eigenvector corresponding to a bulk eigenvalue. 
		(c)--(f)  Exponentially localized eigenvectors associated with defect eigenvalues, showing localization around the corresponding defect sites. 
	}\label{Eigenvector_With_diffDefects}
\end{figure}

\subsection{Simultaneously localized defect modes at multiple defect sites}

In this subsection, we present numerical simulations to verify the existence of quasi-degenerate eigenvalues and their sharper asymptotic expansions \eqref{MLexp}, together with the amplitude ratios \eqref{amplitude_ratio} of the corresponding eigenvectors at each defect site. Moreover, we include numerical validations for hierarchical configurations (see Remark \ref{rem53}), in which the set $\{(i_l,\eta_l)\}_{l=1}^M$ is partitioned into $b$ blocks, each consisting of equally spaced defect locations with identical perturbation parameters, while the spacing and perturbation parameters may vary across different blocks. In addition, we numerically compute the associated localized defect modes using the elastic multipole expansion method \eqref{ext_solution}--\eqref{lame_solution2}.

We first consider material-parameter defects with equidistant defect sites and identical perturbation parameter $(i_l,\eta)$ for $1\leq l\leq M$.
We set
$N=50$, $r_1^+ = 2.4$, $\beta^{\textup{R}} = 9500$, We further take  four material-parameter defects $\{(i_l,\eta)\}_{l=1}^4=\{(10,-0.25),(20,-0.25),(30,-0.25),(40,-0.25)\}$. 

The spectral analysis of the generalized elastic stiffness eigenvalue problem \eqref{GEP14}, illustrated in Figure \ref{Eigenvector_With_sameDefects}(a), clearly reveals the presence of several quasi-degenerate eigenvalues located above the bulk spectrum. The number of such eigenvalues coincides with the number of defects, and they are well separated from the bulk eigenvalues.  
We further compare the asymptotic predictions \eqref{MLexp} for $\lambda^{\mathrm{m}}_{(i_l,\eta),d}$ with the numerically computed eigenvalues $\lambda^{\mathrm{R}}_{N-M+l}$ of \eqref{GEP14}. The comparison, presented in Figure \ref{Eigenvector_With_sameDefects}(b), shows excellent agreement between the two approaches, thereby validating the accuracy of the asymptotic expansions \eqref{MLexp}.
Moreover, as displayed in Figure \ref{Eigenvector_With_sameDefects}(c)–(f), the amplitudes of the eigenvectors corresponding to the defect eigenvalues, evaluated at the defect sites, are given respectively by
\[
\(\bm{v}_{47}^{\mathrm{R},{(10s)}}\)_{1\leq s\leq 4} \approx 0.54\(\sin\frac{4s\pi}{5}\)_{1\leq s\leq 4},\quad (\bm{v}_{48}^{\mathrm{R},{(10s)}})_{1\leq s\leq 4} \approx  0.54\(\sin\frac{3s\pi}{5}\)_{1\leq s\leq 4},
\]
\[(\bm{v}_{49}^{\mathrm{R},{(10s)}})_{1\leq s\leq 4} \approx  0.54\(\sin\frac{2s\pi}{5}\)_{1\leq s\leq 4},\quad (\bm{v}_{50}^{\mathrm{R},{(10s)}})_{1\leq s\leq 4} \approx  0.54 \(\sin\frac{s\pi}{5}\)_{1\leq s\leq 4}.\]
These numerical results are in excellent agreement with the theoretical amplitude ratio \eqref{amplitude_ratio} for the even case $d=10$. 

From a broader perspective, this behavior can be viewed as a macroscopic analogue of the classical vibrational analysis of conjugated molecules such as butadiene \cite{AF_book2011}, where molecular orbital theory combined with the H\"{u}ckel approximation predicts similar sine-type amplitude distributions. In this sense, our results provide a natural extension from molecular-scale vibrational modes to multimodal defect-induced patterns in macroscopic structured media, thereby offering a unified framework that connects subwavelength metamaterials with classical quantum-inspired models.

\begin{figure}[H]
	\centering  %图片全局居中
	\subfigbottomskip=-5pt %两行子图之间的行间距
	\subfigcapskip=-5pt %设置子图与子标题之间的距离
	\subfigure[]{
		\includegraphics[width=0.322\linewidth]{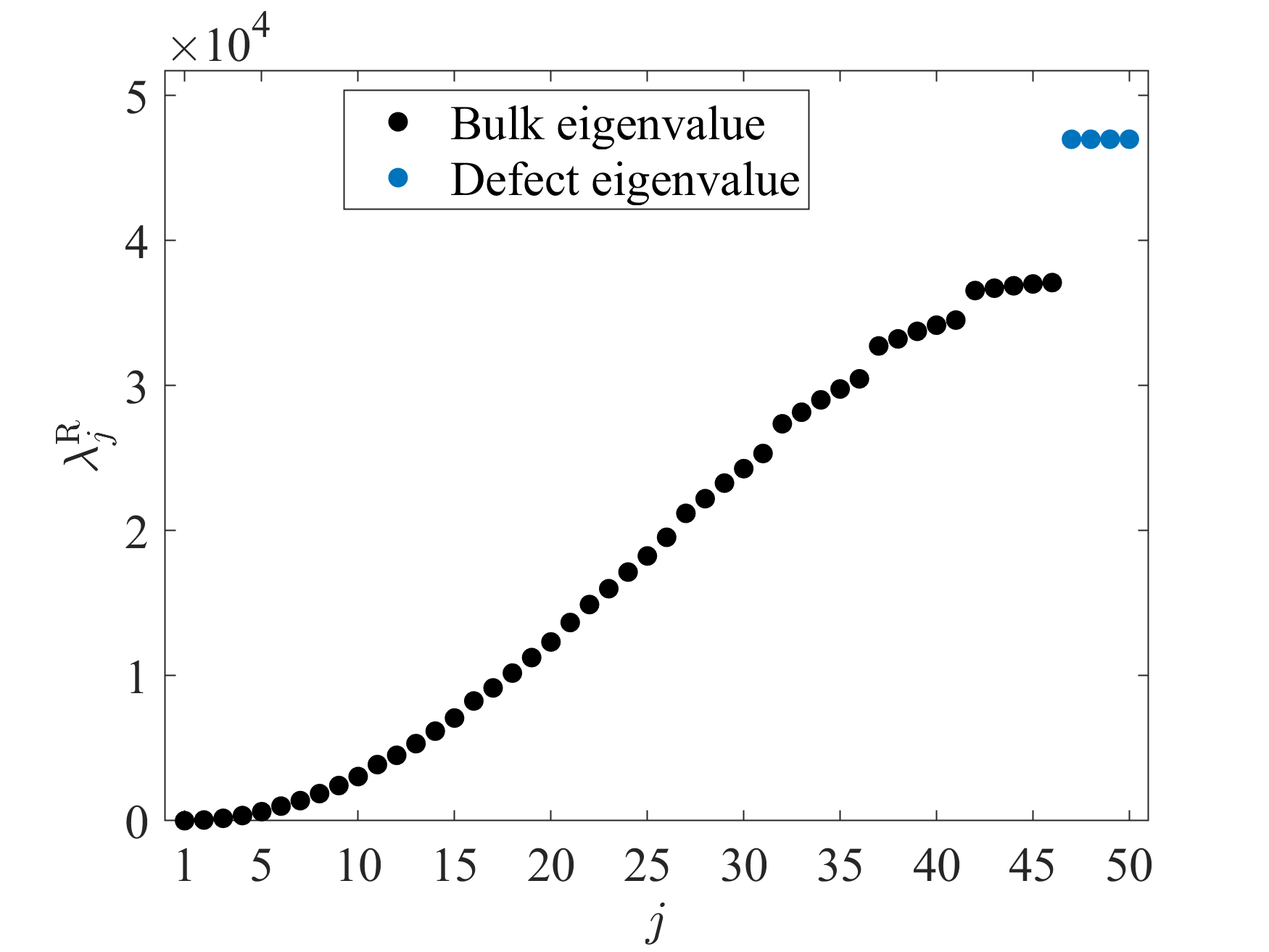}}
	\subfigure[]{
		\includegraphics[width=0.33\linewidth]{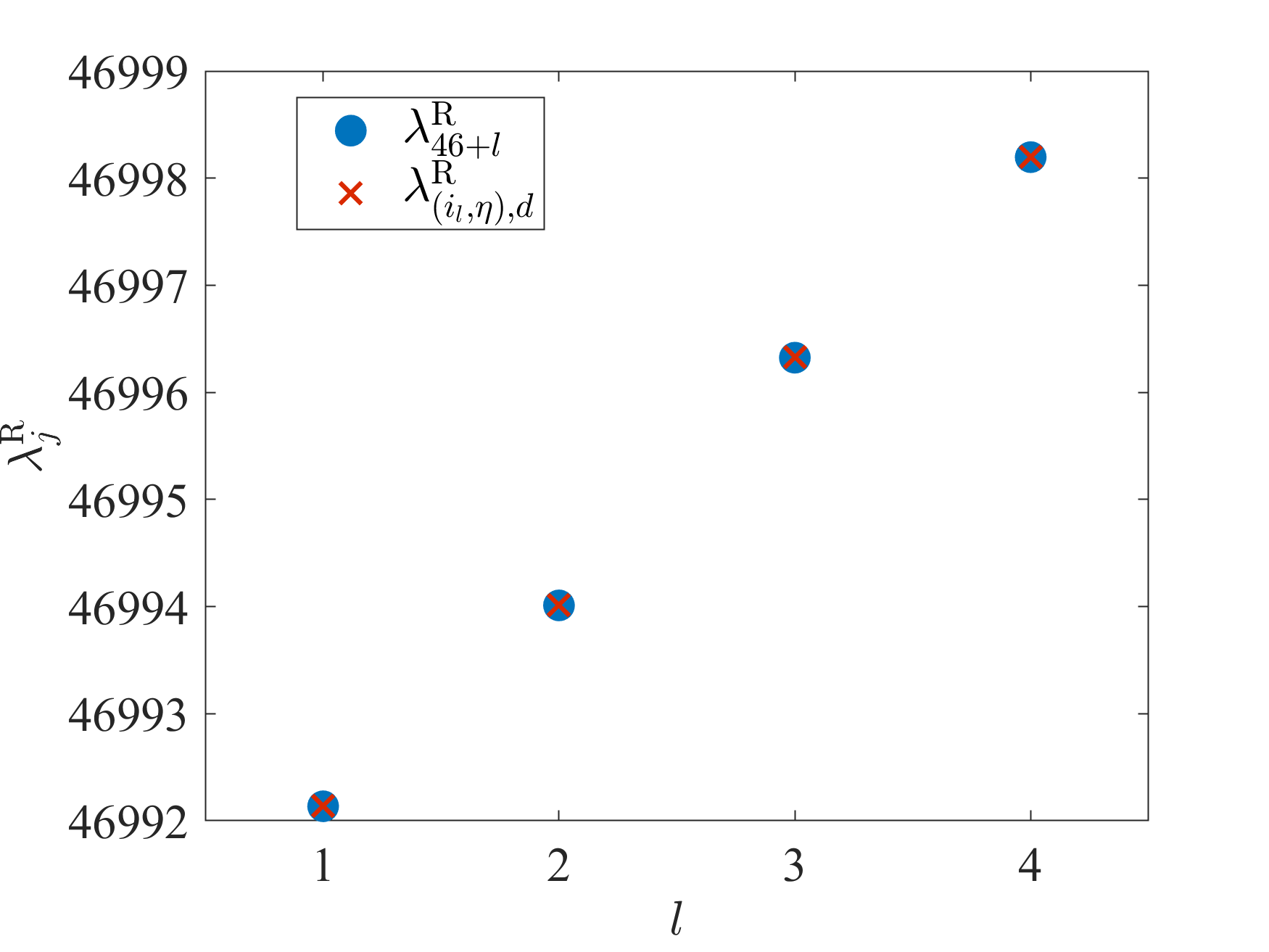}}
	\subfigure[]{
		\includegraphics[width=0.322\linewidth]{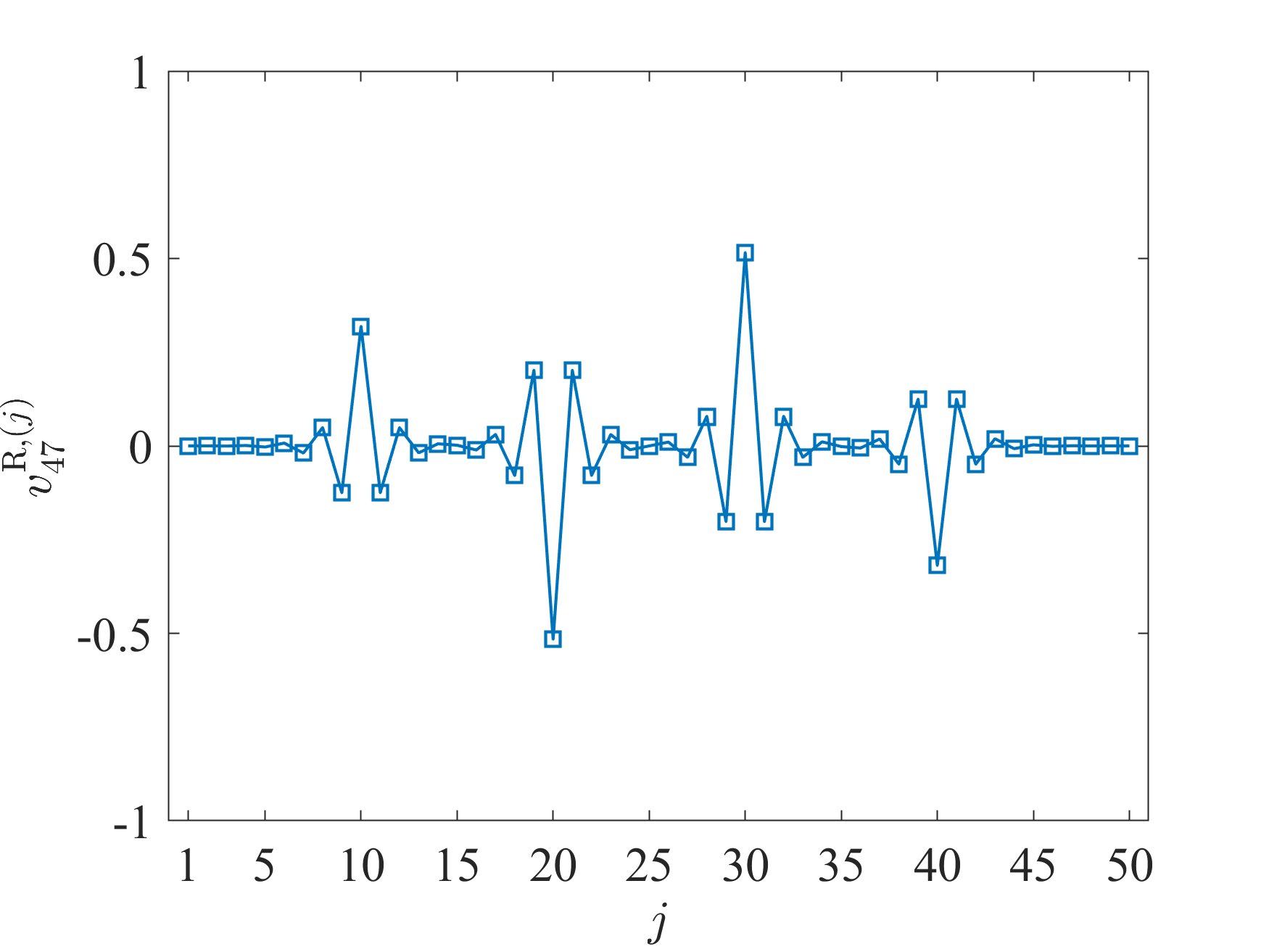}}
	%\quad
	\subfigure[]{
		\includegraphics[width=0.322\linewidth]{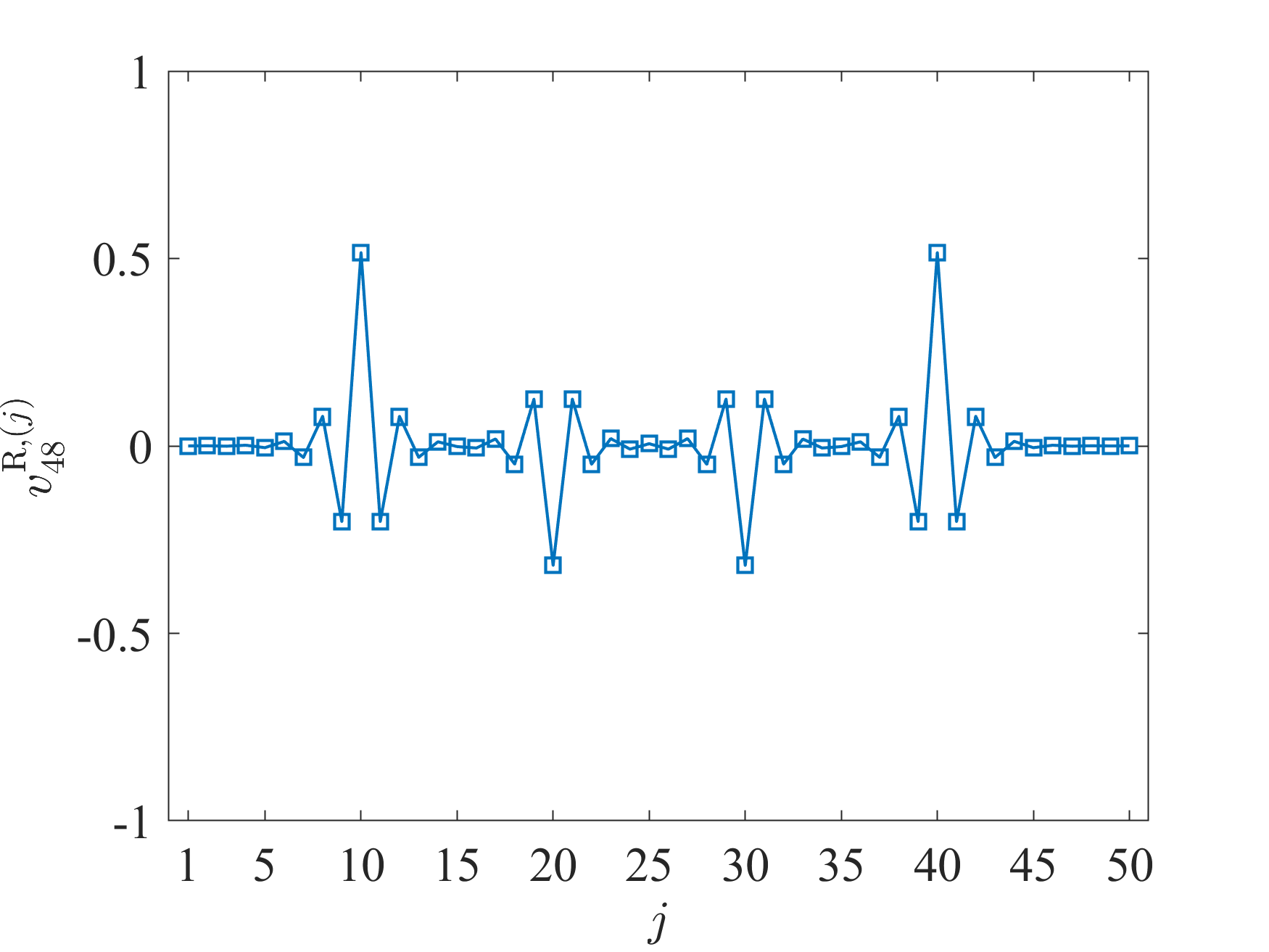}}
	\subfigure[]{
		\includegraphics[width=0.322\linewidth]{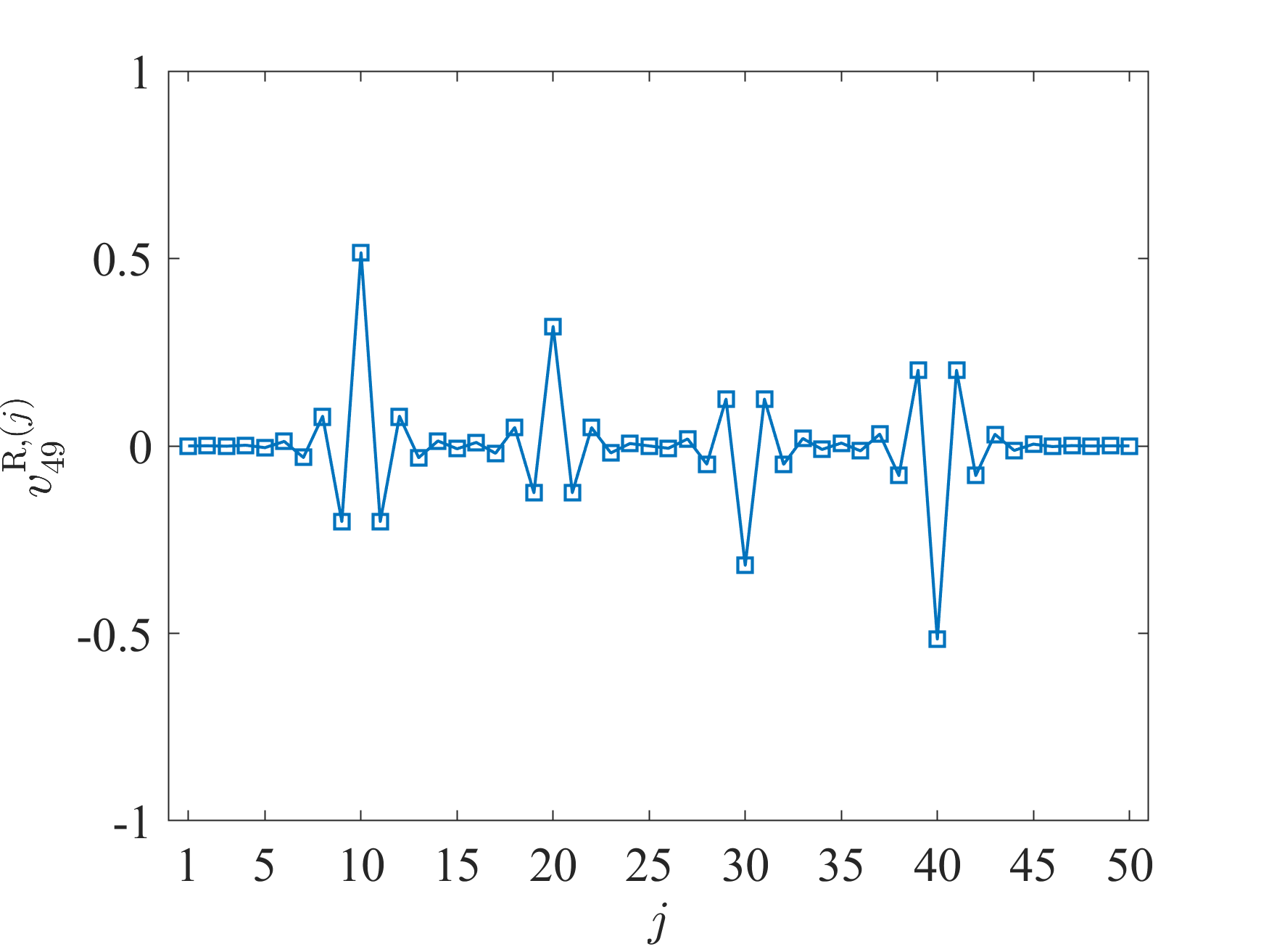}}
	\subfigure[]{
		\includegraphics[width=0.322\linewidth]{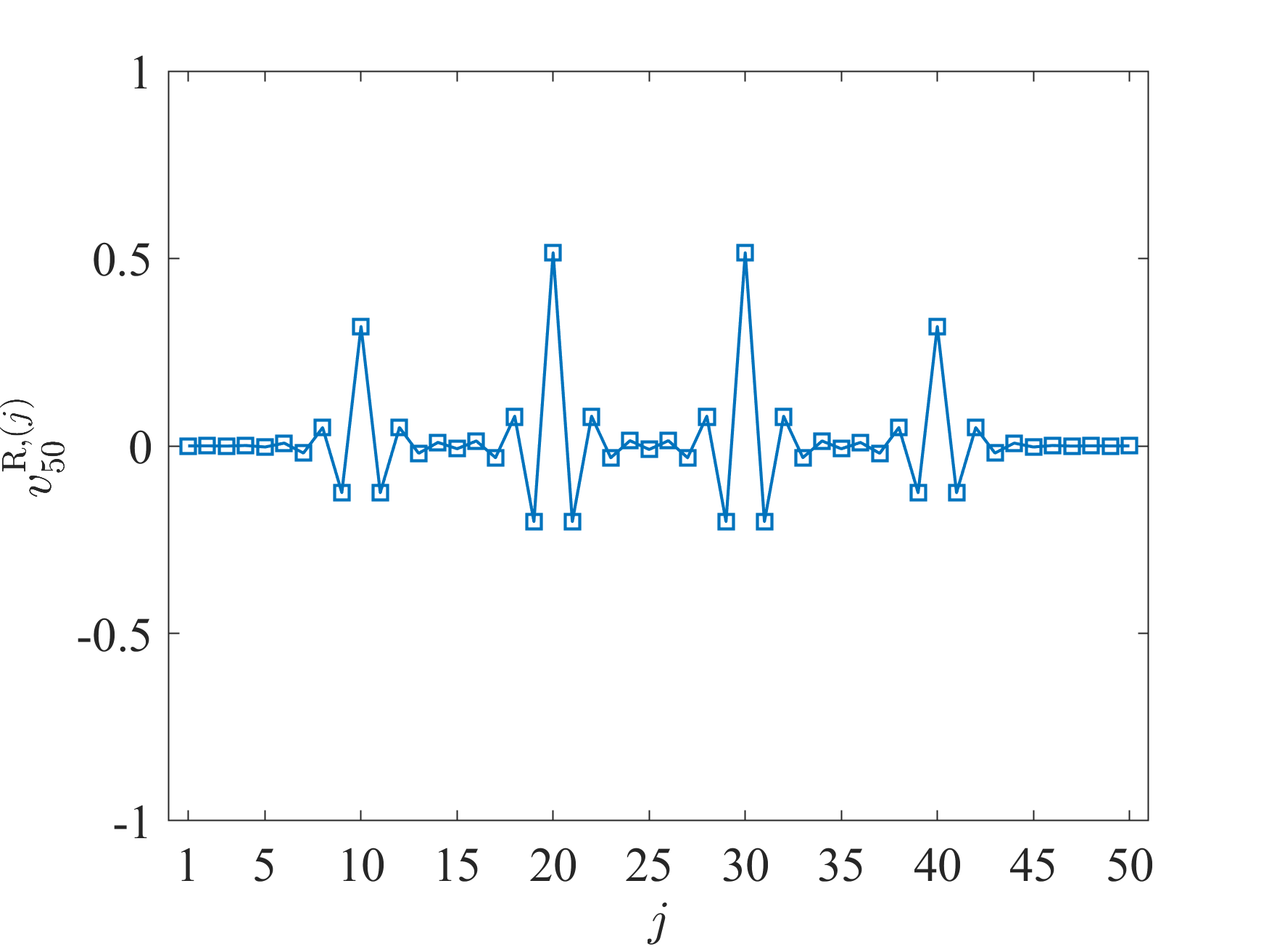}}
	\caption{Eigenpair profiles of the generalized elastic stiffness eigenvalue problem \eqref{GEP14} with $N=50$, $r_1^+=2.4$, $\beta^{\textup{R}} = 9500$ and $\{(i_l,\eta_l)\}_{l=1}^4=\{(10,-0.25),(20,-0.25),(30,-0.25),(40,-0.25)\}$. 
		(a) Distribution of eigenvalues, where black dots represent bulk eigenvalues and blue dots represent quasi-fourfold degenerate defect eigenvalues.
		(b) 
		Comparison between the asymptotic predictions \eqref{MLexp} for $\lambda^{\mathrm{m}}_{(i_l,\eta),d}$ and the numerically computed eigenvalues $\lambda^{\mathrm{R}}_{45+l}$ of \eqref{GEP14}.
		(c)--(f)  Exponentially localized eigenvectors associated with defect eigenvalues, showing localization around the corresponding defect sites, and exhibiting amplitude ratios consistent with the theoretical prediction \eqref{amplitude_ratio} for the even case $d=10$.
	}\label{Eigenvector_With_sameDefects}
\end{figure}

%Simultaneously Localized Defect Modes at Multiple Defect Sites

We next consider material-parameter defects with hierarchical configurations (see Remark \ref{rem53}).
We also set
$N=50$, $r_1^+ = 2.4$, $\beta^{\textup{R}} = 9500$, We further take  five material-parameter defects $\{(i_l,\eta_l)\}_{l=1}^5=\{(10,-0.2),(20,-0.2)\}\cup\{(32,-0.3),(37,-0.3),(42,-0.3)\}$, i.e., $b=2$, $M_1 = 2$, $M_2 = 3$, $d_1 = 10$, $d_2 = 5$,  $\eta^{(1)} = -0.2$,  and $\eta^{(2)} = -0.3$.

The spectral analysis of the generalized elastic stiffness eigenvalue problem \eqref{GEP14}, illustrated in Figure \ref{Eigenvector_With_same0304Defects}(a), clearly reveals the presence of several hierarchical quasi-degenerate eigenvalues located above the bulk spectrum. The number of such eigenvalues coincides with the number of defects in each block, and they are well separated from the bulk eigenvalues.  
We further compare the asymptotic predictions \eqref{MLexpblock} for $\lambda^{\mathrm{R}}_{(i_l,\eta^{(t)},d_t)}$, denoted by $\lambda^{\mathrm{R}}{(i_l,\eta_l)}$ and arranged in ascending order, with the numerically computed eigenvalues $\lambda^{\mathrm{R}}_{N-M+l}$ of \eqref{GEP14}. As shown in Table \ref{table-results2}, the two sets of results are in excellent agreement, thus confirming the validity and accuracy of the asymptotic expansions \eqref{MLexpblock}. The results indicate that the relative error decreases as the intra-block spacing  $d_t$ increases, while the defect eigenvalues grow with perturbation magnitude  $|\eta^{(t)}|$, in full agreement with the theoretical predictions.
Moreover, as displayed in Figure \ref{Eigenvector_With_sameDefects}(b)–(f), the amplitudes of the eigenvectors corresponding to the defect eigenvalues, evaluated at the defect sites, are given respectively by
\[
\(\bm{v}_{46}^{\mathrm{R},{(10s)}}\)_{1\leq s\leq 2} \approx  0.65 \(\sin\frac{2s\pi}{3}\)_{1\leq s\leq 2},\quad \(\bm{v}_{47}^{\mathrm{R},{(10s)}}\)_{1\leq s\leq 2} \approx  0.65\(\sin\frac{s\pi}{3}\)_{1\leq s\leq 2},
\]
\[\scalebox{0.95}{$\(\bm{v}_{48}^{\mathrm{R},{(27+5s)}}\)_{1\leq s\leq 3} \approx  0.64\(\sin\frac{s\pi}{4}\)_{1\leq s\leq 3},\, \(\bm{v}_{49}^{\mathrm{R},{(27+5s)}}\)_{1\leq s\leq 3} \approx  0.64\(\sin\frac{2s\pi}{4}\)_{1\leq s\leq 3},\, \(\bm{v}_{50}^{\mathrm{R},{(27+5s)}}\)_{1\leq s\leq 3} \approx  0.64\(\sin\frac{3s\pi}{4}\)_{1\leq s\leq 3}.$}\]
These numerical results are in excellent agreement with the theoretical amplitude ratio \eqref{amplitude_ratio} for the even case $d=10$ and the odd case $d=5$, respectively.

\begin{table}[h]% ed d=1
	\centering
	\begin{tabular}{cccc}
		\toprule
		$(i_l,\eta_l)$ & $\lambda^{\mathrm{R}}_{45+l}$ & $\lambda^{\mathrm{R}}_{(i_l,\eta_l)}$ & \text{Relative error} \\
		\midrule
		$(10,-0.2)$ & 43651.1317233221  & 43651.1636864129 & 7.3224$\times 10^{-7}$  \\
		$(20,-0.2)$ & 43664.9812623262   & 43665.0127841719   &  7.2190 $\times 10^{-7}$ \\
		$(32,-0.3)$ & 51169.2124137222 &51174.5453827585  & 1.0442$\times 10^{-4}$   \\	
		$(37,-0.3)$ & 51357.8239773711 &51358.2916610333 & 9.1064$\times 10^{-6}$  \\	
		$(42,-0.3)$ & 51537.1721098435 &51542.0379393080& 9.4414 $\times 10^{-5}$  \\
		\bottomrule
	\end{tabular}
	\caption{A comparison between the eigenvalues
		$\lambda^{\mathrm{R}}_{(i_l,\eta^{(t)},d_t)}$ (see \eqref{MLexpblock}), denoted by $\lambda^{\mathrm{R}}_{(i_l,\eta_l)}$ and arranged in ascending order,
		and the eigenvalues $\lambda^{\mathrm{R}}_{45+l}$  of the generalized elastic stiffness eigenvalue problem \eqref{GEP14} with $N=50$, $r_1^+ = 2.4$, and $\beta^{\textup{R}} = 9500$, over several values of $\{(i_l,\eta_l)\}_{l=1}^5=\{(10,-0.2),(20,-0.2)\}\cup\{(32,-0.3),(37,-0.3),(42,-0.3)\}$.}
	\label{table-results2}
\end{table}

This parity-dependent sinusoidal structure is not only of theoretical interest but also has important practical implications. In particular, the existence of multiple quasi-degenerate defect modes with distinct amplitude profiles enables multi-channel filtering, where different defect modes can selectively localize and guide waves at different spatial channels or frequency bands. Furthermore, in the presence of hierarchical configurations (as described in Remark \ref{rem53}), the superposition of such sinusoidal patterns across different defect groups leads to a rich multi-scale modal structure, offering a systematic way to design broadband and multi-functional wave manipulation devices.
\begin{figure}[H]
	\centering  %图片全局居中
	\subfigbottomskip=-5pt %两行子图之间的行间距
	\subfigcapskip=-5pt %设置子图与子标题之间的距离
	\subfigure[]{
		\includegraphics[width=0.322\linewidth]{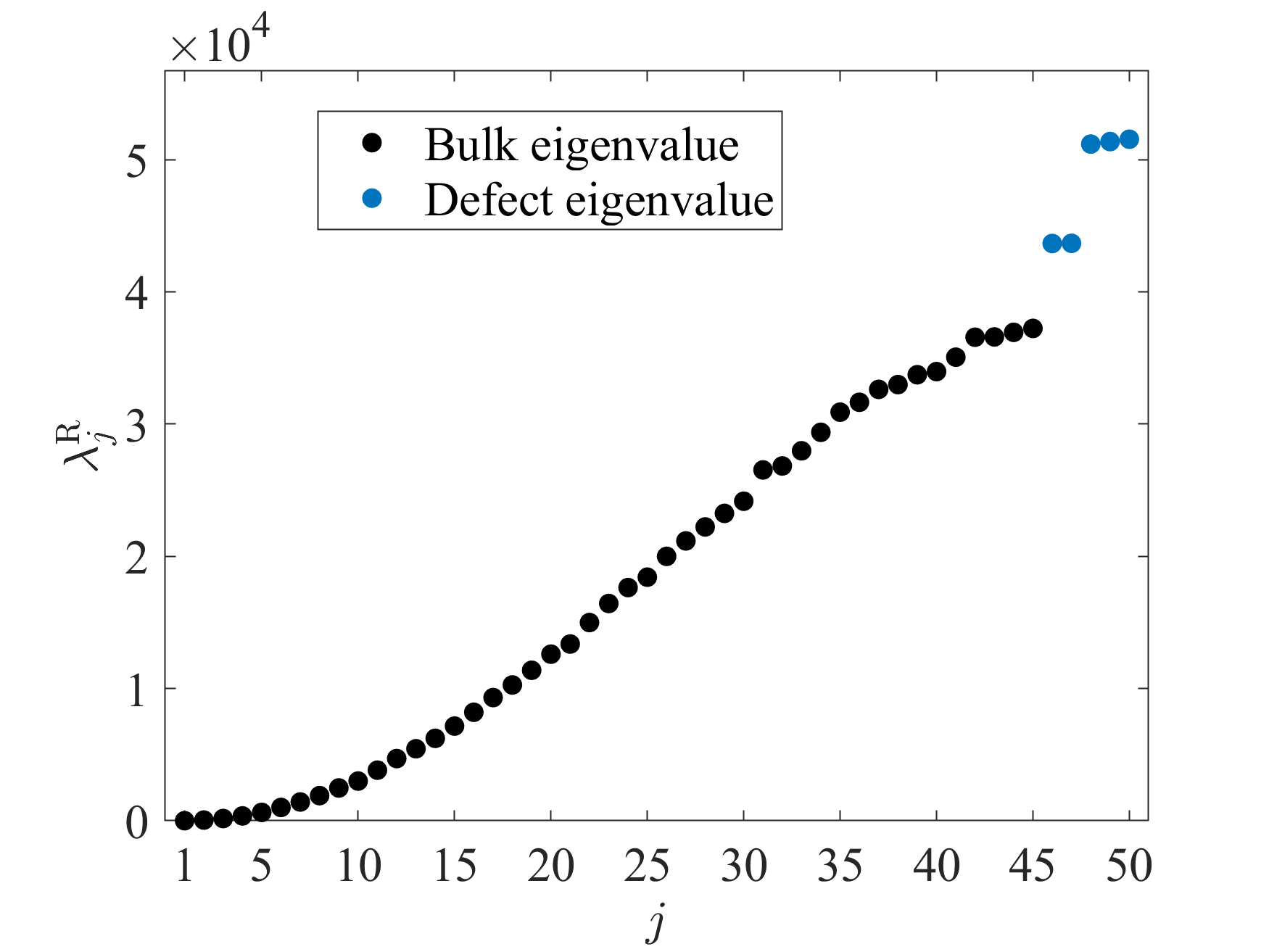}}
	\subfigure[]{
		\includegraphics[width=0.33\linewidth]{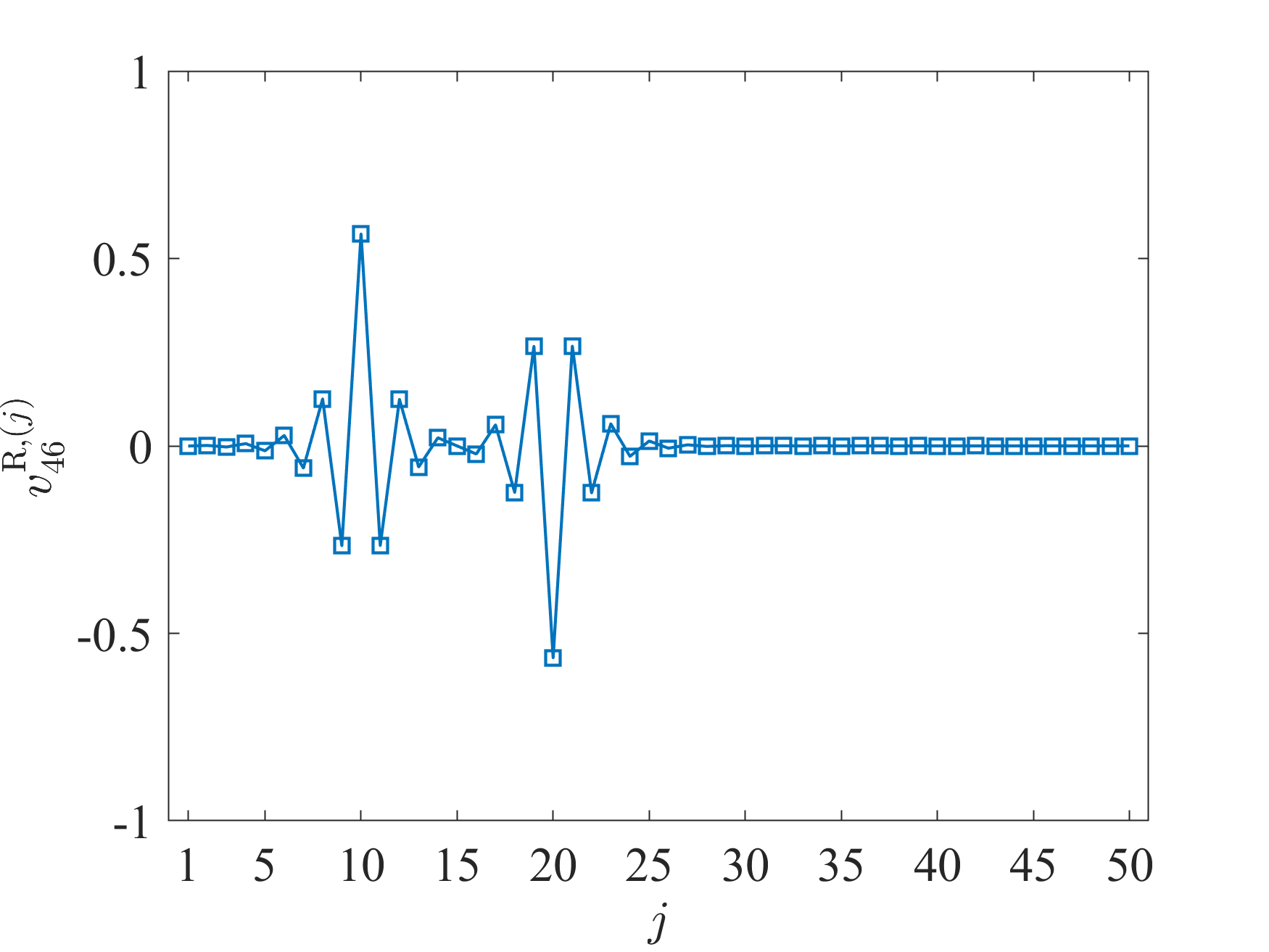}}
	\subfigure[]{
		\includegraphics[width=0.322\linewidth]{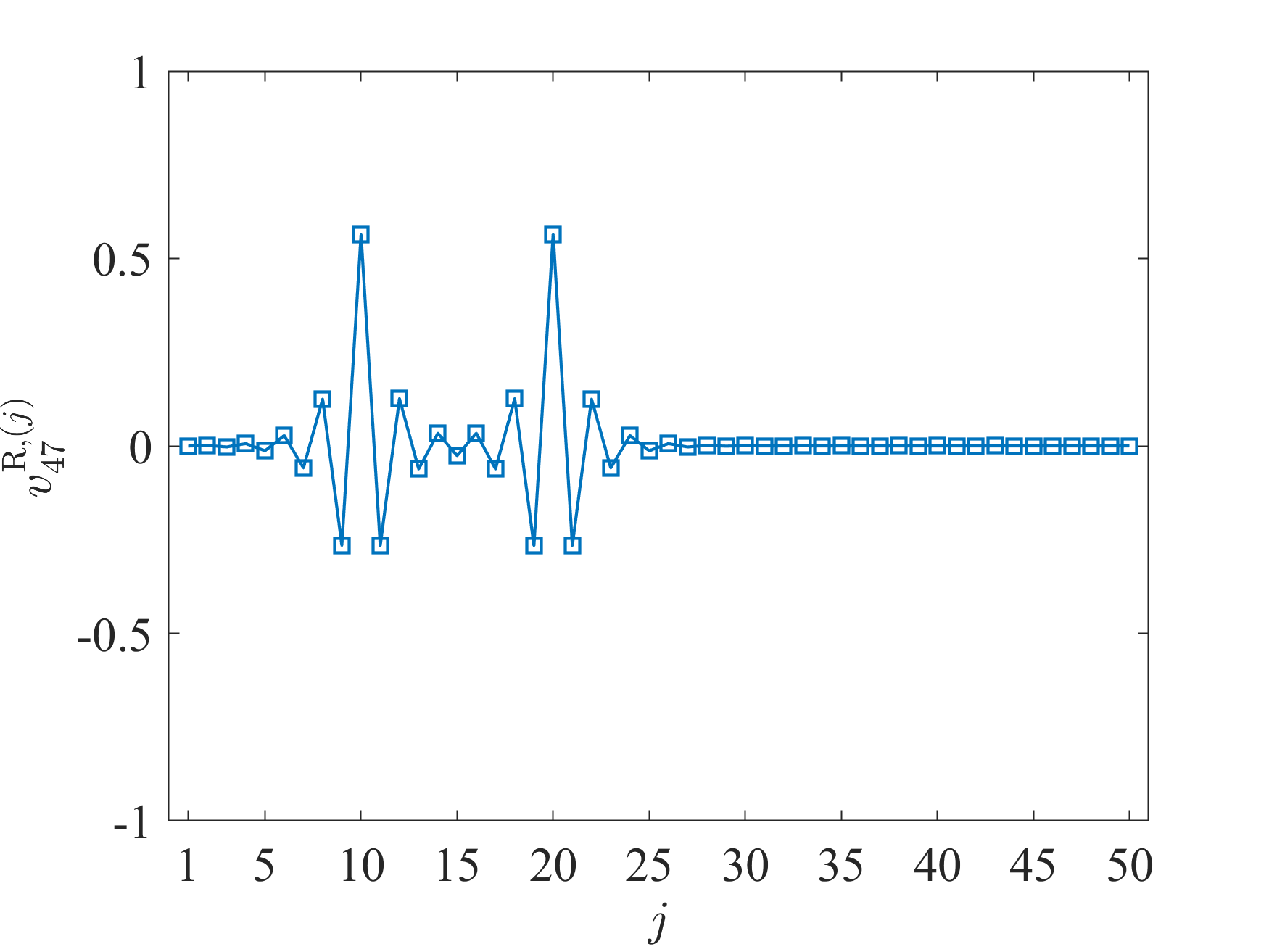}}
	%\quad
	\subfigure[]{
		\includegraphics[width=0.322\linewidth]{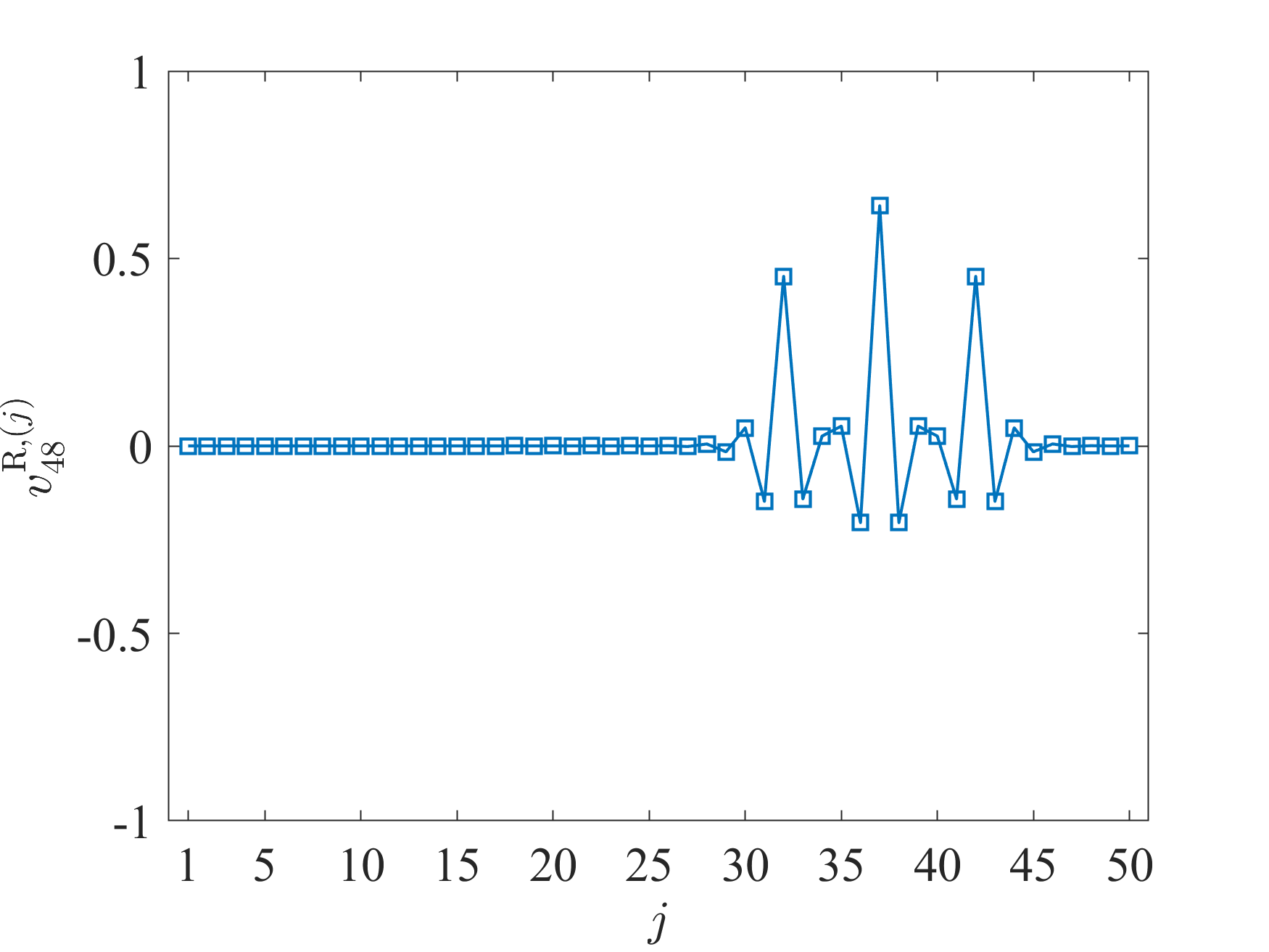}}
	\subfigure[]{
		\includegraphics[width=0.322\linewidth]{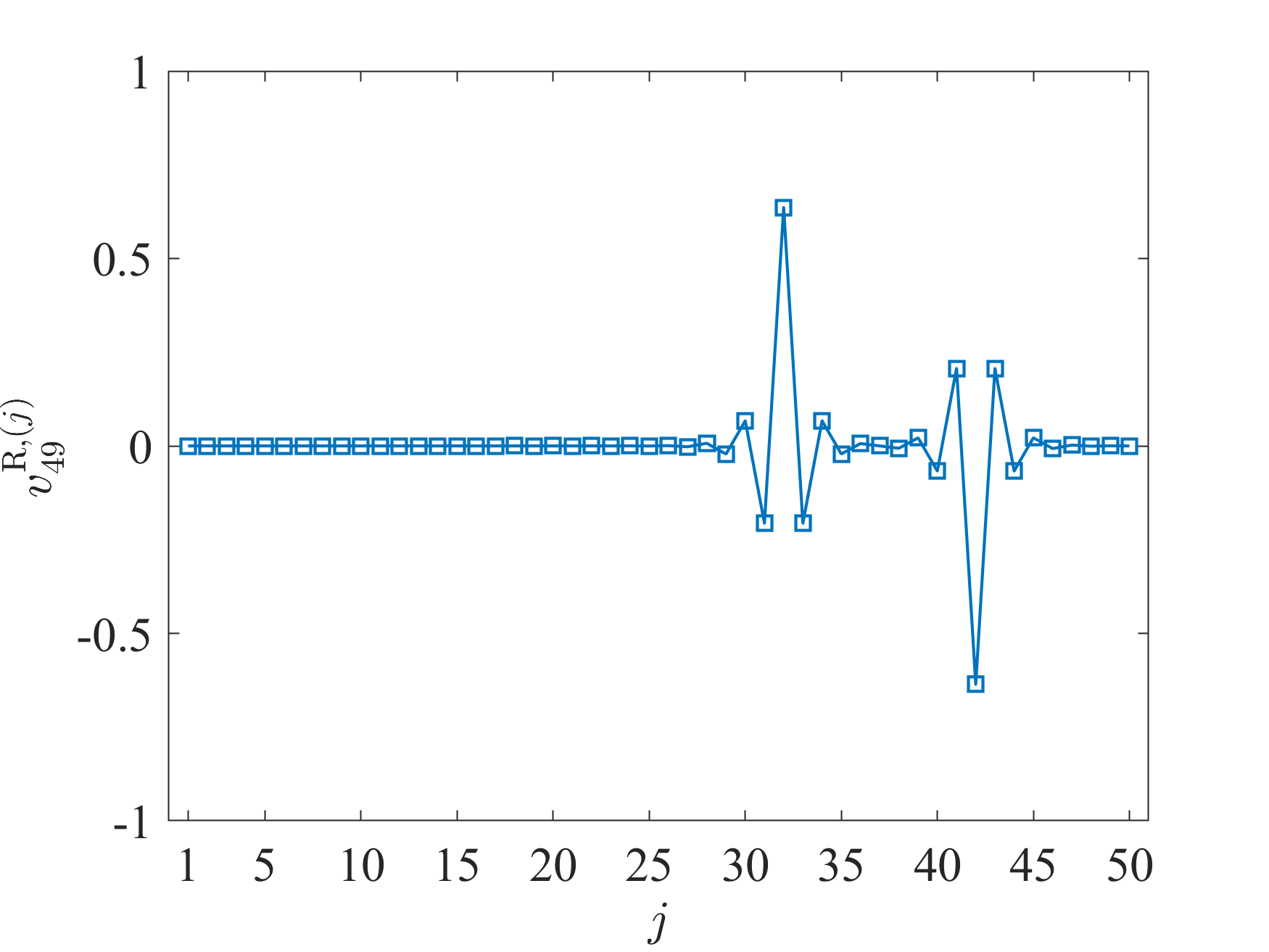}}
	\subfigure[]{
		\includegraphics[width=0.322\linewidth]{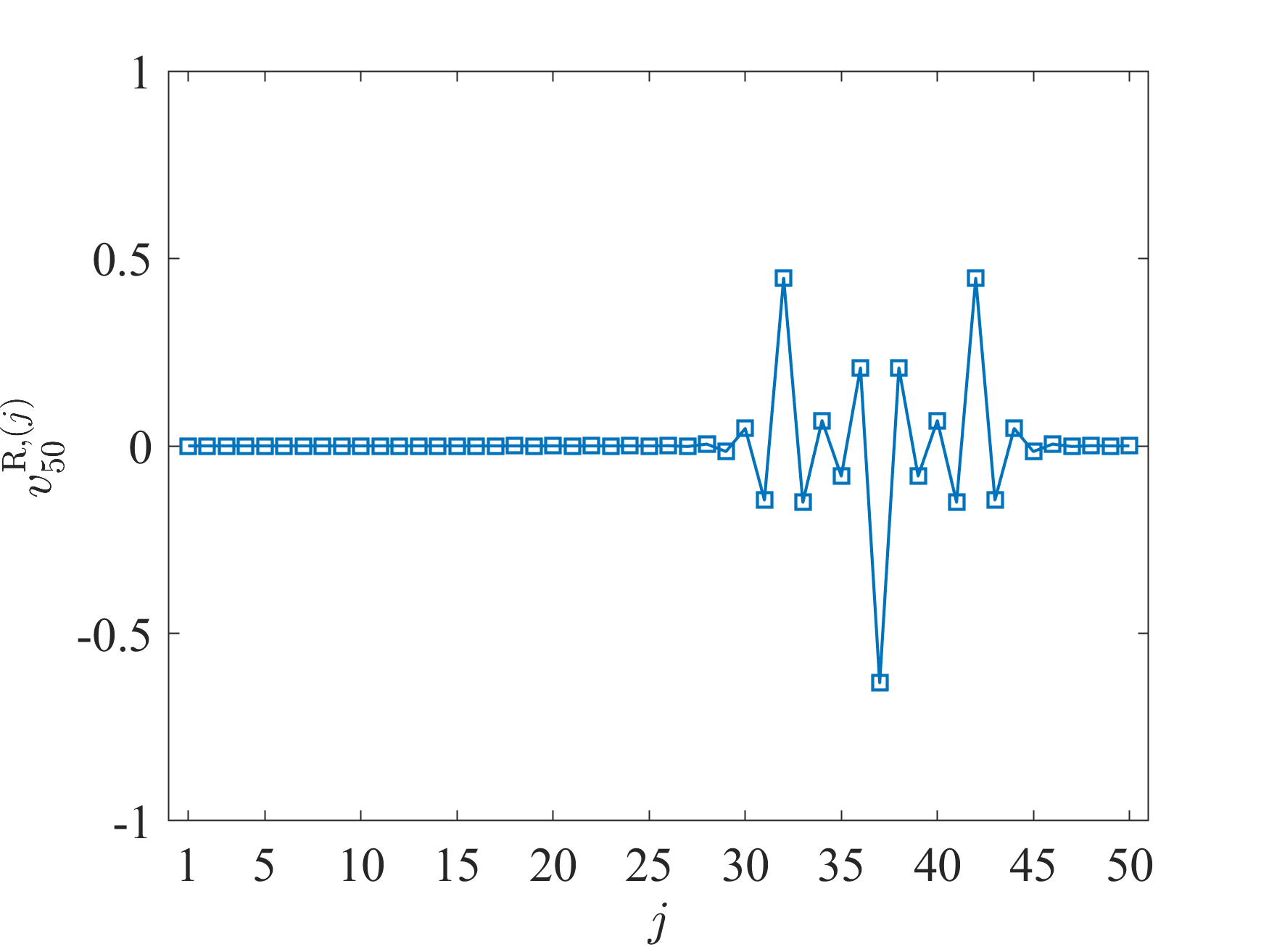}}
	\caption{Eigenpair profiles of the generalized elastic stiffness eigenvalue problem \eqref{GEP14} with $N=50$, $r_1^+=2.4$, $\beta^{\textup{R}} = 9500$ and $\{(i_l,\eta_l)\}_{l=1}^5=\{(10,-0.2),(20,-0.2)\}\cup\{(32,-0.3),(37,-0.3),(42,-0.3)\}$. 
		(a) Distribution of eigenvalues, where black dots represent bulk eigenvalues and blue dots represent quasi-fourfold degenerate defect eigenvalues.
		(b)--(f)  Exponentially localized eigenvectors associated with defect eigenvalues, showing localization around the corresponding defect sites, and exhibiting amplitude ratios consistent with the theoretical prediction \eqref{amplitude_ratio} for the even case $d=10$ and the odd case $d=5$, respectively. 
	}\label{Eigenvector_With_same0304Defects}
\end{figure}

\begin{figure}[H]
	\centering
	\includegraphics[scale=0.5]{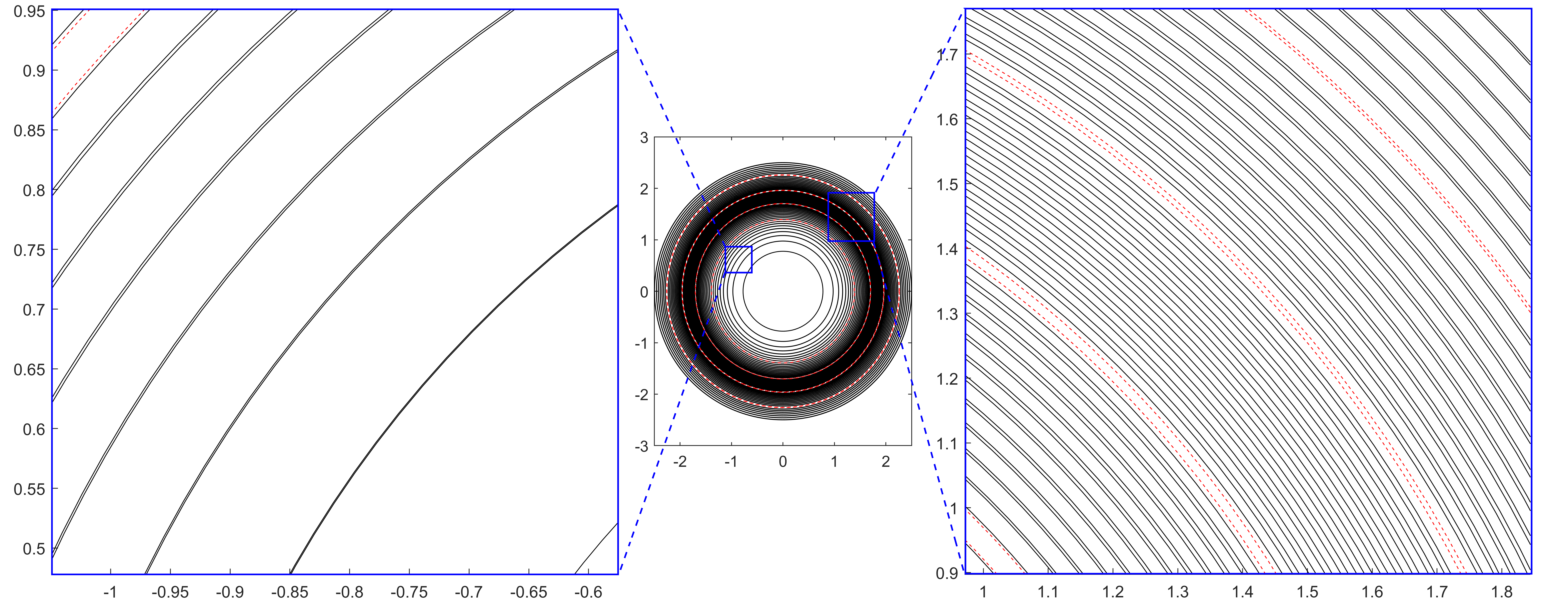}
	\caption{
		Schematic cross-section of the $51$-layered concentric sphere structure with $r_1^+=2.5$ and  $\beta^{\textup{R}} = 8800$. The dashed red lines indicate four ring defects $\{(i_l,\eta_l)\}_{l=1}^4=\{(8,-0.25),(20,-0.25), (32,-0.25),(44,-0.25)\}$.
	}\label{45LHCCBdefect}
\end{figure}

\begin{figure}[H]
	\centering  %图片全局居中
	\subfigbottomskip=-2pt %两行子图之间的行间距
	\subfigcapskip=-2pt %设置子图与子标题之间的距离
	\subfigure[]{
		\includegraphics[width=0.49\linewidth]{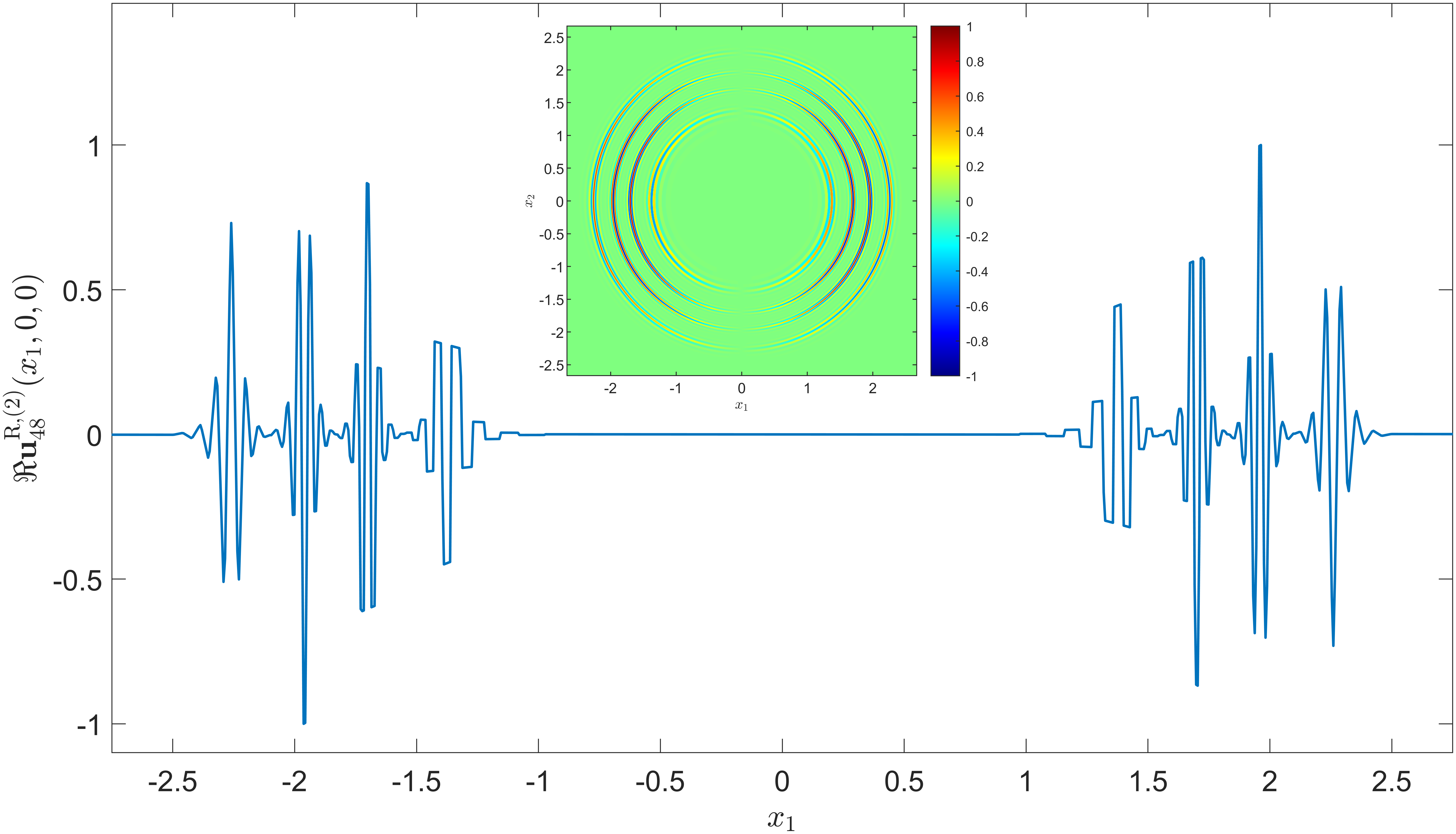}}
	\subfigure[]{
		\includegraphics[width=0.49\linewidth]{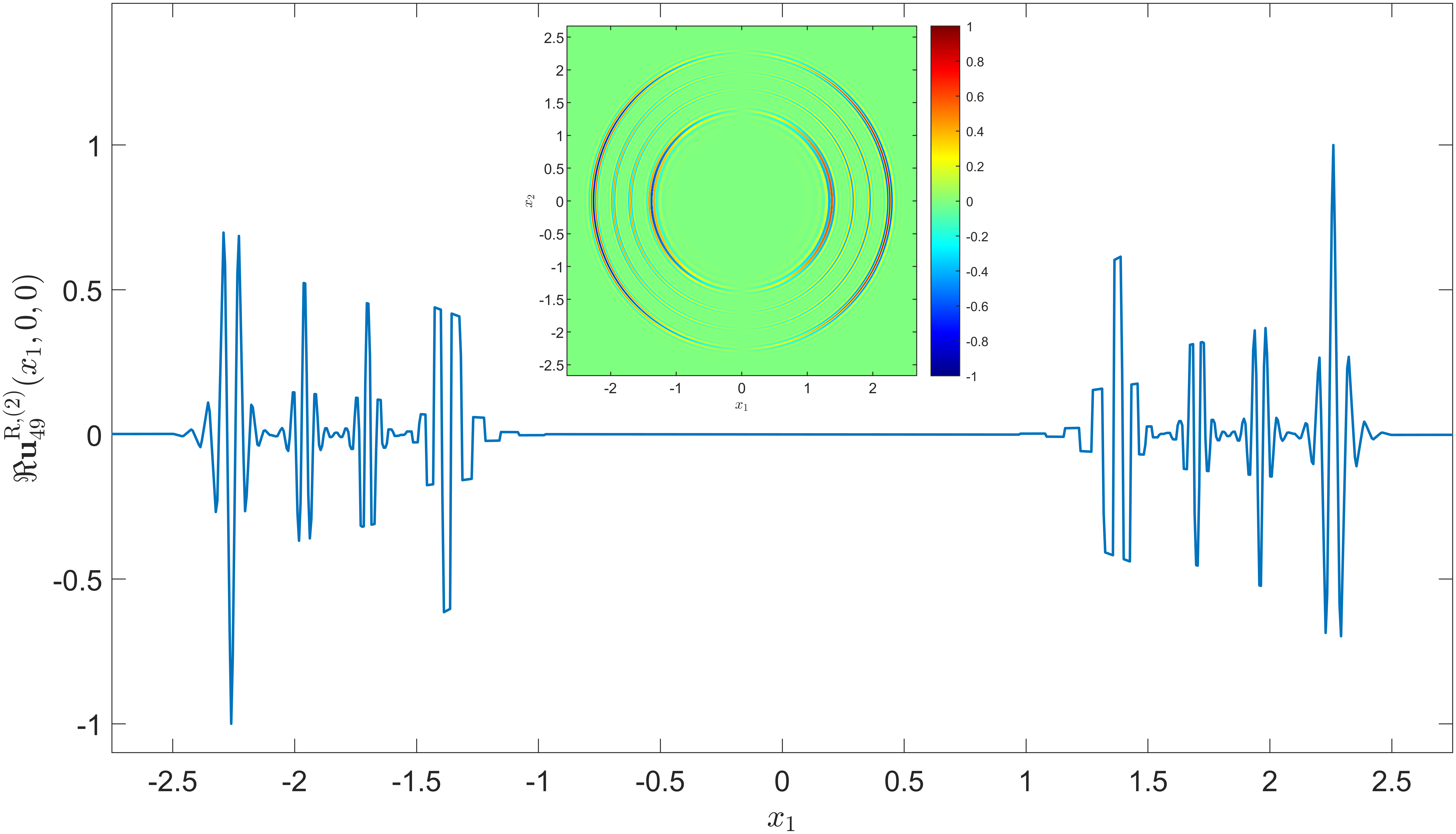}}
	\subfigure[]{
		\includegraphics[width=0.49\linewidth]{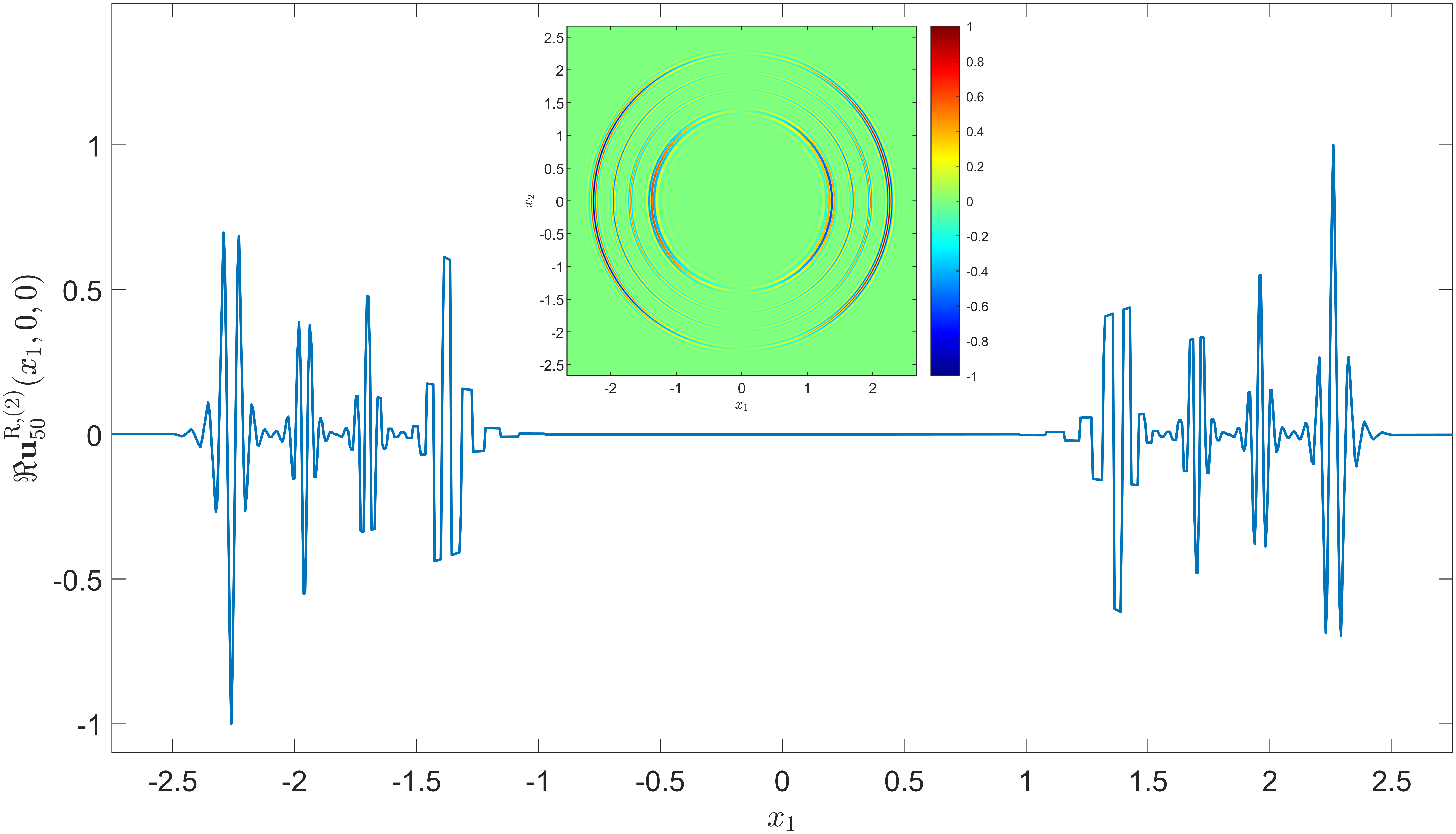}}
	\subfigure[]{
		\includegraphics[width=0.49\linewidth]{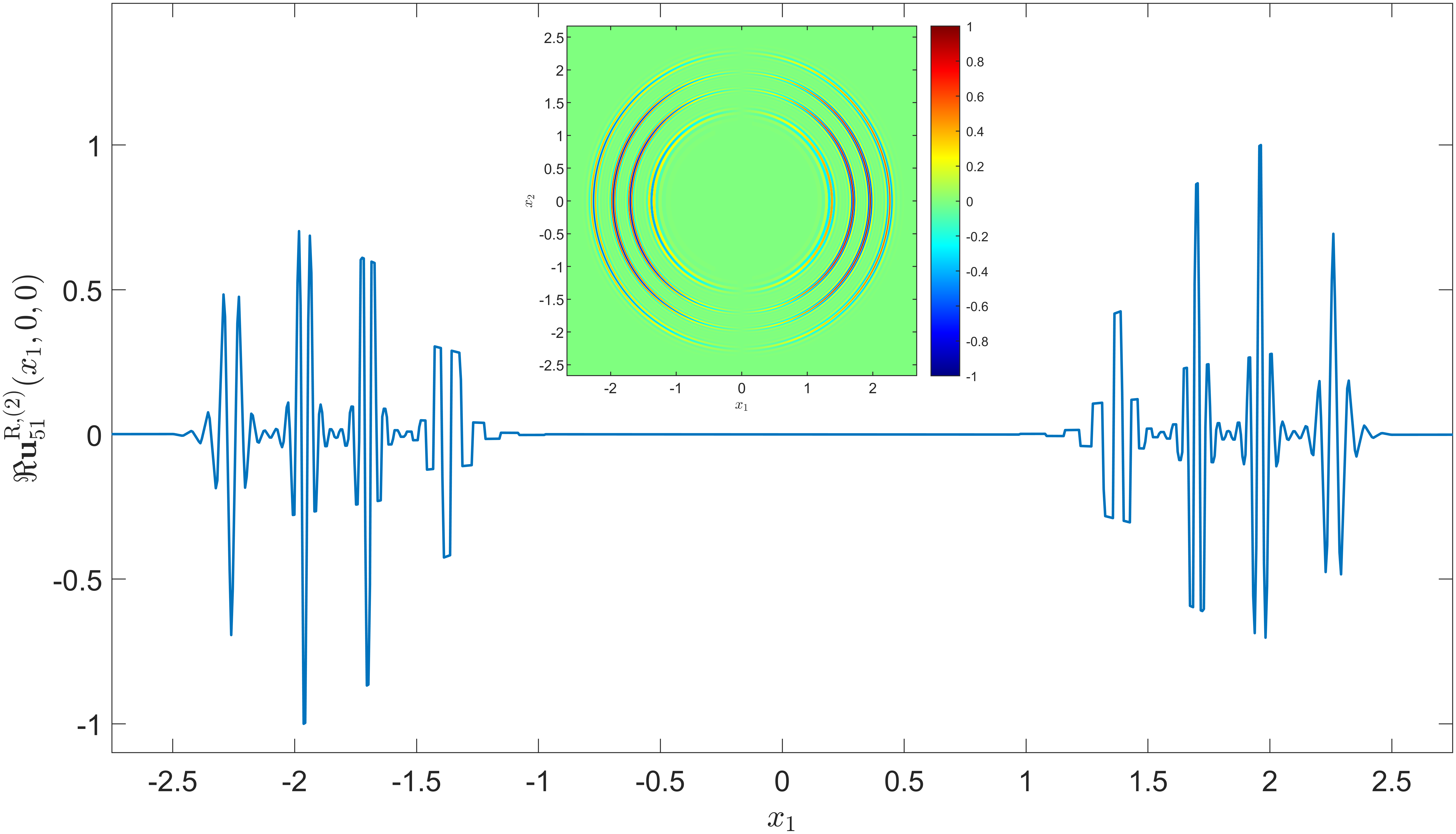}}
	\caption{The second components of the rotational elastic displacement field distributions $\mathbf{u}_{48}^{\mathrm{R},(2)},\mathbf{u}_{49}^{\mathrm{R},(2)},\mathbf{u}_{50}^{\mathrm{R},(2)},\mathbf{u}_{51}^{\mathrm{R},(2)}$, along
		the $x_1$ axis, for the same 51-layered concentric sphere as in Figure \ref{45LHCCBdefect}.
		Each plot corresponds to one of the four eigenfrequencies $\omega^{\mathrm{R},+}_{j}$, $j=48,49,50,51$. 
		The inset displays a contour plot of the function $\Re \mathbf{u}_j^{\mathrm{R},(2)}(x_1, x_2,0)$.
	}\label{Eigenmode_With_same3eta025Defects}
\end{figure}

Finally, we present numerical simulations of localized defect modes.  To facilitate visualization of the results, we consider a  51-layered monomer-type metamaterial with four defects
\[\{(i_l,\eta_l)\}_{l=1}^4=\{(8,-0.25),(20,-0.25), (32,-0.25),(44,-0.25)\},\] under the assumptions \eqref{interface}, \eqref{identicalvolume}, and \eqref{monomersetupdefect1} with $r_1^+=2.5$ and  $\beta^{\textup{R}} = 8800$.
A schematic cross section of this structure is shown in Figure \ref{45LHCCBdefect}, where the
dashed red lines indicate the four ring defects located at $D_8 = \{2.257<|\Bx|<2.262\}$, $D_{20} = \{1.956<|\Bx|<1.964\}$, $D_{32} = \{1.693<|\Bx|<1.707\}$, and $D_{44} = \{1.359<|\Bx|<1.393\}$. Figure \ref{Eigenmode_With_same3eta025Defects} shows the localized defect mode for the same structure as in Figure \ref{45LHCCBdefect}.
The figure clearly shows that the eigenmodes corresponding to the eigenfrequencies above the bulk spectrum are localized to the ring defects and quickly decay in both directions away from the ring defects. 
Specifically,
the amplitudes of the eigenvectors corresponding to the defect eigenvalues, evaluated at the defect sites, are in excellent agreement with the theoretical amplitude ratio \eqref{amplitude_ratio} for the even case $d=12$, and are given respectively by
\[\(\bm{v}_{48}^{\mathrm{R},{(12s-4)}}\)_{1\leq s\leq 4} \approx  0.55\(\sin\frac{4s\pi}{5}\)_{1\leq s\leq 4},\;\(\bm{v}_{49}^{\mathrm{R},{(12s-4)}}\)_{1\leq s\leq 4} \approx -0.55\(\sin\frac{3s\pi}{5}\)_{1\leq s\leq 4},\]
\[\(\bm{v}_{50}^{\mathrm{R},{(12s-4)}}\)_{1\leq s\leq 4} \approx  0.55\(\sin\frac{2s\pi}{5}\)_{1\leq s\leq 4},\;
\(\bm{v}_{51}^{\mathrm{R},{(12s-4)}}\)_{1\leq s\leq 4} \approx  0.55\(\sin\frac{s\pi}{5}\)_{1\leq s\leq 4}.
\]
Let $r_s = (r_s^- + r_s^+)/2$  and 
$ \mathbf{u}_j^{\mathrm{R},(2)}|_{D_s}:=\Re \mathbf{u}_j^{\mathrm{R},(2)}(r_s,0,0).
$
The amplitudes of localized defect modes, evaluated at the defect sites, are given respectively by
\[
\(\mathbf{u}_{48}^{\mathrm{R},(2)}|_{D_{12s-4}}\)_{1\leq s\leq 4} \approx 0.86\(r_s\bm{v}_{48}^{\mathrm{R},{(12s-4)}}\)_{1\leq s\leq 4},\;\(\mathbf{u}_{49}^{\mathrm{R},(2)}|_{D_{12s-4}}\)_{1\leq s\leq 4} \approx 0.86\(r_s\bm{v}_{49}^{\mathrm{R},{(12s-4)}}\)_{1\leq s\leq 4},\] \[\(\mathbf{u}_{50}^{\mathrm{R},(2)}|_{D_{12s-4}}\)_{1\leq s\leq 4} \approx 0.86\(r_s\bm{v}_{50}^{\mathrm{R},{(12s-4)}}\)_{1\leq s\leq 4}\;\(\mathbf{u}_{51}^{\mathrm{R},(2)}|_{D_{12s-4}}\)_{1\leq s\leq 4} \approx 0.86\(r_s\bm{v}_{51}^{\mathrm{R},{(12s-4)}}\)_{1\leq s\leq 4}.
\]
These numerical results show excellent
agreement with the theoretical predictions of Proposition \ref{propMLCB}.
This  dipole-type localized mode with an antisymmetric profile may have relevant applications in directional sensors.

\section{Concluding remarks}\label{sec6}

In this work, we established a rigorous framework for the analysis of subwavelength resonances and seismic-wave localization in Matryoshka-type elastic metamaterials. By combining the displacement-to-traction map, variational methods, with the Gohberg--Sigal theory, we proved the existence and asymptotic behavior of subwavelength resonances through the eigensystem of a generalized stiffness tensor, which serves as the elastic analogue of the generalized capacitance matrix in Minnaert acoustic-cavitation systems. Explicit formulas for concentric radial resonators confirmed the general theory and revealed a block-diagonal tensor structure separating translational and rotational modes. Based on this structure, we show that strategically placed ring defects can create nearly degenerate eigenfrequencies above the bulk spectrum, whose corresponding eigenmodes are exponentially localized simultaneously at multiple defects. This enables multimode localization in  radially laminated architecture, effectively realizing a tunable seismic waveguide that supports both localized trapping and guided propagation along prescribed defect paths. These results provide a mathematical foundation for designing elastic metamaterials capable of controlling destructive low-frequency seismic waves.

\section*{Acknowledgment}
The research was done when L. Kong visited School of Physical and Mathematical Sciences, Nanyang Technological University at Singapore under the support of China Scholarship Council (No. 202406370156), and L. Kong thanks Professor Li-Lian Wang for his invitation and School of Physical and Mathematical Sciences, NTU for their support and kind hospitality. 

\section*{Conflict of interest}
The author declares no conflicts of interest.

\section*{Data availability}
No data was used for the research described in the article.

\end{document}